\documentclass[10pt]{amsart}
\usepackage[utf8]{inputenc}

\usepackage{etoolbox}
\usepackage{comment}

\makeatletter
\let\old@tocline\@tocline
\let\section@tocline\@tocline
\newcommand{\subsection@dotsep}{4.5}
\newcommand{\subsubsection@dotsep}{4.5}
\patchcmd{\@tocline}
  {\hfil}
  {\nobreak
     \leaders\hbox{$\m@th
        \mkern \subsection@dotsep mu\hbox{.}\mkern \subsection@dotsep mu$}\hfill
     \nobreak}{}{}
\let\subsection@tocline\@tocline
\let\@tocline\old@tocline

\patchcmd{\@tocline}
  {\hfil}
  {\nobreak
     \leaders\hbox{$\m@th
        \mkern \subsubsection@dotsep mu\hbox{.}\mkern \subsubsection@dotsep mu$}\hfill
     \nobreak}{}{}
\let\subsubsection@tocline\@tocline
\let\@tocline\old@tocline

\let\old@l@subsection\l@subsection
\let\old@l@subsubsection\l@subsubsection

\def\@tocwriteb#1#2#3{%
  \begingroup
    \@xp\def\csname #2@tocline\endcsname##1##2##3##4##5##6{%
      \ifnum##1>\c@tocdepth
      \else \sbox\z@{##5\let\indentlabel\@tochangmeasure##6}\fi}%
    \csname l@#2\endcsname{#1{\csname#2name\endcsname}{\@secnumber}{}}%
  \endgroup
  \addcontentsline{toc}{#2}%
    {\protect#1{\csname#2name\endcsname}{\@secnumber}{#3}}}%

\newlength{\@tocsectionindent}
\newlength{\@tocsubsectionindent}
\newlength{\@tocsubsubsectionindent}
\newlength{\@tocsectionnumwidth}
\newlength{\@tocsubsectionnumwidth}
\newlength{\@tocsubsubsectionnumwidth}
\newcommand{\settocsectionnumwidth}[1]{\setlength{\@tocsectionnumwidth}{#1}}
\newcommand{\settocsubsectionnumwidth}[1]{\setlength{\@tocsubsectionnumwidth}{#1}}
\newcommand{\settocsubsubsectionnumwidth}[1]{\setlength{\@tocsubsubsectionnumwidth}{#1}}
\newcommand{\settocsectionindent}[1]{\setlength{\@tocsectionindent}{#1}}
\newcommand{\settocsubsectionindent}[1]{\setlength{\@tocsubsectionindent}{#1}}
\newcommand{\settocsubsubsectionindent}[1]{\setlength{\@tocsubsubsectionindent}{#1}}

\renewcommand{\l@section}{\section@tocline{1}{\@tocsectionvskip}{\@tocsectionindent}{}{\@tocsectionformat}}%
\renewcommand{\l@subsection}{\subsection@tocline{2}{\@tocsubsectionvskip}{\@tocsubsectionindent}{}{\@tocsubsectionformat}}%
\renewcommand{\l@subsubsection}{\subsubsection@tocline{3}{\@tocsubsubsectionvskip}{\@tocsubsubsectionindent}{}{\@tocsubsubsectionformat}}%
\newcommand{\@tocsectionformat}{}
\newcommand{\@tocsubsectionformat}{}
\newcommand{\@tocsubsubsectionformat}{}
\expandafter\def\csname toc@1format\endcsname{\@tocsectionformat}
\expandafter\def\csname toc@2format\endcsname{\@tocsubsectionformat}
\expandafter\def\csname toc@3format\endcsname{\@tocsubsubsectionformat}
\newcommand{\settocsectionformat}[1]{\renewcommand{\@tocsectionformat}{#1}}
\newcommand{\settocsubsectionformat}[1]{\renewcommand{\@tocsubsectionformat}{#1}}
\newcommand{\settocsubsubsectionformat}[1]{\renewcommand{\@tocsubsubsectionformat}{#1}}
\newlength{\@tocsectionvskip}
\newcommand{\settocsectionvskip}[1]{\setlength{\@tocsectionvskip}{#1}}
\newlength{\@tocsubsectionvskip}
\newcommand{\settocsubsectionvskip}[1]{\setlength{\@tocsubsectionvskip}{#1}}
\newlength{\@tocsubsubsectionvskip}
\newcommand{\settocsubsubsectionvskip}[1]{\setlength{\@tocsubsubsectionvskip}{#1}}

\patchcmd{\tocsection}{\indentlabel}{\makebox[\@tocsectionnumwidth][l]}{}{}
\patchcmd{\tocsubsection}{\indentlabel}{\makebox[\@tocsubsectionnumwidth][l]}{}{}
\patchcmd{\tocsubsubsection}{\indentlabel}{\makebox[\@tocsubsubsectionnumwidth][l]}{}{}

\newcommand{\@sectypepnumformat}{}
\renewcommand{\contentsline}[1]{%
  \expandafter\let\expandafter\@sectypepnumformat\csname @toc#1pnumformat\endcsname%
  \csname l@#1\endcsname}
\newcommand{\@tocsectionpnumformat}{}
\newcommand{\@tocsubsectionpnumformat}{}
\newcommand{\@tocsubsubsectionpnumformat}{}
\newcommand{\setsectionpnumformat}[1]{\renewcommand{\@tocsectionpnumformat}{#1}}
\newcommand{\setsubsectionpnumformat}[1]{\renewcommand{\@tocsubsectionpnumformat}{#1}}
\newcommand{\setsubsubsectionpnumformat}[1]{\renewcommand{\@tocsubsubsectionpnumformat}{#1}}
\renewcommand{\@tocpagenum}[1]{%
  \hfill {\mdseries\@sectypepnumformat #1}}

\let\oldappendix\appendix
\renewcommand{\appendix}{%
  \leavevmode\oldappendix%
  \addtocontents{toc}{%
    \protect\settowidth{\protect\@tocsectionnumwidth}{\protect\@tocsectionformat\sectionname\space}%
    \protect\addtolength{\protect\@tocsectionnumwidth}{2em}}%
}
\makeatother

\makeatletter
\settocsectionnumwidth{2em}
\settocsubsectionnumwidth{2.5em}
\settocsubsubsectionnumwidth{3em}
\settocsectionindent{1pc}%
\settocsubsectionindent{\dimexpr\@tocsectionindent+\@tocsectionnumwidth}%
\settocsubsubsectionindent{\dimexpr\@tocsubsectionindent+\@tocsubsectionnumwidth}%
\makeatother

\settocsectionvskip{10pt}
\settocsubsectionvskip{0pt}
\settocsubsubsectionvskip{0pt}
    
\settocsectionformat{\bfseries}
\settocsubsectionformat{\mdseries}
\settocsubsubsectionformat{\mdseries}
\setsectionpnumformat{\bfseries}
\setsubsectionpnumformat{\mdseries}
\setsubsubsectionpnumformat{\mdseries}

\let\oldtableofcontents\tableofcontents
\renewcommand{\tableofcontents}{%
  \vspace*{-\linespacing}
  \oldtableofcontents}

\usepackage{amsmath,amssymb,amsxtra,anysize,hyperref}
\usepackage{graphics,graphicx}

\usepackage{enumitem}
\usepackage{tikz}
\usepackage{tikz-cd}

\newcommand{\s}{\mathbf{s}}
\newcommand{\R}{\mathbb{R}}
\newcommand{\C}{\mathbb{C}}
\newcommand{\Z}{\mathbb{Z}}
\newcommand{\Q}{\mathbb{Q}}
\newcommand{\La}{\Lambda}
\newcommand{\vt}{\vartheta}
\newcommand{\sse}{\subset}

\newcommand{\up}[1]{\mathrm{up}(#1)}
\newcommand{\cluster}[1]{\mathcal{A}(#1)}

\DeclareMathOperator{\sign}{sign}
\DeclareMathOperator{\Conf}{Conf}
\DeclareMathOperator{\Spec}{Spec}

\DeclareMathOperator{\uf}{uf}
\DeclareMathOperator{\DT}{DT}

\newcommand{\fM}{\mathfrak{M}}
\newcommand{\fW}{\mathfrak{W}}
\newcommand{\fD}{\mathfrak{D}}
\newcommand{\tz}{\widetilde{z}}

\newcommand{\rind}[1]{\overrightarrow{\mathfrak{w}}(#1)}
\newcommand{\lind}[1]{\overleftarrow{\mathfrak{w}}(#1)}

\DeclareMathOperator{\Br}{Br}
\DeclareMathOperator{\SL}{SL}
\newcommand{\borel}[0]{\mathsf{B}}
\newcommand{\uni}[0]{\mathsf{U}}
\newcommand{\si}[0]{\sigma}
\newcommand{\dynkin}[0]{\mathsf{D}}
\newcommand{\G}[0]{\mathsf{G}}
\newcommand{\g}[0]{\gamma}

\newcommand{\Pic}{\mathrm{Pic}}

\newcommand{\Tr}{\mathrm{Tr}}

\newcommand{\wt}{\widetilde}
\newcommand{\dd}{\partial}

\newcommand{\CA}{\mathcal{A}}
\newcommand{\CX}{\mathcal{X}}

\newcommand{\n}{\mathsf{n}}
\newcommand{\so}{\mathsf{s}}
\newcommand{\e}{\mathsf{e}}
\newcommand{\w}{\mathsf{w}}
\newcommand{\nw}{\mathsf{nw}}
\newcommand{\sw}{\mathsf{sw}}
\newcommand{\neast}{\mathsf{ne}}
\newcommand{\se}{\mathsf{se}}

\DeclareRobustCommand{\ateb}{\text{\reflectbox{$\beta$}}}

\newtheorem{theorem}{Theorem}[section]

\newtheorem{corollary}[theorem]{Corollary} 
\newtheorem{lemma}[theorem]{Lemma} 
\newtheorem{proposition}[theorem]{Proposition} 
\newtheorem{example}[theorem]{Example}

\newtheorem{remark}[theorem]{Remark}  
\newtheorem{definition}[theorem]{Definition}

\author{Roger Casals}
\address{University of California Davis, Dept.~of Mathematics, USA}
\email{casals@math.ucdavis.edu}

\author{Eugene Gorsky}
\address{University of California Davis, Dept.~of Mathematics, USA}
\email{egorskiy@math.ucdavis.edu}
	
\author{Mikhail Gorsky}
\address{Fakult\"at f\"ur Mathematik, Universit\"at Wien, Oskar-Morgenstern-Platz 1, 1090 Wien, Austria}
\email{mikhail.gorskii@univie.ac.at}

\author{Ian Le}
\address{Australian National University, Mathematical Sciences Institute, Australia}
\email{ian.le@anu.edu.au}

\author{Linhui Shen}
\address{Michigan State University, Dept.~of Mathematics, USA}
\email{linhui@math.msu.edu}

\author{Jos\'e Simental}
\address{Instituto de Matem\'aticas, Universidad Nacional Aut\'onoma de M\'exico, M\'exico}
\email{simental@im.unam.mx}

\title{Cluster structures on braid varieties}

\begin{document}

\begin{abstract}
We show the existence of cluster $\mathcal{A}$-structures and cluster Poisson structures on any braid variety, for any simple Lie group. The construction is achieved via weave calculus and a tropicalization of Lusztig's coordinates. Several explicit seeds are provided and the quiver and cluster variables are readily computable. We prove that these upper cluster algebras equal their cluster algebras, show local acyclicity, and explicitly determine their DT-transformations as the twist automorphisms of braid varieties. The main result also resolves the conjecture of B. Leclerc on the existence of cluster algebra structures on the coordinate rings of open Richardson varieties.
\end{abstract}

\maketitle
\setcounter{tocdepth}{2}
\tableofcontents


\section{Introduction}

The object of this article will be to show the existence of cluster $\mathrm{K}_2$-structures and cluster Poisson structures on braid varieties for any simple algebraic Lie group. The construction of such cluster structures is achieved via the study of Demazure weaves and their cycles. The initial seed is explicitly obtained by using weaves and a tropicalization of Lie group identities in Lusztig's coordinates, yielding both a readily computable exchange matrix and an initial set of cluster $\mathcal{A}$-variables. In particular, a conjecture of B.~Leclerc on open Richardson varieties is resolved. We also establish general properties of these cluster structures for braid varieties, including local acyclicity and the explicit construction of a Donaldson-Thomas transformation.

\subsection{Scientific Context} 
Cluster algebras, introduced by S.~Fomin and A.~Zelevinsky \cite{BFZ,FZcluster,FominZelevinsky_ClusterII} in the study of Lie groups, are commutative rings endowed with a set of distinguished generators satisfying remarkable combinatorial and geometric properties. Cluster varieties, a geometric enrichment of cluster algebras introduced by V.~Fock and A.~Goncharov \cite{FockGoncharovII,FockGoncharov_ModuliLocSys,FockGoncharov_ensemble}, are algebraic varieties equipped with an atlas of toric charts whose transition maps obey certain combinatorial rules, closely related to the rules of mutation in a cluster algebra. Cluster varieties come in pairs consisting of a cluster $\mathrm{K}_2$-variety, also known as a cluster $\mathcal{A}$-variety, and a cluster Poisson variety, also known as a cluster $\mathcal{X}$-variety. In particular, the coordinate ring of a cluster $\mathcal{A}$-variety coincides with an upper cluster algebra \cite{BFZ}.\\

The existence of a cluster structure on an algebraic variety has consequences for its geometry, including the existence of a canonical holomorphic $2$-form \cite{GSV}, canonical bases on its algebra of regular functions, and the splitting of the mixed Hodge structure on its cohomology \cite{LamSpeyer}. A wealth of Lie-theoretic varieties have been shown to admit cluster structures, including the affine cones over partial flag varieties of a simply connected Lie group, double Bott-Samelson varieties generalizing double Bruhat cells, and open positroid varieties, see \cite{BFZ,FZcluster,GLpositroid,GLSpartial,Scott, serhiyenko2019cluster, SW} and references therein. The existence of cluster structures on open Richardson varieties has also been a subject of study, see \cite{CaoKeller,GLtwist,GLpositroid,Ingermanson,Leclerc, MarshRietsch,Menard,WebsterYakimov}. Cluster algebras and cluster varieties have been constructed for a wide gamut of moduli spaces, especially in the context of Teichm\"uller theory \cite{FockGoncharov_ModuliLocSys,FST08,GSV,SG}, birational geometry \cite{GHK15,GHKK,HK18} and more recently symplectic geometry \cite{CasalsGao23,CW,CZ,GSW}. Braid varieties, as introduced in \cite{CGGS1,CGGS2,Kalman-braid,Mellit,SW}, are moduli spaces of certain configuration of flags; they generalize open Richardson varieties and double Bott-Samelson varieties and have appeared in many areas of algebra and geometry, including the microlocal theory of constructible sheaves \cite{CW,CZ,GSW} and the study of character varieties \cite{Boalch01,Boalch07,Boalch18,Mellit,Sibuya75}.\\

The study of cluster structures on braid varieties is the central focus of this paper. The main ingredient that we employ is the theory of weaves, introduced in \cite{CZ}. As explained in \cite[Section 7.1]{CZ}, an application of weaves is the study of exact Lagrangian fillings $L\sse(\mathbb{D}^4,\lambda_{st})$ of Legendrian links $\La\sse(\partial\mathbb{D}^4,\xi_{st})$. Specializing to the case that $\La$ has a front given by the ($-1$)-closure of a positive braid $\beta\in\mbox{Br}_n^+$, see \cite[Section 2.2]{CN}, a weave is a planar diagrammatic representation of a sequence of moves from $\beta$ to (a lift of) its Demazure product. The allowed moves are the two braid relations, i.e.~a Reidemeister III move and commutation for non-adjacent Artin generators, and the 0-Hecke product $\sigma_i^2\to\sigma_i$, which inputs the square of an Artin generator $\sigma_i\in \mbox{Br}_n^+$ and outputs the Artin generator itself. Such weaves were studied in \cite[Section 4]{CGGS1}, under the name Demazure weaves, where several results regarding equivalences and mutations were proven. A core contribution of this paper is the construction of a specific collection of cycles in Demazure weaves, for any simple Lie group type, through a tropicalization of the braid identities in Lusztig's coordinates and an intersection form between them: given a Demazure weave $\fW$ for $\beta$, this allows us to construct an exchange matrix $\varepsilon_\fW$.

\subsection{Main Results} Let $\G$ be a simple algebraic group with Weyl group $W(\G)$. We fix a Borel subgroup  $\borel\sse\G$ and a Cartan subgroup $T\subset \borel$. 
Pairs of flags $\borel_1,\borel_2\in \G/\borel$ in relative position $w\in W(\G)$ are denoted by $\borel_1 \buildrel w \over \longrightarrow \borel_2$. Let $\mbox{Br}(\G)$ be the braid group associated with $W(\G)$. The Artin generators of $\mbox{Br}(\G)$ are denoted by $\sigma_i$, which lift the Coxeter generators $s_i\in W(\G)$,  where the index $i$ runs through the simple positive roots of (the Lie algebra of) $\G$. Let $\beta = \sigma_{i_1}\cdots \sigma_{i_{r}}$ be a positive braid word 
and $\delta(\beta) \in W(\G)$ its Demazure product. The {\it braid variety} associated with $\beta$ is
\[
X(\beta) := \{(\borel_1, \dots, \borel_{r+1}) \in (\G/\borel)^{r + 1} \mid \borel_1 = \borel, \borel_{k} \buildrel s_{i_{k}} \over \longrightarrow \borel_{k+1}, \borel_{r + 1} = \delta(\beta) \borel \},
\]
where $\delta(\beta) \in W(\G) \cong N_{\G}(T)/T$ has been lifted to $N_{\G}(T)$; this is well-defined since the flag $\delta(\beta) \borel$ does not depend on such a lift. See \cite{CGGS1,CGGS2} for basic properties and results on braid varieties, including the fact that they are smooth affine varieties. The cluster algebra, resp.~upper cluster algebra, associated with an exchange matrix $\varepsilon$ is denoted by $\mathcal{A}(\varepsilon)$, resp.~$\up{\varepsilon}$. The main result of this paper reads as follows:

\begin{theorem}\label{thm:main}
Let $\G$ be a simple algebraic Lie group and $\beta\in\Br(\G)$ a positive braid. Then the coordinate ring $\C[X(\beta)]$ of the braid variety $X(\beta)$ is isomorphic to the cluster algebra $\cluster{\varepsilon_{\fW}}$, where $\fW$ is an arbitrary Demazure weave for $\beta$. In fact, each Demazure weave $\fW$ gives a cluster seed in $\C[X(\beta)]$ and two non-equivalent Demazure weaves give rise to mutation equivalent cluster seeds.
\end{theorem}

\noindent Theorem \ref{thm:main} is proven by first establishing that $\C[X(\beta)]$ is isomorphic to the upper cluster algebra $\up{\varepsilon_{\fW}}$, which contains $\cluster{\varepsilon_{\fW}}$ as a subalgebra, and then showing that $\up{\varepsilon_{\fW}}=\cluster{\varepsilon_{\fW}}$. The equality $\C[X(\beta)]=\up{\varepsilon_{\fW}}$ is proven by combining our previous work on double Bott-Samelson varieties \cite{SW}, see also \cite{CW,GSW}, and a localization procedure. The argument also shows that the Lusztig cycles associated to two equivalent Demazure weaves, as defined in \cite{CGGS1,CZ}, yield the same exchange matrix. Note that both Demazure weaves and their associated exchange matrices $\varepsilon_\fW$ can be readily constructed, and we provide an algorithmic procedure in the form of inductive weaves. The cluster $\mathcal{A}$-coordinates for $\mathcal{A}(\varepsilon_\fW)$ are subtly extracted from generalized minors associated with the (generic) configuration of flags specified by $\fW$, geometrically measuring relative positions of such flags, see Section \ref{sec: cluster variables}.\\

 Following \cite{CGGS2}, the open Richardson varieties $\mathcal{R}_\G(v, w)$, where $v,w\in W(\G)$, are particular instances of braid varieties. See Section \ref{ssec:openRichardson} where the braid $\beta$ is described in terms of $v,w$. Theorem \ref{thm:main} thus implies the following result:

\begin{corollary}[Leclerc's Conjecture \cite{Leclerc}]\label{cor:Leclerc} Let $\G$ be a simply-laced simple algebraic Lie group and $v, w \in W(\G)$. Then the open Richardson variety $\mathcal{R}_\G(v, w)$ admits a cluster structure.
\end{corollary}

\noindent Previous work on Leclerc's Conjecture includes the original source \cite{Leclerc}, where the category of modules over the preprojective algebra of $\G$ is used to construct an upper cluster algebra contained in $\C[\mathcal{R}_\G(v, w)]$ and equality proven in a number of special cases (e.g.~$v$ is a suffix of $w$). The recent articles \cite{GLpositroid,Ingermanson, serhiyenko2019cluster} construct upper cluster algebra structures for $\C[\mathcal{R}_\G(v, w)]$ for the case $\G = \SL_{n}$ and cluster algebra structures on coordinate rings of positroid varieties. Note that the initial seed in \cite{Leclerc} is constructed in a rather indirect way; see also the algorithm recently provided by E.~M\'enard \cite{Menard} and \cite{GLtwist}. In \cite{CaoKeller}, it is proved that the seed defined via M\'enard's algorithm defines an upper cluster algebra structure on $\C[\mathcal{R}_\G(v, w)]$, for $\G$ simply-laced, as in the conjecture. As emphasized above, our construction with weaves and Lusztig cycles directly provides an explicit initial seed, with exchange matrix being constructed by essentially linearly reading the braid, and the cluster variables are explicitly presented as regular functions on $\C[\mathcal{R}_\G(v, w)]$. In addition, Theorem \ref{thm:main} proves the equality between the upper cluster algebra and the cluster algebra, and applies to open Richardson varieties for non simply-laced types, i.e.~we prove Corollary \ref{cor:Leclerc} even without the simply-laced hypothesis; the hypothesis is only stated so as to match the original conjecture.\\

As a second corollary of the (proof of) Theorem \ref{thm:main}, the braid variety $X(\beta)$ is simultaneously equipped with a cluster $\mathcal{X}$-structure associated with $\varepsilon_{\fW}$. Therefore, $X(\beta)$ admits a natural cluster quantization. 

\begin{corollary}\label{cor:ensemble} Let $\G$ be a simple algebraic group and $\beta\in\Br(\G)$ a positive braid. Then the affine algebraic variety $X(\beta)$ admits the structure of a cluster $\mathcal{X}$-variety. In addition, it admits a Donaldson-Thomas transformation which is realized by a twist automorphism and cluster duality holds.
\end{corollary}

\noindent In Corollary \ref{cor:ensemble}, we establish the existence of the Donaldson-Thomas transformation by showing that a reddening sequence exists, which suffices by the combinatorial characterization of B.~Keller \cite{KelDT}. In this case, the cluster duality conjecture of V.~Fock and A.~Goncharov \cite{FockGoncharov_ModuliLocSys} states that the coordinate ring $\C[X(\beta)]$ admits a linear basis naturally parameterized by the integer tropicalization of the braid variety $X^{\vee}(\beta)$ associated with the Langlands dual group $\G^{\vee}$. By \cite{GHKK}, cluster duality follows from the fact that our exchange matrices are of full rank, which we prove in Section \ref{sec:Poisson}, and the existence of a DT-transformation \cite{GHKK}. Moreover, as stated in Corollary \ref{cor:ensemble}, we explicitly construct the Donaldson-Thomas transformation on $X(\beta)$ as the twist automorphism, see Theorem \ref{thm: DT}.\\

Finally, the present paper develops several new ingredients in the theory of Demazure weaves, used to prove Theorem \ref{thm:main} and its corollaries, and establishes further properties of these cluster $\mathcal{A}$-structures and $\mathcal{X}$-structures. These properties include local acyclicity for the exchange matrices associated to Demazure weaves, the quasi-cluster equivalences induced by cyclic rotations in a braid word, the comparison of the cluster Gekhtman-Shapiro-Vainshtein 2-form with the holomorphic structure constructed in \cite[Theorem 1.1]{CGGS1}, and the comparison of the cluster structures in Theorem \ref{thm:main} with the construction of E.~M\'enard \cite{Menard} in the case of open Richardson varieties.\\


\noindent {\bf Organization of the article}. Section \ref{sec: background} contains background on cluster algebras. Section \ref{sec: braid varieties} defines braid varieties and summarizes their basic properties. In particular, we show that open Richardson varieties and double Bott-Samelson cells are instances of braid varieties. Section \ref{sec: weaves} develops results for Demazure weaves in arbitrary simply-laced type. First, weave equivalences and weave mutations are defined and Lemma \ref{lem: demazure classification} concludes that any two Demazure weaves are related by such local moves. Second, we define Lusztig cycles in a Demazure weave and study their intersections, which leads to the construction of a quiver from a Demazure weave. Section \ref{sec: cluster variables} defines cluster variables associated with cycles in a Demazure weave and concludes Theorem \ref{thm:main} in the simply-laced case. Theorem \ref{thm:main} is proven by first showing that $\C[X(\beta)]$ admits the structure of an {\em upper} cluster algebra for the quivers associated to Demazure weaves and then proving the equality $\mathcal{A}=\mathcal{U}$. The upper cluster structure is constructed by considering the Bott-Samelson cluster structure constructed in \cite{SW} and showing that erasing the letters in a braid word amounts to freezing and deleting vertices in the quiver, cf.~Lemma \ref{lem:upper}. The second step $\mathcal{A}=\mathcal{U}$ is obtained by showing that cyclic rotations of a braid word lead to quasi-cluster transformations; see Theorem \ref{thm: quasi cluster}. Section \ref{sec: non simply laced} proves Theorem \ref{thm:main} in the non simply-laced cases. Section \ref{sec:properties} discusses properties of the cluster structures in Theorem \ref{thm:main}. Section \ref{sec:Poisson} proves Corollary \ref{cor:ensemble} and discusses cluster Donaldson-Thomas transformations. Section \ref{sec: form} studies the 2-form on $X(\beta)$ built in \cite{CGGS1,Mellit}, proving that it agrees with the cluster 2-form in our cluster structure. Section \ref{sec: Richardson} shows that, in the case of open Richardson varieties, the cluster structures in Theorem \ref{thm:main}  recover and generalize the seed construction of E.~M\'enard \cite{Menard}. Finally, Section \ref{sec: examples} provides examples.\\


\noindent {\bf Acknowledgements}. We are grateful to Ben Elias, Brian Hwang, Bernhard Keller, Allen Knutson, Bernard Leclerc, Anton Mellit, Etienne M\'enard, Catharina Stroppel, Daping Weng, and Lauren Williams for many useful discussions. In the development of this manuscript, we have learned of independent work in progress of Pavel Galashin, Thomas Lam, Melissa Sherman-Bennett and David Speyer \cite{GLSS2,GLSS1} towards results in line with Theorem \ref{thm:main}. We are thankful to them for many conversations and useful discussions. We are very much grateful to the referee for their carefully reading of the manuscript and the various improvements they suggested.
R.~Casals is supported by the NSF CAREER DMS-1942363 and a Sloan Research Fellowship of the Alfred P. Sloan Foundation. E.~Gorsky is partially supported by the NSF grant DMS-1760329. M.~Gorsky is supported by the French ANR grant CHARMS~(ANR-19-CE40-0017). This work is a part of a project that has received funding from the European Research Council (ERC) under the European Union’s Horizon 2020 research and innovation programme (grant agreement No. 101001159). Parts of this work were done during stays of M.~Gorsky at the University of Stuttgart, and he is very grateful to Steffen Koenig for the hospitality. L.~Shen is partially supported by the Collaboration Grant for Mathematicians from the Simons Foundation~(\#711926) and the NSF grant DMS-2200738. J.~Simental is grateful for the financial support and hospitality of the Max Planck Institute for Mathematics, where his work was carried out.\\



\section{Preliminaries}
\label{sec: background}

Let us first review the key definitions and notations on cluster algebras, following \cite{FockGoncharov_ensemble} and see also \cite{FWZ, FZcluster} for more details. By definition, a {\bf seed} is a tuple $\s := (I, I^{\uf}\!, 
\varepsilon, d)$, where $I$ is a finite set, $I^{\uf} \!\subseteq I$ is a subset, $
\varepsilon \in \Q^{|I| \times |I|}$ is a rational matrix, $d \in \Z_{>0}^{|I|}$ is a positive integer vector, and they satisfy:
\begin{itemize}
    \item[-] $
    \varepsilon_{ij} \in \Z$ unless $i, j \in I \setminus I^{\uf}$.
    \item[-] The vector $d$ is primitive, i.e. $\gcd(d_{i})_{i \in I} = 1$, and the matrix $
    \widetilde{\varepsilon}_{ij} = \varepsilon_{ij}d^{-1}_j$ is skew-symmetric. 
\end{itemize}

The elements of $I^{\uf}$ are referred to as \emph{unfrozen} elements or \emph{mutable} elements. 
The matrix $
\varepsilon$ is known as the \emph{exchange matrix} of the seed; it is by definition skew-symmetrizable. If $d_{i} = 1$ for every $i \in I$, the seed itself is said to be skew-symmetric: in this case, the data of the matrix $
\varepsilon$ can be visualized by drawing a quiver $Q$ with vertex-set $Q_{0} = I$ and $\max(0, 
\varepsilon_{ij})$ arrows from vertex $i$ to vertex $j$. We mainly work with skew-symmetric seeds in this manuscript. The greater generality of skew-symmetrizable seeds is only needed when discussing braid varieties on non simply-laced groups, see Section \ref{sec: non simply laced}.\\

 Given $k \in I^{\uf}$, mutation $\mu_k(\s) := \mu_{k}(I, I^{\uf}\!, 
\varepsilon, d)$ provides a new seed $(I, I^{\uf}\!, 
\varepsilon', d)$, where the new exchange matrix $
\varepsilon'$ is defined as follows:
  $$
    \varepsilon'_{ij}:=\begin{cases}
    -\varepsilon_{ij}& \text{if}\ i=k\ \text{or}\ j=k,\\
    \varepsilon_{ij}+\frac{|\varepsilon_{ik}|\varepsilon_{kj}+\varepsilon_{ik}|\varepsilon_{kj}|}{2} & \text{otherwise}.
    \end{cases}
    $$

\noindent Mutation is involutive: $\mu_{k}^{2} = \operatorname{id}$. A seed $\s'$ is said to be mutation equivalent to $\s$ if there exists a finite sequence of mutations that turn $\s$ into $\s'$.\\

Consider the field of rational functions $\C(x_{i})_{i \in I}$. For each seed $\s'$ mutation equivalent to $\s$, we consider a collection of algebraically independent rational functions $(A_{\s', i})_{i \in I} \subseteq \C(x_{i})_{i \in I}$. These rational functions are compatible with mutations in that if $\s'' = \mu_{k}(\s')$ then $A_{\s'', i} = A_{\s', i}$ for $i \neq k$, but 
$$
A_{\s'',k}=\frac{\prod_{\varepsilon'_{ki}\ge 0}A_{\s', i}^{\varepsilon'_{ki}}+\prod_{\varepsilon'_{ki}\le 0}A_{\s', i}^{-\varepsilon'_{ki}}}{A_{\s',k}}.
$$
Note that $A_{\s', i}$ is independent of $\s'$ if $i \in I \setminus I^{\uf}$. By definition, the \emph{cluster algebra} $\cluster{\s}$ associated with the seed $\s$ is the $\C[A_{\s, i}^{\pm 1} \mid i \in I \setminus I^{\uf}]$-subalgebra of $\C(x_i)_{i \in I}$ generated by the set
\[
\bigcup_{\s'}\{A_{\s', i} \mid i \in I\},
\]
where the union runs over all the seeds $\s'$ which are mutation-equivalent to $\s$. Since all the combinatorics are encoded by the exchange matrix $\varepsilon$, we will denote the cluster algebra $\cluster{\s}$ simply by $\cluster{\varepsilon}$, or $\cluster{Q}$ when the exchange matrix $\varepsilon$ is skew-symmetric with quiver $Q$.

The \emph{upper cluster algebra} $\up{\varepsilon}$ is defined as
\[
\up{\varepsilon} := \bigcap_{\s'} \C[A_{\s', i}^{\pm 1} \mid i \in I],
\]
where the intersection again runs over all seeds $\s'$ which are mutation equivalent to $\s$.  The Laurent phenomenon \cite{FZcluster} states that $\cluster{\varepsilon} \subseteq \up{\varepsilon}$. Thus, for every seed $\s'$, the localization $\cluster{\varepsilon}[\prod_{i \in I}A_{\s', i}^{-1}]$ is a Laurent polynomial algebra. Geometrically, every seed $\s'$ defines a rank $|I|$ open algebraic torus
\[
\mathbb{T}_{\s'} \subseteq \Spec(\cluster{\varepsilon}),
\]
known as a \emph{cluster torus}. 

\begin{remark}
In the notation $[x]_{+}:=\max(x,0)$ and $[x]_{-}:=\min(x,0)$, the cluster mutation rules can be then written as
\begin{equation}
\label{eq: quiver mutation plus minus}
\varepsilon'_{ij}=\begin{cases}
    -\varepsilon_{ij}& \text{if}\ i=k\ \text{or}\ j=k,\\
    \varepsilon_{ij}+[\varepsilon_{ik}]_{+}[\varepsilon_{kj}]_{+}-[\varepsilon_{ik}]_{-}[\varepsilon_{kj}]_{-} & \text{otherwise}.
    \end{cases}
\end{equation}
and
\begin{equation}
\label{eq: cluster mutation plus minus}
A'_k=\frac{\prod A_i^{[\varepsilon_{ki}]_{+}}+\prod A_i^{-[\varepsilon_{ki}]_{-}}}{A_k}.
\end{equation}
\end{remark}

Finally, the idea of tropicalization also plays a role in this manuscript. Let $(\mathbb{Q}(t)_{>0},+,\cdot)$ denote the semifield of subtraction-free rational functions and consider the standard discrete valuation map from $(\mathbb{Q}(t)_{>0},+,\cdot)$ to the semifield $(\Z,\min,+)$. The tropicalization of $1+t^a$ is $\min(0,a)=[a]_{-}$ and the tropicalization of $\frac{t^a}{1+t^a}$ is $a-\min(0,a)=[a]_{+}$. Part of the identities we use are tropicalizations of explicit identities with rational functions and can be proven directly. Nevertheless, other identities use abstract results on total positivity, e.g.~ see Lemma \ref{lem: cycle output}. In either case, the idea of tropicalization guides the definition of Lusztig cycles on weaves and significantly clarifies the constructions in the paper.


\section{Braid varieties}
\label{sec: braid varieties}

This section discusses braid varieties and their properties, including the use of pinnings, framings, and their relation to open Richardson varieties and double Bott-Samelson varieties.


\subsection{Notations}\label{sec: notations}

Throughout the paper we fix an algebraic group $\G$, which for now we assume to be of simply laced type, and choose a pair of opposite Borel subgroups $(\borel_{+}, \borel_{-})$, with unipotent subgroups $\uni_{\pm} = [\borel_{\pm}, \borel_{\pm}]$ and maximal torus $T = \borel_{+}\cap \borel_{-}$. We will also frequently write $\borel = \borel_{+}$. The flag variety is the quotient $\G/\borel$ and we refer to its points as flags; the point $\borel \in \G/\borel$ is said to be the standard flag. Elements of $\G/\borel$ are in correspondence with the set of Borel subgroups of $\G$, in such a way that the Borel subgroup $\borel$ corresponds to $\borel \in \G/\borel$.\\

\noindent We denote the vertex set of the Dynkin diagram of $\G$ by $\dynkin$,  the corresponding Weyl group by  $W=W(\G)$, and its longest element by $w_0\in W$. The simple reflections in $W$ are denoted by $s_i,i\in \dynkin$. Note that, upon identification $W = N_{\G}(T)/T$, we have $\borel_{-} = w_{0}\borel w_{0}$, where we abuse the notation and denote by $w_0$ a lift of the longest element to $N_\G(T)$. We also consider the associated braid group $\Br_W=\Br(\G)$, generated by elements $\sigma_i,i\in \dynkin$ modulo the relations:
\begin{equation}
\label{eq: braid rels}
\begin{cases}
\sigma_i\sigma_j=\sigma_j\sigma_i & \text{if}\ i,j\ \text{are not adjacent in}\ \dynkin\\  
\sigma_i\sigma_j\sigma_i=\sigma_j\sigma_i\sigma_j & \text{if}\ i,j\ \text{are adjacent in}\ \dynkin.\\  
\end{cases}
\end{equation}
An arbitrary product $\beta=\sigma_{i_1}\cdots \sigma_{i_r}$ is said to be a positive braid word of length $\ell(\beta)=r$, and we denote by $\Br_W^{+}$ the positive braid monoid consisting of such words. There is a homomorphism from $\Br_W$ to $W$ that sends $\sigma_i$ to $s_i$. Conversely, given $w\in W$ we can define its minimal-length positive braid lift $\beta(w)\in \Br^{+}_W$. We denote a minimal lift of $w_0$ by $\Delta := \beta(w_0) \in \Br^{+}_W$ and we refer to $\Delta$ as the \emph{half twist}. \\

\noindent Following \cite[Definition 1.3]{Escobar}, the Demazure product map $\delta:\Br^{+}_W\to W$ is inductively defined by 
$$
\delta(\sigma_i):=s_i,\ \delta(\beta\sigma_i):=\begin{cases}
\delta(\beta)s_i & \text{if}\  \ell(\delta(\beta)s_i)=\ell(\delta(\beta))+1\\
\delta(\beta) & \text{if}\ \ell(\delta(\beta)s_i)=\ell(\delta(\beta))-1.\\
\end{cases}
$$
The map $\delta$ is well-defined and we have
\[
\delta(\sigma_i\beta) = \begin{cases} s_{i}\delta(\beta) & \text{if}\ \ell(s_i\delta(\beta)) = \ell(\delta(\beta)) + 1 \\ \delta(\beta) & \text{if}\ \ell(s_i\delta(\beta)) = \ell(\delta(\beta)) - 1. \end{cases}
\]
Note that $\delta$ is not a homomorphism of monoids, e.g.~ $\delta(\sigma_i^k) = s_i$ for $k \ge 1$, however $\delta(\beta(w)) = w, w \in W$. For $u,v\in W$ we will sometimes write $u*v = \delta(\beta(u)\beta(v))$.


\subsection{Relative position}\label{sec: rel position} Following the identification $W = N_{\G}(T)/T$, we have a bijection between the Weyl group $W$ and the set of double coset representatives $\borel\backslash\G/\borel$, see \cite{CG}. Moreover, we have the Bruhat and Birkhoff decompositions:
\begin{equation}\label{eq: bruhat dec}
\G = \bigsqcup_{w \in W} \borel w\borel = \bigsqcup_{w \in W}\borel_{-}w\borel.
\end{equation}
We say that a pair $(x\borel, y\borel) \in \G/\borel \times \G/\borel$ is in relative position $w \in W$ if $x^{-1}y \in \borel w\borel$. We denote this relationship by $x\borel \buildrel w \over \longrightarrow y\borel$. The relative position of flags satisfies many properties related to the Coxeter group structure of $W$:

\begin{lemma}\label{lem: rel position}
Let $\G$ be a simple Lie group and $\borel\sse\G$ a Borel subgroup. Then the following holds:
\begin{itemize}
    \item[$(1)$] If $x\borel \buildrel w \over \longrightarrow y\borel$, $y\borel \buildrel s_{i} \over \longrightarrow z\borel$ and $w < ws_{i}$, then $x\borel \buildrel ws_{i} \over \longrightarrow z\borel$. 
    \item[$(2)$] If $i, j \in \dynkin$ are not adjacent and we have a sequence of flags in the corresponding relative positions
    \[
    x\borel \buildrel s_{i} \over \longrightarrow y\borel \buildrel s_{j} \over \longrightarrow z\borel,
    \]
    then there exists a unique flag $y'\borel$ that fits in the following diagram:
    \[
    x\borel \buildrel s_{j} \over \longrightarrow y'\borel \buildrel s_{i} \over \longrightarrow z\borel.
    \]
    \item[$(3)$] If $i, j \in \dynkin$ are adjacent and we are given the sequence of flags:
    \[
    x\borel \buildrel s_{i} \over \longrightarrow y_{1}\borel \buildrel s_{j} \over \longrightarrow y_{2}\borel \buildrel s_{i} \over \longrightarrow z\borel,
    \]
    then there exist unique flags $y'_{1}\borel$ and $y'_{2}\borel$ that fit in the following diagram:
    \[
    x\borel \buildrel s_{j} \over \longrightarrow y'_{1}\borel \buildrel s_{i} \over \longrightarrow y'_{2}\borel \buildrel s_{j} \over \longrightarrow z\borel.
    \]
\end{itemize}
\end{lemma}

\noindent Lemma \ref{lem: rel position}.(1) follows from the following property of the Bruhat decomposition:
\begin{equation}\label{eq: bruhat}
(\borel w\borel)(\borel s_{i}\borel) = \begin{cases} \borel ws_{i}\borel, & \mbox{if } w < ws_{i}, \\ \borel w\borel \sqcup \borel ws_{i}\borel, & \text{else.} \end{cases}
\end{equation}
Lemma \ref{lem: rel position}.(2) and (3) are deduced from the following result:

\begin{lemma}\label{lem: rel position 2}
Let $w \in W$ and assume that $w<ws_{i}$ for some $i\in \dynkin$. Consider $x\borel, z\borel \in \G/\borel$ such that $x\borel \buildrel ws_{i} \over \longrightarrow z\borel$. Then, there exists a unique flag $y\borel$ such that
\[
x\borel \buildrel w \over \longrightarrow y\borel \buildrel s_{i} \over \longrightarrow z\borel.
\]
\end{lemma}
\begin{proof}
Existence follows from \eqref{eq: bruhat}. For uniqueness, assume that we have $y\borel, y'\borel$ satisfying the conclusion of the lemma. Then  $z^{-1}y \in \borel s_{i}\borel$, $z^{-1}y' \in \borel s_{i}\borel$. Since $s_{i} = s_{i}^{-1}$, we have that $y^{-1}y' \in (\borel s_{i}\borel)(\borel s_{i}\borel) = \borel s_{i}\borel \sqcup \borel$; it thus suffices to show that $y^{-1}y' \not\in \borel s_{i}\borel$. By contradiction, suppose that $y^{-1}y' \in \borel s_{i}\borel$. Then, since $x\borel \buildrel w \over \longrightarrow y\borel$, we have $x^{-1}y' = x^{-1}yy^{-1}y' \in (\borel w\borel)(\borel s_{i}\borel) = \borel ws_{i}\borel$, where we have used $w < ws_{i}$. Nevertheless, this contradicts $x\borel \buildrel w \over \longrightarrow y'\borel$, and the result follows. 
\end{proof}


\subsection{Braid varieties} Let $\beta = \sigma_{i_1}\cdots \sigma_{i_{r}} \in \Br_W^{+}$ be a positive braid word and let $\delta(\beta) \in W$ be its Demazure product. The notation $\delta:=\delta(\beta)$ will be used for $\delta(\beta)$ if $\beta$ is clear by context. The \emph{braid variety} associated with $\beta$ is
\begin{equation}\label{eq:def braid variety}
X(\beta) := \{(\borel_1, \dots, \borel_{r+1}) \in (\G/\borel)^{r + 1} \mid \borel_1 = \borel, \borel_{k} \buildrel s_{i_{k}} \over \longrightarrow \borel_{k+1}, \borel_{r + 1} = \delta \borel \}
\end{equation}
where $\delta$ is a lift of $\delta \in W \cong N_{\G}(T)/T$ to $N_{\G}(T)$. (The flag $\delta \borel$ does not depend on such a lift.) Note that $X(\beta)$ does \emph{not} depend on the chosen braid word for $\beta$, cf. Lemma \ref{lem: rel position}, and there is an isomorphism between the braid varieties associated to two representatives of the same braid \cite[Theorem 2.18]{SW}. These have been studied at least in \cite{Boalch07,CGGS1,Escobar,Kalman-braid,Mellit,SW} under different names and contexts.

\begin{remark}\label{rmk: other fibers}
Instead of requiring $\borel_{r+1} = \delta \borel$, we can require $\borel_{r + 1} = x\borel$ for any flag $x\borel \in \borel\delta \borel/\borel$, and obtain an isomorphic variety, see \cite[Theorem 3.3]{Escobar} and its proof. The choice of $\borel_{r + 1} = \delta \borel$ allows for certain torus actions to be defined on $X(\beta)$, cf. \cite[Section 2.2]{CGGS1}.
\end{remark}

By \cite[Theorem 20]{Escobar}, $X(\beta)$ is a smooth, irreducible affine variety of dimension $\ell(\beta) - \ell(\delta)$. The result \cite[Theorem 3.7]{CLS} shows that $X(\beta) \cong X(\beta\sigma_{i})$ if $\delta(\beta\sigma_{i}) = \delta(\beta)s_{i}$. In particular, we can assume that $\delta(\beta) = w_0$ in many arguments. In fact, these isomorphisms can be refined as follows:

\begin{lemma}\label{lem: loc closed}
Let $\beta \in \Br_W^+$ and $i \in \dynkin$.
\begin{itemize}
\item[$(1)$] If $\delta(\beta\sigma_{i}) = \delta s_{i}$, then $X(\beta\sigma_{i}) = X(\beta)$.
\item[$(2)$] If $\delta(\beta\sigma_{i}) = \delta$, then $X(\beta)$ is isomorphic to a locally closed subvariety of $X(\beta\sigma_{i})$. 
\item[$(3)$] If $\delta(\sigma_i\beta) = \delta$, then $X(\beta)$ is isomorphic to a locally closed subvariety of $X(\sigma_i\beta)$.
\end{itemize}
\end{lemma}
\begin{proof}
Part (1) is \cite[Theorem 3.7]{CLS} but we provide a proof for the sake of completeness, as follows. Assume that $\delta(\beta\sigma_{i}) = \delta s_{i}$. It suffices to show that given 
\begin{equation}\label{eqn: seq flags}
(x_1\borel \buildrel s_{i_{1}} \over \longrightarrow x_2\borel \longrightarrow \cdots \longrightarrow x_{r+1}\borel \buildrel s_{i} \over \longrightarrow x_{r+2}\borel) \in X(\beta\sigma_{i})
\end{equation}
we are then forced to have $x_{r+1}\borel = \delta \borel$. Thanks to Lemma \ref{lem: rel position 2}, it is enough to show that $x_1\borel \buildrel \delta \over \longrightarrow x_{r+1}\borel$. Since we must have $x_1\borel \buildrel \delta' \over \longrightarrow x_{r+1}B$ for some $\delta' \leq \delta$, but if $\delta' < \delta$ then $\delta's_{i} < \delta s_{i}$, we cannot have $x_{r+1}\borel = \delta s_{i} \borel$ and Part (1) follows.

\noindent For Part (2), consider an element as in \eqref{eqn: seq flags} above. We must have $x_{1}\borel \buildrel \delta' \over \longrightarrow x_{r+1}\borel$ for some $\delta' \leq \delta$. If $\delta' < \delta$, then we are forced to have $\delta' = \delta s_{i}$ and, using Lemma \ref{lem: rel position 2} again, $x_{r+1}\borel = \delta s_{i}\borel$. Thus, the locus
\[
X^{\circ}(\beta\sigma_{i}) := \{(x_1\borel \buildrel s_{i_{1}} \over \longrightarrow x_2\borel \longrightarrow \cdots \longrightarrow x_{r+1}\borel \buildrel s_{i} \over \longrightarrow x_{r+2}\borel) \in X(\beta\sigma_{i}) \mid x_1\borel \buildrel \delta \over \longrightarrow x_{r+1}\borel\} \subseteq X(\beta\sigma_{i})
\]
coincides with the locus $x_{r+1}\borel \neq \delta s_{i}\borel$ and is therefore open in $X(\beta\sigma_{i})$. Let us now fix a flag $x\borel$ such that $\borel = x_{1}\borel \buildrel \delta \over \longrightarrow x\borel \buildrel s_{i} \over \longrightarrow x_{r+1}\borel = \delta B$. Note, in particular, that $x\borel \neq \delta \borel$. The locus 
\[
\{(x_1\borel \buildrel s_{i_{1}} \over \longrightarrow x_2\borel \longrightarrow \cdots \longrightarrow x_{r+1}\borel \buildrel s_{i} \over \longrightarrow x_{r+2}\borel) \in X^{\circ}(\beta\sigma_{i}) \mid x_{r+1}\borel = x\borel\} \subseteq X^{\circ}(\beta\sigma_{i})
\]
is closed in $X^{\circ}(\beta\sigma_{i})$ and it is isomorphic to $X(\beta)$, by Remark \ref{rmk: other fibers}. The proof of (3) is analogous.
\end{proof}


\subsection{Coordinates and pinnings}\label{sect: pinnings} In this subsection, we provide ambient affine coordinates to describe the braid varieties $X(\beta)$. In particular, we construct an explicit collection of polynomials in $\C[z_1, \dots, z_\ell]$ defining them, where $\ell = \ell(\beta)$. In order to give such coordinates, we first fix a \emph{pinning} of the group $\G$, see \cite{Lusztig, SW}. Namely, for every $i \in \dynkin$ we select isomorphisms $x_{i}: \C \to \uni_{i}^{+}$ and $y_{i}: \C \to \uni_{i}^{-}$, where $\uni_{i}^{+}$ and $\uni_{i}^{-}$ are the corresponding root subgroups of $i \in \dynkin$, such that the assignment
\[
\left(\begin{matrix} 1 & z \\ 0 & 1 \end{matrix}\right) \mapsto x_{i}(z), \quad \left(\begin{matrix} b & 0 \\ 0 & b^{-1} \end{matrix}\right) \mapsto \chi_{i}(b), \quad \left(\begin{matrix} 1 & 0 \\ z & 1 \end{matrix}\right) \mapsto y_{i}(z)
\]
gives a morphism $\varphi_{i}: \SL_{2}(\C) \to \G$, where $\chi_{i}: \C^{\times} \to T$ is the simple coroot  corresponding to $i \in \dynkin$. Every simple algebraic group $\G$ admits a pinning and any two pinnings are conjugate, cf. \cite{Lusztig}. Given a pinning $(x_{i}, y_{i})_{i \in \dynkin}$, define 
\[
s_{i} := \varphi_{i}\left(\begin{matrix}0 & -1 \\ 1 & 0 \end{matrix} \right)\in \G.
\]
Note that $s_{i} \in N_{\G}(T)$ is a lift of the simple reflection corresponding to $i \in \dynkin$. Given a permutation $u\in W$, we can define its lift to $\G$ by choosing an arbitrary reduced expression and multiplying $s_i$ accordingly.
For $z \in \C$, we define
\[
B_{i}(z):= x_{i}(z)s_{i} = \varphi_{i}\left(\begin{matrix} z & -1 \\ 1 & 0\end{matrix}\right) \in \G.
\]

\noindent By \cite[Proposition 2.5]{Lusztig}, the group elements $B_{i}(z)$ satisfy the following properties.

\begin{lemma}
\label{lem: braid matrix relations}
Let $i,j \in \dynkin$ be two distinct vertices of the Dynkin diagram. Then the following holds:
\begin{itemize}
    \item[$(1)$] If $i$ and $j$ are not adjacent in $\dynkin$, then $B_{i}(z)B_{j}(w) = B_{j}(w)B_{i}(z)$.
    \item[$(2)$] If $i$ and $j$ are adjacent in $\dynkin$, then
    \[
    B_{i}(z_1)B_{j}(z_2)B_{i}(z_3) = B_j(z_3)B_i(z_1z_3 - z_2)B_j(z_1)
    \]
\end{itemize}
\end{lemma}

\noindent The elements $B_i(z)$ can be used for an alternative description of flags in $s_i$-relative position:

\begin{proposition}\label{prop:coordinates}
Fix a flag $x\borel \in \G/\borel$. Then
$
\{y\borel \in \G/\borel \mid x\borel \buildrel s_{i} \over \longrightarrow y\borel\} = \{xB_{i}(z)\borel | z \in \C\}.
$
In addition, $xB_{i}(z)\borel = xB_{i}(z')\borel$ only if $z = z'$. 
\end{proposition}
\begin{proof}
The former statement is \cite[Lemma A.6]{SW}, and the latter follows since, in $\SL_{2}(\C)$, the matrix $\varphi_{i}^{-1}(B_{i}(z)B^{-1}_{i}(z'))$ is upper triangular if and only if $z = z'$.
\end{proof}

This description readily yields a set of equations for $X(\beta)$:

\begin{corollary}\label{cor: braid via eqns}
If $\beta = \sigma_{i_{1}}\sigma_{i_{2}}\cdots \sigma_{i_{r}}$, then
$
X(\beta) \cong \{(z_1, \dots, z_r) \in \C^{r} \mid \delta^{-1}B_{i_{1}}(z_1)\cdots B_{i_{r}}(z_r) \in \borel\},
$\\
where $\delta^{-1}$ denotes the lift of the Weyl group element to $\G$ using $s_i$, as above.
\end{corollary}
\begin{proof}
By Proposition \ref{prop:coordinates}, for every element $(x_{1}\borel \buildrel s_{i_{1}} \over \longrightarrow \cdots \buildrel s_{i_{r}} \over \longrightarrow x_{r+1}\borel) \in X(\beta)$ there exists a unique element $(z_1, \dots, z_r) \in \C^r$ such that:
\[
x_1\borel = \borel,\qquad x_2\borel = B_{i_{1}}(z_1)\borel, \qquad \dots, \qquad x_{r+1}\borel = B_{i_{1}}(z_1)\cdots B_{i_{r}}(z_r)\borel
\]
and the condition $x_{r+1}\borel = \delta \borel$ translates to $\delta^{-1}B_{i_{1}}(z_1)\cdots B_{i_{r}}(z_r) \in \borel$.
\end{proof}

Note that the condition $\delta^{-1}B_{i_{1}}(z_1)\cdots B_{i_{r}}(z_r) \in \borel$ can be expressed via the vanishing of some generalized minors, which implies that $X(\beta)$ is indeed an affine variety, cf. \cite{FZ}. Note that $X(\beta) = \{0\}$  if $\beta = \beta(w)$ for some element $w \in W$; indeed, in terms of coordinates one verifies that
\begin{equation}\label{eqn: braid reduced}
(z_1, \dots, z_r) \in X(\beta(w)) ~\Longleftrightarrow~ z_1 = \cdots = z_r = 0.
\end{equation}

\begin{definition}
The group element $B_{\beta}(z)$ associated with $\beta = \sigma_{i_{1}}\cdots \sigma_{i_{r}} \in \Br_{W}$ is
\[
B_{\beta}(z) := B_{i_{1}}(z_1)\cdots B_{i_{r}}(z_r) \in \G.
\]
\end{definition}

\noindent The following identity will be useful, compare to \cite[Lemma 2.13]{CGGS1}.

\begin{corollary}\label{cor: move U}
Let $i \in \dynkin, z \in \C$, $U \in \borel$. Then, there exist unique elements $U' \in \borel, z' \in \C$ such that
\[
UB_{i}(z) = B_{i}(z')U'.
\]
\end{corollary}
\begin{proof}
Note that $\borel \buildrel s_{i} \over \longrightarrow UB_{i}(z)\borel$. By Proposition \ref{prop:coordinates}, this implies $UB_{i}(z)\borel = B_{i}(z')\borel$ for some $z' \in \C$. By the same argument as in the proof of Proposition \ref{prop:coordinates} such $z' \in \C$ is unique.
\end{proof}

\noindent Corollary \ref{cor: move U} is used to show rotation invariance, in the following sense.

\begin{lemma}\label{lem: rotation}
Let $\beta = \sigma_{i_{1}}\cdots \sigma_{i_{\ell}}$ and assume that $\delta(\beta) = w_0$. Let $i_{1}^{*} \in \dynkin$ be such that $w_{0}s_{i_{1}}w_{0} = s_{i_{1}^{*}}$.  Then there exists an isomorphism
\[
X(\sigma_{i_{1}}\cdots \sigma_{i_{\ell}}) \cong X(\sigma_{i_{2}}\cdots \sigma_{i_{\ell}}\sigma_{i_{1}^{*}})
\]
such that, in coordinates, it is of the form $(z_1, z_2, \dots, z_{\ell}) \mapsto (z_2, \dots, z_{\ell}, z_{1}')$ for $z'_1$ depending on $z_1, \dots, z_{\ell}$. 
\end{lemma}
\begin{proof}
Let us denote $w_{0} = B_{\Delta}(0) \in \G$, and we claim that there exist unique $\tilde{z} \in \C$ and $\tilde{U} \in \borel$ such that 
\begin{equation}\label{eq: conjugate B}
w_0B_{i_{1}}(z)w_0 = B_{i_{1}^{*}}(\tilde{z})\tilde{U}
\end{equation}
In order to see this, first note that:
\[
s_iB_{i}(z)s_i = \varphi_i\left[\left(\begin{matrix}0 & -1 \\ 1 & 0 \end{matrix} \right)\left(\begin{matrix}z & -1 \\ 1 & 0 \end{matrix} \right)\left(\begin{matrix}0 & -1 \\ 1 & 0 \end{matrix} \right)\right] = \varphi_i\left(\begin{matrix}0 & 1 \\ -1 & -z \end{matrix} \right) = B_{i}^{-1}(-z).
\]

\noindent Choose a reduced word for $\Delta$ of the form $\Delta = \si_{i_{1}}\beta(w) = \beta(w)\si_{i^{*}_{1}}$ for a reduced word $w$. Then
\[
B_{i_1}^{-1}(-z)w_0 = B_{i_1}^{-1}(-z)s_iB_{w}(0) = s_iB_{i_1}(z)B_w(0) = s_iB_w(0)B_{i_1^{*}}(z) = w_0B_{i_{1}^{*}}(z)
\]
where the next-to-last equality follows from Lemma \ref{lem: braid matrix relations}. Thus, we get $w_0B_{i_1}(z)w_0 = B_{i_{1}^{*}}^{-1}(-z)$. Now, $B_{i_{1}^{*}}(-z) \in \borel s_{i_{1}^{*}}\borel$, so the same is true for $B_{i_{1}^{*}}^{-1}(-z)$. It follows that $\borel \buildrel s_{i^{*}_1} \over \longrightarrow B_{i_{1}^{*}}^{-1}(-z)\borel$ and thus, using Proposition \ref{prop:coordinates}, that $B_{i_{1}^{*}}^{-1}(-z) = B_{i_{1}^{*}}(\tilde{z})\tilde{U}$ for a unique $\tilde{z} \in \C$ and $\tilde{U} \in \borel$, which is precisely \eqref{eq: conjugate B}.\\

\noindent Now assume that $(z_1, \dots, z_\ell) \in X(\sigma_{i_{1}}\cdots \sigma_{i_{\ell}})$. Then, $w_0B_{\beta}(z_1, \dots, z_{\ell}) = U \in \borel$ and we get:
\[
\begin{array}{rl}
   B_{i_{2}}(z_2)\cdots B_{i_{\ell}}(z_\ell)  & = B_{i_{1}}^{-1}(z_1)w_0U  \\
     & = w_0\tilde{U}B^{-1}_{i_{1}^{*}}(\tilde{z}_1)U \\
     & = w_0\tilde{U}U'B^{-1}_{i_{1}^{*}}(z'_1)
\end{array}
\]
where in the last equality we have used Corollary \ref{cor: move U}. Thus, $w_0B_{i_{2}}(z_2)\cdots B_{i_{\ell}}(z_\ell)B_{i_{1}^{*}}(z'_1) = \tilde{U}U' \in \borel$.
\end{proof}


\subsection{Framings}\label{ssec:framings}
Consider the basic affine space $\G/\uni$, where $\uni$ is the unipotent radical of $\borel$. There is a natural projection $\pi:\G/\uni\to \G/\borel$ with fibers isomorphic to $\borel/\uni=T$. A point of $\G/\uni$ will be called a framed flag, and its image in $\G/\borel$ is referred to as its underlying flag. We often denote framed flags by cosets $x\uni$ in $\G/\uni$ of elements $x\in\G$. The following is a straightforward analogue of Proposition \ref{prop:coordinates}.

\begin{proposition}
\label{prop: coordinates framed}
Let $x\uni\in \G/\uni$ be a framed flag and consider $\tz, \tz' \in \C$, $u, u' \in \C^{*}$. Suppose that $xB_i(\tz)\chi_i(u)\uni = xB_i(\tz')\chi_i(u')\uni$. Then $\tz = \tz'$ and $u = u'$.
\end{proposition}
 
The framed version of Lemma \ref{lem: braid matrix relations} reads as follows:.

\begin{lemma}
\label{lem: braid matrix relations framed}
Let $i,j \in \dynkin$ be two distinct vertices of the Dynkin diagram. Then the following holds:
\begin{itemize}
    \item[$(1)$] If $i$ and $j$ are not adjacent in $\dynkin$, then $B_{i}(\tz_1)\chi_i(u_1)B_{j}(\tz_2)\chi_j(u_2) = B_{j}(\tz_2)\chi_j(u_2)B_{i}(\tz_1)\chi_i(u_1)$.
    \item[$(2)$] If $i$ and $j$ are adjacent in $\dynkin$, then
    \[
    B_{i}(\tz_1)\chi_i(u_1)B_{j}(\tz_2)\chi_j(u_2)B_{i}(\tz_3)\chi_i(u_3) = B_j(\tz'_1)\chi_j(u'_1)B_i(\tz'_2)\chi_i(u'_2)B_j(\tz'_3)\chi_j(u'_3)
    \]
    provided that 
    $$
    u_1u_2=u'_2u'_3,\ u_2u_3=u'_1u'_2.
    $$
    Here $\tz'_i$ are uniquely determined by $\tz_i, u_i$ and $u'_i$.
\end{itemize}
\end{lemma}

\begin{proof}
This follows from an $\SL_3$-computation, and it is directly verified that
$$
\tz'_1=\tz_3\frac{u_1}{u_2},\ \tz'_2=\tz_1\tz_3\frac{u_1u_3}{u'_2} - \tz_2\frac{u'_1}{u_1},\  
\tz'_3=\tz_1\frac{u'_2}{u'_1}.
$$
\end{proof}

\noindent A relation between the $z$-coordinates and the $\tz$-coordinates is as follows:

\begin{lemma}
\label{lem: z to z tilde}
Let $\beta=\sigma_{i_1}\cdots \sigma_{i_{\ell}}$ be a positive braid word and fix $u_1,\ldots,u_{\ell}\in \C^*$. Then the variety
$$
\left\{(\tz_1,\ldots,\tz_{\ell}) \in \C^{\ell}:\delta^{-1}B_{i_1}(\tz_1)\chi_{i_1}(u_1)\cdots B_{i_{\ell}}(\tz_{\ell})\chi_{i_{\ell}}(u_{\ell})\in \borel\right\}\subset \C^{\ell}
$$
is isomorphic to the variety $X(\beta)$. Furthermore, there is an isomorphism such that the ratios $\tz_i/z_i$ are Laurent monomials in the $u_1,\ldots,u_{\ell}$ parameters.
\end{lemma}

\begin{proof}
Similarly to Corollary \ref{cor: move U} we have $DB_i(z)=B_i(z')D^{s_i}$ where $D\in T,D^{s_i}=s_iDs_i$ and $z'$ is related to $z$ by a monomial in the elements $\chi^{\vee}_{i}(D)$. 
Using this identity, we move all $\chi_{i_j}(u_j)$ to the right and get
$$
\delta^{-1}B_{i_1}(\tz_1)\chi_{i_1}(u_1)\cdots B_{i_{\ell}}(\tz_{\ell})\chi_{i_{\ell}}(u_{\ell})=\delta^{-1}B_{i_1}(z_1) \cdots B_{i_{\ell}}(z_{\ell})D
$$
for some $D\in T$ and some $z_i$ related to $\tz_i$ by monomials in the $u_1,\ldots,u_{\ell}$ parameters. In particular, in this change of coordinates the ratios $\tz_i/z_i$ are expressed as Laurent monomials in the $u_1,\ldots,u_{\ell}$. Since 
$$
\delta^{-1}B_{i_1}(z_1) \cdots B_{i_{\ell}}(z_{\ell})D\in B\Leftrightarrow \delta^{-1}B_{i_1}(z_1) \cdots B_{i_{\ell}}(z_{\ell})\in\borel,
$$
we have that $(z_1,\ldots,z_{\ell})$ defines a point in $X(\beta)$, establishing the desired isomorphism.
\end{proof}


\subsection{Open Richardson varieties}\label{ssec:openRichardson} In the rest of this section, we study the relationship that braid varieties bear to two families of previously studied varieties: open Richardson varieties and half-decorated double Bott-Samelson varieties. Braid varieties generalize both of these families of varieties in a sense that we now make precise.\\

\noindent Let us recall that we have fixed both a Borel subgroup $\borel$ as well as its opposite Borel $\borel_{-}$. By the Bruhat (resp.~Birkhoff) decomposition \eqref{eq: bruhat dec}, every $\borel$ (resp.~ $\borel_{-}$) orbit in $\G/\borel$ is of the form $\mathcal{S}_{w} := \borel w\borel/\borel$ (resp.~ $\mathcal{S}^{-}_{w} := \borel_{-}w\borel/\borel$) for a unique element $w \in W$. Moreover, the space $\mathcal{S}_{w}$ (resp.~$\mathcal{S}^{-}_{w}$) is an affine cell of dimension $\ell(w)$ (resp. $\ell(w_{0}) - \ell(w)$) and it is known as a \emph{Schubert cell} (resp.~\emph{opposite Schubert cell}) of the flag variety $\G/\borel$. Note that we can describe the Schubert cells in terms of relative positions:
\[
\mathcal{S}_{w} = \{x\borel \in \G/\borel \mid \borel \buildrel w \over \longrightarrow x\borel\}, \qquad \mathcal{S}^{-}_{v} = \{y\borel \in \G/\borel \mid y\borel \buildrel v^{-1}w_{0} \over \longrightarrow w_{0}\borel\}.
\]

\noindent By definition, the \emph{open Richardson variety} associated with a pair $v,w\in W$ is
\[
\mathcal{R}(v, w) := \mathcal{S}^{-}_{v} \cap \mathcal{S}_{w}.
\]

\noindent It is known that the intersection $\mathcal{S}^{-}_{v}\cap\mathcal{S}_{w}$ is nonempty if and only if $v \leq w$ in Bruhat order, in which case it is a transverse intersection of dimension $\ell(w) - \ell(v)$.

\begin{theorem}
\label{thm: richardson}
Let $v, w \in W$ be such that $v \leq w$. Let $\beta(w), \beta(v^{-1}w_{0}) \in \Br_W^{+}$ be minimal lifts, and $\ell := \ell(w)+\ell(v^{-1}w_0)$. Then the map
\[
\begin{array}{c}
X(\beta(w)\beta(v^{-1}w_{0})) \to \mathcal{R}(v, w)\\
(x_{1}\borel, x_{2}\borel, \dots, x_{\ell+1}\borel) \mapsto x_{\ell(w)+1}\borel
\end{array}
\]
is an isomorphism.
\end{theorem}
\begin{proof}
This is analogous to the proof of \cite[Theorem 4.5]{CGGS2}\footnote{Note that in \cite{CGGS2} we defined braid varieties entirely in terms of matrices, which slightly differ from the matrices $B_i(z)$ used here.}. 
Indeed, since $\beta(w)$ is a minimal lift of $w$ and $x_{1}\borel = \borel$, we have $\borel \buildrel w \over \longrightarrow x_{\ell(w)+1}\borel$, i.e. $x_{\ell(w)+1}\borel \in \mathcal{S}_{w}$. Independently, since $v \leq w$ the Demazure product $\delta(\beta(w)\beta(v^{-1}w_{0}))$ is precisely $w_{0}$, and thus we have $x_{\ell+1}\borel = w_{0}\borel$. The minimality of the lift $\beta(v^{-1}w_{0})$ implies that $x_{\ell(w)+1}\borel \buildrel v^{-1}w_{0} \over \longrightarrow w_{0}\borel$, that is, $x_{\ell(w)+1}\borel \in \mathcal{S}^{-}_{v}$. Therefore $x_{\ell(w)+1}\borel \in \mathcal{R}(v, w)$, showing that the map is indeed well-defined.\\

\noindent Given $x\borel \in \mathcal{S}_{w}$, it follows from Lemma \ref{lem: rel position 2} that, given a reduced decomposition $w = s_{i_{1}}\cdots s_{i_{\ell(w)}}$, there is a unique sequence of flags:
\[
\borel \buildrel s_{i_{1}} \over \longrightarrow \borel_{2} \buildrel s_{i_{2}} \over \longrightarrow \cdots \buildrel s_{i_{\ell(w)}} \over \longrightarrow x\borel.
\]
Lemma \ref{lem: rel position 2} also implies that, given a reduced decomposition $v^{-1}w_{0} = s_{i_{\ell(w)+1}}\cdots s_{i_{\ell}}$, there is a unique sequence of flags
\[
x\borel \buildrel s_{i_{\ell(w)+1}} \over \longrightarrow \cdots \buildrel s_{i_{\ell}} \over \longrightarrow w_{0}\borel,
\]
and we conclude that the map is an isomorphism.
\end{proof}


\subsection{Double Bott-Samelson varieties}\label{sec:double bs}
Let us now describe the relationship that braid varieties bear to double Bott-Samelson varieties, which were introduced in \cite{SW}, and see also \cite[Section 4.1]{GSW}.

\begin{definition}
Let $\beta \in \mathrm{Br}_{W}^{+}$, the \emph{(half-decorated) double Bott-Samelson variety} $\Conf(\beta)$ is
\[
\Conf(\beta) := \{(z_1, \dots, z_{r}) \in \C^r \mid B_{\beta}(z) \in \borel_{-}\borel = (w_0\borel w_0)\borel\}.
\]
\end{definition}

\noindent It is shown in \cite[\S 2.4]{SW}, see also \cite[Proposition 4.9]{GSW}, that $\Conf(\beta)$ is a smooth affine variety and that it is an open set in $\C^r$ given by the non-vanishing of a single polynomial. 

\begin{lemma}
\label{lem: conf as braid variety} Let $\beta \in \mathrm{Br}_{W}^{+}$. Then there exists a natural identification
\[
X(\Delta\beta) = \Conf(\beta)
\]
where $\Delta\in\mathrm{Br}_{W}^{+}$ is a minimal lift of the longest element $w_0\in W$. 
\end{lemma}
\begin{proof}
Let us denote by $z_1, \dots, z_r$ the variables corresponding to the letters of $\beta$, and by $w_1, \dots, w_s$ those corresponding to the letters of $\Delta$. Since $\delta(\Delta\beta) = w_0$, we have that $(w,z) \in X(\Delta\beta)$ iff $w_0B_{\Delta}(w)B_{\beta}(z) \in \borel$. Either condition implies $B_{\beta}(z) \in \borel_{-}\borel$ because the map $w \mapsto B_{\Delta}(w)$ gives an isomorphism $\C^r \to \uni w_{0}$ so that $w_{0}B_{\Delta}(w) \in \borel_{-}$. (See Proposition \ref{prop:coordinates}, and Equation \eqref{eq: bruhat}.) Given $z \in \Conf(\beta)$, so that $B_{\beta}(z) \in \borel_{-}\borel$, we can decompose uniquely $B_{\beta}(z) = x_{-}x_{+}$, where $x_{-} \in \uni_{-} = w_{0}\uni_{+}w_{0}$. Therefore there exists a unique $w \in \C^r$ such that $x_{-} = w_{0}B_{\Delta}(w)$ and the identification follows.
\end{proof}

\noindent The varieties $\Conf(\beta)$  admit cluster structures, as proven in \cite{SW}. This was independently shown in \cite{CW} via the microlocal theory of sheaves on weaves for $\G=\SL_n$. Let us now briefly review the cluster structure on $\Conf(\beta)$ as in \cite{SW}, which serves as a starting point for constructing cluster structures on more general braid varieties. The basic combinatorial input in \cite{SW} is that of a \emph{triangulation} of a trapezoid\footnote{We remark that, just as in \cite{GSW}, our notation differs from \cite{SW} by a horizontal flip.}. In our setting, the trapezoid is a triangle and we have a unique triangulation of the form:
\begin{center}
\begin{tikzcd}
& & \bullet \arrow[ddll,dash] \arrow[ddl, dash] \arrow[ddr, dash] \arrow[ddrr, dash]& & \\ & & & & \\
\bullet \arrow[r,dash,"s_{i_{1}}"] & \bullet \arrow[r, dash, "s_{i_{2}}"] & \cdots \arrow[r, dash] & \bullet \arrow[r,dash, "s_{i_{r}}"] & \bullet 
\end{tikzcd}
\end{center}
where $\beta = \sigma_{i_{1}}\cdots \sigma_{i_{r}}$. There is a quiver $Q(\beta)$ associated with this triangulation: the vertices of $Q(\beta)$ correspond to the letters of $\beta$ and are colored by the vertices of the Dynkin diagram $\dynkin$. For each triangle of the form
\begin{center}
    \begin{tikzcd}
     & \bullet \ar[ddr, dash] \ar[ddl, dash] & \\ & \color{blue}{\bullet_{i}} & \\ \bullet \ar[rr, dash, "s_{i}"] & & \bullet
    \end{tikzcd}
\end{center}
we have an $i$-colored vertex in $Q(\beta)$, pictured in blue above. The arrows in the quiver correspond to the following configurations:
\begin{center}
    \begin{tikzcd}
    & & & \bullet \arrow[ddrrr, dash] \arrow[ddr, dash] \arrow[ddlll, dash] \arrow[ddl, dash] & & & \\ & & \color{blue}{\bullet_{i}} \arrow[rr]  & &  \color{blue}{\bullet_{i}}  & & \\
    \bullet \arrow[rr, dash, "s_i"] & & \bullet \arrow[r, dash] & \cdots\arrow[r, dash] & \bullet \arrow[rr, dash, "s_i"] & & \bullet 
    \end{tikzcd}
\end{center}
\begin{center}
    \begin{tikzcd}
    & & & \bullet \arrow[ddrrr, dash] \arrow[ddr, dash] \arrow[ddlll, dash] \arrow[ddl, dash] & & & \\ & & \color{blue}{\bullet_{i}}   & &  \color{red}{\bullet_{j}}\arrow[ll] & & \\
    \bullet \arrow[rr, dash, "s_i"] & & \bullet \arrow[r, dash] & \cdots\arrow[r, dash] & \bullet \arrow[rr, dash, "s_j"] & & \bullet 
    \end{tikzcd}
\end{center}
where, in the first case, there is no $i$-vertex in-between the pictured $i$-vertices and, in the second case, $i$ and $j$ are adjacent in the Dynkin diagram $\dynkin$ and there are neither $i$- nor $j$-vertices in-between the pictured vertices. For each $i \in \dynkin$ the rightmost $i$-vertex is declared to be frozen, and these are all frozen vertices in $Q(\beta)$. Finally, we add a half-weighted arrow from a frozen $i$-vertex to a frozen $j$-vertex if the last appearance of $\sigma_i$ in $\beta$ comes after the last appearance of $\sigma_j$ and $i, j$ are adjacent in $\dynkin$.  \\

\noindent The cluster variables associated with the vertices of $Q(\beta)$ are constructed as follows. First, note that an $i$-vertex of $Q(\beta)$ is nothing but an element $k = 1, \dots, r$ with $i_{k} = i$. For such an element $k$, define
\[
\widetilde{A}_{k} :=  \Delta_{\omega_{i}}(B_{i_{1}}(z_1)\cdots B_{i_{k}}(z_k))
\]
where $\Delta_{\omega_{i}}$ is the generalized principal minor associated to the fundamental weight $\omega_i$, cf.~ \cite{FZ,GLSkm}. By \cite[Theorem 3.45]{SW}, the quiver $Q(\beta)$ together with the variables $\widetilde{A}_{k}$ give rise to a cluster structure on $\C[\Conf(\beta)]$. Recall that we have the identity $w_0=B_{\Delta}(0)$, where $\Delta$ is the braid lift of $w_0$. For a coordinate-free interpretation of the cluster variables $\widetilde{A}_{k}$, we consider the following function on pairs $(x\uni,y\uni)$ of framed flags:
\[
\Delta_{\omega_{i}}(x\uni,y\uni) := \Delta_{\omega_{i}}(w_0^{-1}x^{-1}y).
\] 
Let us denote $\Delta=\sigma_{j_1}\cdots \sigma_{j_l}$. An element
\begin{center}
    \begin{tikzcd}
    \borel_0 \ar[r, "s_{j_1}"] & \borel_1 \ar[r,"s_{j_2}"] & \cdots \ar[r,"s_{j_l}"]  & \borel_l \ar[r, "s_{i_{1}}"] & \borel_{l+1} \ar[r, "s_{i_{2}}"] & \cdots    \ar[r, "s_{i_{r}}"] & \borel_{l+r}
    \end{tikzcd}
\end{center}
in $X(\Delta\beta)=\Conf(\beta)$ admits a unique lift to a sequence of framed flags
\begin{center}
    \begin{tikzcd}
     \uni_0 \ar[r, "s_{j_1}"] & \uni_1 \ar[r,"s_{j_2}"] & \cdots \ar[r,"s_{j_l}"]  &
    \uni_l \ar[r, "s_{i_{1}}"] & \uni_{l+1} \ar[r, "s_{i_{2}}"] & \cdots  \ar[r, "s_{i_{r}}"] & \uni_{r}
    \end{tikzcd}
\end{center}
subject to the condition that $\uni_0 = \uni$, cf. \cite[Lemma B.8]{GSW}. Then, $\widetilde{A}_{k} = \Delta_{\omega_{i}}(\uni, \uni_{l+k})$, where $i = i_{k}$. 
Indeed, following the proof of Lemma \ref{lem: conf as braid variety}, we have $\uni_{l+k}=B_{\Delta}(w) B_{i_1}(z_1)\cdots B_{i_k}(z_k) \uni$, where $w$ are the variables corresponding to the crossings of $\Delta$. Note that $w_0^{-1}B_\Delta(w)\in \uni_-$. Therefore
 \[
 \Delta_{\omega_{i}}(\uni, \uni_{l+k}) = \Delta_{\omega_{i}}(w_0^{-1}B_{\Delta}(w)B_{i_1}(z_1)\cdots B_{i_k}(z_k)) = \Delta_{\omega_{i}}(B_{i_1}(z_1)\cdots B_{i_k}(z_k))=\widetilde{A}_k.
 \]


\section{Demazure weaves and Lusztig cycles}\label{sec: weaves}

This section develops the necessary results in the theory of weaves. The core contribution is the construction of Lusztig cycles and their associated quivers. The former are built using a tropicalization of the Lie group braid relations in Lusztig's coordinates, hence the name, and the latter is obtained via a new definition of local intersection numbers of cycles on weaves.


\subsection{Demazure weaves}\label{sec:demazure weaves} The diagrammatic calculus of algebraic weaves is
developed in \cite{CGGS1}, following the original geometric weaves in \cite{CZ}. In this manuscript, we exclusively use Demazure weaves, see \cite[Definition 4.2 (ii)]{CGGS1}, and we thus use the terms `weave' and `Demazure weave' interchangeably. By definition, a Demazure weave $\fW\sse\R^2$ is a planar graph with edges labeled by braid generators $\sigma_i$ and vertices of the types specified in Figure \ref{fig: types vertices}. The set of vertices of $\fW$ is denoted by $V(\fW)$ and its of edges by $E(\fW)$.\\

\begin{figure}[ht!]
\centering
    \includegraphics[scale=1.3]{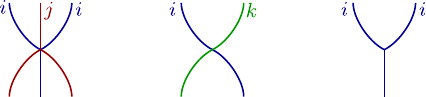}
    \caption{The types of vertices allowed in weaves. Here, we take $i, j$ and $k$ such that $i$ and $j$ are adjacent in $\dynkin$, but $i$ and $k$ are not.}
    \label{fig: types vertices}
\end{figure}

Each (generic) horizontal slice of a weave is a positive braid word, and we interpret weaves as sequences of braid words or ``movies" of braids. By \cite[Lemma 4.5]{CGGS1}, the Demazure products of all these braid words remain constant. In particular, if we start from a braid word $\beta$ on the top and the braid word at the bottom is reduced, then we get $\delta(\beta)$ on the bottom. This is expressed with the notation $\fW: \beta \to \delta(\beta)$. By convention, all our weaves will be oriented downwards.

 Each slice of an algebraic weave carries a variable, with the variables on top being $z_1, \dots, z_r$; this is capturing the variables in Corollary \ref{cor: braid via eqns}. The vertices correspond to the following equations between elements $B_{i}(z)$:
\begin{equation}
\label{eq: braid moves}
B_i(z_1)B_j(z_2)B_i(z_3) = B_j(z_3)B_i(z_1z_3 - z_2)B_j(z_1), \ \quad B_i(z_1)B_k(z_2) = B_k(z_2)B_i(z_1)
\end{equation}
\begin{equation}
\label{eq: trivalent}
B_{i}(z_1)B_{i}(z_2) = B_{i}(z_1 - z_2^{-1})U,\ \quad  U = \varphi_{i}\left(\begin{matrix} z_2 & -1 \\ 0 & z_2^{-1} \end{matrix} \right)
\end{equation}

\noindent The equation \eqref{eq: trivalent} is defined only when 
$z_2 \neq 0$ and can be applied in the middle of a product of several braid matrices. In this case, we apply Corollary \ref{cor: move U} to move the element $U \in \borel$ to the right of all the elements $B_{k}(z)$ appearing to the right of $B_{i}(z_2)$. This implies that at every trivalent vertex we must modify all the variables appearing to the right of this vertex. Finally, we require that all variables on the bottom of the weave are equal to $0$, cf. Equation \eqref{eqn: braid reduced}.\\

\begin{figure}[ht!]
\centering
    \includegraphics[scale=1.3]{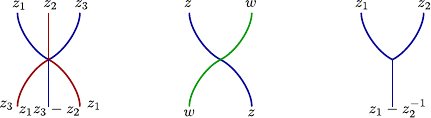}
\caption{The effect that the basic weaves have on variables, which reflects Equations \eqref{eq: braid moves} and \eqref{eq: trivalent}. Note that the rightmost weave is only defined when $z_2 \neq 0$.}
    \label{fig: types vertices + variables}
\end{figure}

The results in \cite{CGGS1} imply the following:

\begin{lemma}$($\cite[Proposition 5.3,Corollary 5.5]{CGGS1}$)$
\label{lem: weave as chart} Let $\beta\in\Br_W^+$ be a positive braid word and $\fW$ a Demazure weave. Then $\fW$ defines an open affine subset $T_{\fW} \subseteq X(\beta)$, isomorphic to the algebraic torus $(\C^*)^d$, where $d$ is the number of trivalent vertices. In addition, the variables on all edges of $\fW$ are rational functions in the initial variables $z_i$, and (Laurent) coordinates on $T_{\fW}$ are given by the variables on the right incoming edges at trivalent vertices.
\end{lemma}

 The following lemma is a more precise coordinate version of Remark \ref{rmk: other fibers}:

\begin{lemma}
\label{lem: slide weave}
Let $U_0 \in \borel$ and consider $X_{U_{0}}(\beta) := \{(z_1, \dots, z_r) \in \C^r \mid \delta^{-1}U_0B_{\beta}(z) \in \borel\}$.
Then there is a canonical isomorphism of varieties 
$$
\Phi:X(\beta)\xrightarrow{\sim} X_{U_0}(\beta).
$$
Furthermore, given any weave for $\beta$, the isomorphism $\Phi$
extends uniquely to all variables in the weave, and for any slice $\gamma$ of the weave we have
$$
U_0B_{\gamma}(\Phi(z))=B_{\gamma}(z)U.
$$
Finally, the right incoming edge at every trivalent vertex is multiplied by a scalar depending only on the projection of $U_0$ to $T$.
\end{lemma}

\begin{proof}
The existence and uniqueness of $\Phi$ follows from Corollary \ref{cor: move U}. To prove that $\Phi$ extends to a weave correctly, it is sufficient to check it for any vertex, and this is verified in \cite[Section 5.2.1]{CGGS1}. The last assertion follows from the identity (compare with \cite[Section 5.2.1]{CGGS1}):
\begin{equation}\label{eq: moving Y}
\left(\begin{matrix} a & b\\
0 & c\end{matrix}\right)\left(\begin{matrix} z_1 & -1\\
1 & 0\end{matrix}\right)\left(\begin{matrix} z_2& -1\\
1 & 0\\\end{matrix}\right)=\left(\begin{matrix} \frac{z_1a + b}{c}&  -1\\
1 & 0\\ \end{matrix}\right)\left(\begin{matrix} \frac{c}{a}z_2 & -1\\
1& 0\end{matrix}\right)\left(\begin{matrix} a & 0\\
0 & c\end{matrix}\right).\qedhere
\end{equation}
\end{proof}

\begin{remark}\label{rmk: slide weave}
The second part of Lemma \ref{lem: slide weave} can be interpreted as an analogue of Lemma \ref{lem: weave as chart} for $X_{U_0}(\beta)$. Note, however, that we do not require that the variables $\Phi(z)$ at the bottom vanish, rather that $\Phi$ determines specific values for them $($which depend on the flag $U_{0}^{-1}\delta \borel/\borel)$. In this sense, the second part of the Lemma \ref{lem: slide weave} states that the isomorphism $\Phi$ preserves the torus $T_{\fW}$.
\end{remark}

Following \cite[Section 5]{CZ}, the torus $T_{\fW}$ has the following moduli interpretation, used repeatedly throughout the manuscript. The weave $\fW\sse R$ is considered inside a rectangle $R$ in such a way that $\fW \cap \partial R$ only has points in the northern and southern edges of $\partial R$. The northern edge intersection points dictate $\beta$ left-to-right, and the southern edge intersection points dictate $\delta(\beta)$ left-to-right. Then the weave itself $\fW$ describes an incidence problem in the flag variety $\G/\borel$ as follows. For each connected component $C$ of $R \setminus \fW$, assign a flag $\borel_{C} \in \G/\borel$ such that:
\begin{enumerate}
    \item $\borel_{C_{-}} = \borel$ for the unique connected component $C_{-}$ of $R\setminus \fW$ intersecting the left boundary of $R$.
    \item $\borel_{C_{+}} = \delta\borel$ for the unique component  $C_{+}$ of $R \setminus \fW$ intersecting the right boundary of $R$.
    \item If $C,D\sse R \setminus \fW$ are separated by an edge of $\fW$ of color $i$, then we require $\borel_{C} \buildrel s_{i} \over \longrightarrow \borel_{D}$. 
\end{enumerate}

\noindent See Figure \ref{fig: weave moduli} for a depiction. Indeed, equations  \eqref{eq: braid moves} and \eqref{eq: trivalent} imply that all flags $\borel_{C}$ are determined by those flags corresponding to components intersecting the northern boundary of $R$. (In the setting of Lemma \ref{lem: slide weave} the condition (2) should be replaced by $\borel_{C_{+}} = U_{0}^{-1}\delta\borel$, cf. Remark \ref{rmk: slide weave}.)

    \begin{figure}[ht!]
\centering
    \includegraphics[scale=1]{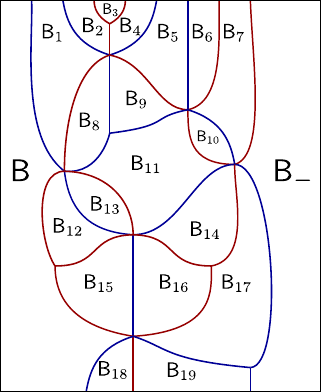}
    \caption{A weave $\fW: \beta \to \delta(\beta)$ with its configuration of flags. Note that the flags $\borel_{8}, \dots, \borel_{19}$ are completely determined by $(\borel, \borel_{1}, \dots, \borel_{7}, \borel_{-}) \in X(\beta)$, and that the flags $\borel_{18}$ and $\borel_{19}$ are coordinate flags.}
    \label{fig: weave moduli}
\end{figure}


\subsection{Weave equivalence and mutations} The notion of weave mutation was introduced in \cite[Section 4.8]{CZ}. Equivalences between weaves, also known as moves, were discussed in \cite[Theorem 1.1]{CZ}. See also \cite[Section 4]{CGGS1}. The equivalence relation on weaves can be defined as follows:
\begin{itemize}
    \item[(i)] Let $\fW,\fW':\beta\to\beta'$ consist only of braid moves, i.e. 4- and 6- valent vertices, where $\beta,\beta'$ are two positive braid words representing the same element in the braid group $\Br^+_W$. Then $\fW$ and $\fW'$ are equivalent.
    \item[(ii)] Suppose that $i,j\in\dynkin$  are adjacent. Then the weaves $ijij\to jijj\to jij$ and $ijij\to iiji\to iji\to jij$  are equivalent. See Figure \ref{fig: weave eq1}.
    \item[(iii)] Suppose that $i,j\in \dynkin$ are not adjacent. Then the weaves $iji\to iij\to ij\to ji$ and $iji\to jii\to ji$ are equivalent. In other words, one can move a $j$-colored strand through an $i$-colored trivalent vertex. 
\end{itemize}

    \begin{figure}[ht!]
\centering
    \includegraphics[scale=1.5]{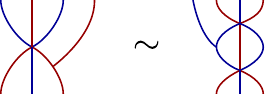}
    \caption{The two equivalent weaves in $(ii)$: the two weaves $ijij \to jijj \to jij$, depicted on the left, and $ijij \to iiji \to iji \to jij$, on the right, are declared equivalent.}
    \label{fig: weave eq1}
\end{figure}

\noindent The relations (ii) and (iii) are parameterized by rank $2$ subdiagrams of $\dynkin$ which are of types $A_2$ and $A_1\times A_1$ respectively. To ease notation, we often write $i=1$ and $j=2$ for the second case, so that we have an $A_2$ subdiagram of $\dynkin$; we therefore refer to the braid word $ijij$  on top of Figure \ref{fig: weave eq1} as 1212. Note that the weave calculus in \cite{CGGS1,CZ} used two more equivalence relations. The first relation was that all weaves from 12121 to 121 are equivalent -- by \cite[Section 4.2.5]{CGGS1} this is a consequence of our equivalence relation (ii) for 1212. The second relation was the Zamolodchikov relation for different paths of reduced expressions for the longest element in $A_3$. Such reduced expressions are related by a sequence of braid moves, and hence any two weaves of this type are equivalent by item (i). In the same vein, applying the same braid relation twice $121\to 212\to 121$ is equivalent to doing nothing. 
Finally, \cite[Section 5]{CGGS1} shows that two equivalent weaves $\fW_1$ and $\fW_2$ yield equal tori, i.e. $T_{\fW_1} = T_{\fW_2}$.\\

The two weaves for $iii\to i$ depicted in Figure \ref{fig: weave mutation} are not equivalent. By definition, these two weaves are said to be are related by {\em weave mutation}. Two weaves $\fW_1,\fW_2$ that differ by a weave mutation do {\it not} yield equal tori, i.e. $T_{\fW_1} \neq T_{\fW_2}$.\\

\begin{figure}[ht!]
\centering
\includegraphics[scale=0.8]{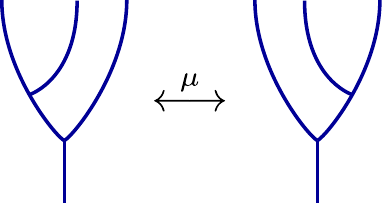}
\caption{Weave mutation}
\label{fig: weave mutation}
\end{figure}

\begin{lemma}
\label{lem: demazure classification}
Let $\fW_1,\fW_2:\beta\to\delta(\beta)$ be Demazure weaves, where we have fixed a braid word for $\delta(\beta)$. Then $\fW_1$ and $\fW_2$ are related by a sequence of equivalence moves and mutations.
\end{lemma}

\begin{proof}
In type $A$ this is proved in \cite[Theorem 4.6]{CGGS1}. For arbitrary simply laced type, we consider all possible positions in a braid word where one can apply the operations $ii\to i$ and braid relations. If such positions do not overlap, the operations commute. If they overlap, then these involve at most 3 different simple reflections, hence the problem is reduced to a rank 3 subgroup of $W$. Since any rank 3 subgroup is of type $A$, the result follows. A direct proof can also be provided by arguing as in \cite[Theorem 4.11]{CGGS1}.
\end{proof}

\begin{figure} [ht!]
\centering
\includegraphics[scale=0.35]{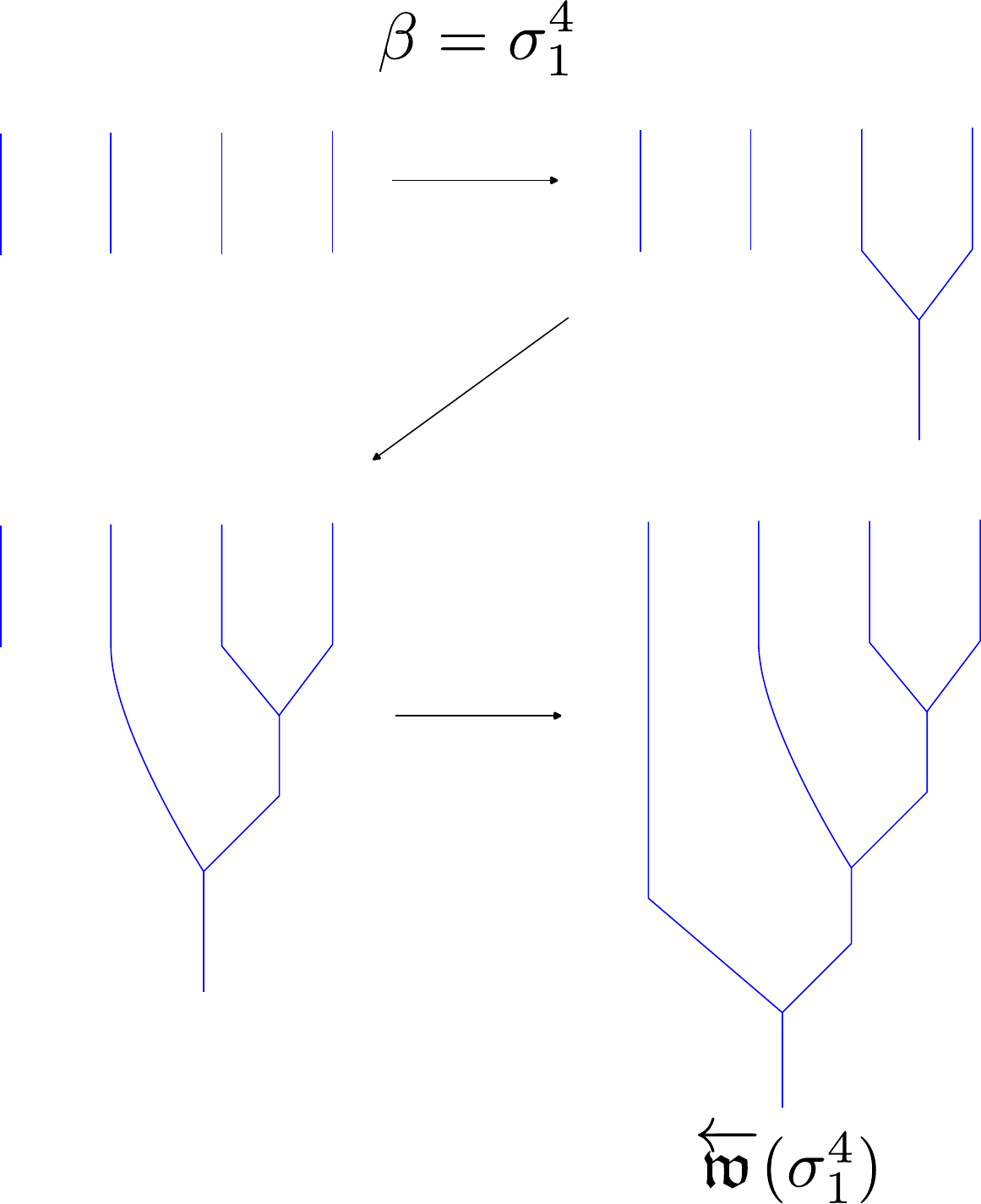}
\caption{A step by step depiction of how to construct the left inductive weave for $\beta=\sigma_1^4$. The first step is drawn in the upper-left, the second in the upper-right. The third step is drawn in the bottom-left and the final step, which is the left inductive weave, is drawn in the bottom-right.} 
\label{fig: Left_inductive_Example1}    
\end{figure}

\subsection{Inductive weaves}\label{sec: inductive weaves} 
In this subsection, we introduce the Demazure weaves $\rind{\beta},\lind{\beta}:\beta\to\delta(\beta)$ associated to a braid word $\beta$. They will yield the initial cluster seeds in our proofs in Section \ref{sec: cluster variables}. These weaves, $\lind{\beta}$ being called left inductive and $\rind{\beta}$ right inductive, are defined uniquely, up to weave equivalence. Their definition is as follows.

\begin{definition}\label{def:inductive_weave}
The \emph{left inductive weave} $\lind{\beta}:\beta\to\delta(\beta)$ is the weave constructed as follows:

\begin{itemize}
    \item[$(i)$] $\lind{\beta}$ is the empty weave if $\beta$ is the empty word.
    
    \item[$(ii)$] Suppose that $\delta(\sigma_{i}\beta)=s_i\delta(\beta)$. Then  $\lind{\sigma_{i}\beta}$ is obtained as the concatenation of  $\lind{\beta}$ and a vertical $s_i$-strand to its left.
    
    \item[$(iii)$] Suppose that 
$\delta(\sigma_i\beta)=\delta(\beta)$. Then, choose a braid word for $\delta(\beta)$ which starts at $s_i$ and form $\lind{\sigma_i\beta}$ by appending a trivalent vertex labeled by $s_i$ to the bottom left of $\lind{\beta}$.
\end{itemize}

\noindent The \emph{right inductive weave} $\rind{\beta}$ is defined analogously, instead reading the braid word $\beta$ left-to-right and having all the trivalent vertices to its right.
\end{definition}

\begin{example}\label{ex:leftinductive} $(i)$ Consider the positive braid word $\beta=\sigma_1^4$ in 2-strands. Its Demazure product is $\delta(\beta)=\sigma_1$. The left inductive weave $\lind{\beta}:\beta\to\delta(\beta)$ is drawn in Figure \ref{fig: Left_inductive_Example1}.\\

\noindent $(ii)$ Consider the positive braid word $\beta=\sigma_2^2\sigma_1\sigma_2\sigma_1\sigma_2^2\sigma_1$ in 3-strands. Its Demazure product is $\delta(\beta)=\sigma_1\sigma_2\sigma_1$. The left inductive weave $\lind{\beta}:\beta\to\delta(\beta)$ is drawn in Figure \ref{fig: Left_inductive_Example2}.$(12)$. In fact, Figure \ref{fig: Left_inductive_Example2} depicts each of the steps constructing $\lind{\beta}:\beta\to\delta(\beta)$. We draw the strands on the eventual northern boundary of $\lind{\beta}$, spelling the word $\beta$, in each intermediate step. We also split each application of step (iii) in Definition~\ref{def:inductive_weave} further into steps, adding hexavalent vertices corresponding to braid moves.
\\

\noindent $(iii)$ Figure \ref{fig:mutation 1} (left) and Figure \ref{fig:mutation 2} (left) each give an example of a right inductive weave.
\end{example}

\begin{figure} [ht!]
\centering
\includegraphics[scale=0.7]{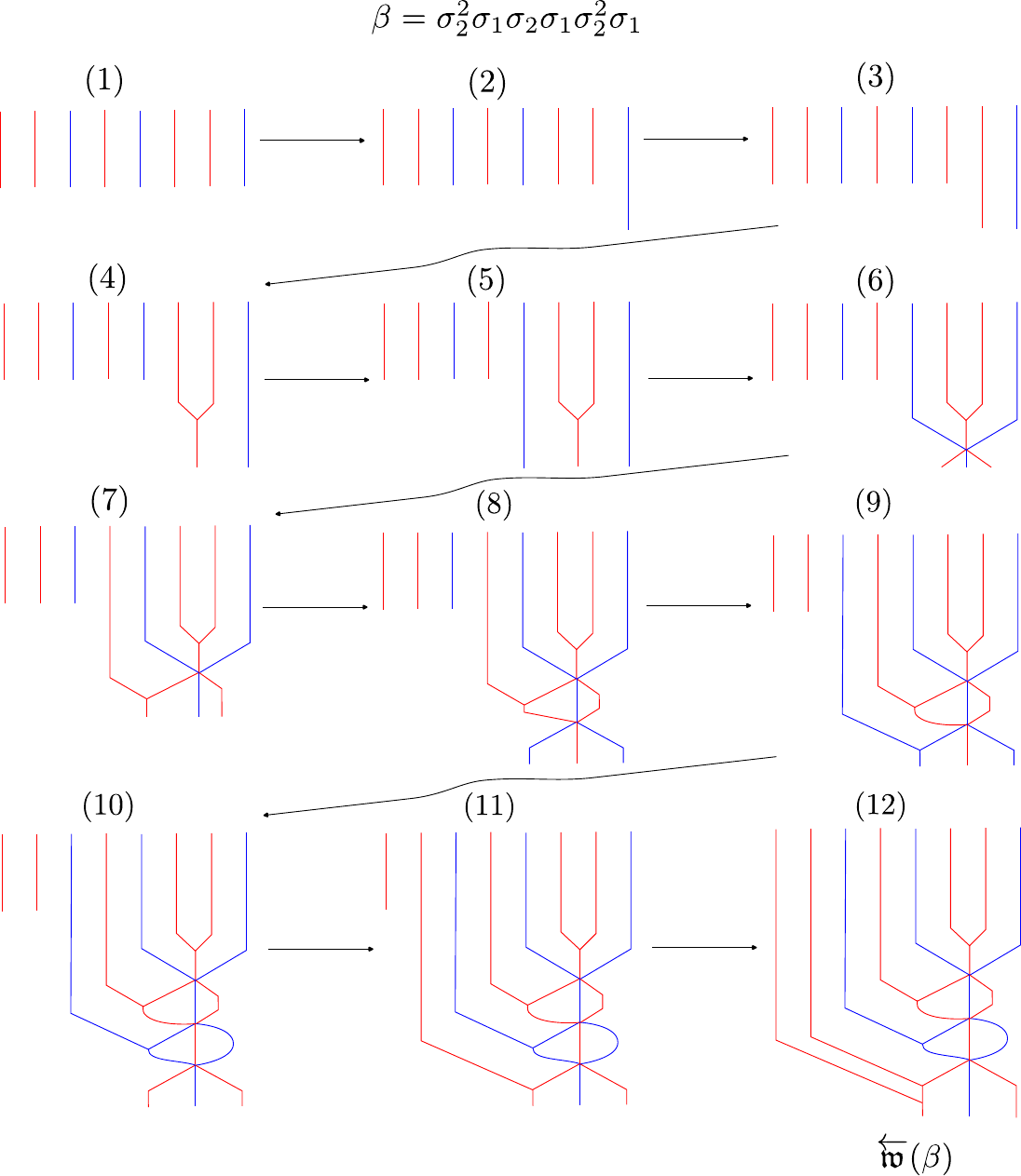}
\caption{A step by step depiction of how to construct the left inductive weave for the 3-stranded positive braid word $\beta=\sigma_2^2\sigma_1\sigma_2\sigma_1\sigma_2^2\sigma_1$. The first step is drawn in the upper-left, the second in the upper-center and so on. The $k$th step is labeled by $(k)$. There are three types of steps: a strand is added (brought down to the left) and the Demazure product increases, a strand is added and the Demazure product stays constant, or a braid move occurs. Steps $(1)\to(2)$, $(2)\to(3)$ and $(4)\to(5)$ are of the first type. Steps $(3)\to(4)$, $(6)\to(7)$, $(8)\to(9)$, $(10)\to(11)$ and $(11)\to(12)$ are of the second type, each adding a trivalent vertex. Steps $(5)\to(6)$, $(7)\to(8)$ and $(9)\to(10)$ are of the third type, with a braid move. The final left inductive weave is drawn in $(12)$.}

\label{fig: Left_inductive_Example2}    
\end{figure}

\begin{remark}\label{rmk: ind weaves arms}
By construction, a weave $\fW: \beta \to \delta(\beta)$ is left $($resp.~right$)$ inductive if and only if the left $($resp.~right$)$ edge of each trivalent vertex $v$ goes all the way to the top. Thus, trivalent vertices in such weaves can be identified with certain letters in $\beta$. The trivalent vertices in a left $($resp.~right$)$ inductive weave are parameterized by the letters in the complement of the rightmost $($resp.~leftmost$)$ reduced subword for $\delta(\beta)$ inside the word for $\beta$. 
\end{remark}

\noindent Both left and right inductive weaves are special cases of \emph{double inductive weaves}, defined in Section \ref{sec: double inductive}.

\subsection{Lusztig cycles}\label{sec: cycles} Following the geometry of 1-cycles on surfaces represented by weaves, as developed in \cite[Section 2]{CZ}, we now present the algebraic notion of a cycle on a weave $\fW$ that works for any $\G$.

\begin{definition}\label{def:weave_cycle}
A \emph{cycle} in $\mathfrak{W}$ is a function $C: E(\mathfrak{W}) \to \Z_{\geq 0}$ that assigns a non-negative integer to each edge of the weave. The values of $C$ are referred to as the \emph{weights} of the edges in $C$.
\end{definition}

 If two weaves $\mathfrak{W}_{1},\mathfrak{W}_{2}$ can be vertically concatenated (i.e.~the southern boundary of $\mathfrak{W}_{1}$ coincides with the northern boundary of $\mathfrak{W}_{2}$) and $C_{i}$ is a cycle on $\mathfrak{W}_{i}$, then the cycles $C_{i}$ can be concatenated provided that their values agree on the southern edges of $\mathfrak{W}_{1}$, which are the northern edges of $\mathfrak{W}_{2}$. We denote concatenation of cycles by $C_{2}\circ C_{1} : E(\mathfrak{W}_{2}\circ \mathfrak{W}_{1}) \to \Z_{\geq 0}$.\\

Given a weave $\fW:\beta\to\delta(\beta)$, we will extract a quiver from a particular collection of cycles and an intersection form defined on that collection. Let us focus on constructing such a collection, motivated by work of G.~Lusztig on total positivity \cite{Lusztig}. For that, let $x_i(t)=\exp(E_it)$ be the one-parameter subgroup in $\G$ corresponding to the positive simple root $\alpha_i$; in particular, $x_i(t_1)x_i(t_2)=x_i(t_1+t_2)
$. In addition, if $i,j\in \dynkin$ are not adjacent, then $$x_i(t_1)x_j(t_2)=x_j(t_2)x_i(t_1).$$
If $i,j\in\dynkin$ are adjacent, and $t_1+t_3\neq 0$, then $$x_i(t_1)x_j(t_2)x_i(t_3)=
x_j\left(\frac{t_2t_3}{t_1+t_3}\right)x_i(t_1+t_3)x_j\left(\frac{t_1t_2}{t_1+t_3}\right).$$
\noindent These can be verified directly \cite[Proposition 2.5]{Lusztig}. These relations can be considered as rational maps 
\begin{equation}
\label{eq: lusztig}
\varphi_1:(t_1,t_2)\mapsto t_1+t_2, \quad  \varphi_2:(t_1,t_2)\mapsto (t_2,t_1),\ \quad  
\varphi_3: (t_1,t_2,t_3)\mapsto  \left(\frac{t_2t_3}{t_1+t_3},t_1+t_3,\frac{t_1t_2}{t_1+t_3}\right).
\end{equation}
The coordinates $(t_i)_{i\in\dynkin}$ are referred to as Lusztig's coordinates for $\G$ in \cite[Section 1.2.6]{FockGoncharovII} and as Lusztig factorization coordinates in \cite[Definition 3.12]{SW}. A tropical version of the maps $\varphi_1,\varphi_2,\varphi_3$ is obtained by replacing multiplication with addition and addition with $\min$. The rational maps $\varphi_1,\varphi_2,\varphi_3$ then become
\begin{multline}
\label{eq: lusztig trop}
\Phi_1:(a_1,a_2)\mapsto \min(a_1,a_2), \quad \Phi_2:(a_1,a_2)\mapsto (a_2,a_1),\\ \Phi_3: (a_1,a_2,a_3)\mapsto  \left(a_2+a_3-\min(a_1,a_3),\min(a_1,a_3),a_1+a_2-\min(a_1,a_3)\right). \hskip 2.8cm
\end{multline}

\noindent Note that the equations for $\Phi_{1},\Phi_2$ and $\Phi_3$ do not depend on the indices $i,j$ of the corresponding simple roots, and $\Phi_3^2(a_1,a_2,a_3)=(a_1,a_2,a_3)$. These tropicalization maps define the following collection of cycles on a Demazure weave.

\begin{definition}\label{def:Lusztig_cycles}
Let $\mathfrak{W}$ be a Demazure weave. A \emph{Lusztig cycle} is a cycle $C: E(\mathfrak{W}) \to \Z_{\ge 0}$ satisfying the following conditions. 
\begin{itemize}
    \item[$(1)$] For a trivalent vertex  with incoming edges $e_1,e_2$ and outgoing edge $e$, $C$ satisfies $$C(e)=\Phi_1(C(e_1),C(e_2)).$$
    \item[$(2)$] For a 4-valent vertex  with incoming edges $e_1,e_2$ and outgoing edges $e_1',e_2'$, $C$ satisfies $$(C(e'_1),C(e'_2))=\Phi_2(C(e_1),C(e_2)).$$
    \item[$(3)$] For a 6-valent vertex  with incoming edges $e_1,e_2,e_3$ and outgoing edges $e_1',e_2',e'_3$, $C$ satisfies $$(C(e'_1),C(e'_2),C(e'_3))=\Phi_3(C(e_1),C(e_2),C(e_3)).$$
\end{itemize}
\end{definition}

Definition \ref{def:Lusztig_cycles} implies that the weights of a Lusztig cycle on a weave are completely determined by the weights of the top edges. In fact, the following strengthening holds.

\begin{lemma}
\label{lem: cycle output}
Let $\fW:\beta\to u$ be a Demazure weave, where $u=\delta(\beta)$ is a choice of reduced braid word, and $C$ a Lusztig cycle. Then, given the input values of $C$ on $\beta$, the output values on $u$ do not depend on the weave $\fW$.  
\end{lemma}

\begin{proof}
Suppose that $\beta=\sigma_{i_1}\cdots \sigma_{i_\ell}$ and $u=\sigma_{j_1}\cdots \sigma_{j_k}$, and choose variables $t_1,\ldots,t_{\ell},t'_1,\ldots,t'_{k}\in \C$. Consider the factorization problem
$$
x_{\beta}(t)=x_{i_1}(t_1)\cdots x_{i_{\ell}}(t_\ell)=x_{j_1}(t'_1)\cdots x_{j_k}(t'_k)=x_u(t').
$$
For a fixed weave, Equation \eqref{eq: lusztig} implies that the variables $t'_j$ can be written as certain rational functions in $t_1,\ldots,t_\ell$, where both numerator and denominator have nonnegative coefficients. Indeed,  apply $\varphi_1,\varphi_2,\varphi_3$ at every 3-,4- and 6-valent vertex, respectively. By \cite[Proposition 2.18]{FZ}, see also \cite{Lusztig}, the map 
$$
(t'_1,\ldots,t'_k)\mapsto x_u(t')=x_{j_1}(t'_1)\ldots x_{j_k}(t'_k)
$$
is an isomorphism between $(\C^*)^k$ and a Zariski open subset of a Schubert cell. In particular, $t'_1,\ldots,t'_k$ are uniquely determined by $x_u(t')$ and hence by $t_1,\ldots,t_\ell$. Then the lemma follows by tropicalization of the above argument.
\end{proof}

The following identity will be useful.

\begin{lemma}
\label{lem: tropical}
Let $a,b,c,d\in\Z$, then
$$
\min\left(a,c+d-\min(b,d)\right)+\min(b,d)=\min\left(d,a+b-\min(a,c)\right)+\min(a,c).
$$
\end{lemma}

\begin{proof}
This is a tropicalization of the following identity, which is readily verified by direct computation:
\[
\left(t^a+\frac{t^ct^d}{t^b+t^d}\right)(t^b+t^d)=\left(t^d+\frac{t^at^b}{t^a+t^c}\right)(t^a+t^c).
\qedhere
\]
\end{proof}

\begin{example}
\label{ex: 1212}
Consider the pair of Demazure weaves  $\fW_1,\fW_2$ for the braid word $\beta=1212$ as in Figure \ref{fig: weave eq1}, where $\fW_1$ is the left figure and $\fW_2$ is the right figure. Suppose that the incoming edges for a cycle have weights $a,b,c,d$. Then $\fW_1$ has the form $1212\to 2122\to 212$ and the weights transform as follows:
\begin{multline*}
(a,b,c,d)\mapsto \left(b+c-\min(a,c),\min(a,c),a+b-\min(a,c),d\right)\mapsto\\ \left(b+c-\min(a,c),\min(a,c),\min(a+b-\min(a,c),d)\right)=:(a',b',c').
\end{multline*}
The weave $\fW_2$ has the form $
1212\to 1121\to 121\to 212$ and the weights transform as:
\begin{multline*}
(a,b,c,d)\to (a,c+d-\min(b,d),\min(b,d),b+c-\min(b,d))\to \\
(\min(a,c+d-\min(b,d)),\min(b,d),b+c-\min(b,d))\to (a'',b'',c''),
\end{multline*}
where we have that
$$
b''=\min(\min(a,c+d-\min(b,d)),b+c-\min(b,d))=\min(a,c+d-\min(b,d),b+c-\min(b,d))=
$$
$$
\min(a,\min(b,d)+c-\min(b,d))=\min(a,c).
$$
By Lemma \ref{lem: tropical}, the weights also satisfy $a''=b+c-b''=b+c-\min(a,c)$ and
$$
c''=\min(a,c+d-\min(b,d))+\min(b,d)-b''=c'.
$$
\end{example}

The cycles that lead to an initial quiver are associated to trivalent vertices of a weave. These cycles are not directly Lusztig cycles, but  are ``Lusztig cycles below the trivalent vertex $v$'', in the following sense.


\begin{definition}\label{def:vertex_cycle}
Let $\mathfrak{W}$ be a Demazure weave and $v\in\fW$ be a trivalent vertex. Given the decomposition $\mathfrak{W} = \mathfrak{W}_{2} \circ \mathfrak{W}_{1}$, where the southernmost edge of $\mathfrak{W}_{1}$ is the outgoing edge of the trivalent vertex $v$, the cycle $\g_{v}$ is defined to be the concatenation $\g_{v} := C_2 \circ C_1$, where
\begin{itemize}
    \item[-] $C_1: E(\mathfrak{W}_{1}) \to \Z_{\ge 0}$ is the cycle that assigns weight $0$ to all edges, except for the (downwards) outgoing edge of the trivalent vertex $v$, to which $C_1$ assigns weight $1$.
    \item[-] $C_2: E(\mathfrak{W}_{2}) \to \Z_{\ge 0}$ is the unique Lusztig cycle that can be concatenated with $C_1$.
\end{itemize}
\end{definition}

\noindent The cycles $\gamma_v$ in Definition \ref{def:vertex_cycle}, $v\in\fW$ a trivalent vertex, will often be referred to as Lusztig cycles as well, in a minor abuse of notation and only when the context is clear, given that they are Lusztig cycles except at their origin vertex $v$. The following terminology is also useful.

\begin{definition}
Let $\mathfrak{W}$ be a Demazure weave and $v\in\fW$ be a trivalent vertex. By definition, $\g_v$ is said to \emph{bifurcate} at a $6$-valent vertex with incoming edges $e_1,e_2,e_3$ and outgoing edges $e_1',e_2',e'_3$ if  
\[\g_v(e_1) = \g_v(e_3) = 0, \g_v(e_2) \neq 0.\]
Note that this implies that $\g_v(e'_1), \g_v(e'_3) \neq 0$ and $\g_v(e'_2) = 0$, justifying the terminology. 
By definition, $\g_v$ is \emph{non-bifurcating} if it never bifurcates.
\end{definition}

\begin{figure}[ht!]
\centering
    \includegraphics[scale=1.3]{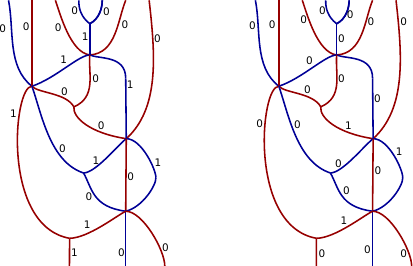}
    \caption{The cycles $\g_v$ associated to the topmost trivalent vertex (left) and second topmost trivalent vertices (right) of the weave.}
    \label{fig: cycles weights}
\end{figure}

\begin{example}
Consider the weave $\fW:\beta = \sigma_1\sigma_2\sigma_2\sigma_1\sigma_1\sigma_2\sigma_2 \to \sigma_2\sigma_1\sigma_2$ in Figure \ref{fig: cycles weights}. The cycles $\g_v$ for the topmost (resp.~second topmost) trivalent vertices are also depicted in Figure \ref{fig: cycles weights} left (resp.~right). The cycle on the left bifurcates at two $6$-valent vertices, while the cycle on the right is non-bifurcating. 
\end{example}

\begin{remark}
Note that relations similar to those in Definition \ref{def:Lusztig_cycles} appear in the definition of Mirkovi\'c-Vilonen polytopes \cite[Proposition 5.2]{Kamnitzer}. The connection between the cycles on Demazure weaves and Mirkovi\'c-Vilonen polytopes is intriguing and we plan to investigate it in the future.
\end{remark}


\subsection{Local intersections}\label{sec: intersections}

In our construction, the arrows of the quiver $Q_\fW$, which we discuss momentarily, are determined by considering (local) intersection numbers between cycles on $\fW$. Given two cycles $C,C': E(\fW)\to\Z_{\geq0}$ on a weave $\fW$, we now define their intersection number as a sum of local contributions from intersections at the 3-valent and 6-valent vertices and a boundary intersection term.

\begin{definition}[Local intersection at 3-valent vertex]\label{def:locint_3valent} Let $\fW$ be a Demazure weave, $v\in\fW$ a trivalent vertex, and $C,C': E(\fW)\to\Z_{\geq0}$ two cycles. Suppose that $C$ $($resp.~ $C')$ has weights $a_1,a_2$ $($resp.~$b_1,b_2)$ on the top left and top right incoming edges of a trivalent vertex $v$, respectively, and weight $a'$ $($resp.~$b')$ on the outgoing bottom edge. By definition, the local intersection number of $C,C'$ at $v$ is
$$
\sharp_{v}(C\cdot C')=\left|
\begin{matrix}
1 & 1 & 1\\
a_1 & a' & a_2\\
b_1 & b' & b_2\\
\end{matrix}
\right|.
$$
\end{definition}


\begin{definition}[Local intersection at 6-valent vertex]\label{def:locint_6valent} Let $\fW$ be a Demazure weave, $v\in\fW$ a hexavalent vertex, and $C,C': E(\fW)\to\Z_{\geq0}$ two cycles.
Suppose that $C$ $($resp.~$C')$ has weights $a_1,a_2,a_3$ $($resp.~$b_1,b_2,b_3)$ on the incoming edges of a 6-valent vertex $v$, and weights $a'_1,a'_2,a'_3$ $($resp.~$b'_1,b'_2,b'_3)$ on the outgoing edges. By definition, the local intersection number of $C,C'$ at $v$ is
$$
\sharp_{v}(C\cdot C')=\frac{1}{2}\left(
\left|
\begin{matrix}
1 & 1 & 1\\
a_1 & a_2 & a_3\\
b_1 & b_2 & b_3\\
\end{matrix}
\right|-
\left|
\begin{matrix}
1 & 1 & 1\\
a'_1 & a'_2 & a'_3\\
b'_1 & b'_2 & b'_3\\
\end{matrix}
\right|
\right).
$$
\end{definition}

\noindent Subsection \ref{ssec:first_example} provides an example computing these local intersections for cycles with weights 0 and 1.

\begin{remark}
See Section \ref{sec:topological-cycles} for one of the geometric motivations behind these definitions. In particular, Lemma \ref{lem:topological-intersection} asserts that for $\G = \SL_n$ the formulas in Definitions \ref{def:locint_3valent} and \ref{def:locint_6valent} compute actual intersection numbers between 1-dimensional homology classes on surfaces. This case was first studied in \cite[Section 2]{CZ}, cf.~also \cite[Section 3]{CW}. Figures \ref{fig:Cycles_Trivalent_Perturb} and \ref{fig:Cycles_Hexavalent_Perturb} in Section \ref{sec:topological-cycles} illustrate how to associate a curve to a Lusztig cycle, and Figures \ref{fig:Cycles_Lusztig} and \ref{fig:Cycles_Trivalent_Lusztig_Intersections} depict some of their geometric intersections.
\end{remark}



\begin{example}
\label{ex: local intersection 100}
Let $C,C'$ be Lusztig cycles. Suppose $C$ has weights $(a_1,a_2,a_3)=(1,0,0)$ on the top of $v$. Therefore $(a'_1,a'_2,a'_3)=(0,0,1)$. Then their local intersection at $v$ is
$$\sharp_v(C\cdot C')=\frac{1}{2}\left((b_2-b_3)-(b'_1-b'_2)\right).
$$
Since $b'_1=b_2+b_3-b'_2$, we get
$\sharp_v(C\cdot C')=b'_2-b_3$, and  $\sharp_v(C'\cdot C)=b_3-b'_2.$

\end{example}

\begin{lemma}
\label{lem: add 101}
Let $\fW$ be a weave and a 6-valent vertex $v\in\fW$. Consider three Lusztig cycles  $C,C',C''$ whose weights are $(a_1,a_2,a_3)$, $(b_1,b_2,b_3)$ and $(c_1,c_2,c_3)$ on the top of $v$. Then the following holds:

\begin{itemize}
    \item[(1)] If $(c_1,c_2,c_3)=(b_1,b_2,b_3)+(1,0,1)$ then $\sharp_v(C\cdot C')=\sharp_v(C\cdot C'')$.
    
    \item[(2)] If $(c_1,c_2,c_3)=(b_1,b_2,b_3)+(0,1,0)$ then $\sharp_v(C\cdot C')=\sharp_v(C\cdot C'')$.
\end{itemize}
\end{lemma}

\begin{proof}
For $(1)$, we have $\min(c_1,c_3)=\min(b_1,b_3)+1$ and Lusztig's rules in Definition \ref{def:Lusztig_cycles} imply $(c'_1,c'_2,c'_3)=(b'_1,b'_2,b'_3)+(0,1,0)$. Consider the cycle $\overline{C}$ with weights $(d_1,d_2,d_3)=(1,0,1)$ above $v$, and weights $(d'_1,d'_2,d'_3)=(0,1,0)$ below $v$, which is a Lusztig cycle. Then $C''=C'+\overline{C}$, addition here being understood as adding the upper weights with upper weights and adding lower weights with lower weights. In consequence, $\sharp_v(C\cdot C'')=\sharp_v(C\cdot C')+\sharp_v(C\cdot \overline{C})$, since determinants are multilinear. It thus suffices to compute $\sharp_v(C\cdot \overline{C})$, which is
$$
\sharp_v(C\cdot \overline{C})=\frac{1}{2}\left(
\left|
\begin{matrix}
1 & 1 & 1\\
a_1 & a_2 & a_3\\
1 & 0 & 1\\
\end{matrix}\right|-
\left|
\begin{matrix}
1 & 1 & 1\\
a'_1 & a'_2 & a'_3\\
0 & 1 & 0\\
\end{matrix}\right|\right)=
\frac{1}{2}\left[(a_3-a_1)-(a'_1-a'_3)\right]=$$
$$=\frac{1}{2}\left[(a_3-a_1)-((a_2+a_3-\min(a_1,a_3))-(a_2+a_1-\min(a_1,a_3)))\right]=0.
$$
Here we have used that $C$ is a Lusztig cycle to express $a'_1,a'_3$ in terms of $a_1,a_2,a_3$. From this computation we conclude that $\sharp_v(C\cdot C'')=\sharp_v(C\cdot C')+\sharp_v(C\cdot \overline{C})=\sharp_v(C\cdot C')$, as required. The proof of $(2)$ is similar.
\end{proof}

\begin{remark} Consider the notation of Lemma \ref{lem: add 101}. We can write the top weights $(c_1,c_2,c_3)$ of the Lusztig cycle $C''$ as a positive linear combination of $(1,0,1)$, $(0,1,0)$ and one of either $(1,0,0)$ or $(0,0,1)$:
$$
(c_1,c_2,c_3)=c_3\cdot (1,0,1)+c_2\cdot(0,1,0)+(c_1-c_3)\cdot(1,0,0),\qquad\mbox{if $c_3\leq c_1$},\quad \mbox{or}$$
$$
(c_1,c_2,c_3)=c_1\cdot (1,0,1)+c_2\cdot(0,1,0)+(c_3-c_1)\cdot(0,0,1),\qquad\mbox{if $c_1\leq c_3$}.$$

The local intersection $\sharp_v(C\cdot C'')$ at the hexavalent vertex $v$ of any Lusztig cycle $C$ with $C''$ can then be simplified by using Lemma \ref{lem: add 101}. Indeed, iteratively using the lemma, the computation of $\sharp_v(C\cdot C'')$ is reduced to computing the local intersection numbers of $C$ with Lusztig cycles that have weights $(c,0,0)$ or $(0,0,c)$ at the top, for some natural number $c$. In particular, following Example \ref{ex: local intersection 100}, we conclude that the local interesection number of Lusztig cycles at a 6-valent vertex is always an integer number.
\end{remark}

\begin{definition}\label{def:intersection_number}
Let $\fW$ be a weave and $C, C':E(\fW)\to\Z_{\geq0}$ cycles. By definition, the intersection number $\sharp_{\fW}(C \cdot C')$ of $C$ and $C'$ is
\[
\sharp_{\fW}(C \cdot C') := \sum_{v\, 3\text{-valent}}\sharp_{v}(C \cdot C') + \sum_{v\, 6\text{-valent}}\sharp_{v}(C\cdot C').
\]
\end{definition}

\noindent Note that $\sharp_{\fW}(C \cdot C') = -\sharp_{\fW}(C' \cdot C)$ and that $\sharp_{\fW}(C \cdot C')$ is an integer when  $C, C'$ are both Lusztig cycles. 


\subsection{Quiver from local intersections}\label{ssec:quiver} Let $\fW$ be a Demazure weave. Definition \ref{def:intersection_number}, along with the following notion of boundary intersections in Definition \ref{def: boundary intersection}, allow us to associate a quiver $Q_\fW$ to $\fW$. 

Recall that we denote the Cartan subgroup of $\G$ by $T$. Denote by $X$ and $X^\vee$ the lattices of characters and cocharacters of $T$. Consider the perfect pairing  
\begin{equation}
\label{22.7.20.1}
(\cdot, \cdot): ~ X \times X^\vee \longrightarrow \mathbb{Z}.
\end{equation}
Let $\{\alpha_i\}$ and $\{\alpha_i^\vee\}$ be the set of simple roots and simple coroots, indexed by the vertices in $\dynkin$.
Now given a braid word $\beta=\sigma_{i_1}\cdots\sigma_{i_k}$, we consider the following sequences  of roots and coroots (cf. \cite{CLS}):
\begin{equation}
\label{eq: def gamma}
\rho_j:= s_{i_1}\cdots s_{i_{j-1}}(\alpha_{i_j}), \qquad \rho_j^\vee:=s_{i_1}\cdots s_{i_{j-1}}(\alpha_{i_j}^\vee),\qquad \forall 1\leq j \leq k.
\end{equation}
Note that $\rho_{j} = s_{i_{1}}\cdots s_{i_{j}}(-\alpha_{i_j})$.  For a weave $\fW$ and a cycle $C$, such that $\beta$ appears as a horizontal section of $\fW$, we denote by $c_j$ the weight of $C$ on the $j$-th letter of $\beta$.

\begin{definition}
\label{def: boundary intersection}
Let $\fW$ be a weave and $C, C':E(\fW)\to\Z_{\geq0}$ cycles and $\beta = \sigma_{i_{1}}\cdots \sigma_{i_{r}}$  a braid word which is a horizontal section of $\fW$. By definition, the boundary intersection $\sharp_{\beta}(C\cdot C')$ of $C,C'$ at $\beta$ is 
\[
\sharp_{\beta}(C\cdot C'):=\frac{1}{2}\sum_{i,j = 1}^{r}\sign(j-i)c_{i}c'_{j}\cdot(\rho_{i}, \rho_{j}^\vee)
\]
where $(\cdot, \cdot)$ is the pairing defined via \eqref{22.7.20.1}, and
\[
\sign(k) = \begin{cases} 1 & \mbox{if }k >0, \\ 0 & \mbox{if }k = 0, \\ -1 & \mbox{if }k < 0. \end{cases}
\]
\end{definition}

\begin{remark}
In Definition \ref{def: boundary intersection} and throughout this section we are assuming that the group $\G$ is of simply laced type. For non-simply laced type Definition \ref{def: boundary intersection} has to be modified to take into account cycles for the Langlands dual group $\G^{\vee}$ of $\G$, see Section \ref{sec: non simply laced defs} and in particular \eqref{eq: boundary intersection non simply laced} below.
\end{remark}

\begin{definition}\label{def:quiver}
Let $\fW$ be a Demazure weave. By definition, the quiver $Q_{\fW}$ is the quiver whose vertices are (in bijective correspondence with) the trivalent vertices of $\fW$, and whose adjacency matrix is given by%
\[
\varepsilon_{v,v'} := \sharp_{\fW}(\g_{v}\cdot \g_{v'}) + \sharp_{\delta(\beta)}(\g_{v} \cdot \g_{v'}).
\]
where $\delta(\beta)$ is the bottom slice of the weave $\fW$.
\end{definition}

\begin{remark}
The entries  $\varepsilon_{v,v'}$ in Definition \ref{def:quiver} are always half-integers but not necessarily integers. Note also that the boundary intersection terms for $\varepsilon_{v,v'}$ vanish for cycles $\gamma_v,\gamma_{v'}$ (either of) which do not reach the bottom part of the weave: in the language of Subsection \ref{ssec:frozen}, the boundary intersection terms only appear between frozen vertices, and the weights of arrows involving a mutable vertex are always integers. 
\end{remark}

Let us now continue our study of $Q_\fW$ and its dependence on the weave $\fW$.

\begin{lemma}\label{lem: boundary int}
Let $\fW$ be a weave with no trivalent vertices. Then for any two Lusztig cycles $C,C'$ the sum of local intersection numbers equals the difference of boundary intersection numbers. 
\end{lemma}

\begin{proof}
It suffices to verify this for a single 6-valent vertex and a single $4$-valent vertex, which are local computations. For the former, suppose that the Lusztig cycle $C$ has weights $(a_1, a_2, a_3)$ on top of a  $6$-valent vertex, while the Lusztig cycle $C'$ has weights $(b_1, b_2, b_3)$, also on top. By Lemma \ref{lem: add 101}, we can assume that $a_2 = a_3 = b_2 = b_1 = 0$, so the intersection number around the $6$-valent vertex is $\sharp(C \cdot C') = -a_1b_3$. We may also assume that the roots at the top boundary are $\rho_3 = \alpha_j$, $\rho_2 = \alpha_i+\alpha_j$ and $\rho_1 = \alpha_i$ where $i,j\in\dynkin$ are adjacent. Thus, the top intersection number is
\[
\#_{\mathrm{top}}(C\cdot C') = \frac{1}{2}\sign(3-1)a_1b_3(\alpha_i, \alpha_j^\vee) = -\frac{1}{2}a_1b_3
\]
and the bottom intersection number is
\[
\#_{\mathrm{bottom}}(C \cdot C') = \frac{1}{2}\sign(1-3)a'_3b'_1(\alpha_j, \alpha_i^\vee) = \frac{1}{2}a_1b_3.
\]
The result for $6$-valent vertices thus follows. For a $4$-valent vertex, suppose that we have Lusztig cycles $C$, $C'$ with weights $a_1, a_2$ and $b_1, b_2$ at the top, respectively. Then the top boundary intersection number is $a_1b_2 - a_2b_1$, while the bottom boundary intersection number is $a'_1b'_2 - a'_2b'_1 = a_2b_1 - a_1b_2$. Thus, the difference between the boundary intersection numbers is $0$, as required. 
\end{proof}

\begin{corollary}
\label{cor: quiver Zam}
Let $\mathfrak{W}_1,\mathfrak{W}_2:\beta\to\beta'$ be two weaves with no trivalent vertices. Suppose that $C_{\mathfrak{W}_1},C'_{\mathfrak{W}_1}$ are Lusztig cycles in $\mathfrak{W}_1$, $C_{\mathfrak{W}_2},C'_{\mathfrak{W}_2}$ are Lusztig cycles in $\mathfrak{W}_2$, and the initial weights of $C_{\mathfrak{W}_{1}}$ $($resp.~$C'_{\mathfrak{W}_{1}})$ are the same as those of $C_{\mathfrak{W}_{2}}$ $($resp.~$C'_{\mathfrak{W}_{2}})$. Then $
\sharp_{\mathfrak{W}_1}(C\cdot C')=\sharp_{\mathfrak{W}_2} (C\cdot C').
$
\end{corollary}

\begin{proof}
Indeed, both intersection numbers are equal to $
\sharp_{\beta}(C\cdot C')-\sharp_{\beta'}(C\cdot C').
$
Note that by Lemma \ref{lem: cycle output} the output weights of $C_{\mathfrak{W}_{1}}$ $($resp.~$C'_{\mathfrak{W}_{1}})$ are the same as those of $C_{\mathfrak{W}_{2}}$ $($resp.~$C'_{\mathfrak{W}_{2}})$.
\end{proof}

\noindent Corollary \ref{cor: quiver Zam} implies that if two weaves $\fW,\fW'$ that are equivalent via an equivalence that uses only $4$- and $6$-valent vertices, then the corresponding quivers $Q_{\fW}$ and $Q_{\fW'}$ coincide. Let us now prove the stronger result that any two equivalent weaves yield the same quiver. For that, it suffices to study weave equivalences that involve $3$-valent vertices, which locally are those in Example \ref{ex: 1212}, see Figure \ref{fig: weave eq1}.



\begin{lemma}
\label{lem: quiver 1212}
Let $\fW_1:1212\to 2122\to 212$ and $\fW_2:
1212\to 1121\to 121\to 212$ and $C_{i}, C'_{i}$ be Lusztig cycles on $\fW_i$, $i=1,2$. Suppose that the initial weights of $C_{1}$ $($resp.~$C'_{1})$ coincide with those of $C_{2}$ $($resp. $C'_{2})$, and let $\g^{i}_{v}$ be the cycle originating at the unique trivalent vertex of $\fW_{i}$. Then we have the equalities:

\begin{itemize}
    \item[(1)] $\sharp_{\fW_{1}}(C_{1},\g^{1}_v)=\sharp_{\fW_{2}}(C_{2},\g^{2}_v)$
    
    \item[(2)] $\sharp_{\fW_1}(C_1,C'_1)=\sharp_{\fW_2}(C_2,C'_2)$
\end{itemize}
\end{lemma}
\begin{proof}
For (1), we follow the notations of Example \ref{ex: 1212}, so $C_1,C_2$ have weights $a,b,c,d$ on the top. For the weave $\fW_1$, the only local intersection is at trivalent vertex and thus
$$
\sharp_{\fW_1}(C_1,\g_v^{1})=a+b-\min(a,c)-d.
$$
For $\fW_1$, the local intersection at trivalent vertex equals $a-c-d+\min(b,d)$, while the local intersection at the bottom 6-valent vertex equals 
$b+c-\min(b,d)-\min(a,c)$, as in Example \ref{ex: local intersection 100}. By combining these together we also obtain
$$
\sharp_{\fW_2}(G,G_v)=(a-c-d+\min(b,d))+(b+c-\min(b,d)-\min(a,c))=a+b-\min(a,c)-d.
$$
\noindent For (2), Lemma \ref{lem: add 101}, implies that adding $(1,0,1,0)$ and $(0,1,0,1)$ on top of either weave does not change the intersection number at any vertex of either weave. Thus we assume that $C_i$ and $C'_i$ have weights $(a,0,0,b)$ and $(c,0,0,d)$ on top, where $a,b,c,d\in\Z$. Note that $a,b,c,d$ could be negative here. Denote $m_{ab}:=\min([a]_{+},b)$ and $m_{cd}:=\min([c]_{+},d)$ and let us compute the intersection numbers for $\fW_1$. At the 6-valent vertex we have
$$
\frac{1}{2}\left|
\begin{matrix}
1 & 1 & 1\\
a & 0 & 0\\
c & 0 & 0\\
\end{matrix}
\right|-
\frac{1}{2}\left|
\begin{matrix}
1 & 1 & 1\\
-[a]_{-} & [a]_{-} & [a]_{+}\\
-[c]_{-} & [c]_{-} & [c]_{+}\\
\end{matrix}
\right|=
[a]_{+}[c]_{-}-[a]_{-}[c]_{+}.
$$
At the 3-valent vertex we have
$$
\left|
\begin{matrix}
1 & 1 & 1\\
[a]_{+} & m_{ab} & b\\
[c]_{+} & m_{cd} & d\\
\end{matrix}
\right|=m_{ab}([d]_{+}+[d]_{-}-[c]_{+})+m_{cd}([a]_{+}-[b]_{+}-[b]_{-})+
$$
$$
[b]_{+}[c]_{+}+[b]_{-}[c]_{+}-[a]_{+}[d]_{+}-[a]_{+}[d]_{-}.
$$
The intersection numbers for $\fW_2$ are as follows; at the top 6-valent vertex we have
$$
\frac{1}{2}\left|
\begin{matrix}
1 & 1 & 1\\
0 & 0 & b\\
0 & 0 & d\\
\end{matrix}
\right|-
\frac{1}{2}\left|
\begin{matrix}
1 & 1 & 1\\
[b]_{+} & [b]_{-} & -[b]_{-}\\
[d]_{+} & [d]_{-} & -[d]_{-}\\
\end{matrix}
\right|=
[b]_{-}[d]_{+}-[b]_{+}[d]_{-}.
$$
At the 3-valent vertex we have:
$$
\left|
\begin{matrix}
1 & 1 & 1\\
a & m_{ab}+[a]_{-}-[b]_{-} & [b]_{+}\\
c & m_{cd}+[c]_{-}-[d]_{-} & [d]_{+}\\
\end{matrix}
\right|=
m_{ab}([d]_{+}-[c]_{+}-[c]_{-})+m_{cd}([a]_{+}+[a]_{-}-[b]_{+})+
$$
$$
-[b]_{-}[d]_{+}+[b]_{+}[c]_{+}-[a]_{-}[d]_{-}+[a]_{+}[c]_{-}-[a]_{+}[d]_{-}-[a]_{-}[c]_{+}+[b]_{-}[c]_{-}+[b]_{-}[c]_{+}+[b]_{+}[d]_{-}-[a]_{+}[d]_{+},
$$
where we have used the equality $\min(a,[b]_+)=\min([a]_+,b)+[a]_{-}-[b]_{-}$. Finally, at the bottom 6-valent vertex we have
$$
\frac{1}{2}\left|
\begin{matrix}
1 & 1 & 1\\
m_{ab}+[a]_{-}-[b]_{-} & [b]_{-} & -[b]_{-}\\
m_{cd}+[c]_{-}-[d]_{-} & [d]_{-} & -[d]_{-}\\
\end{matrix}
\right|-
\frac{1}{2}\left|
\begin{matrix}
1 & 1 & 1\\
-[a]_{-} & [a]_{-} & m_{ab}\\
-[c]_{-} & [c]_{-} & m_{cd}\\
\end{matrix}
\right|=
m_{ab}([d]_{-}+[c]_{-})-m_{cd}([b]_{-}+[a]_{-})+
$$
$$
[a]_{-}[d]_{-}-[b]_{-}[c]_{-}.
$$
By adding these local intersection indices, we obtain $\sharp_{\fW_1}(C_1,C'_1)=\sharp_{\fW_2}(C_2,C'_2)$ as required.
\end{proof}

\begin{corollary}\label{cor: quivers coincide}
Let $\fW_1,\fW_2:\beta\to\delta(\beta)$ be  two equivalent weaves, where the same braid word has been fixed for $\delta(\beta)$. Then the quivers $Q_{\fW_1}$ and $Q_{\fW_2}$ coincide.
\end{corollary}

Let us now study the effect that weave mutation has on the associated quivers. We use the following: 

\begin{lemma}
\label{lem: arrows to I cycle}
Let $a,b,c,d\in\Z$, then the following two identities hold:
\begin{itemize}
    \item[$(1)$] $[b-a+\min(a,b,c)-c]_{-}=-[a+c-b-\min(a,b,c)]_{+}=\min(a,b)-c-a+\min(b,c).$
    \item[$(2)$] $[b-a+\min(a,b,c)-c]_{+}=-[a+c-b-\min(a,b,c)]_{-}=b+\min(a,b,c)-\min(a,b)-\min(b,c).$
\end{itemize}
\end{lemma}

\begin{proof}
Part (1) is a tropicalization of the identity
$$
1+\frac{t^b(t^a+t^b+t^c)}{t^at^c}=\frac{(t^a+t^b)(t^b+t^c)}{t^at^c},
$$
and (2) follows by
$
[(b-a)+(\min(a,b,c)-c),0]_{+}+[(b-a)+(\min(a,b,c)-c)]_{-}=(b-a)+(\min(a,b,c)-c).
$
\end{proof}

\begin{lemma}
\label{lem: quiver mutation}
Let $\fW_1,\fW_2$ be the two Demazure weaves for $\sigma_1^3$ depicted in Figure \ref{fig: weave mutation}. Consider the following three types of cycles: $C,C'$ are Lusztig cycles with initial weights $a,b,c$ and $a',b',c'$; $\g_{v_1}$ is the short cycle connecting the trivalent vertices; $\g_{v_2}$ is the cycle exiting the bottom trivalent vertex. Then:
\begin{itemize}[itemsep=7pt]
\item[$(1)$] $\sharp_{\fW_1}(C,\g_{v_1})=-\sharp_{\fW_2}(C,\g_{v_1})$.

\item[$(2)$] $\sharp_{\fW_1}(\g_{v_1},\g_{v_2})=1,\sharp_{\fW_2}(\g_{v_1},\g_{v_2})=-1.$

\item[$(3)$] $ \sharp_{\fW_2}(C,\g_{v_2})=\sharp_{\fW_1}(C,\g_{v_2})+\left[\sharp_{\fW_1}(C,\g_{v_1})\right]_{+}=\sharp_{\fW_1}(C,\g_{v_2})-\left[\sharp_{\fW_1}(C,\g_{v_1})\right]_{+}\left[\sharp_{\fW_1}(\g_{v_2},\g_{v_1})\right]_{-}.$

\item[$(4)$] $
\sharp_{\fW_2}(C,C')=\sharp_{\fW_1}(C,C')-[\sharp_{\fW_1}(C,\g_{v_1})]_{+}[\sharp_{\fW_1}(C',\g_{v_1})]_{-}+ [\sharp_{\fW_1}(C,\g_{v_1})]_{-}[\sharp_{\fW_1}(C',\g_{v_1})]_{+}.$
\end{itemize}
\end{lemma}

\begin{proof}
For (1), we have
$
\sharp_{\fW_1}(C,\g_{v_1})=(a-b)+(c-\min(a,b,c)),
$ and similarly
$
\sharp_{\fW_2}(C,\g_{v_1})=(b-c)+(\min(a,b,c)-a),
$. Part (2) is also immediate. For (3) we have 
$$
\sharp_{\fW_1}(C,\g_{v_2})=
\min(a,b)-c,\quad 
\sharp_{\fW_2}(C,\g_{v_2})=
a-\min(b,c). 
$$
By Lemma \ref{lem: arrows to I cycle} we obtain 
$
\sharp_{\fW_2}(C,\g_{v_2})-\sharp_{\fW_1}(C,\g_{v_2})=a+c-\min(a,b)-\min(b,c)=
\left[\sharp_{\fW_1}(G,G_{v_1})\right]_{+},
$ as required. Finally, for Part (4), let us denote $m=\min(a,b,c),m'=\min(a',b',c')$. Then we have
$$
\sharp_{\fW_2}(C,C')-\sharp_{\fW_1}(C,C')=
$$
$$
\left|
\begin{matrix}
1 & 1 & 1\\
b & \min(b,c) & c\\
b' & \min(b',c') & c'\\
\end{matrix}
\right|+
\left|
\begin{matrix}
1 & 1 & 1\\
a & m & \min(b,c)\\
a' & m' & \min(b',c')\\
\end{matrix}
\right|-
\left|
\begin{matrix}
1 & 1 & 1\\
a & \min(a,b) & b\\
a' & \min(a',b') & b'\\
\end{matrix}
\right|-
\left|
\begin{matrix}
1 & 1 & 1\\
\min(a,b) & m & c\\
\min(a',b') & m' & c'\\
\end{matrix}
\right|=
$$
$$
(a+c)(b'+m')-(b+m)(a'+c')-
$$
$$
(\min(a,b)+\min(b,c))(b'+m'-a'-c')+(\min(a',b')+\min(b',c'))(b+m-a-c)=
$$
$$
(a+c-\min(a,b)-\min(b,c))(b'+m'-\min(a',b')-\min(b',c'))-
$$
$$
(b+m-\min(a,b)-\min(b,c))(a'+c'-\min(a',b')-\min(b',c')).
$$
By Lemma \ref{lem: arrows to I cycle} this equals
$
-[\sharp_{\fW_1}(C,\g_{v_1})]_{+}[\sharp_{\fW_1}(C',\g_{v_1})]_{-}+ [\sharp_{\fW_1}(C,\g_{v_1})]_{-}[\sharp_{\fW_1}(C',\g_{v_1})]_{+}.
$\end{proof}

\begin{theorem}\label{thm:mutation quivers}
Let $\fW_1,\fW_2:\beta\to\delta(\beta)$ be  two Demazure weaves, where the same braid word has been fixed for $\delta(\beta)$. Then the corresponding quivers $Q_{\fW_1}$ and $Q_{\fW_2}$ are related by a sequence of mutations. 
\end{theorem}

\begin{proof}
By Lemma \ref{lem: demazure classification} any two such Demazure weaves are related by a sequence of equivalence moves and weave mutations. By Corollary \ref{cor: quiver Zam} and Lemma \ref{lem: quiver 1212} equivalence moves for weaves do not change the quiver. By Lemma \ref{lem: quiver mutation} and \eqref{eq: quiver mutation plus minus} a weave mutation corresponds to the quiver mutation in the cycle $\g_{v_1}$ connecting two trivalent vertices.
\end{proof}


\subsection{Frozen vertices}\label{ssec:frozen} Let $\fW: \beta \to \delta(\beta)$ be a Demazure weave and let $Q_{\fW}$ be its associated quiver. 
Recall that the vertices of $Q_{\fW}$ are in bijection with the trivalent vertices of $\fW$. In this section, we specify which vertices of $Q_{\fW}$ are frozen.

\begin{definition}\label{def: frozen}
Let $v$ be a trivalent vertex of $\fW$, equivalently a vertex of the quiver $Q_{\fW}$, and $\gamma_v$ its associated cycle. We say that $v$ is \emph{frozen} if there exists an edge $e\in E(\fW)$ on the southern boundary of $\fW$ such that $\g_{v}(e) \neq 0$.
\end{definition}

 Definition \ref{def: frozen} allows us to upgrade $Q_\fW$ to an ice quiver. Corollary \ref{cor: quivers coincide} is refined as follows. 

\begin{lemma}
Let $\fW_1,\fW_2:\beta\to\delta(\beta)$ be  two equivalent weaves. 
Then the quivers $Q_{\fW_1}$ and $Q_{\fW_2}$ coincide as ice quivers, i.e.~their frozen vertices coincide. 
\end{lemma}
\begin{proof}
For equivalences with only $4$- and $6$-valent vertices, let $v$ be a frozen trivalent vertex and assume that the equivalence moves in the weave are performed after the appearance of the trivalent vertex $v$; otherwise the result is clear. Then, in the area where the moves are performed, $\g_v$ is a Lusztig cycle and the result in this case now follows by Lemma \ref{lem: cycle output}.
\noindent Now assume that $\fW_1$ and $\fW_2$ are related by a single equivalence involving a $3$-valent vertex, i.e.~they are related by a move as in Example \ref{ex: 1212}. If $v$ is not the trivalent vertex involved in the move, the computations in Example \ref{ex: 1212} imply the result. Else, the values of $\g_{v}$ on the bottom of both weaves in Figure \ref{fig: weave eq1} are $(0,0,1)$ and the result follows. 
\end{proof} 

\noindent The behavior that weave mutation has on these ice quivers is readily computed as well:

\begin{lemma}
Let $\fW_1,\fW_2$ be two weaves related by one mutation at a trivalent vertex $v \in Q_{\fW_1}$. Then:
\begin{itemize}
    \item[$(1)$] The trivalent vertex $v$ is not frozen.
    \item[$(2)$] The quivers $\mu_{v}(Q_{\fW_{1}})$ and $Q_{\fW_{2}}$ coincide as ice quivers. 
\end{itemize}
\end{lemma}
\begin{proof}
Part (1) is clear by the definition of weave mutation and Definition \ref{def: frozen}. For Part (2) it suffices to notice that, if $C$ is a cycle entering either one of the weaves in Figure \ref{fig: weave mutation} with weights $(a,b,c)$, the exiting weight is $\min(a,b,c)$, independently of the weave.
\end{proof}


\subsection{Quiver comparison for $\Delta\beta$}\label{sec:SW quiver} In the study of braid varieties of the form $X(\Delta\beta)$, Lemma \ref{lem: conf as braid variety} established the isomorphism $X(\Delta\beta)\cong\Conf(\beta)$. Subsection \ref{sec:double bs} also described the quiver $Q(\beta)$, following \cite{SW}, which gives a cluster structure on the configuration space $\Conf(\beta)$. The purpose of the present subsection is to show that the quiver $Q_{\rind{\Delta\beta}}$ for the right inductive weave $\rind{\Delta\beta}$, see Definitions \ref{def:inductive_weave} and \ref{def:quiver}, coincides with the quiver $Q(\beta)$.\\

\noindent In Subsection \ref{ssec:quiver} we assigned a sequence of roots $\rho_1, \dots, \rho_r$, via Equation \eqref{eq: def gamma}, to a horizontal slice of a weave spelling the word $\sigma_{i_{1}}\cdots \sigma_{i_{r}}$. By definition, in that case $\rho_k$ is said to label the $k$-th strand of the weave. We now explain how strands labeled by simple roots are of particular relevance, starting with the following observation:

\begin{lemma}\label{lem: exchange}
$(1)$ The word $\beta$ is reduced if and only if all roots $\rho_1, \dots, \rho_k$ are positive.

$(2)$ Let $w = s_{i_{1}}\cdots s_{i_{r}} \in W$ satisfy $\ell(w) = r$ and assume that there exists a simple root $\alpha_j$, $j \in \dynkin$, such that $w(-\alpha_{i_{r}}) = \alpha_j$. Then $w$ has a reduced expression starting with $s_j$.
\end{lemma}
\begin{proof}
Part (1) is well known, see e.g \cite[Proposition 4.2.5]{BB}. Let us prove Part (2). Since $s_{j}w(-\alpha_{i_{r}}) = -\alpha_{j}$ is a negative root, by (a) the word $s_{j}s_{i_{1}}\cdots s_{i_{r}}$ is not reduced; since $s_{i_{1}}\cdots s_{i_{r}}$ is reduced the result follows.  
\end{proof}



\begin{lemma}\label{lem: no middle}
Let $\sigma_{i_{1}}\cdots \sigma_{i_{r}}$ be a horizontal slice of a weave which is reduced. Suppose that the $k$-th strand of this weave is labeled by a simple root $\alpha_j$. Then the following holds:
\begin{itemize}
    \item[$(1)$] The $k$-th strand cannot enter a six-valent vertex through the middle.
    \item[$(2)$] If the $k$-th strand enters a six-valent vertex through the right $($resp.~left$)$ then the $k-2$-nd $($resp.~ $k+2$-nd$)$ strand of the next horizontal slice is labeled by $\alpha_j$.
    \item[$(3)$] If the $k$-th strand enters a 4-valent vertex through the right $($resp.~left$)$ then the $k-1$-st $($resp.~$k+1$-st$)$ strand of the next horizontal slice is labeled by $\alpha_j$. 
\end{itemize}
\end{lemma}
\begin{proof}
The assumption states $s_{i_{1}}\cdots s_{i_{k}}(-\alpha_{i_{k}}) = \alpha_j$, for each of the items we then have:
\begin{enumerate}
    \item Assume that $i_{k-1}$ is adjacent to $i_{k}$ and $i_{k+1}=i_{k-1}$. Let $w=s_{i_{1}}\cdots s_{i_{k-2}}$. Then $\alpha_j = s_{i_{1}}\cdots s_{i_{k}}(-\alpha_{i_k}) =  w(\alpha_{i_{k}} + \alpha_{i_{k-1}})$ and it follows from Lemma \ref{lem: exchange}(a) that $ws_{i_{k-1}}$ and $ws_{i_{k}}$ cannot be simultaneously reduced.   Since $ws_{i_{k-1}}$ is reduced, $ws_{i_{k}}$ is not and $ws_{i_{k-1}}s_{i_k}s_{i_{k+1}}=ws_{i_k}s_{i_{k-1}}s_{i_{k}}$ is not reduced either. Contradiction.
    \item This is a check based on $s_is_js_i(-\alpha_i) = \alpha_j$ if $i$ and $j$ are adjacent.
    \item This is also a check. \qedhere
\end{enumerate} 
\end{proof}

\begin{corollary}\label{cor: simple root}
Let $\Delta = \sigma_{i_{1}}\cdots \sigma_{i_{r}}$ be any reduced lift of $w_0$ defining the positive roots $\rho_k$, $k = 1, \dots, r$. 
For each $j \in \dynkin$, consider
$
j_0 := \min\{1 \leq k \leq r \mid \sigma_{j}\sigma_{i_{1}}\cdots \sigma_{i_{k}} \; \text{is not reduced}\}$ and $\ j_1 := \min\{1 \leq k \leq r \mid \rho_k = \alpha_j\}.
$
Then, $j_0 = j_1$. 
\end{corollary}
\begin{proof}
Lemma \ref{lem: exchange} implies $j_1 \geq j_0$. For $j_1 \leq j_0$, Lemma \ref{lem: no middle} (b) and (c) imply that it is enough to find one reduced expression for $w_0$ that satisfies this property. The expressions given in e.g. \cite[Table 1]{benkart} work, i.e.~there exist such reduced expressions.  
\end{proof}

\begin{remark}\label{rmk: gamma opposite}
We have defined the sequence of roots $\rho_1, \dots, \rho_r$ by reading $\beta$ in a left-to-right fashion. We may read it in the opposite order to get a different sequence of roots $\rho'_r = \alpha_{i_r}$, $\rho'_{r-1} = s_{i_r}(\alpha_{i_{r-1}}), \rho'_{r-2} = s_{i_r}s_{i_{r-1}}(\alpha_{i_{r-2}}) \dots, \rho'_1 = s_{i_r}\cdots s_{i_2}(\alpha_{i_1})$. Alternatively, $\rho'_1, \dots, \rho'_r$ is the sequence $\rho$ of roots for the opposite word $\ateb := \sigma_{i_{r}}\cdots \sigma_{i_{1}}$, but ordered oppositely: $\rho'(\beta)_{i} = \rho(\ateb)_{r-i+1}$. Lemmas \ref{lem: exchange}, \ref{lem: no middle} and Corollary \ref{cor: simple root} are still valid with the appropriate modifications.
\end{remark}

Let us now study the weaves of type $\rind{\Delta\beta}$ inductively. Suppose that $\rind{\Delta\beta}$ has been given, with its lower boundary being reduced expression for $\Delta$, that we also refer to as $\Delta$. By the right-handed version of Corollary \ref{cor: simple root}, the weave $\rind{\Delta\beta\sigma_{i}}$ is obtained by taking the first strand (counting right to left) such that $\rho'_{k} = \alpha_{i}$, and move this strand to the right in order to obtain a reduced word for $\Delta$ that ends in $\sigma_i$. By Lemma \ref{lem: no middle}, in the process of doing this the strand will \emph{not} enter a $6$-valent vertex from the middle so, if there was a cycle containing this strand, it will not bifurcate. Lemma \ref{lem: no middle} also implies that any cycle containing a strand labeled by a simple root will not bifurcate. Once we have finished moving the strand to the right, we pair it with the strand coming from the rightmost $\sigma_i$. This ends the cycle containing the strand that has been moved (if any) and creates a new cycle starting at the new trivalent vertex. Note that this new cycle is labeled by the positive root $\alpha_i$. This discussion implies the following:

\begin{lemma}\label{lem: no bifurcations}
Every cycle in the inductive weave $\rind{\Delta\beta}$ is  non-bifurcating, and all of its weights are equal to 0 or 1.
\end{lemma}

For finer information, we first fix some notation. For every enumeration $\kappa$ of vertices of the Dynkin diagram $\dynkin$, we have a reduced expression $\Delta(\kappa) = \Delta_{n}^{(\kappa)}\cdots \Delta_{1}^{(\kappa)}$ of the half-twist $\Delta$, so that $\Delta_{m}^{(\kappa)}\cdots \Delta_{1}^{(\kappa)}$ is a reduced expression of the longest element of the Weyl group of the Dynkin diagram consisting of the first $m$ vertices (under the enumeration $\kappa$) of $\dynkin$. Let us {\bf fix} an enumeration of $\dynkin$, and we denote $\Delta = \Delta_n\cdots \Delta_1$ the reduced decomposition of $\Delta$ corresponding to this fixed enumeration.  Note that this implies that the first strand (reading from right to left) that is labeled by $\alpha_i$ (as in Remark \ref{rmk: gamma opposite}) is the leftmost strand on $\Delta_i$. Note also that every other enumeration 
corresponds to an element in the symmetric group $S_n$. For $i = 1, \dots, n$, let us denote by $\Delta(i)$ a reduced expression of $\Delta$ corresponding to the enumeration given by the permutation $(12\cdots i)$, that is, corresponding to the enumeration $(i, 1, \dots, i-1, i+1, \dots, n)$ of the vertices of $\dynkin$. Note that the rightmost strand of $\Delta(i)$ has color $i$ and is labeled by $\sigma_{i}$.\\

\noindent In order to obtain the inductive weave $\rind{\Delta\beta}$, we iteratively build the weaves $\mathfrak{w}_0 := \rind{\Delta}$, $\mathfrak{w}_{1} := \rind{\Delta\sigma_{i_{1}}}$, $\mathfrak{w}_{2} := \rind{\Delta\sigma_{i_{1}}\sigma_{i_2}}, \cdots, \mathfrak{w}_{r} := \rind{\Delta\beta}$. In fact, we build these weaves as follows:
\begin{itemize}
    \item[-] The bottom boundary of the weave $\mathfrak{w}_{k}$ is  $\Delta = \Delta_n\cdots \Delta_1$ for every $k = 0, \dots, r$.
    \item[-] To build $\mathfrak{w}_{k+1}$ from $\mathfrak{w}_{k}$, we use braid moves to change the bottom boundary of $\mathfrak{w}_{k}$ to $\Delta(i_{k+1})$.  
    The rightmost strand of $\Delta(i_{k+1})$ is labeled by $\sigma_{i_{k+1}}$, and we may form a new trivalent vertex in $\mathfrak{w}_{k+1}$. After, we use braid moves to return the bottom boundary to $\Delta$. 
\end{itemize}

\begin{definition}
Let $\mathfrak{W}$ be a weave and $Q_{\mathfrak{W}}$ its corresponding quiver. For $i \in \dynkin$, a vertex of $Q_{\mathfrak{W}}$ is said to be an $i$-vertex if it corresponds to an $i$-colored trivalent vertex of $\mathfrak{W}$. $($Compare with Section \ref{sec:double bs}.$)$
\end{definition}

\noindent By Lemma \ref{lem: no bifurcations}, the quiver $Q_{\mathfrak{w}_{k}}$ has  a frozen $i$-vertex if and only if there exists a (necessarily unique) cycle that has a nonzero weight on the leftmost strand of $\Delta_{i}$ in the bottom boundary. In particular, $Q_{\mathfrak{w}_{k}}$ has at most one frozen $i$-vertex for every $i \in \dynkin$.

\begin{proposition}\label{prop: conf quiver}
Let $i \in \dynkin$ and let $f(i) \in Q_{\rind{\Delta\beta}}$ be the unique $($if any$)$ frozen $i$-vertex. Then, the quiver $Q_{\rind{\Delta\beta\sigma_i}}$ is obtained from $Q_{\rind{\Delta\beta}}$ by the following procedure.
\begin{itemize}
\item[$(1)$] Thaw the vertex $f(i)$ and add a new frozen vertex $f'(i)$, together with an arrow $f(i) \to f'(i)$. 
\item[$(2)$] If $j$ is adjacent to $i$ in $\dynkin$ and the vertex $f(j)$ was added after $f(i)$, add an arrow $f(j) \to f(i)$. 
\item[$(3)$] If $i$ is adjacent to $j$ in $\dynkin$, add an arrow of weight $1/2$ from the frozen vertex $f'(i)$ to the frozen vertex $f(j)$. 
\end{itemize}
\end{proposition}
\begin{proof}
To obtain the weave $\rind{\Delta\beta\sigma_i}$ from $\rind{\Delta\beta}$, we have to take the left-most strand of $\Delta_i$ and move it to the right. The vertex $f(i)$ exists if and only if this strand is carrying a cycle, that we call $C(i)$. By Lemma \ref{lem: no bifurcations}, the cycle will end at the new trivalent vertex in $\rind{\Delta\beta\sigma_{i}}$, which corresponds to $f'(i)$. Thus, Part (1) is clear. For Part (2), we use Lemma \ref{lem: boundary int} to count the new intersections that are formed in $\rind{\Delta\beta\sigma_i}$. We only look at the portion of the weave that is between the bottom boundary of $\rind{\Delta\beta}$ (corresponding to the braid word $\Delta$) and the new trivalent vertex in $\rind{\Delta\beta\sigma_{i}}$ (so that the bottom boundary is $\Delta(i)$). 
By Lemma \ref{lem: no middle}, the top and bottom boundaries of the cycle $C(j)$ consist of a single strand labeled by $\alpha_j$. The permutation $p_{i} = (12\cdots i)$ satisfies the property that if $a < b$ but $p_{i}(b) < p_{i}(a)$, then $b = i$. Thus, the only new intersections involve the cycle $C(i)$, and these intersections may only involve cycles $C(j)$ where $j$ is adjacent to $i$ in $\dynkin$ and $j < i$. Now we compute
\begin{equation}\label{eqn: int ind weave 1}
\sharp_{\mathrm{top}}C(i)\cdot C(j) - \sharp_{\mathrm{bottom}}C(i)\cdot C(j) = \begin{cases} -1 & j < i, \, \text{and} \, (\alpha_i, \alpha_j^\vee) \neq 0 \\ 0 & \text{else.} \end{cases}
\end{equation}
Let us now look at the intersections that are formed after the trivalent vertex. These intersections will only involve $C'(i)$, where $C'(i)$ is the cycle that has started at this trivalent vertex. Similarly, we have
\begin{equation}\label{eqn: int ind weave 2}
\sharp_{\mathrm{top}}C'(i)\cdot C(j) - \sharp_{\mathrm{bottom}}C'(i)\cdot C(j) = \begin{cases} +1 & j < i, \, \text{and} \, (\alpha_i, \alpha_j^\vee) \neq 0 \\ 0 & \text{else.} \end{cases}
\end{equation}
Now, if $(\alpha_i, \alpha_j^\vee) \neq 0$ and $f(j)$ was added after $f(i)$, then we observe an arrow $f(j) \to f(i)$ from \eqref{eqn: int ind weave 1} if $j < i$ and from \eqref{eqn: int ind weave 2} if $i < j$. If $(\alpha_i, \alpha_j^\vee) \neq 0$ and $f(j)$ was added before $f(i)$, then either we do not observe any intersections (if $i < j$) or the terms \eqref{eqn: int ind weave 1} and \eqref{eqn: int ind weave 2} cancel. Finally, we need to study the (half-weighted) arrows between $f'(i)$ and $f(j)$ for $j \neq i$. It follows easily from \eqref{eqn: int ind weave 2} and the fact that $f(j)$ is on a strand with root $\alpha_j$ that (3) above holds. The result follows.
\end{proof}

\noindent Inductively, the analysis above concludes the following result. 
\begin{corollary}\label{cor: coincidence quivers}
Let $\beta\in\Br_W^+$ be a braid word and $Q(\beta)$ the quiver for the initial seed for $\Conf(\beta)$, as introduced in Section \ref{sec:double bs}. Then $Q_{\rind{\Delta\beta}}=Q(\beta)$.
\end{corollary}


\subsection{Quivers for inductive weaves}\label{sect:ind quiver} The inductive weaves $\lind{\beta}, \rind{\beta}$ in Definition \ref{def:inductive_weave} depend on the braid \emph{word} for $\beta$ and not only on the braid $\beta$. In this section, we examine the dependency of the quivers $Q_{\lind{\beta}}$ and $Q_{\rind{\beta}}$ on the choice of braid word. First, we have the following result.

\begin{proposition}\label{prop:inductive mutation}
Let $i, j\in\dynkin$ be adjacent vertices in the Dynkin diagram $\dynkin$. Consider the two braid words $\beta = \beta_2\sigma_{i}\sigma_{j}\sigma_{i}\beta_1$ and $\beta' = \beta_2\sigma_{j}\sigma_{i}\sigma_{j}\beta_1$, which differ by a single braid move. Then the following holds:

\begin{itemize}
    \item[$(1)$] The quivers $Q_{\lind{\beta}}$ and $Q_{\lind{\beta'}}$ are identical if $\delta(\sigma_{i}\sigma_{j}\sigma_{i}\beta_1) \neq \delta(\beta_1)$. Similarly, $Q_{\rind{\beta}}$ and $Q_{\rind{\beta'}}$ are identical if $\delta(\beta_2\sigma_i\sigma_j\sigma_i) \neq \delta(\beta_2)$.
    
    \item[$(2)$] Else, the quivers $Q_{\lind{\beta}}$ and $Q_{\lind{\beta'}}$ $($resp. $Q_{\rind{\beta}}$ and $Q_{\rind{\beta'}})$ are related by a single mutation at the vertex given by the middle letter in the braid move.
\end{itemize}
\end{proposition}

\begin{proof}
Let us focus on right inductive weaves, as the proof for the left inductive weave $\overleftarrow{\mathfrak{w}}$ is analogous. The statement (2) follows by studying the two right inductive weaves in Figures \ref{fig:mutation 1} and \ref{fig:mutation 2}, which correspond to those for $\beta$ and $\beta'$ respectively. In these figures, the cycles are indicated with colors, as depicted on the right, and the quivers are related by a mutation at the green vertex. 

The proof of (1) is similar. The key observation is that any two weaves starting from the same  braid word in the list
$$
\sigma_i\sigma_j\sigma_i\sigma_j;\ \sigma_i\sigma_j\sigma_i\sigma_i\sim \sigma_j\sigma_i\sigma_j\sigma_i;\ \sigma_i\sigma_j\sigma_i\sigma_j\sigma_i; \ \sigma_i\sigma_j\sigma_i\sigma_i\sigma_j\sim \sigma_j\sigma_i\sigma_j\sigma_i\sigma_j
$$
and ending at $\sigma_i\sigma_j\sigma_i$ are equivalent.
\end{proof}

\begin{figure}[ht!]
\centering
\includegraphics[scale=2]{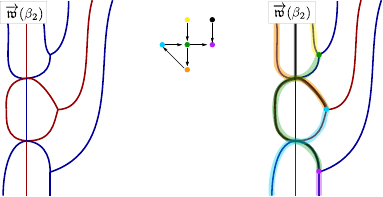}
\caption{(Left) Right inductive weave for $\beta_2 \sigma_i\sigma_j\sigma_i$. (Center) The intersection quiver associated to some of the Lusztig cycles. (Right) Some of the Lusztig cycles depicted in the weave. Their colors match the colors of the corresponding vertices in the quiver.}
\label{fig:mutation 1}
\end{figure}

\begin{figure}[ht!]
\centering
\includegraphics[scale=2]{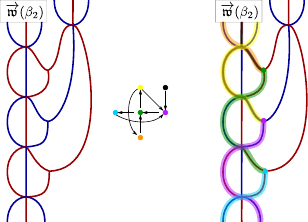}
\caption{(Left) The right inductive weave for $\beta_2 \sigma_{j}\sigma_{i}\sigma_{j}$. (Center) The intersection quiver for some of its Lusztig cycles. The quiver here is obtained from that in Figure \ref{fig:mutation 1} by mutating at the green vertex. Note that the arrow from the blue vertex to the purple vertex appears only if these cycles are not frozen in the right inductive weave of $\beta$. (Right) Some of the Lusztig cycles in the weave.}
\label{fig:mutation 2}
\end{figure}

\noindent Finally, we can establish the relation between the quivers $Q_{\lind{\beta}}$ and $Q_{\lind{\sigma_i\beta}}$. It reads as follows:

\begin{lemma}\label{lem: add trivalent quiver} Let $\beta$ be a braid word and consider the possible two cases:  $\delta(\sigma_i\beta) = \delta(\beta)$ or  $\delta(\sigma_i\beta) \neq \delta(\beta)$.\\

If $\delta(\sigma_i\beta) = \delta(\beta)$, let $v\in\lind{\sigma_i\beta}$ be the last trivalent vertex of the weave $\lind{\sigma_i\beta}$. Then:
\begin{itemize}
    \item[$(1)$] The vertex $v$ is frozen and it is a source in $Q_{\lind{\sigma_i\beta}}$.
    \item[$(2)$] The quiver $Q_{\lind{\beta}}$ can be obtained from $Q_{\lind{\sigma_i\beta}}$ by the following procedure:
    \begin{itemize}
        \item Remove the frozen vertex $v$.
        \item Freeze all vertices that were incident with $v$.
        \item Remove possible arrows between frozen vertices. 
    \end{itemize}
\end{itemize}
\noindent If else $\delta(\sigma_i\beta) = s_i\delta(\beta)$, then the quivers $Q_{\lind{\sigma_i\beta}}$ and $Q_{\lind{\beta}}$ coincide.
\end{lemma}

\begin{proof} First we consider the case  $\delta(\sigma_i\beta) = \delta(\beta)$. By Definition \ref{def:inductive_weave}, the left inductive weave $\lind{\sigma_i\beta}$ is obtained from the weave $\lind{\beta}$ by adding a new $i$-colored trivalent vertex $v$. The vertex in the quiver associated to the Lusztig cycle for this trivalent vertex $v$ is frozen, because of Definition \ref{def: frozen} and the fact that the Lusztig cycle $\gamma_v$ flows straight down to the southern boundary of $\lind{\sigma_i\beta}$. Independently, the (upper) left arm of the trivalent vertex $v$ goes all the way to the top of $\lind{\sigma_i\beta}$ and thus $v$ is a source in the quiver $Q_{\lind{\sigma_i\beta}}$. (See for example Definition \ref{def:locint_3valent} or cf.~ Figure \ref{fig:arrows-trivalent-vertices} in Subsection \ref{ssec:first_example}.) This directly establishes $(1)$ and, by construction, also $(2)$. Second, in the case that $\delta(\sigma_i\beta) = s_i\delta(\beta)$, there are no trivalent vertices added because of Definition \ref{def:inductive_weave}. Therefore the quivers are identical in this case.
\end{proof}

\begin{remark}\label{rmk: add trivalent quiver}
The appropriate modification of Lemma \ref{lem: add trivalent quiver} is valid for the right inductive weaves $\overrightarrow{\mathfrak{w}}$. The quiver $Q_{\rind{\beta}}$ is obtained from $Q_{\rind{\beta\sigma_i}}$ by removing a frozen sink, provided that $\delta(\beta) = \delta(\beta\sigma_i)$.
\end{remark}


\section{Construction of cluster structures}
\label{sec: cluster variables}

In this section we focus on simply laced cases. We introduce the (to be) cluster $\mathcal{A}$-variables, which will be indexed by trivalent vertices of a Demazure weave, study their properties and  prove Theorem \ref{thm:main}.


\subsection{Framed weaves and framed flags} Given a Demazure weave $\fW$, the cluster $\CA$-variables associated to $\fW$ will be extracted from the information of framed flags compatible with $\fW$. Intuitively, this translates to studying all possible assignments of a framed flag to every connected component of the complement of $\fW$ satisfying certain incidence conditions dictated by $\fW$. See \cite[Section 5]{CZ} or \cite[Section 4]{CW} for the origin of such ideas, related to the microlocal theory of sheaves.

In order to make the cluster $\CA$-variables computable, we introduce appropriate coordinates. This is done in a manner that we effectively assign an element 
$g\in\G$ to each component of the complement of $\fW$, not just a flag. These elements $g\in\G$ indeed parametrize flags, so the associated flags $g\uni\in\G/\uni$ and $g\borel\in\G/\borel$ are the main geometric objects, but the elements themselves are useful in our construction and when performing computations. We now introduce the notions of a raked weave and a labeling, following \cite[Section 4]{CGGS1}, which allow us to describe such matters with precision.

\subsubsection{Raked weaves} Consider a Demazure weave $\fW\sse R$ for a positive braid word $\beta$. Following Subsection \ref{sec:demazure weaves}, $R\sse\R^2$ denotes a fixed rectangle and the weave $\fW\sse\R^2$ is considered inside of $R\sse\R^2$ in such a way that $\fW \cap \partial R$ only has points in the northern and southern edges of $\partial R$. The boundary $\dd R$ of the rectangle $R$ is piecewise linear: we refer to its four linear components $\dd_\n R,\dd_\so R,\dd_\e R$ and $\dd_\w R$ as the north, south, east and west boundaries, respectively. Here $\dd_\n R$ and $\dd_\so R$, resp.~$\dd_\e R$ and $\dd_\w R$, are parallel.\\

By definition, a horizontal slice of $\fW$ is any segment in $R$ parallel to $\dd_\n R$ that starts at $\dd_\w R$ and ends at $\dd_\e R$. We henceforth assume that Demazure weaves $\fW\sse R$ have the property that any horizontal slice contains at most one vertex $v\in V(\fW)$, i.e.~no two different vertices in $\fW$ are at the same horizontal height. Therefore, any vertex $v\in V(\fW)$ uniquely defines a horizontal slice $H_v\sse R$ by requiring $v\in H_v$. Following \cite[Section 4.1]{CGGS1}, we use certain decorations added to the weave, as follows.

\begin{definition}\label{def:raked_weave}
Let $\fW\sse R$ be a Demazure weave. For each trivalent vertex $v\in V(\fW)$, consider the horizontal segment $r_v\sse H_v$ that starts at $v$ and ends at $\dd_\e R$. By definition, the raked weave $\fW^=\sse R$ associated to $\fW$ is the planar graph given by
$$\fW^=:=\fW\cup\left(\bigcup_{v\in V(\fW)}r_v\right),$$
where vertices $V(\fW^=)$ and edges $E(\fW^=)$ are defined as follows:
\begin{enumerate}
    \item Every vertex $v\in V(\fW)$ is a vertex $v\in V(\fW^=)$, i.e.~ $V(\fW)\sse V(\fW^=)$, and the remaining vertices $V(\fW^=)\setminus V(\fW)$ are in bijection with the collection of intersection points of the form $(r_v\cap\fW)\setminus\{v\}$, where $v\in V(\fW)$ is a trivalent vertex.\\

    \item An edge 
    $e\in E(\fW^=)$
    exists between a pair of vertices 
    $v,v'\in V(\fW^=)$
    if and only if there exists a connected component of $\fW^=\setminus V(\fW^=)$ whose closure contains $v$ and $v'$.\footnote{In particular, every edge $e\in E(\fW)$ that does not intersect any $r_v$, $v\in V(\fW)$, defines a unique edge of $E(\fW^=)$. Similarly, if an edge $e\in E(\fW)$ intersects a collection of rays $r_{v_1},\ldots r_{v_q}$, for some trivalent vertices $v_1,\ldots,v_q\in V(\fW)$, then there is a unique edge between the vertices defined by $r_{v_i}\cap e$ and $r_{v_{i+1}}\cap e$. Finally, if a ray $r_v$ intersects a collection of edges $e_1,\ldots, e_q$, for some $e_1,\ldots,e_q\in E(\fW)$, then there is a unique edge between $r_v \cap e_i$ and $r_v \cap e_{i+1}$.}
\end{enumerate}
The vertices in $V(\fW^=)\setminus V(\fW)$ are referred to as virtual vertices. An edge in $E(\fW^=)$ which is not contained in any edge of $\fW$ is referred to as a dashed edge; a non-dashed edge is said to be solid.
\end{definition}

\begin{figure}[ht!]
\centering
\includegraphics[scale=0.95]{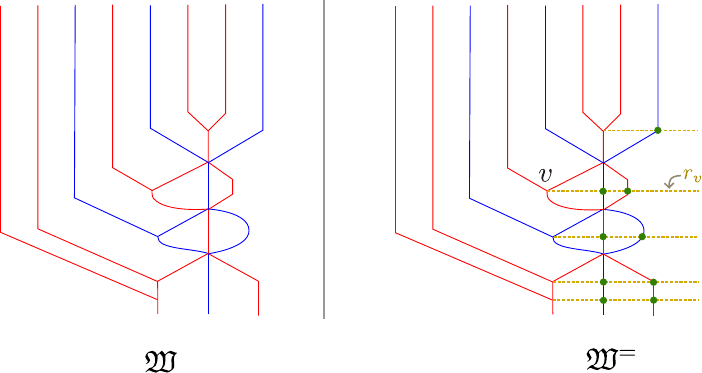}
    \caption{A Demazure weave $\fW$ (left) and its associated raked weave $\fW^=$ (right). The 
    virtual vertices have been marked with green dots, for clarity. The raking rays $r_v$ emanating from trivalent vertices $v$ of $\fW$ are always drawn with dashed yellow lines.}
    \label{fig:LeftInductive_Example3}
\end{figure}

Figure \ref{fig:LeftInductive_Example3} depicts the raked weave associated to the weave in Figure \ref{fig: Left_inductive_Example2} from Example \ref{ex:leftinductive}.(ii). Definition \ref{def:raked_weave} is used in Definition \ref{def:framed_weave} below, which introduces a key notion in our construction. To simplify notation, we denote the subset of dashed edges in $E(\fW^=)$ by $E_{dh}(\fW)$ and its complement by $E_{sd}(\fW)$, the subscript abbreviating {\it solid}. In figures, we draw the raking rays $r_v$ by {\it dashed yellow} lines, as in \cite[Section 4]{CGGS1}.

\begin{remark}
Note that $E_{sd}(\fW)$ contains more edges than $E(\fW)$. The edges in $E_{sd}(\fW)$ are in bijection with the edges of $\fW^\circ$, where $\fW^\circ$ is the graph obtained from $\fW$ by adding one bivalent vertex per each intersection point of the form $r_v\cap e$, $e\in E(\fW)$.
\end{remark}

\subsubsection{Labeled weaves} Let $\fW$ be a Demazure weave. In order to construct the cluster $\mathcal{A}$-variables, we will label the solid edges of $\fW^=$ by pairs of $\C$-valued rational functions on $X(\beta)$ and the dashed edges by $\uni$-valued rational functions on $X(\beta)$. For each solid edge $e\in E_{sd}(\fW)$, we denote such pair of rational functions by $\wt z_e,u_e:X(\beta)\dashrightarrow\C$ and typically write $(\wt z_e,u_e)$ to the right of the weave edge to indicate this assignment, e.g.~see Figure \ref{fig: Models_Propagation3}. These $\wt z_e$ and $u_e$ are referred to as the $\wt z$ and $u$-variables of the edge $e$. This labeling is a framed enhancement of the labeling defined in Section \ref{sec:demazure weaves}, see e.g. Figure \ref{fig: types vertices + variables}.\\

\noindent Given a rational function $f:X(\beta)\dashrightarrow\C$, we denote its (maximal) domain of definition by $\fD(f)\sse X(\beta)$.

\begin{definition}\label{def:labeled_weave}
A labeled weave $(\fW,\zeta)$ is a pair consisting of a Demazure weave $\fW$ for a positive braid word $\beta$ together with a pair of functions $\zeta=(\zeta_{sd},\zeta_{dh})$ such that\\

\begin{itemize}
    \item[$(i)$] $\zeta_{sd}:E_{sd}(\fW)\longrightarrow \C(X(\beta))\times \C(X(\beta))$ assigns two rational functions $\zeta(e):=(\wt z_e,u_e)$ to each solid edge $e\in E_{sd}(\fW)$.\\

    \item[$(ii)$] $\zeta_{dh}:E_{dh}(\fW)\longrightarrow \C(X(\beta),\uni)$ assigns a $\uni$-valued rational function on $X(\beta)$ to each dashed edge $e\in E_{dh}(\fW)$.\\
\end{itemize}

\noindent Given a labeled weave $(\fW,\zeta)$, we denote by
$$\mathfrak{D}(\fW,\zeta):=\left(\bigcap_{e\in E_{sd}(\fW)} (\fD(\wt z_e)\cap \fD(u_e)\cap \fD(u_e^{-1}))\right)\cap\left(\bigcap_{e\in E_{dh}(\fW)}\fD(\zeta_{dh}(e)) \right)$$
the domain of definition of the labeling $\zeta$. That is, $\mathfrak{D}(\fW,\zeta)$ is the the maximal open subset of $X(\beta)$ where all the rational functions $\tz_e$, $u_e$ and $\zeta_{dh}(e)$ in the labels are defined, and $u_e$ are invertible.
\end{definition}

\noindent Given a labeled weave $(\fW,\zeta)$ as in Definition \ref{def:labeled_weave}, any point $p\in\mathfrak{D}(\fW,\zeta)$ specifies another labeled weave $(\fW,\zeta_p)$ whose labels are the constant rational functions $\zeta_p(e):=(z_e(p),u_e(p))$, if $e\in E_{sd}(\fW)$, or $\zeta_p(e):=\zeta_{dh}(e)(p)$ if $e\in E_{dh}(\fW)$.\\

\begin{figure}[ht!]
\centering
\includegraphics[scale=0.75]{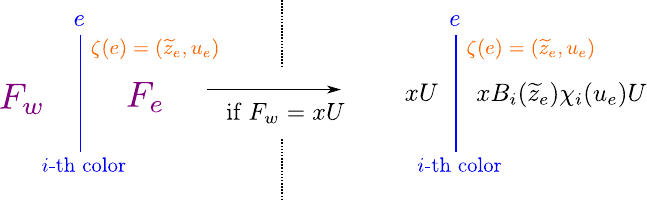}
    \caption{Framed flags near a labeled solid edge $e\in E_{sd}(\fW)$ of a weave $\fW$.}
    \label{fig: Models_Propagation3}
\end{figure}

\noindent The purpose of the labels in Definition \ref{def:labeled_weave} is to record information about framed flags, using the coordinates introduced in Subsections \ref{sect: pinnings} and \ref{ssec:framings}. 
We recall that we parameterize framed flags as $x\uni\in \G/\uni$, where $x\in \G$ and $\uni\sse\G$ is a fixed unipotent subgroup, cf.~Subsection \ref{ssec:framings}. Also, given a Demazure weave $\fW\sse R$, we refer to the connected components of $R\setminus\fW$ as regions, and similarly for the associated raked weave $\fW^=\sse R$. A region of $R\setminus\fW^=$ whose closure intersects the east boundary $\dd_\e R$ is said to be a rightmost region. There are as many rightmost regions as there are trivalent vertices in $\fW$ plus one. Note that there is a unique well-defined leftmost region, whose closure intersects the west boundary $\dd_\w R$.
 
\begin{definition}
\label{def: compatible flags}
Let $(\fW,\zeta)$ be a labeled weave and $p\in\mathfrak{D}(\fW,\zeta)$ a point in its domain. By definition, a collection of framed flags indexed by the regions of $R\setminus\fW^=$ is said to be compatible with the labeled weave $(\fW,\zeta_p)$, if it satisfies the following conditions:

\begin{enumerate}
    \item The framed flag indexed by the leftmost region is the standard framed flag $\uni$, and the framed flag indexed by any of the rightmost regions projects to $\delta(\beta)\borel$ in $\G/\borel$. That is, the flag underlying the rightmost framed flag is  $\delta(\beta)\borel\in\G/\borel$.\\

    \item Given a solid edge $e\in E_{sd}(\fW)$ of  color $i\in\dynkin$, we denote by $F_\w$ and $F_\e$ the two framed flags respectively west and east of $e$, i.e.~$F_\w$ is to the left of $e$ and $F_\e$ is to the right of $e$. We impose the condition that there exists $x\in\G$ such that $F_\w=x\uni$ and
    \begin{equation}\label{eq:compatible_flags}
    F_\e=xB_i(\wt z_e(p))\chi_i(u_e(p))\uni.
    \end{equation}
    That is, the framed flag $F_\e$ to the right of $e$ is obtained, as dictated by this formula, from the framed flag $F_\w$ to the left of $e$ through the information in the label $\zeta_p(e)$.\\

    \item Given a dashed edge $e\in E_{dh}(\fW)$, we denote by $F_\n$ and $F_\so$ the two framed flags respectively north and south of $e$. Then we require $F_\n=F_\so$.\\
\end{enumerate}

\noindent By definition, the moduli space $\fM(\fW,\zeta)$ of framed flags associated to $(\fW,\zeta)$ is the space of collections of framed flags compatible with $(\fW,\zeta_p)$ for some $p\in\mathfrak{D}(\fW,\zeta)$.
\end{definition}

Note that $\fM(\fW,\zeta)$ in Definition \ref{def: compatible flags} is naturally an algebraic variety. Indeed, it is a Zariski closed subset of the algebraic variety $\fD(\fW,\zeta)\times(\G/\uni)^r$, where $r=|\pi_0(R\setminus\fW)|$ is the number of regions of $R\setminus\fW$. In the unframed case and $\G=\SL_m$, these assignments of flags for each region in $R\setminus\fW$, with transversality conditions as imposed by $\fW$, led to the flag moduli space introduced and studied in \cite[Section 5]{CZ}, cf.~also Section \ref{sec: weaves} above.\footnote{In that case, one can work rather directly with $\fW$, without using its raked refinement $\fW^=$.}

\begin{remark} For convenience, we label the edges in local models near vertices of the weave according to their (inter)cardinal directions, as depicted in Figure \ref{fig: Models_Propagation}, e.g.~ $e_{\n\w}$ stands for the edge pointing {\it northwest}.\footnote{In particular, $e$ can denote either an edge or the East direction, but the meaning is always clear from the context and we use different fonts ($e$ for edges and $\e$ for East).} To ease notation, we also write $\zeta(e_{\nw}) = (\wt z_{\nw}, u_{\nw})$ instead of $\zeta_{sd}(e_{\nw})=(\wt z_{e_{\nw}}, u_{e_{\nw}})$, and similarly for all the other directions and for dashed edges. In particular, we suppress the subscript of the labeling $\zeta(e)$ if it is clear by context if it is applied to a solid edge, and thus $\zeta(e)=\zeta_{sd}(e)$, or to a dashed edge, where the notation would be read as $\zeta(e)=\zeta_{dh}(e)$.
\end{remark}

\begin{figure}[ht!]
\centering
\includegraphics[scale=0.65]{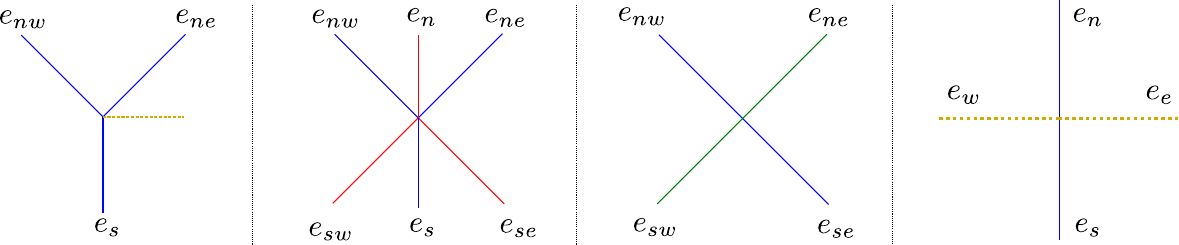}
    \caption{Cardinal notation for edges near the types of vertices in $\fW$. From left to right: a trivalent vertex $v$, with the dashed ray $r_v$ in yellow, a hexavalent vertex, a tetravalent vertex and a virtual vertex. Technically, virtual vertices are tetravalent, but we reserve that notation for the vertices in the third column, strictly coming from intersections of (solid) weave edges. Similarly, trivalent vertices $v\in V(\fW)$ become tetravalent in $\fW^=$, but we still refer to them as trivalent.}
    \label{fig: Models_Propagation}
\end{figure}

\begin{remark} \label{rem: dashed-help-with-matrices}
The data of a Demazure weave $\fW$ alone is enough to describe a moduli space of compatible framed flags as in Definition \ref{def: compatible flags}, without using $\fW^=$. That said, adding the raking rays $r_v$, which is the additional information contained in $\fW^=$, essentially allows us to assign elements of $\G$ in every region. This data of an element $g\in\G$ in every region, and not just its framed flag coset $g\uni\in\G/\uni$, has the advantage of simplifying certain computations in our arguments. This choice of an element $g\in\G$ for each region of $R\setminus\fW^=$ would not be well-defined if we did not refine $\fW$ to $\fW^=$, only their framed flag cosets would be well-defined.
\end{remark}

Now, the intuition is that we want $\fM(\fW,\zeta)$ in Definition \ref{def: compatible flags} to be isomorphic to a torus, which will underlie a cluster torus in $X(\beta)$. Nevertheless, an arbitrary labeling $\zeta$ of a weave $\fW$ for $\beta$ is not sufficient: we must add conditions to a labeling $\zeta$ for that to hold. We also need the following piece of notation:

\begin{definition}
Let $\dynkin$ be a Dynkin diagram. For each vertex $i\in\dynkin$, we denote by $\xi_i:\uni\longrightarrow (\C,+)$ the unique additive character on $\uni$ such that 
$$
\xi_i\left[\varphi_j\left(\begin{matrix}
1 & a\\ 0 & 1\end{matrix}\right)\right]=\delta_{ij}a.
$$
\end{definition}

\subsubsection{Framed weaves} Let us impose additional conditions on a labeled weave $(\fW,\zeta)$ so as to proceed with our construction of cluster $\mathcal{A}$-variables, as follows. 

\begin{definition}\label{def:framed_weave}
A labeled weave $(\fW,\zeta)$ is said to be a framed weave if $\zeta$ satisfies the following conditions:

\begin{enumerate}
    \item If $e_i\in E_{sd}(\fW)$ is the solid weave edge corresponding to the $i$th crossing of $\beta$, starting at the (top) northern boundary of $\fW$, then we require the condition
    $$(\wt z_{e_i},u_{e_i})=(z_i,1),$$
    where $z_i:X(\beta)\longrightarrow\C$ is the regular function introduced in Subsection \ref{sect: pinnings}.\\
    
    \item In a trivalent vertex of $\fW$, using the notation in Figure \ref{fig: Models_Propagation} (left), we require
     $$\wt z_{\so}=\wt z_{\nw}-u_{\nw}^{-2}\wt z_{\neast}^{-1}\qquad\mbox{and}\qquad u_{\so}=\wt z_{\neast}u_{\nw}u_{\neast}.$$
     
     \noindent In additional, the unique dashed edge $e_v$ starting at this vertex is labeled by
     $$
     \zeta_{dh}(e_v)=\varphi_i\left(\begin{matrix}
1 & -\tz_{\neast}^{-1}u_{\neast}^{-2}\\ 0 & 1\end{matrix}\right)\in \uni.
     $$
     
    \item In a hexavalent vertex of $\fW$, using the notation in Figure \ref{fig: Models_Propagation} (center), we require
    $$\tz_{\sw}u_{\n}=\tz_{\neast}u_{\nw},\qquad \tz_su_{\nw}u_\so=\tz_{\nw}\tz_{\neast}u_{\nw}^2u_{\neast} - \tz_nu_{\sw}u_\so,\qquad \tz_{\se}u_{\sw}=\tz_{\nw}u_\so,$$
    $$\qquad u_{\nw}u_\n=u_\so u_{\se}\qquad\mbox{and}\qquad u_\n u_{\neast}=u_{\sw}u_\so.$$

    \item In a tetravalent vertex of $\fW$, using the notation in Figure \ref{fig: Models_Propagation} (right), we require
    $$\zeta(e_{\nw})=\zeta(e_{\se})\qquad\mbox{and}\qquad \zeta(e_{\neast})=\zeta(e_{\sw}).$$

    \item In a virtual vertex we write $Y_\w:=\zeta_{dh}(e_\w)$ and $Y_\e:=\zeta_{dh}(e_\e)$, where $e_\w,e_\e$ are the dashed edges west and east of the virtual vertex. Suppose that the solid edge $e_\n$ north of the virtual vertex has color $i\in\dynkin$. Then we require
    $$
    \tz_{\so}=\tz_\n+\xi_i(Y_\w),\quad u_{\so}=u_\n,\quad\mbox{and} \quad Y_\e=(B_i(\tz_\so)\chi_i(u_\s))^{-1}\cdot Y_\w\cdot (B_i(\tz_\n)\chi_i(u_\n)).
    $$
\end{enumerate}
\end{definition}

 \begin{figure}[ht!]
\centering
\includegraphics[scale=0.7]{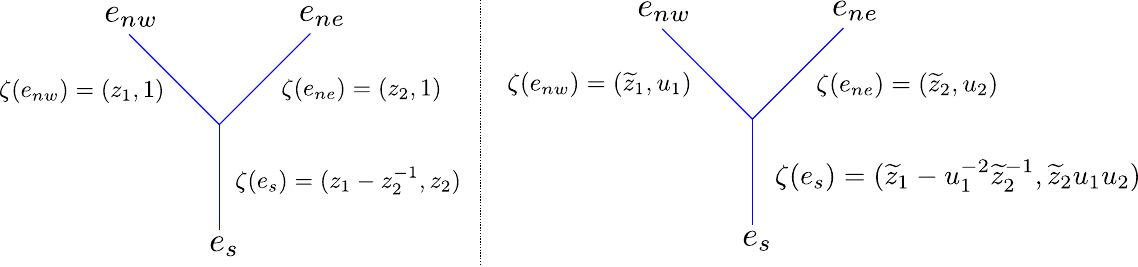}
    \caption{(Left) Example of framed weave $(\fW,\zeta)$ with the weave $\fW$ being a trivalent vertex, considered at the top of $\fW$ so that $u_1=u_2=1$ and $\wt z_1=z_1$, $\wt z_2=z_2$. (Right) A general framed weave $(\fW,\zeta)$ with $\fW$ a trivalent vertex, possibly inside of a weave $\fW$. The dashed yellow lines have not been depicted to improve visual ease.}
    \label{fig: Models_Propagation4}
\end{figure}

The conditions in Definition \ref{def:framed_weave} are designed to describe framed flags compatible with $(\fW,\zeta)$ in a manner that $\fM(\fW,\zeta)$ is (isomorphic to) the initial cluster torus in $X(\beta)$ associated to the Demazure weave $\fW$ and cluster $\mathcal{A}$-variables can be read directly from $\zeta$. This will be apparent in the proof of Theorem \ref{thm:cluster vars}. In particular, Condition (1) in Definition \ref{def:framed_weave} captures the equation in Corollary \ref{cor: braid via eqns}. Conditions (2), (3), (4) and  (5) are consistency conditions for such collection of framed flags to exist, as explained by the following result.

\begin{lemma}
\label{lem: framed flags from labeled weaves}
Let $(\mathfrak{W},\zeta)$ be a framed weave. Then the following holds:\\

  \item[$(a)$] For any $p\in\mathfrak{D}(\fW,\zeta)$, there exists a unique collection of framed flags compatible with $(\mathfrak{W},\zeta_p)$. In particular, the moduli space $\fM(\fW,\zeta)$ is isomorphic to $\mathfrak{D}(\fW,\zeta)$.\\

  \item[$(b)$] For any (solid) weave edge $e\in E_{sd}(\fW)$ at the southern boundary of $\fW$, we have $\tz_e=0$.\\
 \end{lemma}

\begin{proof}
For Part (a), start with the standard framed flag $\uni$ in the leftmost (westmost) region. For any other region, choose a path connecting that region with the leftmost region that avoids all vertices of $\fW^=$ and is transverse to the edges of $\fW^=$. Then assign framed flags to regions along the path using condition \eqref{eq:compatible_flags} in Definition \ref{def: compatible flags}. If we choose a horizontal path near the top of $\fW^=$ (which coincides with the top of $\fW$), parallel to $\dd_\n R$, Condition (1) in Definition~\ref{def:framed_weave} and the definition of $X(\beta)$ ensure that the framed flag assigned to the (top) rightmost region of that path has underlying flag $\delta(\beta)\borel\in\G/\borel$. Condition (3) in Definition \ref{def: compatible flags} then implies that the framed flag assigned to any rightmost region in $R\setminus\fW^=$ also projects to $\delta(\beta)\borel\in\G/\borel$.\\

\noindent It remains to verify that this assignment of flags is independent of the choice of paths. This can be checked locally near each vertex: comparing the condition imposed by a path above and below the vertex. For a trivalent vertex, Condition (2) captures the framed version of identity \eqref{eq: trivalent}, which reads:
\begin{equation}
\label{eq: trivalent framed}
B_i(\tz_{\nw})\chi_i(u_{\nw})B_i(\tz_{\neast})\chi_i(u_{\neast})=B_i(\tz_\so)\chi_i\left(\tz_{\neast}u_{\nw}u_{\neast}\right)\varphi_i\left(\begin{matrix}
1 & -\tz_{\neast}^{-1}u_{\neast}^{-2}\\ 0 & 1\end{matrix}\right).
\end{equation} 
Thus, the assignment of flags is independent of whether we cross $e_{\nw}$ and $e_{\neast}$ or $e_\so$. Similarly, for a hexavalent vertex, Condition (3) captures the property discussed in Lemma \ref{lem: braid matrix relations framed}(2), which implies independence of the path. Condition (4) captures the usual fact that nothing particularly interesting happens at a tetravalent vertex. Finally, for a virtual vertex, Condition (5) is equivalent to 
$$
Y_\w B_i(\tz_\n)\chi_i(u_\n)=B_i(\tz_\so)\chi_i(u_\so)Y_\e
$$
and Lemma \ref{lem: slide weave} implies that $Y_\e\in \uni$, see \eqref{eq: moving Y}. In summary, conditions $(2)$, $(3)$, $(4)$ and $(5)$ in Definition~\ref{def:framed_weave} imply independence of the chosen path.\\

For Part $(b)$, choose a horizontal path near the bottom (southern boundary) of $\fW^=$, parallel to $\dd_\s R$, where the edges spell a reduced expression for the Demazure product $\delta(\beta)$. Lemma \ref{lem: rel position 2} then directly implies that $\tz_{e}=0$ for all solid edges $e\in E_{sd}(\fW)$ at the bottom of $\fW^=$, i.e.~for those edges intersecting $\dd_\s R$. (Alternatively, this also follows from identity (\ref{eqn: braid reduced}).) 
\end{proof}

\subsubsection{The $\wt z$-variables as Laurent polynomials in the $u$-variables} The following algebraic property will be a key step in the construction of cluster $\CA$-variables in Theorem \ref{thm:cluster vars}. Intuitively, it states that given $u$-variables satisfying the conditions of a labeled weave, there is a unique way to find $\wt z$-variables such that they are Laurent polynomials in the $u$-variables and they all together form a framed weave.

\begin{lemma}
\label{lem: framed flags from labeled weaves_2}
Let $\mathfrak{W}$ be a weave. Consider a collection of variables $\underline{u}=\{u_e\}_{e\in E_{sd}(\fW)}$ satisfying Condition $(1)$, the two equations on the bottom line of Condition $(3)$, and the equations in the $u$-variables in Conditions $(4)$ and $(5)$, i.e.~the identities in the $u$-arguments of a labeling.\footnote{All these are conditions from Definition \ref{def:framed_weave}. Also, in this lemma we do not require that the $u$-variables $\underline{u}$ are rational functions in the initial top variables $\{z_i\}$.}\\

\noindent Then there exists a unique collection of Laurent polynomials $\{\tz_{e}(\underline{u})\}_{e\in E_{sd}(\fW)}$ and $\{Y_e(\underline{u})\}_{e\in E_{dh}(\fW)}$ such that $(\fW,\zeta)$ is a framed weave, where $\zeta$ is the labeling given by
$$\zeta_{sd}(e):=(\tz_{e}(\underline{u}),u_e)\qquad\mbox{and}\qquad \zeta_{dh}(e):=Y_e(u).$$
In particular, the variables $\{z_i\}_{i\in[\ell(\beta)]}$ labeling the top edges of the weave $\fW$ are Laurent polynomials in the $u$-variables $\underline{u}$. In addition, the variables $\{z_i\}$ satisfy the defining equations of $X(\beta)$.
 \end{lemma}

\begin{proof}
We prove the statement by scanning the weave $\fW$ from bottom to top. At the bottom (southern) boundary of $\fW^=$, near $\dd_\so R$, we declare all $\tz$-variables to be zero and assign framed flags to the corresponding regions as in the proof of Lemma \ref{lem: framed flags from labeled weaves}.(a). In particular, this implies that the framed flag assigned to that bottom rightmost (eastmost) region must have underlying flag $\delta(\beta)\borel\in\G/\borel$.\footnote{Definition \ref{def: compatible flags}.(3) then implies that all rightmost regions will be assigned the same
framed flag that projects to $\delta(\beta)\borel$.} Now we start scanning up the weave $\fW$, bottom to top by horizontal slices, and argue inductively. The base case is the bottom boundary, just discussed in this paragraph. Let us proceed with the inductive step:

\begin{enumerate}
    \item If we cross a trivalent vertex, the inductive assumption is that $\tz_\so$ is a Laurent polynomial in the $u$-variables, where $\tz_\so$ is the $z$-variable of the edge $e_\so$ at a trivalent vertex. Condition (2) implies that $\tz_{\neast}$ is a Laurent monomial in $\underline{u}$, while $\tz_{\nw}$ is a linear combination of $\tz_\so$ and another Laurent monomial in $\underline{u}$. Therefore, both  $\tz_{\neast}$ and $\tz_{\nw}$ are Laurent polynomials in the $u$-variables.\\

    \item Now, when crossing the trivalent vertex $v$ in item $(1)$ above, there are $\wt z$-variables to the right of $v$ that also change when crossing the dashed edges inside the raking ray $r_v$ associated to $v$. First, by Definition \ref{def:framed_weave}.(2), the label $\zeta_{dh}(e_v)=Y_{e_v}$ for the dashed edge $e_v\in E_{dh}(\fW)$ contained in $r_v$ and intersecting $v$ is
    \begin{equation}
    \label{eq:YLaurent_basecase}
    \zeta_{dh}(e_v)=\varphi_i\left(\begin{matrix}
1 & -\tz_{\neast}^{-1}u_{\neast}^{-2}\\ 0 & 1\end{matrix}\right)\in \uni.\end{equation}
By item $(1)$ above, the label $\tz_{\neast}$ is a Laurent polynomial in the $u$-variables. Therefore the components of $\zeta_{dh}(e_v)$ are indeed Laurent polynomials in the $u$-variables, i.e.~ $\zeta_{dh}(e_v)$ defines a regular morphism $\Spec\C[\underline{u}^{\pm1}]\to \uni$.\footnote{It is a matrix with Laurent polynomial entries in the $u$-variables if $\G$ is a matrix group.}\\

Let $v_0,v_1,\ldots,v_q\in V(\fW^=)$ be the vertices in $r_v$, starting with the trivalent vertex $v_0=v$ and continuing with the virtual vertices $v_1,\ldots,v_q$ to the right of $v$, ordered left to right. Let $e_v^i$ be the dashed edge between $v_i$ and $v_{i+1}$, so that $e_v^0=e_v$. Denote by $\wt z^{(i)}_\n$ and $\wt z^{(i)}_\so$ the $\wt z$-variables north and south of the vertex $v_i$, $i\in[q]$, and similarly for the corresponding $u$-variables. We now argue by induction on $i$ that $\zeta_{dh}(e^i_v)$ has components being Laurent polynomials in the $u$-variables and $\wt z^{(i)}_\n$ are Laurent polynomials in the $u$-variables. By the bottom to top overall induction, we can assume that all $\wt z^{(i)}_\so$ are Laurent polynomials in the $u$-variables.\\

\noindent For this (left to right) induction on $i$, the base case $i=0$ is established in the paragraph above: $\zeta_{dh}(e_v)$ has components that are Laurent polynomials in the $u$-variables by Equation (\ref{eq:YLaurent_basecase}). Let us proceed with the inductive step, assuming that the components of $Y_i:=Y_{e^i_v}=\zeta_{dh}(e^i_v)$ and $\wt z^{(i+1)}_\so$ are all Laurent polynomials in the $u$-variables. We want to show that the associated $\tz^{i+1}_\n$, satisfying Condition (5), is a Laurent polynomial in the $u$-variables and that so are the components of $Y_{i+1}:=\zeta_{dh}(e^{i+1}_v)$. First, note that the element $\xi_j(Y_i)$ is a Laurent polynomial in the $u$-variables in this case, where $j\in\dynkin$ is the color of the solid edges north and south of $v_i$. Indeed, this holds because both the components of $Y_i$ and $\wt z^{(i+1)}_\so$ are Laurent polynomials in the $u$-variables, by the induction hypothesis. Second, since we must satisfy Condition (5), we must have
$$\tz^{(i+1)}_\n=\tz^{(i+1)}_\so-\xi_j(Y_i),$$
and it follows that $\tz^{(i+1)}_\n$ is a Laurent polynomial in the $u$-variables, as this holds for both summands $\tz^{(i+1)}_\so$ and $-\xi_j(Y_i)$. Finally, it suffices to argue that the components of $Y_{i+1}$ are Laurent polynomials in the $u$-variables. This follows from the third equation of Condition (5):
$$Y_{i+1}=(B_j(\tz^{i+1}_\so)\chi_j(u^{i+1}_\so))^{-1}\cdot Y_i\cdot (B_j(\tz^{i+1}_\n)\chi_j(u^{i+1}_\n)).$$
Indeed, all entries on the right hand side have components that are Laurent polynomials in the $u$-variables, and thus this also holds for their product. This completes the left-to-right induction on $i$ and concludes the inductive step of the bottom-to-top induction in the case that we are crossing a trivalent vertex.\\

\item If we cross a hexavalent vertex, then the inductive assumption is that $\tz_{\sw},\tz_{\so},\tz_{\se}$ are Laurent polynomials in the $u$-variables. Then Condition (3) implies that $\tz_{\nw},\tz_{\n},\tz_{\neast}$ are Laurent polynomials in $u$-variables, as required.\\
    
    \item The case of crossing a tetravalent vertex is immediate, as the labels do not change.\\
\end{enumerate}

\noindent Finally, we must verify that the variables $\{z_i\}$ define a point in $X(\beta)$. For that, we note that this inductive process leads to a collection of framed flags with $\uni$ on the leftmost region, and a framed version of $\delta(\beta)\borel$ on the top rightmost region. By Condition $(1)$ in Definition \ref{def:framed_weave}, this defines a point in $X(\beta)$.
\end{proof}

To conclude this subsection, we emphasize the following two properties of framed weaves $(\fW,\zeta)$, entirely to do with solid edges and vertices of $\fW$ itself:

\begin{itemize}
    \item[(i)] Near a trivalent or tetravalent vertex, the values of the labeling $\zeta$ for the (solid) edges above the vertex uniquely determine the values of the labeling for the edges below the vertex. That is, for a trivalent vertex, $\zeta(e_{\nw})$ and $\zeta(e_{\neast})$ uniquely determine $\zeta(e_{\so})$. This is depicted in Figure \ref{fig: Models_Propagation4}. Similarly, for a tetravalent vertex, $\zeta(e_{\nw})$ and $\zeta(e_{\neast})$ uniquely determine $\zeta(e_{\se})$, and $\zeta(e_{\sw})$.\\

    \item[(ii)] For a hexavalent vertex, the values of the labeling $\zeta$ for the (solid) edges above the vertex do {\it not} uniquely determine the values of the labeling for the (solid) edges below the vertex. Nevertheless, they determine them up to a $\C^*$ choice. For instance, $\zeta(e_{\nw}),\zeta(e_{\n}),\zeta(e_{\neast})$ together with the value $u_{\so}$ do determine $\zeta(e_{\sw}),\zeta(e_{\so})$ and $\zeta(e_{\se})$. This indicates that having a defining rule for the $u$-variables will potentially allow for the construction of a unique framed weave via a propagation argument, from the top of the weave to its bottom. This is indeed what occurs in the proof of Theorem \ref{thm:cluster vars}, where the (to be) cluster $\mathcal{A}$-variables will uniquely specify a framed weave by determining the $u$-variables.\\
\end{itemize}

\noindent In other words, by item (ii) above, a weave $\fW$ often underlies many framed weaves $(\fW,\zeta)$, i.e.~the choice of $\zeta$ is not unique. 

\subsubsection{Independence of choices for raked weaves} Given a Demazure weave $\fW$, we can perform a compactly supported isotopy that moves a trivalent vertex upwards (or downwards). The resulting Demazure weave $\fW'$ is effectively the same: the planar graphs $\fW$ and $\fW'$ are identical. That said, the associated raked weaves $\fW^=$ and $(\fW')^=$ might differ by a sliding of dashed yellow lines through vertices of the weave. We want constructions to be independent of whether we have started with $\fW$ or $\fW'$, i.e.~independent of the exact heights of the vertices of $\fW$. This independence was essentially established in \cite[Section 5]{CGGS1}. We include the necessary details here for completeness and refer to \cite[Section 5.2.1]{CGGS1} for further discussions:

\begin{lemma}\label{lem:independence_dashedlines}
Let $v\in\fW$ be a vertex in a Demazure weave. Consider a dashed yellow line $y_\n$, resp.~$y_\so$, right above the vertex $v$, resp.~below. Then all the $(\tz,u)$-variables above $y_\n$ and below $v$ coincide with the corresponding $(\tz,u)$-variables above $v$ and below $y_\so$, i.e.~they are independent of whether we choose the dashed yellow line above or below the vertex.
\end{lemma} 

\begin{proof}
There are three cases to verify, depending on the type of vertex $v\in\fW$.

\begin{enumerate}
    \item If $v$ is a trivalent vertex of color $i\in\dynkin$, we use Identity \eqref{eq: moving Y} in the proof of Lemma 4.2. Indeed, let us set $a=c=1$ in that identity and observe that $z_2$ does not change. Then Identity \eqref{eq: moving Y} implies that the variable $\tz_{\neast}$ is independent of whether we are using the dashed yellow line $y_\n$ or $y_\so$. If we are using the dashed line $y_\n$ above $v$, we have the following changes of variables:
\[
(\tz_{\nw},u_{\nw},\tz_{\neast},u_{\neast})\to 
(\tz_{\nw}+\xi_i(Y),u_{\nw},\tz_{\neast},u_{\neast})\to
(\tz_{\nw}+\xi_i(Y)-u_{\nw}^{-2}\tz_{\neast}^{-1},\tz_{\neast}u_{\nw}u_{\neast}).
\]
If we are instead using the dashed line $y_\so$ below $v$, then we obtain
\[
(\tz_{\nw},u_{\nw},\tz_{\neast},u_{\neast})\to 
(\tz_{\nw}-u_{\nw}^{-2}\tz_{\neast}^{-1},\tz_{\neast}u_{\nw}u_{\neast})\to 
(\tz_{\nw}+\xi_i(Y)-u_{\nw}^{-2}\tz_{\neast}^{-1},\tz_{\neast}u_{\nw}u_{\neast}).
\]

Let us now study how the unipotent-valued labelings of dashed edges change. On the one hand, if the dashed line is $y_\n$, above the trivalent vertex, then the dashed edge to the right of $e_{\n\e}$ will be labeled by $Y''$, where $Y'' \in \uni$ is defined by the identity
\begin{equation}
\label{eq:Y-to-Y''}
YB_i(\tz_{\nw})\chi_i(u_{\nw})B_{i}(\tz_{\neast})\chi_{i}(u_{\neast}) = B_i(\tz_{\nw} + \xi_i(Y))\chi_i(u_{\nw})B_i(\tz_{\neast})\chi_i(u_{\neast})Y''.
\end{equation}
Here we used, just as in the previous paragraph, that the variable $\wt z_{\neast}$ is the same  as the $\wt z$-variable on the solid edge continuing $e_{\neast}$ upwards just above $y_\n$.
On the other hand, if the dashed line is $y_\so$, below the trivalent vertex, then the dashed edge to the right of $e_{\so}$ will be labeled by $Y'$, where $Y' \in \uni$ is defined by
\begin{equation}
\label{eq:Y-to-Y'}
YB_i(\tz_{\nw} - u_{\nw}^{-2}\tz_{\neast}^{-1})\chi_i(\tz_{\neast}u_{\nw}u_{\neast}) = B_i(\tz_{\nw} - u_{\nw}^{-2}\tz_{\neast}^{-1} + \xi_i(Y))\chi_i(\tz_{\neast}u_{\nw}u_{\neast})Y'.
\end{equation}
By applying identity \eqref{eq: trivalent framed} twice, first to the left hand sides of \eqref{eq:Y-to-Y''} and \eqref{eq:Y-to-Y'}, and then to their right hand sides, we obtain that
\[
Y'\varphi_i\left(\begin{matrix}
1 & -\tz_{\neast}^{-1}u_{\neast}^{-2}\\ 0 & 1\end{matrix}\right) = \varphi_i\left(\begin{matrix}
1 & -\tz_{\neast}^{-1}u_{\neast}^{-2}\\ 0 & 1\end{matrix}\right)Y''.
\]
This implies that the following two situations for a (vertical) solid edge, to the right of the trivalent vertex and disjoint from it, lead to the same variables above and below. First, a solid edge first crosses the dashed edge emanating from the trivalent vertex and then the dashed edge labeled by $Y'$. Second, this same solid edge first crosses the dashed line labeled by $Y''$ and then the dashed line emanating from the trivalent vertex. Thus, the variables attached to solid edges to the right of this trivalent vertex and below both dashed edges do not depend on the relative position of the dashed edges, as required.\\

\item For a hexavalent vertex, using notation analogous to item (1) above, the statement follows from the identities
$$
YB_i(\tz_{\nw})\chi_i(u_{\nw})B_j(\tz_{\n})\chi_j(u_{\n})B_i(\tz_{\neast})\chi_i(u_{\neast})=
$$
$$
B_i(\tz_{\nw}')\chi_i(u_{\nw})B_j(\tz_{\n}')\chi_j(u_{\n})B_i(\tz_{\neast}')\chi_i(u_{\neast})Y'=
$$
$$
B_j(\tz_{\sw})\chi_j(u_{\sw})B_i(\tz_{\so}')\chi_i(u_{\so})B_j(\tz_{\se}')\chi_j(u_{\se})Y'
$$
and the fact that all the variables (and $Y'$) are uniquely determined by the variables $$Y,\tz_{\nw},u_{\nw},\tz_\n,u_\n,\tz_{\neast},u_{\neast},u_{\sw},u_\so,u_{\se}.$$ 

\item The case of a tetravalent vertex is similar.
\end{enumerate}
\end{proof}


\subsection{Cluster variables in Demazure weaves} Given a Demazure weave $\mathfrak{W}$, let $\g_v:E(\mathfrak{W})\to \Z_{\ge 0}$ be the cycles associated to the trivalent vertices $v\in\fW$ as in Definition \ref{def:vertex_cycle}. Let us now construct a collection of functions $\{A_v\}$ on $X(\beta)$, a priori rational, that will later be proven to 
be cluster $\mathcal{A}$-variables 
in a cluster seed for the regular ring of functions $\C[X(\beta)]$. Such functions are indexed by the trivalent vertices $v\in\fW$ of the given Demazure weave. They are constructed recursively and, simultaneously, determine a framed weave $(\fW,\zeta)$.

\begin{theorem}\label{thm:cluster vars} Let $\fW$ be a Demazure weave for a positive braid word $\beta$ and $\fW_3$ its set of trivalent vertices. Then there exists a unique collection 
$\{A_v\}_{v\in\fW_3}$ of rational functions $A_v:=A_v(z_1,\ldots,z_{\ell})\in\C(z_1,\ldots,z_{\ell})$, indexed by the trivalent vertices $v\in\fW_3$, and a unique labeling $\zeta$ of $\fW$ such that

\begin{itemize}
    \item[(i)] The $u$-variable $u_e$ of the label $\zeta_{sd}(e)$, $e\in E_{sd}(\fW)$, is given by
    \begin{equation}
    \label{eq: u from cluster}   u_e=\prod_{v} A_v^{\g_v(e)}.
    \end{equation}
    
    \item[(ii)] The labeled weave $(\fW,\zeta)$ is a framed weave. That is, this assignment of $u$-variables, $\tz$-variables, and the dashed labels $\zeta_{dh}$ of $\zeta$ satisfy all the conditions in Definition \ref{def:framed_weave}.
\end{itemize}

\end{theorem}

\begin{proof}
The construction of the rational functions
$\{A_v\}_{v\in\fW_3}$ is achieved by simultaneously building the labeling $\zeta$ such that $(\fW,\zeta)$ is a framed weave. The key principles in the construction are:\\

\begin{enumerate}
    \item The labeling $\zeta$ at the top (solid) edges of the weave $e\in E_{sd}(\fW)$ is determined from the start by virtue of Condition (1) in Definition \ref{def:framed_weave} of a framed weave. That is, we declare $\zeta(e_i):=(z_i,1)$ if $e_i$ is the solid weave edge corresponding to the $i$-th crossing of $\beta$, starting at the northern boundary of $\fW^=$.
    Note that $\gamma_v(e_i)=0$ for all Lusztig cycles, which is compatible with the condition $u_{e_i}=1$.\\ 

    \item Both $\{A_v\}_v$ and $\zeta$ are then built by scanning down the weave $\fW^=$ top to bottom. Namely, for each vertex $\vt$ of the weave
    $\fW$, we consider two horizontal slices $H^\pm_\vt$ of $\fW^=$ so that $H^+_\vt$ is above the vertex $\vt$, $H^-_\vt$ is below the vertex $\vt$ and there are no other vertices of $\fW^=$ (except for $\vt$) in the closed strip region between $H^+_\vt$ and $H^-_\vt$. We will construct $\{A_v\}_v$ and $\zeta$ inductively, by assuming that we have constructed them for all edges above (and intersecting) $H^+_\vt$ and then propagating it down to the edges that intersect $H^-_\vt$. The base case is in item (1) above. The propagation will be governed by the conditions in Definition \ref{def:framed_weave}.\\

\end{enumerate}

\noindent The two important cases are those where $\vt$ is a trivalent vertex and where $\vt$ is a hexavalent vertex. The case where $\vt$ is a tetravalent vertex is immediate by construction, as the labeling propagates essentially without changing by Condition (4) in Definition \ref{def:framed_weave}, and thus the functions $\{A_v\}$ do not change. Virtual vertices are treated as part of the case where $\vt$ is a trivalent vertex, since they appear to the right of each trivalent vertex.\\

The nature of these two important cases, $\vt$ trivalent or hexavalent, is a bit different:

\begin{itemize}
    \item[(i)] At a trivalent vertex $\vt=v$, a new variable $A_v$ will be introduced. It will be defined by Equation \eqref{eq: u from cluster} and one must first argue that such an assignment gives a well-defined unique $A_v$. Independently, we must ensure that the resulting labeling $\zeta$, also defined using equation \eqref{eq: u from cluster}, satisfies Condition (2) in Definition \ref{def:framed_weave}, so that it propagates from $H^+_v$ to $H^-_v$ in a manner that $(\fW,\zeta)$ will eventually be a framed weave.\\

    The dashed labels $\zeta_{dh}$ of the dashed edges contained in $r_v$ are specified by this data as well. Indeed, Condition (2) in Definition \ref{def:framed_weave} uniquely specifies $\zeta_{dh}(e_v)$ for the unique dashed edge starting at $v$. From there, Condition (5) implies that $\zeta_{dh}(e)$ is uniquely determined by $\zeta_{dh}(e_v)$ and the solid labels $\zeta_{sd}(e_i)=(z_i,1)$ above, for any dashed edge $e\in E_{dh}(\fW)$ contained in $r_v$.\\

    \item[(ii)] At a hexavalent vertex $\vt$, no new variable $A_v$ is introduced but we must still propagate the labeling $\zeta$ from $H^+_\vt$ to $H^-_\vt$. In this case, we must ensure that the labeling $\zeta$ satisfies Condition (3) in Definition \ref{def:framed_weave} and, equally important, propagates down in a {\it unique} way. Neither of these two conditions is trivial: $\zeta$ depends on the $u$-variables, which are themselves constrained by equations of the form $u_e=\prod_{v} A_v^{\g_v(e)}$. Therefore, one must use properties of the cycles $\gamma_v$ and the structure of these equations so as to ensure that $\zeta$ propagates following the conditions in Definition \ref{def:framed_weave} and it is unique.\\
\end{itemize}

Let us analyze these two cases in detail.\\

\noindent {\it Propagating down through a trivalent vertex}. Let $\vt=v$ be a trivalent vertex of color $i\in D$. Suppose that $\{A_v\}$ and $\zeta(e)$ are defined for the trivalent vertices $v$ and edges $e$ of the weave $\fW_v^+\sse \fW^=$ above the slice $H^+_v$ and satisfy the conditions of the statement of the theorem. In particular, the functions 
$A_v$, the dashed labels, and all the $\wt z$ and $u$-variables of $\zeta(e)$, for the vertices and edges above $H^+_v$, are rational functions of the variables $(z_1,\ldots,z_\ell)$. The three (solid) edges incident to $v$ will be denoted $e_{\nw},e_{\neast}$ and $e_\so$, as in Figure \ref{fig: Models_Propagation} (left). The associated cycle $\gamma_v$, cf.~Definition \ref{def:vertex_cycle}, satisfies
$$\g_v(e_{\nw})=\g_v(e_{\neast})=0\qquad\mbox{and}\qquad \g_{v}(e_\so)=1.$$
Condition (2) in Definition \ref{def:framed_weave} then defines $(\tz_\so,u_\so)$ as explicit rational functions of $(\tz_{\nw},u_{\nw},\tz_{\neast},u_{\neast})$.
Therefore, we can define $\zeta(e_\so)$ at the weave edge $e_\so$ according to Condition (2). Since $e_\so$ is the only solid edge in $\fW_v^-$ which is not in $\fW_v^+$, the labeling $\zeta$ on $\fW_v^+$ and $\zeta(e_\so)$ together determine a unique solid labeling $\zeta_{sd}$ on $\fW_v^-$. For the dashed edges, we extend the dashed labeling from $\fW_v^+$ to $\fW_v^-$ by specifying the dashed labels of the dashed edges contained in $r_v$ as in item $(i)$ above.\\


\noindent Let us now argue that the equation \eqref{eq: u from cluster}
gives a well-defined unique function $A_v$ associated to $v$. Since we have a labeling $\zeta$ on $\fW_v^-$, we in particular have the $u$-variable of $\zeta(e_\so)$, which is $u_\so=\tz_{\neast}u_{\nw}u_{\neast}$. We can therefore combine this identity with equation~\eqref{eq: u from cluster}, which is another equation for $u_\so$. Furthermore, we have the identity $\gamma_v(e_\so)=1$ and the corresponding equations for $u_{\nw}$ and $u_{\neast}$ in terms of the functions $\{A_{v'}\}_{v'}$ for trivalent vertices $v'$ of $\fW^+_v$, above $v$. Combining all of these we then obtain the identity

$$
u_{\so} =A_{v}^{\gamma_v(e_\so)}\cdot\prod_{v'\neq v}A_{v'}^{\g_{v'}(e_\so)}=\tz_{\neast}\cdot \prod_{v'\neq v}A_{v'}^{\g_{v'}(e_{\nw})+\g_{v'}(e_{\neast})}.
$$
This implies the identity
\begin{equation}
\label{eq: def cluster}
    A_{v}=\tz_{\neast}\cdot\prod_{v'\neq v}A_{v'}^{\g_{v'}(e_{\nw})+\g_{v'}(e_{\neast})-\g_{v'}(e_\so)}.
\end{equation}
Since $\gamma_w(e_{\nw})=\gamma_w(e_{\neast})=\gamma_w(e_{\so})=0$ for any trivalent vertex $w$ in $\fW$ below $v$, the right hand side of Equation \ref{eq: def cluster} involves only the trivalent vertices $v'$ above $v$, i.e.~the trivalent vertices in $\fW^+_v$. By induction hypothesis, the functions $\{A_{v'}\}_{v'}$ and $\tz_{\neast}$ are rational functions on the initial variables $(z_1,\ldots,z_\ell)$. Therefore, Equation \ref{eq: def cluster} uniquely defines a rational function $A_v\in\C(z_1,\ldots,z_\ell)$ in terms of the data assigned to $\fW_v^+$. In conclusion, this allows us to inductively propagate the labeling $\zeta$ downwards through a trivalent vertex, from $\fW_v^+$ to $\fW_v^-$, and define $A_v$ in the process.\\

\noindent {\it Propagating down through a hexavalent vertex}. Let $\vt$ be a hexavalent vertex with colors $i,j\in D$, so that the top edges have colors $i,j$ and $i$. Suppose that $\{A_v\}$ and $\zeta(e)$ are defined for the trivalent vertices $v$ and edges $e$ of the weave $\fW_\vt^+$ above the slice $H^+_\vt$ and satisfy the conditions of the statement of the theorem. The six edges incident to $\vt$ will be denoted as in Figure \ref{fig: Models_Propagation} (center).\\ 

\noindent There is {\it no} new variable $A_v$ being introduced at a hexavalent vertex. This is a marked difference with the case of the trivalent vertex treated above. Rather, in this case of a hexavalent vertex, the six equations
\begin{equation}
\label{eq: identities_uA}
u_e=\prod_{v} A_v^{\g_v(e)},\qquad e\in\{e_{\nw},e_\n,e_{\neast},e_{\sw},e_\so,e_{\se}\}
\end{equation}
are used to {\it uniquely} propagate the labeling $\zeta$ on $\fW_\vt^+$ to a unique labeling on $\fW_\vt^-$, so that $(\fW_\vt^-,\zeta)$ is still a framed weave. In particular, the variables $\tz_{\sw},\tz_{\so},\tz_{\se}$ are determined by $\tz_{\nw},\tz_{\n},\tz_{\neast}$ and the $u$-variables by Condition (3). 

A priori, this propagation of the labeling, which should satisfy Condition (3) in Definition \ref{def:framed_weave}, might be incompatible with the values of the cycles $\gamma_v$ at the edges $e_{\sw},e_\so,e_{\se}$. This compatibility is what we need to verify. Condition (3) in Definition \ref{def:framed_weave} reads
\begin{equation}\label{eq: Condition3}
u_{\nw}u_\n=u_\so u_{\se}\qquad\mbox{and}\qquad u_\n u_{\neast}=u_{\sw}u_\so,
\end{equation}
and we must argue that this is consistent with the identities \eqref{eq: identities_uA}.
For that, consider a cycle $\g_v$ arriving from the top at the hexavalent vertex $\vt$, with incoming top weights are $\g_v(e_{\nw})$, $\g_v(e_{\n})$ and $\g_v(e_{\neast})$. By virtue of being a Lusztig cycle, cf.~Definition \ref{def:Lusztig_cycles}, the outgoing weights $\g_v(e_{\sw})$, $\g_v(e_{\so})$ and $\g_v(e_{\se})$
satisfy
\begin{equation}
\label{eq:lusz_hex}
\g_v(e_{\nw})+\g_v(e_{\n})=\g_v(e_{\so})+\g_v(e_{\se}),\qquad \g_v(e_{\n})+\g_v(e_{\neast})=\g_v(e_{\sw})+\g_v(e_{\so}).
\end{equation}

\noindent Equations \eqref{eq: identities_uA} and \eqref{eq:lusz_hex} imply
$$u_{\nw}u_{\n}=\prod_{v} A_v^{\g_v(e_{\nw})}\cdot \prod_{v} A_v^{\g_v(e_\n)}=\prod_{v} A_v^{\g_v(e_{\nw})+\g_v(e_\n)}=\prod_{v} A_v^{\g_v(e_{\so})+\g_v(e_{\se})}=u_\so u_{\se}.$$
Thus, we conclude that the first identity in Equation \eqref{eq: Condition3} is indeed satisfied. The second identity $u_\n u_{\neast}=u_{\sw} u_\s$ is verified similarly.
\end{proof}

\noindent Given a Demazure weave $\fW$, we denote by $\zeta_\CA:=\zeta_\CA(\fW)$ the labeling constructed in Theorem \ref{thm:cluster vars}. Intuitively, the functions $\{A_v\}_{v\in\fW_3}$ in Theorem \ref{thm:cluster vars} measure the mutual relative positions for any given collection of framed flags in $\fM(\fW,\zeta_\CA)$, in a compatible way.\footnote{To wit, insert Equation \eqref{eq: u from cluster} for the $u$-variables in terms of $\{A_v\}$ into Condition \eqref{eq:compatible_flags} in Definition~\ref{def: compatible flags}.}

\begin{remark}
By Lemma \ref{lem:independence_dashedlines}, the rational functions $\{A_v\}$ in Theorem \ref{thm:cluster vars} do not depend on the exact heights of the vertices of the weave $\fW$. In consequence, we sometimes use the notation $A_v=A_v(\fW)$ for such rational functions, without explicitly mentioning $\fW^=$.
\end{remark}
\color{black}

\noindent Let us argue that the rational functions $\{A_v\}$ in Theorem \ref{thm:cluster vars} define an open toric
chart in $X(\beta)$. This toric chart, along with the functions $\{A_v\}$, will become the cluster torus associated to the weave $\fW$.


\begin{lemma}\label{lem:framedflag_moduli_torus}
Let $\fW$ be a Demazure weave for a positive braid word $\beta$ and $\fW_3$ its set of trivalent vertices. Then the open algebraic set
$$\mathfrak{D}_{A}:=\bigcap_{v\in\fW_3}(\fD(A_v)\cap \fD(A_v^{-1}))\sse X(\beta)$$
is an  algebraic torus, i.e.~ it is  isomorphic to $(\C^*)^{\dim X(\beta)}$. In addition, it is isomorphic to the moduli space $\fM(\fW,\zeta_\CA)$ of framed flags compatible with the labeling $\zeta_\CA$.
\end{lemma}

\begin{proof}
By Theorem \ref{thm:cluster vars}, $A_v$ are rational functions on $X(\beta)$. 
Let us show that $\mathfrak{D}_{A}\cong (\C^*)^d$, where $d=\dim X(\beta)$ is the number of trivalent vertices in $\fW$, as follows. Consider the rational functions $u_e$, indexed by the solid edges $e\in E_{sd}(\fW)$, as defined by Equation 
\eqref{eq: u from cluster}. By Lemma \ref{lem: framed flags from labeled weaves_2}, there exists a unique collection of Laurent polynomials $\tz_e(\underline{u})$ and $Y_e(\underline{u})$
in the $u$-variables $\underline{u}=\{u_e\}_{e\in E(\fW)}$ which defines a framed weave. Now, the functions $z_i$ coincide with $\tz_{e_i}(\underline{u})$ for the edges $e_i\in E(\fW)$ on the northern boundary of the weave $\fW$. Since $u_e$ are Laurent  monomials in $A_v$ by \eqref{eq: u from cluster}, we conclude that the functions $z_i$ are themselves Laurent polynomials in the variables $A_v$.\\

\noindent Let us write these Laurent dependencies as $\underline{u}=\underline{u}(\{A_v\})$ and $z_i=z_i(\{A_v\})=\tz_{e_i}(\underline{u}(\{A_v\}))$. Then we obtain the regular morphism 
$$
\phi:\Spec\C[A_{v}^{\pm1}]\simeq (\C^{*})^{d}\longrightarrow X(\beta),\quad d=|\fW_3|,\quad \phi(A_1,\ldots,A_d)\longmapsto (z_1(\{A_v\}),\ldots,z_{\ell(\beta)}(\{A_v\})). 
$$
Independently, $A_v$ are rational functions in the functions $z_i$, and therefore $\phi$ is an isomorphism onto its image which coincides with $\mathfrak{D}_A$. This concludes that $\fD_A$ is indeed an algebraic torus.\\ 

For the isomorphism $\fD_A\cong\fM(\fW,\zeta_\CA)$ we proceed as follows. First, we claim that all the rational functions $A_v$ are Laurent monomials in the rational functions $\{u_e\}$. Indeed, at a trivalent vertex $v$ we have the identity
\begin{equation}
\label{eq: A_v-via-u}
u_\so =A_v\cdot\prod_{v'\neq v}A_{v'}^{\gamma_{v'}(e_\so)},\qquad\mbox{and thus}\qquad A_v=u_\so \cdot\prod_{v'\neq v}A_{v'}^{-\gamma_{v'}(e_\so)}.
\end{equation}
For $v' \neq v$, $\gamma_{v'}(e_\so)$ is nonzero only if $v'$ is above $v$, and therefore we can assume by induction that $A_{v'}$ are already Laurent monomials in the $\{u_e\}$. In consequence, $A_v$ is a Laurent monomial in the $\{u_e\}$ as well. Second, the fact that each $A_v$ is a Laurent monomial in the $u$-variables $\{u_e\}$ implies that
$\mathfrak{D}(\fW,\zeta_\CA)$ coincides with $\mathfrak{D}_{A}$. Indeed, all $A_v$ are Laurent monomials in $u_e$, and therefore $A_v$ are well-defined on the domain $\mathfrak{D}(\fW,\zeta_\CA)$. Conversely, the functions $u_e$ are Laurent monomials in $A_v$ by \eqref{eq: u from cluster}, so they are well-defined and invertible on $\mathfrak{D}_{A}$. Furthermore, $\tz_e$ are Laurent polynomials in $A_v$, so these are well-defined on $\mathfrak{D}_{A}$ as well. This concludes $\mathfrak{D}(\fW,\zeta_\CA)=\mathfrak{D}_{A}$. Finally, Lemma \ref{lem: framed flags from labeled weaves}.(a) implies that the moduli space $\fM(\fW,\zeta_\CA)$ of framed flags compatible with the labeling $\zeta_\CA$ is indeed isomorphic to $\mathfrak{D}(\fW,\zeta_\CA)$ and hence to $\mathfrak{D}_A$, as required.
\end{proof}

\noindent The following two lemmas establish how the functions $A_v$ in Theorem \ref{thm:cluster vars} change under weave equivalences, in Lemma \ref{lem: variables 1212}, and under mutations, in Lemma \ref{lem: variables mutation}.

\begin{lemma}
\label{lem: variables 1212}
Let $\fW_1,\fW_2$ be the two weaves for $1212$, as in Example \ref{ex: 1212} $($see Fig.~\ref{fig: weave eq1}$)$, and denote their unique trivalent vertices by $v_1\in\fW_1,v_2\in\fW_2$. Then the variables $A_{v_1}(\fW_1)$ and $A_{v_2}(\fW_2)$ agree.
\end{lemma}

\begin{proof}
Let us denote both $v_1$ and $v_2$ by $v$, as the weaves determine the index. Suppose that the $\tz$-variables on the top are $\tz_1,\tz_2,\tz_3,\tz_4$ and for $v'\neq v$ the incoming edges have multiplicities $a_{v'},b_{v'},c_{v'},d_{v'}$. For $\fW_1$, the right incoming (red) edge at $v$ has position variable $\tz_4$, hence
$$
A_{v}(\fW_1)=\tz_4\prod_{v\neq v'}A_{v'}^{m_1(v')}
$$
where $m_1(v')=\left(a_{v'}+b_{v'}-\min(a_{v'},c_{v'})\right)+d_{v'}-\min(a_{v'}+b_{v'}-\min(a_{v'},c_{v'}),d_{v'}).$
\noindent For $\fW_2$, the right incoming (blue) edge at $v$ has position variable $$\tz_4\prod_{v'\neq v}A_{v'}^{ b_{v'}-c_{v'}},$$ by Lemma \ref{lem: braid matrix relations framed}. Therefore
$$
A_v(\fW_2)=\tz_4\prod_{v'\neq v}A_{v'}^{m_2(v')}
$$
where 
$$
m_2(v')=(b_{v'}-c_{v'})+a_{v'}+\left(c_{v'}+d_{v'}-\min(b_{v'},d_{v'})\right)-
$$
$$\min(a_{v'},c_{v'}+d_{v'}-\min(b_{v'},d_{v'}))=m_1(v')
$$
by Lemma \ref{lem: tropical}. Since $m_1(v')=m_2(v')$, this shows that $A_v(\fW_1)=A_v(\fW_2)$.
\end{proof}

\begin{lemma}
\label{lem: variables mutation}
Let $\fW_1,\fW_2$ be the two weaves for $\sigma_1^3$, as in Figure \ref{fig: weave mutation}, and for each of them let $v_1,v_2$ be the two trivalent vertices, $v_1$ on top of $v_2$. Then the cluster variables $\{A_v\}:=\{A_v(\fW_1)\}$ and $\{\overline{A_v}\}:=\{A_v(\fW_2)\}$ satisfy
$$
A_{v_1}\cdot \overline{A_{v_1}}=A_{v_2}\prod_{v'\neq v_1,v_2}A_{v'}^{ \left[\sharp_{\mathfrak{W_1}}(\g_{v'}\cdot\g_{v_1})\right]_{+}}+\prod_{v'\neq v_1,v_2}A_{v'}^{ -\left[\sharp_{\mathfrak{W_1}}(\g_{v'}\cdot\g_{v_1})\right]_{-}}
$$
and $A_v=\overline{A_{v}}$ for $v\neq v_1$.
\end{lemma}

\begin{proof}
We need to verify the statement for two trivalent vertices $v_1,v_2$. Suppose that the $\tz$-variables on the top are $\tz_1,\tz_2,\tz_3$ and for $v'\neq v_1,v_2$ the incoming edges have multiplicities $a_{v'},b_{v'},c_{v'}$.\\

\noindent In $\fW_1:(11)1\to 11\to 1$ the position variable at the right incoming edge at $v_1$ is $\tz_2$ and the cluster variable is
$$
A_{v_1}=\tz_2\prod_{v'}A_{v'}^{a_{v'}+b_{v'}-\min(a_{v'},b_{v'})}.
$$
The position variable at the right incoming edge at $v_2$ is $\tz_3-\tz_2^{-1}\prod_{ v'}A_v^{-2b_{v'}}$. Indeed, we have
$$
\varphi_i\left(\begin{matrix}
1 & -\tz_2^{-1}u_2^{-2}\\ 0 & 1\end{matrix}\right)B_i(\tz_3)\chi_i(u)=
B_i(\tz_3-\tz_2^{-1}u_2^{-2})\chi_i(u).
$$
Therefore the function $A_{v_2}$ at $v_2$ equals
$$
A_{v_2}=(\tz_3-\tz_2^{-1}\prod_{v\neq v'}A_v^{-2b_{v'}})A_{v_1}\prod_{v'}A_{v'}^{\min(a_{v'},b_{v'})+c_{v'}-\min(a_{v'},b_{v'},c_{v'})}=
$$
$$
(\tz_2\tz_3-\prod_{ v'}A_v^{-2b_{v'}})\prod_{v'}A_{v'}^{a_{v'}+b_{v'}+c_{v'}-\min(a_{v'},b_{v'},c_{v'})}.
$$
For $\fW_2:1(11)\to 11\to 1$, the function $\overline{A_{v_1}}$ at the top vertex $v_1$ is
$$
\overline{A_{v_1}}=\tz_3\prod_{v'}A_{v'}^{b_{v'}+c_{v'}-\min(b_{v'},c_{v'})}
$$
and  by \eqref{eq: trivalent framed}
$$
\overline{A_{v_2}}=(\tz_2-\tz_3^{-1}\prod_{ v'}A_v^{-2b_{v'}})\overline{A_{v_1}}\prod_{v'}A_{v'}^{a_{v'}+\min(b_{v'},c_{v'})-\min(a_{v'},b_{v'},c_{v'})}=
$$
$$
(\tz_2\tz_3-\prod_{ v'}A_v^{-2b_{v'}})\prod_{v'}A_{v'}^{a_{v'}+b_{v'}+c_{v'}-\min(a_{v'},b_{v'},c_{v'})}=A_{v_2}.
$$
Finally,
\begin{multline}
\label{eq: mutation vars}
A_{v_1}\overline{A_{v_1}}=\tz_2\tz_3\prod_{v'}A_{v'}^{a_{v'}+2b_{v'}+c_{v'}-\min(b_{v'},c_{v'})-\min(a_{v'},b_{v'})}=\\
A_{v_2}\prod_{v'}A_{v'}^{b_{v'}-\min(b_{v'},c_{v'})-\min(a_{v'},b_{v'})+\min(a_{v'},b_{v'},c_{v'})}+
\prod_{v'}A_{v'}^{ a_{v'}+c_{v'}-\min(b_{v'},c_{v'})-\min(a_{v'},b_{v'})}.
\end{multline}
By Lemma \ref{lem: arrows to I cycle}, we have that $\sharp_{\fW_1}(G_{v_2}\cdot G_{v_1})=-1$ and 
$$
b_{v'}-\min(b_{v'},c_{v'})-\min(a_{v'},b_{v'})+\min(a_{v'},b_{v'},c_{v'})=-\left[\sharp_{\fW_1}(G_{v'}\cdot G_{v_1})\right]_{-}
$$
and
$$
a_{v'}+c_{v'}-\min(b_{v'},c_{v'})-\min(a_{v'},b_{v'})=\left[\sharp_{\fW_1}(G_{v'}\cdot G_{v_1})\right]_{+},
$$
concluding that Equation \eqref{eq: mutation vars} coincides with the equation in the statement.
\end{proof}

\noindent The transformation described in Lemma \ref{lem: variables mutation} is precisely a cluster mutation, see Section \ref{sec: background}. Therefore, from this moment forward we refer to the functions $\{A_v(\fW)\}$ as the cluster variables associated to a Demazure weave $\fW$. Theorem \ref{thm:upper} at the end of this section proves that these functions are indeed cluster variables for a cluster structure.

\begin{theorem}\label{thm:mutation variables}
Let $\fW_1,\fW_2:\beta\to\delta(\beta)$ be two Demazure weaves. Then the collections of functions $A_v(\fW_1)$ and $A_v(\fW_2)$ are related by a sequence of cluster mutations.
\end{theorem}

\begin{proof}
By Lemma \ref{lem: demazure classification}, any two such Demazure weaves are related by a sequence of equivalence moves and weave mutations. By Lemma \ref{lem: variables 1212}, the equivalence moves for weaves do not change the collection $A_v$ and by Lemma \ref{lem: variables mutation} a weave mutation corresponds to a cluster mutation.
\end{proof}


\subsection{Cluster variables in inductive weaves}
\label{sec: SW cluster variables} 

In an inductive weave, the procedure for computing the cluster variables $A_v$ from Theorem \ref{thm:cluster vars} can be made more explicit, as we now describe. At each trivalent vertex of a left inductive weave, the northwest edge $e_{\nw}$ 
goes all the way to the top, but the northeast edge $e_{\neast}$ 
may be contained in some cycles. 

\begin{definition}
Let $\fW$ be a left inductive weave and $v,v'\in\fW$ trivalent vertices. By definition, $v'$ is said to cover $v$ if  $\g_{v'}(e_{\neast}) \neq 0$ where $e_\neast$ is the northeast edge of $v$. 
\end{definition}

\begin{theorem}\label{thm: ind weave vars}
Let $\fW$ be a left inductive weave, $v\in\fW$  a trivalent vertex with color $i \in \dynkin$. 
Then:

\begin{itemize}
    \item[$(1)$] The $\tz$-variable $s_v$ on the edge $e_{\neast}$ agrees with the corresponding $z$-variable.
    
\item[$(2)$]  The cluster variable $A_v(\fW)$ associated to $v\in\fW$ satisfies the equation
$$
A_v=s_v\cdot \prod_{v'\ \mathrm{covers}\ v} A_{v'}^{\g_{v'}(e_{\neast})}.
$$

\item[$(3)$]  Let 
$u_{v} := u_{\s}$ be the $u$-variable associated to the south edge $e_{\s}$ 
of $v$. Then \[
A_v= u_{v}.
\]
\end{itemize}
Part $(3)$ also holds for a right inductive weave.
\end{theorem}

\begin{proof}
For Part (1), all edges $e$ of the weave to the left of $v$, including the northwest edge 
$e_{\nw}$ at $v$, go all the way to the top, and thus we have $\g_{v'}(e)=0$ and $u_e=1$. On the northeast edge $e_{\neast}$ at $v$, we have the matrix $B_i(\tz)\chi_i(u)$ and, if we move $\chi_i(u)$ to the right as in Lemma \ref{lem: z to z tilde}, then $\tz$ would not change. Therefore $\tz=z$.\\

\noindent For Part (2), consider the edges $e_{\nw}, e_{\neast}$ and $e_{\so}$ at $v$. We have $\g_v'(e_{\nw}) = \g_v'(e_{\so}) = 0$ for all $v' \neq v$, and $\g_{v'}(e_{\neast}) \neq 0$ if and only if $v'$ covers $v$, so the result follows from Equation \eqref{eq: def cluster}. \\

\noindent For Part (3), we have $u_v= \prod_{v' \in \fW_3} A_{v'}^{\g_{v'}(e_{\so})}$ 
by \eqref{eq: u from cluster}. (Alternatively, see Equation \eqref{eq: A_v-via-u} and note that $\gamma_{v}(e_\so)=1$.) By the above, we have $\g_{v'}(e_\so) = \delta_{v, v'}$, which implies the result.
The proof for a right inductive weave is identical.
\end{proof}

Next, we consider the right inductive weave $\rind{\Delta\beta}$, constructed in Section \ref{sec:SW quiver}, and compare the variables $A_{v}(\rind{\Delta\beta})$, for the trivalent vertices $v\in \rind{\Delta\beta}$, with those cluster variables coming from the cluster structure on $\Conf(\beta) \cong X(\Delta\beta)$, as defined in \cite{SW} and described in Section \ref{sec:double bs} above. 
We will denote by $w$ the variables corresponding to the crossings of $\Delta$, and by $z$ the variables corresponding to the crossings of $\beta$.\\

As explained in Section \ref{sec:SW quiver}, the trivalent vertices of $\rind{\Delta\beta}$ correspond to the letters of $\beta = \sigma_{i_{1}}\cdots \sigma_{i_{r}}$. Let us denote by $A_{k}:=A_{v_k}(\rind{\Delta\beta})$ the cluster variables constructed in Theorem \ref{thm:cluster vars}, $1\leq k\leq r$, which are associated to the corresponding trivalent vertices in $\rind{\Delta\beta}$. Similarly, let us denote by $\widetilde{A}_{k}$ the rational function (in fact, polynomial) defined in Section \ref{sec:double bs}, i.e. $\widetilde{A}_{k}(z_1, \dots, z_r):= \Delta_{\omega_{i_{k}}}(B_{i_{1}}(z_1)\cdots B_{i_{k}}(z_k))$. 

\begin{proposition}\label{prop: coincidence cluster}
In a right inductive weave, we have $A_{k} = \widetilde{A}_{k}$ for all $1\leq k\leq r$.
\end{proposition}
\begin{proof}
For that, we use the description of the variables $A_{k}, \widetilde{A}_{k}$ in terms of distances between framed flags, as follows. First, consider a configuration of flags
\begin{center}
    \begin{tikzcd}
    (\borel \ar[r, "s_{j_{1}}"] & \borel_1 \ar[r, "s_{j_{2}}"] & \cdots \ar[r,"s_{j_{\ell}}"] & \borel_{{\ell}} \ar[r, "s_{i_{1}}"] &   \borel_{\ell+1} \ar[r, "s_{i_{2}}"] & \cdots \ar[r, "s_{i_{r}}"] &  \borel_{\ell + r})
    \end{tikzcd} $\in X(\Delta\beta)$
    \end{center}
where $\Delta = \sigma_{j_{1}}\cdots \sigma_{j_{\ell}}$. This admits a unique lift with the condition that $\borel$ is lifted to $\uni$. We denote by $\uni_s$ the lift of $\borel_s$. Let $k$ be such that $1\leq k\leq r$ and $i_{k} = i$. As established in  Section \ref{sec:double bs}, $\widetilde{A}_{k} = \Delta_{\omega_{i}}(\uni, \uni_{\ell+k})$.

\begin{figure}[ht!]
\centering
    \includegraphics[scale=1.2]{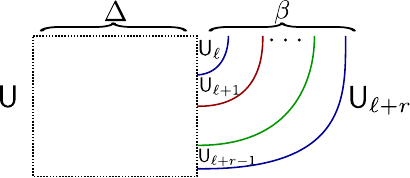}
    \caption{The right inductive weave for $\Delta\beta$ together with its collection of framed flags. Every horizontal slice inside the rectangle is a reduced word for $w_0$.}
    \label{fig: ind_deltabeta}
\end{figure}

\noindent 
Let $v$ be the trivalent vertex corresponding to $\sigma_{i_k}$ and let $u_v$ be the $u$-variable associated to the south edge $e_\so$ of $v$.
By Part (3) of Theorem \ref{thm: ind weave vars},
it suffices to show that $u_v=\Delta_{\omega_{i}}(\uni, \uni_{\ell+k})$. 
Consider a slice of the weave right below $v$, which gives rise to a sequence of framed flags $\mathsf{A}, \mathsf{A_{1}}, \dots, \mathsf{A_{\ell}}$, with $\mathsf{A} = \uni$ and $\mathsf{A_{\ell}} = \uni_{\ell+k}$, see Figure \ref{fig: ind_deltabeta}.
This slice of the weave spells a reduced decomposition of $\Delta$. Let $\rho_1',\ldots, \rho_{\ell}'$ be the sequence of positive roots as in Remark \ref{rmk: gamma opposite}. 
Following Lemma \ref{lem: no bifurcations}, the Lusztig cycles are only supported at the edges where the corresponding root 
$\rho_m'$ is simple, $m \in [l]$. Hence the $u$-variables satify 
$u_m=1$
unless 
$\rho_m'$
is simple. 
Therefore, the sequence $\mathsf{A}, \mathsf{A_{1}}, \dots, \mathsf{A_{\ell}}$ satisfies the properties of \cite[Lemma-Definition 8.3]{SG} and note that the only appearance of the root $\alpha_i$ is at the last step, i.e.~$\rho_{\ell}'=\alpha_i$ only for $m=\ell$. The required equality now follows from \cite[Proposition 3.5 (2)]{SG}.
\end{proof}

\begin{corollary}\label{cor: SW comparison}
Let $\beta\in\Br_W^+$ be a positive braid word. Then
\[
\C[X(\Delta\beta)] \cong \up{Q_{\rind{\Delta\beta}}} = \cluster{Q_{\rind{\Delta\beta}}}.
\]
In fact, $\C[X(\Delta\beta)] \cong \up{Q_{\mathfrak{w}}} = \cluster{Q_{\mathfrak{w}}}$ where $\mathfrak{w}: \Delta\beta \to \Delta$ is any Demazure weave.
\end{corollary}
\begin{proof}
By Proposition \ref{prop: coincidence cluster} and Corollary \ref{cor: coincidence quivers}, the first statement is equivalent to \cite[Theorem 3.45]{SW}. The second statement follows from Theorems \ref{thm:mutation quivers} and \ref{thm:mutation variables} above.
\end{proof}


\subsection{Existence of upper cluster structures}\label{ssec:upper_cluster} Theorem \ref{thm:main}, in the simply-laced case, is proven in two steps at this stage. First, for any braid $\beta \in \Br_{W}^{+}$, we now show that the algebra of regular functions $\C[X(\beta)]$ is an upper cluster algebra. Second, we prove $\mathcal{A}=\mathcal{U}$ in Subsection \ref{ssec:cyclic_rotation}, i.e.~we show that in this case the cluster algebra coincides with upper cluster algebra. The main result of this subsection is the following:


\begin{theorem}\label{thm:upper}
Let $\beta \in \Br_{W}^{+}$ be a positive braid word, $\lind{\beta}$ the left inductive weave and $Q_{\lind{\beta}}$ its corresponding quiver. Then we have
\[
\C[X(\beta)] \cong \up{Q_{\lind{\beta}}}.
\]
\end{theorem}

\begin{remark}\label{rmk:upper}
The use of the left inductive weave $\lind{\beta}$ simplifies part of the  arguments in the proof of Theorem \ref{thm:upper}. By Theorems \ref{thm:mutation quivers} and \ref{thm:mutation variables}, we will then also have $\C[X(\beta)] \cong \up{Q_\fW}$ where $\fW: \beta \to \delta(\beta)$ is any Demazure weave.
\end{remark}

\noindent In order to prove Theorem \ref{thm:upper} we need the following preparatory lemmata, describing how the braid variety $X(\beta)$ and the quiver $Q_{\overleftarrow{\mathfrak{w}}(\beta)}$ change upon adding a new crossing on the left of $\beta$.  The following is a more precise version of Lemma \ref{lem: loc closed}:

\begin{lemma}
\label{lem: add crossing}  Let $\beta$ be a positive braid word, $\delta=\delta(\beta)$ its Demazure product, and 
let $z=z_1\in\C[X(\sigma_i\beta)]$ be the $($restriction of the$)$ coordinate associated to the first crossing in $\sigma_i\beta$. Then the following holds:

\begin{itemize}
    \item[$(1)$] If $\delta(\sigma_i\beta)=\sigma_i\delta$, then $z=0$ on the braid variety $X(\sigma_i\beta)$ and $X(\sigma_i\beta)\cong X(\beta)$.
    
    \item[$(2)$] If $\delta(\sigma_i\beta)=\delta$, then we have an isomorphism
$$
X(\beta)\times \C^*\cong \{p\in X(\sigma_i\beta): z(p)\neq 0\}\subset X(\sigma_i\beta).
$$
\end{itemize}

\end{lemma}

\begin{proof}
For Part (1), note that the variety $X(\sigma_i\beta)$ is cut out by the conditions
$$
(\sigma_i\delta)^{-1}B_{\beta\sigma_i}(z,z_2,\ldots,z_{\ell(\beta)+1})=U,\quad
B_{\beta\sigma_i}(z,z_2,\ldots,z_{\ell(\beta)+1})=B_i(z)B_{\beta}=\sigma_i\delta U=B_{\sigma_i\delta}(0,\ldots,0)U
$$
for some $U\in B$. We can uniquely write $B_{\beta}=B_{w}(a_1,\ldots,a_{\ell(w)})U'$ for some reduced expression $w\leq \delta$ and some $a_i\in \C$. If we had $w<\delta$, then $\sigma_iw<\sigma\delta$, but we have
$$
B_i(z)B_{\beta}=B_{\sigma_iw}(z,a_1,\ldots,a_{\ell(w)})U',
$$
which is a contradiction. For $w=\delta$, we have $z=a_1=\ldots=a_{\ell(\delta)}=0$, and so
$U=U'$ and $B_{\beta}=\delta U$.\\

\noindent For Part (2), let us assume $z\neq 0$. Then we can decompose
\begin{equation}
\label{eq: Ui}
\left(\begin{matrix} z & -1\\
1 & 0\\ \end{matrix}\right)=\left(\begin{matrix} z& 0\\
1 & z^{-1}\end{matrix}\right)\left(\begin{matrix} 1 & -z^{-1}\\ 0 & 1\\ \end{matrix}\right),
\end{equation}
and factor $B_i(z)=L_i(z)U_i(z)$ accordingly. Now we also have
$$
\delta^{-1}B_i(z)B_{\beta}=\delta^{-1}L_i(z)U_i(z)B_{\beta}=U'\delta^{-1}\widetilde{B_{\beta}}U''
$$
where $U',U''$ are in $\borel$. Indeed, $U_i(z)B_{\beta}=\widetilde{B_{\beta}}U''$ by Corollary \ref{cor: move U}, and $\delta^{-1}L_i(z)=U'\delta^{-1}$ since $\ell(\sigma_i\delta)<\ell(\delta)$ and $\ell(\delta^{-1}\sigma_i)<\ell(\delta^{-1})$.
Therefore, for a fixed $z\neq 0$, the matrix $\delta^{-1}B_i(z)B_{\beta}$ is in $\borel$ if and only if $\delta^{-1}\widetilde{B_{\beta}}$ is in $\borel$, and the result follows.
\end{proof}

\noindent In the notation of Lemma \ref{lem: add crossing}, the next statement follows from the construction:

\begin{lemma}
\label{lem: freeze easy}
Let $\beta$ be a positive braid word, $\delta=\delta(\beta)$ its Demazure product, and assume $\delta(\sigma_i\beta)=\sigma_i\delta$. Then the inductive weave  $\lind{\sigma_i\beta}$ is obtained from $\lind{\beta}$ by adding a disjoint line, and the cycles and cluster variables for $X(\sigma_i\beta)$ and $X(\beta)$ agree. 
\end{lemma}

For the the case $\delta(\sigma_i\beta)=\delta$, the inductive weave  $\lind{\sigma_i\beta}$ is obtained by adding a trivalent vertex $v$ at the bottom left corner of $\lind{\beta}$.
Then the isomorphism $X(\beta)\times \C^*\cong \{z\neq 0\}$ from Lemma \ref{lem: add crossing}(b) can be extended to the weave $\lind{\beta}$ as in Lemma \ref{lem: slide weave}.

\begin{lemma}
\label{lem: add crossing vars}
Let $\beta\in\Br^+_W$ be a positive braid word with $\delta(\sigma_i\beta)=\delta(\beta)$, $i\in\dynkin$, and $v\in\lind{\sigma_i\beta}$ the trivalent vertex for $\sigma_i$. Let $z,a_1,\ldots,a_{\ell}$ be the variables associated to the slice of  $\lind{\sigma_i\beta}$ above $v$, read left-to-right, so that $z$ and $a_1$ are the incoming variables at the vertex $v$. Then we have:
\begin{itemize}
    \item[$(1)$] $
a_1=z^{-1},\ a_2=\ldots=a_{\ell}=0.
$
    
    \item[$(2)$] The cycles and the quiver for $\lind{\sigma_i\beta}$ and $\lind{\beta}$ agree, up to removing the vertex for $v$. The frozen vertices of the quiver for $\lind{\beta}$ are precisely the frozen vertices of the quiver for $\lind{\sigma_i\beta}$ together with those mutable vertices that have an arrow to the vertex for $v$.
    
    \item[$(3)$] The cluster variables for $\lind{\sigma_i\beta}$, except for $A_{v}$, and the cluster variables for $\lind{\beta}$ agree.
    
    \item[$(4)$] The variable $z$ is a cluster monomial for $\lind{\sigma_i\beta}$.
\end{itemize}
\end{lemma}

\begin{proof}
For Part (1), since the first output variable for the weave vanishes, we have $z-a_{1}^{-1}=0$ and $a_{1}=z^{-1}$. For the other output variables, we use the change of variables prescribed by Lemma \ref{lem: slide weave}. On the top of $\lind{\beta}$, we use the change of variables from Lemma \ref{lem: add crossing}(b) which is determined by the matrix $U_i(z)$ from \eqref{eq: Ui}.
At the bottom of $\lind{\beta}$, we can write $\delta=\sigma_i\gamma$, then
$$
\delta^{-1}U_i(z)B_{\delta}(a_1,\ldots,a_{\ell})=\delta^{-1}U_i(z)B_i(a_{1})B_{\gamma}(a_2,\ldots,a_{\ell})=
$$
$$
\delta^{-1}\varphi_i\left(\begin{matrix} 1 & -z^{-1}\\
0 & 1\\
\end{matrix} \right)\varphi_i\left(\begin{matrix} z^{-1} & -1\\
1 & 0\\
\end{matrix} \right)B_{\gamma}(a_2,\ldots,a_{\ell})=\delta^{-1}\left(\begin{matrix} 0 & -1\\
    1 & 0\\
\end{matrix} \right)B_{\gamma}(a_2,\ldots,a_{\ell})=$$
$$
\gamma^{-1}B_{\gamma}(a_1,\ldots,a_{\ell-1}) 
$$
This belongs to the Borel subgroup if and only if  $a_1=\ldots=a_{\ell-1}=0$. This concludes Part (1). Part (2) is immediate by construction, see also Lemma \ref{lem: add trivalent quiver}. Part (3) follows from Lemma \ref{lem: slide weave}. Indeed, the matrix $U_i(z)$ has a 1 on each diagonal entry, so it does not change the variable at the right incoming edge of any trivalent vertex. By Equation \eqref{eq: def cluster}, this implies that the cluster variables do not change as well.\\

\noindent For Part (4), note that the cycles in the weave $\lind{\sigma_i\beta}$ can approach $v$ only from the right. If the cycle corresponding to the cluster variable $A_i$ has weight $w_i$ on the northeast edge at $v$, then it has weight $\min(w_i,0)=0$ on the south edge. By Equation \eqref{eq: def cluster}, see also Theorem \ref{thm: ind weave vars}, we have 
$$
A_{v}=a_1\prod_{i}A_i^{w_i}=z^{-1}\prod_{i}A_i^{w_i},
$$
and therefore 
\begin{equation}\label{eqn: freeze last}
z=A_{v}^{-1}\prod_{i}A_i^{w_i}.
\end{equation}
\end{proof}

\begin{lemma}
\label{lem: freeze interesting}
Let $\beta\in\Br^+_W$ be a positive braid word with $\delta(\sigma_i\beta)=\delta(\beta)$, $i\in\dynkin$, and $v\in\lind{\sigma_i\beta}$ the trivalent vertex for $\sigma_i$. Suppose that the cluster variables $A_{u}$ for the braid variety $X(\sigma_i\beta)$ are regular functions, $A_u\in \C[X(\sigma_i\beta)]$, $u \in\lind{\sigma_i\beta}$ trivalent vertices with $u\neq v$, and that the cluster variable $A_v$ associated to $v$ is invertible. Then all cluster variables for $X(\beta)$ are regular functions, and all cluster variables with nonzero weight at $v$ are invertible.
\end{lemma}

\begin{proof}
By Lemma \ref{lem: add crossing vars}  all cluster variables for $\beta$ are regular on $X(\beta)\times \C^*$ and do not depend on $z$. Therefore all cluster variables are regular on $X(\beta)$. Furthermore, by assumption, $A_v$ is invertible on $X(\sigma_i\beta)$, and $z=A_v^{-1}\prod_{i}A_i^{w_i}$ is invertible on $X(\beta)\times \C^*$.
Since a product of regular functions is invertible if and only if each factor is invertible, we conclude that the cluster variables $A_i$ are invertible on $X(\beta)$ provided that $w_i>0$.
\end{proof}

\noindent The following lemma shows that Theorem \ref{thm:upper} holds for a braid word $\beta$
if it holds for the braid word $\sigma_{i}\beta$, for any $i\in\dynkin$.

\begin{lemma}\label{lem:upper}
Let $\beta \in \Br_{W}^{+}$ be a positive braid word and suppose that there exists an isomorphism $\C[X(\sigma_i\beta)] \cong \up{Q_{\lind{\sigma_i\beta}}}$ for some $i\in\dynkin$. Then we have
\[
\C[X(\beta)] \cong \up{Q_{\lind{\beta}}}.
\]
\end{lemma}
\begin{proof}
The case that $\delta(\sigma_{i}\beta) = s_{i}\delta(\beta)$ follows by Lemma \ref{lem: add crossing}(a) and Lemma \ref{lem: freeze easy}. Thus, we assume that $\delta(\sigma_i\beta) = \delta(\beta)$, and use the notation of Lemma \ref{lem: add crossing vars}. By the same Lemma \ref{lem: add crossing vars}, we can identify $X(\beta)$ with the algebraic subvariety $\{p\in X(\sigma_i\beta):z(p) = 1\} \subseteq X(\sigma_i\beta)$. By Equation \eqref{eqn: freeze last}, we also identify the algebra of regular functions $\C[X(\beta)]$ with the algebra obtained from $\C[X(\sigma_i\beta)]$ by freezing all cluster variables that have an arrow to the last variable in $Q_{\lind{\sigma_i\beta}}$ and, moreover, specializing $A_{v} = \prod A_{i}^{-w_{i}}$.\\

\noindent Note that when we freeze all variables adjacent to the last (frozen) variable in $Q_{\lind{\sigma_i\beta}}$ the quiver becomes disconnected and the specialization $A_{v} = \prod A_{i}^{-w_{i}}$ simply deletes the isolated vertex corresponding to $v$ from the quiver. Since $Q_{\lind{\beta}}$ is obtained from $Q_{\lind{\sigma_i\beta}}$  by exactly this procedure, see Lemma \ref{lem: add crossing vars}(b), we obtain the following inclusion, cf. \cite[Proposition 3.1]{Muller}:
\[
\C[X(\beta)] \subseteq \up{Q_{\lind{\beta}}}
\]

Let us show the reverse inclusion. By Lemma \ref{lem: add crossing vars}, the cluster variables for $\lind{\sigma_i\beta}$, without $A_v$, and the cluster variables for $\lind{\beta}$ agree. Every mutable variable in $\lind{\beta}$ is \emph{not}  to the last vertex in $Q_{\lind{\sigma_i\beta}}$ and it follows that in $\C[X(\beta)]$ we can mutate at all these variables and still get regular functions. Now, by Lemma \ref{thm:braid variety UFD} below, the algebra $\C[X(\beta)]$ is a UFD and by \cite[Theorem 3.1]{GLS}, all cluster variables are irreducible in $\up{Q_{\lind{\beta}}}$ and thus they are also irreducible in $\C[X(\beta)]$. Appealing once more to factoriality of $X(\beta)$, as well as to its smoothness, we use \cite[Corollary 6.4.6]{FWZ} (see Remark 6.4.4 in \emph{loc. cit.}) to conclude
\[
\up{Q_{\lind{\beta}}} \subseteq \C[X(\beta)].
\]

\end{proof}

\begin{proof}[Proof of Theorem \ref{thm:upper}]
By Corollary \ref{cor: SW comparison}, Theorem \ref{thm:upper} holds for words of the form $\Delta\beta$, $\beta\in\Br_W^+$ a positive braid word. Then the general result follows by Lemma \ref{lem:upper}, as we can use it to delete each crossing of $\Delta$, left to right, until we obtain the desired result for $\beta$. 
\end{proof}


\noindent We now complete the proof of Lemma \ref{lem:upper} by establishing the UFD property.

\begin{lemma}\label{thm:braid variety UFD} Let $\beta$ be a positive braid. Then the coordinate ring $\C[X(\beta)]$ is factorial.
\end{lemma}
\begin{proof}
The algebraic variety $X(\beta)$ is smooth and irreducible. By \cite[Exercise 14.2.T]{Vakil_rising_sea}, this implies that $\C[X(\beta)]$ is factorial if and only if the class group vanishes, i.e.~ $\mathrm{Cl}(X(\beta)) = 0$. In order to show that $\mathrm{Cl}(X(\beta))$ vanishes, we apply the excision exact sequence for class groups using an appropriate compactification of $X(\beta)$. Note that we have an isomorphism $\mathrm{Cl}(X(\beta)) \cong \mathrm{Pic}(X(\beta))$ because $X(\beta)$ is smooth and irreducible, see e.g. \cite[Corollary II.6.16]{Hartshorne}. It therefore suffices to argue that $\mathrm{Pic}(X(\beta))=0$.\\

Let us write $\beta = (i_1, \dots, i_{\ell})$. The braid variety $X(\beta)$ admits a smooth compactification by the closed brick manifold $Y(\beta)$ for the same braid word $\beta$. Brick manifolds are introduced and studied in \cite{Escobar}. In a nutshell, $Y(\beta)$ is defined similarly to $X(\beta)$: using Equation \eqref{eq:def braid variety} but with the condition $\borel_{k} \buildrel s_{i_{k}} \over \longrightarrow \borel_{k+1}$ instead replaced by the more general condition that either $\borel_{k} \buildrel s_{i_{k}} \over \longrightarrow \borel_{k+1}$ or $\borel_{k} = \borel_{k+1}$. By \cite[Theorem 3.3]{Escobar}, $Y(\beta)$ is a smooth projective variety of dimension $\ell(\beta) - \ell(\delta(\beta)) = \dim(X(\beta))$. By construction, $X(\beta)\sse Y(\beta)$ is open and, moreover, the irreducible components of $Y(\beta)\setminus X(\beta)$ are of the form $Y(\beta_k)$, where $\beta_k := (i_1, \dots, \widehat{i_k}, \dots, i_{\ell})$ satisfies that $\delta(\beta_k) = \delta(\beta)$. It follows from the excision exact sequence of class groups, cf.~\cite[Section 14.2.8]{Vakil_rising_sea}, that $\C[X(\beta)]$ is a unique factorization domain if and only if the irreducible components of $Y(\beta)\setminus X(\beta)$ span the Picard group $\Pic(Y(\beta))$. In order to show that these irreducible components span Picard group $\Pic(Y(\beta))$, we describe generators of $\Pic(Y(\beta))$, as follows.\\

\noindent The variety $Y(\beta)$ is a closed subset of the Bott-Samelson manifold $Z(\beta)$, see \cite[Section 3]{Escobar}. Here $Z(\beta)$ is defined identical to $Y(\beta)$ but without the conditions $\borel_1 = \borel, \borel_{\ell +1} = \delta\borel$, i.e.~in $Z(\beta)$, $\borel_1$ and $\borel_{\ell+1}$ can be arbitrary flags. By \cite{Anderson,LauTho}, the Picard group $\Pic(Z(\beta))$ is generated by line bundles $\{\mathcal{L}_k\}_{k = 1}^{\ell}$, one for each crossing in $\beta$. Moreover, by \cite{Harterich, Shchigolev}, restriction induces a surjection $\Pic(Z(\beta)) \to \Pic(Y(\beta))$. We conclude that $\Pic(Y(\beta))$ is generated by line bundles $\{\mathcal{L}_k|_{Y(\beta)}\}_{k = 1}^{\ell}$. Finally, to conclude that $\Pic(X(\beta))=0$ we note the following two items:
\begin{itemize}
    \item[(i)] If $\delta(\beta_k) = \delta(\beta)$, then the divisor associated to $\mathcal{L}_{k}|_{Y(\beta)}$ is $Y_{\beta_k}$, which is an irreducible component of the compactifying divisor $Y(\beta) \setminus X(\beta)$.

    \item[(ii)] If otherwise $\delta(\beta_k) < \delta(\beta)$, then the divisor associated to $\mathcal{L}_k|_{Y(\beta)}$ is trivial. In this case $\mathcal{L}_k|_{Y(\beta)}$ is the trivial line bundle on $Y(\beta)$.
\end{itemize}
Thus the irreducible components of $Y(\beta)\setminus X(\beta)$ span $\Pic(Y(\beta))$ and $\Pic(X(\beta))=0$.
\end{proof}


\subsection{Cyclic rotations and quasi-cluster transformations}\label{ssec:cyclic_rotation} In order to show the equality $\up{Q_{\lind{\beta}}} = \cluster{Q_{\lind{\beta}}}$ we use the notion of a quasi-equivalence of cluster structures, following C.~Fraser's work \cite{Fraser} and see also \cite{FraserSB}, as follows. Given a seed $\Sigma$ and a mutable variable $A_i$, consider the following ratio, which is the quotient of the two terms in the mutation formula from Equation \eqref{eq: cluster mutation plus minus}:
\begin{equation}
\label{eq: exchange ratio}
    y_i=\frac{\prod_{\varepsilon_{ij}>0}A_j^{\varepsilon_{ij}}}{\prod_{\varepsilon_{ij}<0}A_j^{-\varepsilon_{ij}}}=\prod_{j}A_j^{\varepsilon_{ij}}.
\end{equation}

\noindent Let $\Sigma,\Sigma'$ be two seeds in different cluster structures. By definition, the seeds $\Sigma,\Sigma'$ are called quasi-equivalent if they satisfy the following conditions:
\begin{itemize}
    \item[-] The groups of  monomials in frozen variables $\Sigma,\Sigma'$ agree. In other words, the frozen variables in $\Sigma'$ are monomials in the frozen variables in $\Sigma$, and vice versa.
    \item[-] The mutable variables in $\Sigma'$ differ from the mutable variables in $\Sigma$ by monomials in frozens.
    \item[-] The ratios \eqref{eq: exchange ratio} in $\Sigma$ and in $\Sigma'$ agree for any mutable variable. 
\end{itemize}

\noindent A key result \cite[Proposition 2.3]{Fraser} is that quasi-equivalence commutes with mutations; if we mutate two quasi-equivalent seeds in their respective vertices, the new seeds will be quasi-equivalent as well.\\

Let us now prove the equality $\up{Q_{\mathfrak{w}}} = \cluster{Q_{\mathfrak{w}}}$ by studying cyclic rotations of braids words. Consider two positive braids words $\beta=\sigma_{i_1}\cdots \sigma_{i_\ell}$ and $\beta'=\sigma_{i_2}\cdots \sigma_{i_\ell}\sigma_{j}$, with $\delta(\beta)=\delta(\beta')=w_0$ and $s_iw_0=w_0s_j$. Then Lemma \ref{lem: rotation} implies that

\begin{itemize}
    \item[$(a)$] The braid varieties $X(\beta)$ and $X(\beta')$ are isomorphic,
    
    \item[$(b)$] The isomorphism in Part (a) changes the variables as follows:
$$
(z_1,z_2,\ldots,z_{\ell})\mapsto (z_2,\ldots,z_{\ell},z'),\quad \mbox{ for some }z':=z'(z_1,z_2,\ldots,z_{\ell}).
$$
\end{itemize}

\noindent The goal is to show that this isomorphism is in fact a quasi-cluster transformation, when we consider the upper cluster structures on $X(\beta)$ and $X(\beta')$ built in Theorem \ref{thm:upper}. Let $\fW$ be an arbitrary Demazure weave for $\sigma_{i_2}\cdots \sigma_{i_{\ell}}$ and $\fW_1$, $\fW_2$ its extensions using $\sigma_{i}$ and $\sigma_j$ respectively, as depicted in Figure \ref{fig:weaves rotation}.

\begin{figure}[ht!]
\centering
    \includegraphics[scale=0.7]{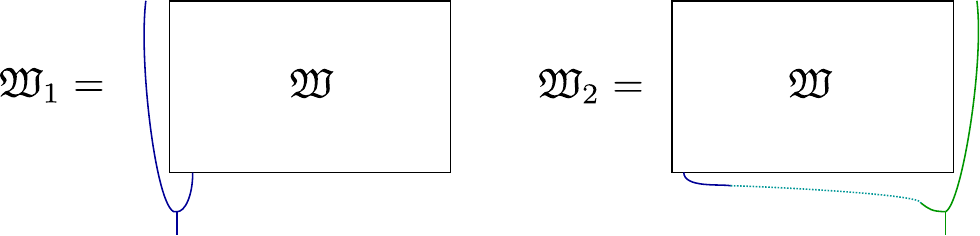}
    \caption{The weaves $\fW_1$ and $\fW_2$. We assume that the southern boundary of $\fW$ is a reduced word for $w_0$ starting with $s_i$. The equality $s_iw_0 = w_0s_j$ assures that we can bring the blue string on the left to the right using braid moves, as in $\fW_2$.}
    \label{fig:weaves rotation}
\end{figure}

\begin{lemma}
\label{lem: cycles through 6 valent}
Suppose that a cycle $C_i$ enters a 6-valent vertex $v$ with weights $(w_i,0,0)$ on top. Assume that we have labels $\zeta(e_{\nw})=(z_1,u_1)$, $\zeta(e_{\n})=(z_2,u_2)$ and $\zeta(e_{\neast})=(z_3,u_3)$ incoming in the top of $v$, and labels $\zeta(e_{\sw})=(z'_1,u'_1)$, $\zeta(e_{\so})=(z'_2,u'_2)$ and $\zeta(e_{\se})=(z'_3,u'_3)$ outgoing at the bottom of $v$. Then the functions $z'_3,u'_3$ are related to the functions $z_1,u_1$ by monomials in $A_j$ and
$$
(z_1u_1)^{-w_i}=(z'_3u'_3)^{-w_i}\prod_{j}A_j^{\sharp_{v}(C_i\cdot C_j)}.
$$
\end{lemma}

\begin{proof}
 The cycle $C_i$ exits $v$ with weights $(0,0,w_i)$. Suppose that other cycles $C_j$ ($j\neq i$) enter the $6$-valent vertex with weights $(a_j,b_j,c_j)$ and exit with weights $(a'_j,b'_j,c'_j)$. By Example \ref{ex: local intersection 100} we obtain $\sharp_{v}( C_i\cdot C_j)=w_i(b'_j-c_j)=w_i(b_j-a'_j)$. 

By Lemma \ref{lem: braid matrix relations framed} we have $z'_3=z_1\frac{u'_2}{u'_1}=z_1\prod{A_j}^{b'_j-a'_j}$, and by construction we have $u_1=A_i^{w_i}\prod A_j^{a_j},u'_3=A_i^{w_i}\prod A_j^{c'_j}$. Now
$$
(z_1u_1)^{-w_i}=z_1^{-w_i}A_i^{-w_i^2}\prod_j A_j^{-w_ia_j}
$$
while 
$$
(z'_3u'_3)^{-w_i}\prod_j A_j^{w_i(b_j-a'_j)}=z_1^{-w_i}\prod_j A_j^{-w_i(b'_j-a'_j)}A_i^{-w^2_i}\prod A_j^{-w_ic'_j}\prod_{j}A_j^{w_i(b_j-a'_j)}=z_1^{-w_i}A_i^{-w_i^2}\prod_j A_j^{-w_ia_j},
$$
where we have used the identity $b'_j+c'_j=a_j+b_j$.
\end{proof}


    

\begin{theorem}
\label{thm: quasi cluster}
Let $\fW_1$, $\fW_2$ as in Figure \ref{fig:weaves rotation}. 
Then:

\begin{itemize}
    \item[$(1)$] All the cluster variables for $\fW_1$ and $\fW_2$ agree, except for the last variables.
        
    \item[$(2)$] The last cluster variables for $\fW_1$ and $\fW_2$ are inverse to each other, up to monomials in frozens.

    \item[$(3)$] The cluster variables in $\fW_1$ and $\fW_2$ are related by a quasi-cluster transformation.
\end{itemize}
\end{theorem}

\begin{proof}
Part (1) holds by construction, as Lemma \ref{lem: rotation} shows that the variables $z_2,\ldots,z_{\ell}$ do not change. 

For Part (2), let $v_1$ and $v_2$ denote the bottom trivalent vertices of $\fW_1$ and $\fW_2$, respectively. Let $\widetilde{z_1}=z_1$ and $\widetilde{z_2}$ denote the variables at the left and right incoming edges of $v_1$, while $\widetilde{z_3}$ and $\widetilde{z_4}$ denote the variables at the left and right incoming edges of $v_2$. Note that $\widetilde{z_4}$ may differ from the variable $z'$ at the top of the weave. For $\fW_1$, we have $A_{v_1}=\widetilde{z_2}u$ and the right incoming edge carries the matrix $B_i(\widetilde{z_2})\chi_i(u)$, where $u=\prod_i A_i^{w_i}$. For $\fW_2$, the left incoming edge carries the matrix $B_j(\widetilde{z_3})\chi_j(u')$ and by Lemma \ref{lem: cycles through 6 valent} the variables  $\widetilde{z_3}$ and $\widetilde{z_2}$ (resp.~$u'$ and $u$) differ by a monomial in frozen variables.   
Equation \eqref{eq: trivalent framed} implies
$$
\widetilde{z_3}-(u')^{-2}\widetilde{z_4}^{-1}=0,\quad \mbox{i.e. } \widetilde{z_4}=(u')^{-2}(\widetilde{z_3})^{-1}.
$$
The cluster variable $A_{v_2}$ equals $\widetilde{z_4}u'=(\widetilde{z_3}u')^{-1}$ and thus it agrees with $(\widetilde{z_2}u)^{-1}=A_{v_1}^{-1}$ up to a monomial in frozen variables, as required.\\




\noindent For part (3), we need to verify that the ratios \eqref{eq: exchange ratio} agree.  
Let $C_i$ be a mutable cycle.
In the weave $\fW_1$ we have 
$$
\sharp_{\fW_1} (C_i\cdot C_{v_1})=\left|\begin{matrix}
1 & 1 & 1\\
0 & 0 & w_i\\
0 & 1 & 0\\
\end{matrix}\right|=-w_i,$$
so we obtain the equality
$$
y_i=A_{v_1}^{-w_i}\prod_j A_j^{\sharp (C_i\cdot C_j)}=(\widetilde{z_2}u)^{-w_i}\prod_j A_j^{\sharp_{\fW_1} (C_i\cdot C_j)}.
$$
By Lemma \ref{lem: cycles through 6 valent}  
this equals
$$
(\widetilde{z_3}u')^{-w_i}\prod_j A_j^{\sharp_{\fW_2} (C_i\cdot C_j)}.
$$
In $\fW_2$, the last cluster variable is $(\widetilde{z_3}u')^{-1}$ and the corresponding intersection index with $C_i$ equals
$$
\left|\begin{matrix}
1 & 1 & 1\\
w_i & 0 & 0\\
0 & 1 & 0\\
\end{matrix}\right|=w_i,
$$
from which the result follows.
\end{proof}

\noindent Let us remark that quasi-cluster transformations may not preserve the mutation class of the ice quiver $Q$, as the following example illustrates.

\begin{example}
Consider the braid word $\beta = \sigma_1\sigma_1\sigma_2\sigma_2\sigma_1\sigma_1\sigma_2$. The quiver $Q_{\fW}$ for any weave $\fW: \beta \to \delta(\beta)$ has three frozen variables, one mutable variables and it is of the form
\[
Q_{\fW} = \square \to \bullet \to \square \qquad \square.
\]
For $\beta' = \sigma_1\sigma_1\sigma_1\sigma_2\sigma_2\sigma_1\sigma_1$ and its right inductive weave $\rind{\beta'}$ the quiver reads
\[
Q_{\rind{\beta'}} = \bullet \to \square \qquad \square \qquad \square.
\]
\end{example}

\begin{remark}\label{lem: rotation mutable} Let $Q^{\uf}$ denote the full subquiver whose vertices are the mutable vertices of $Q$, and $\fW_1,\fW_2$ be weaves as in Figure \ref{fig:weaves rotation}. Then we have an equality $Q_{\fW_1}^{\uf} = Q_{\fW_2}^{\uf}$.
\end{remark}


\subsection{Theorem \ref{thm:main} in simply-laced case} Theorem \ref{thm: quasi cluster} allows us to conclude that the cluster algebra structure we have constructed on $\C[X(\beta)]$ coincides with its upper cluster algebra. This is proven as follows:

\begin{corollary}\label{cor: cluster monomials}
Let $\beta \in \Br_{W}^{+}$ be a positive braid word of length $r=\ell(\beta)$ and consider the upper cluster structure on $\C[X(\beta)]$ for $X(\beta)\sse\Spec\C[z_1,\ldots,z_r]$ constructed in Theorem \ref{thm:upper}. Then, for each $i$, $1\leq i\leq r$, there exists a cluster seed in $\C[X(\beta)]$ such that the restriction of the function $z_i$ to $X(\beta)$ is a cluster monomial in that seed.
\end{corollary}

\begin{proof}
By Lemma \ref{lem: add crossing vars} the variable $z_1$ is a cluster monomial in a cluster seed. By Theorem \ref{thm: quasi cluster}, we can consider the braid variety with variables $(z_2,\ldots,z_{r},z')$ and the corresponding cluster structures are related by a quasi-equivalence and mutations. Therefore $z_2$ is a cluster monomial as well.\footnote{Possibly in another cluster seed.} By repeating this procedure, we conclude that each $z_i$, $1\leq i\leq r$, is a cluster monomial in some cluster seed.
\end{proof}

\begin{theorem}[Theorem \ref{thm:main} in simply-laced case]\label{cor: cluster}
Let $\beta \in \Br_{W(\G)}^{+}$ be a positive braid word in a simply-laced algebraic simple Lie group $\G$ and $\mathfrak{w}: \beta \to \delta(\beta)$ a Demazure weave. Then we have
\[
\C[X(\beta)] \cong \up{Q_{\mathfrak{w}}} = \cluster{Q_{\mathfrak{w}}}.
\]
\end{theorem}
\begin{proof}
That $\C[X(\beta)] \cong \up{Q_{\mathfrak{w}}}$ is Theorem \ref{thm:upper} and Remark \ref{rmk:upper}. It is enough to conclude that $\C[X(\beta)] \subseteq \cluster{Q_{\mathfrak{w}}}$. By construction, see Corollary \ref{cor: braid via eqns}, $\C[X(\beta)]$ is generated by the variables $z_1, \dots, z_r$, $r=\ell(\beta)$, and thus the result follows from Corollary \ref{cor: cluster monomials}. 
\end{proof}

\noindent This concludes the proof of Theorem \ref{thm:main} in the simply-laced case and thus, by Theorem \ref{thm: richardson}, also proves Corollary \ref{cor:Leclerc} in its entirety.

\begin{remark}
Corollary \ref{cor: cluster} implies that any Demazure weave $\fW: \beta \to \delta(\beta)$ defines a cluster torus $T'_{\fW}:=\{\prod_{v} A_{v} \neq 0\} \subseteq X(\beta)$. Independently, Lemma \ref{lem: weave as chart} also associates a torus $T_\fW\subseteq X(\beta)$ to such a weave. It follows from Equation \eqref{eq: def cluster} that these two toric charts coincide, i.e.~ $T'_{\fW} = T_{\fW}$.
\end{remark}


\section{Non simply-laced cases}\label{sec: non simply laced}

In this section, we extend the results in the previous sections to an arbitrary non simply-laced Lie group, concluding the proof of Theorem \ref{thm:main}.  


\subsection{Construction of cluster structure}
\label{sec: non simply laced defs}

The construction of braid varieties for an arbitrary group $\G$ carries over as in Section \ref{sec: braid varieties} verbatim. The braid relations induce isomorphisms of braid varieties by \cite[Lemma 2.5]{SW} which are canonical by \cite[Theorem 2.18]{SW}.\\

We modify the definition of Demazure weaves as follows. Let $d_{ij}$ denote the length of the braid relation between the simple reflections $s_i$ and $s_j$ (that is, 3 for type $A_2$, 4 for $B_2$ and 6 for $G_2$). Instead of 6-valent vertices, we now use $(2d_{ij})$-valent vertices with $d_{ij}$ incoming and $d_{ij}$ outgoing edges (see Figure \ref{fig: B2A3}(left) for a $B_2$ example). This is similar to the Soergel calculus conventions \cite{EW}. The trivalent vertices for each $s_i$ are defined as usual. The proof of Lemma \ref{lem: weave as chart} goes through and any Demazure weave defines an open torus in the braid variety.\\

The definition of Lusztig cycles is generalized as follows. A cycle still starts at an arbitrary trivalent vertex. For a $(2d_{ij})$-valent vertex, one needs to use the more complicated tropical Lusztig rules as in \cite[Proposition 5.2]{Kamnitzer},  see also \cite[Section 7]{BZ} and \cite{Lusztig_QuantumBook} and Section \ref{sec: folding}. The rules for a trivalent vertex remain unchanged. Also, for any Lusztig cycle $\gamma_v$, there is a corresponding cycle $\gamma_v^\vee$ for the Langlands dual group, which satisfies the Langlands dual tropical Lusztig rules. Lemma \ref{lem: langlands dual cycles} below relates the cycles $\g_v$ and $\g^{\vee}_{v}$, but to state it we need some notation first. In what follows, if $\rho$ is a root of $\G$ we denote
\begin{equation}\label{eqn: coroot length}
d_{\rho} := \langle \rho^{\vee}, \rho^{\vee}\rangle
\end{equation}
where the pairing $\langle -,-\rangle$ is normalized so that if $\rho^{\vee}$ is a short coroot then $\langle \rho^{\vee}, \rho^{\vee}\rangle = 1$. Moreover, if $\fW$ is a weave and $v$ (resp. $e$) is a trivalent vertex (resp. edge) of $\fW$ colored by $i \in \dynkin$ then we define
\begin{equation}\label{eqn: multipliers}
d_{e} := d_{\alpha_{i}} =: d_{v}.
\end{equation}
Note that $d_{e}, d_{v} \in \{1, 2\}$ if $\G$ is of type $B,C$ or $F_{4}$, and $d_{e}, d_{v} \in \{1,3\}$ if $\G$ is of type $G_2$.

\begin{lemma}\label{lem: langlands dual cycles}
We have
$\gamma_v^\vee(e) = \gamma_v(e) d_ed_v^{-1}$.
\end{lemma}

\begin{proof}
The identity is clear near $v$, where $\gamma_v(e)=\gamma_v^{\vee}(e)=1$. We need to check that multiplication by $d_e$ changes Lusztig rules to their duals. For simplicity, we consider the doubly laced case and leave triply laced case to the reader. Suppose that we are in type $B_2$. If the root $\alpha_1$ is long and $\alpha_2$ is short then the tropical Lusztig rule is given by
\begin{equation}
\label{eq: phi B2 one}
\Phi_{B_2}(a,b,c,d)=(b+2c+d-p_2,p_2-p_1,2p_1-p_2,a+b+c-p_1),
\end{equation}
while if $\alpha_1$ is short and $\alpha_2$ is long then the tropical Lusztig rule is given by
\begin{equation}
\label{eq: phi B2 two}
\Phi^{*}_{B_2}(a,b,c,d)=(b+c+d-p_1,2p_1-p_2^*,p_2^*-p_1,a+2b+c-p_2^*),
\end{equation}
where
$$
p_1=\min(a+b,a+d,c+d),\ p_2=\min(2a+b,2a+d,2c+d),\ p_2^*=\min(a+2b,a+2d,c+2d).
$$
Observe that $p_1(a,2b,c,2d)=p_2^*$, $p_2(a,2b,c,2d)=2p_1$, so
$$
\Phi_1(a,2b,c,2d)=\left(\begin{matrix}
2 & 0 & 0 & 0\\
0 & 1 & 0 & 0\\
0 & 0 & 2 & 0\\
0 & 0 & 0 & 1\\
\end{matrix}\right)\Phi^*(a,b,c,d).
$$
\end{proof}

\noindent In the non-simply laced case, we can take the boundary intersection between a Langlands dual Lusztig cycle $C^{\vee}$ on $\fW$ -- that, we reiterate, is a cycle in $\fW$ that satisfies the Langlands dual tropical Lusztig rules -- and a Lusztig cycle $C'$ as follows: 
\begin{equation}
\label{eq: boundary intersection non simply laced}
\sharp_{\beta}(C^{\vee}\cdot C'):=\frac{1}{2}\sum_{i,j = 1}^{r}\sign(j-i)c_{i}^\vee c'_{j}\cdot(\rho_{i}, \rho^{\vee}_{j}).
\end{equation}
Note that for a simply-laced group, the Lusztig tropical rules and their Langlands dual coincide, so this formula is consistent with Definition \ref{def: boundary intersection}. By Lemma \ref{lem: boundary int} this allows one to define the intersection number of a cycle and a Langlands dual cycle at a $(2d_{ij})$-valent vertex. The intersection of a cycle and a Langlands dual cycle at a $3$-valent vertex does not change from the simply-laced case. We then define the exchange matrix:
\begin{equation}\label{eqn: exchange non simply laced}
\varepsilon_{i, j} := \sum_{v \; \text{vertex of $\fW$}}\sharp_{v}(\g^{\vee}_{i}, \g_{j}) + \sharp_{\delta(\beta)}(\g^{\vee}_{i}, \g_{j})
\end{equation}
where $\delta(\beta)$ denotes the bottom slice of $\fW$. Note that this specializes to Definition \ref{def:quiver} in the simply-laced case. This completes the definition of the exchange matrix. Note that in non simply-laced case it is not skew-symmetric but skew-symmetrizable as in \cite{SW}, see Lemma \ref{lem: skew symmetrizable} below.
More precisely, there are two separate pieces of data. First, the exchange matrix, which is the important data for the cluster algebra, and which is skew-symmetrizable but not skew-symmetric in non-simply-laced type. Second, there is the intersection form, which encodes the Poisson structure and is skew-symmetric; it is the skew-symmetrization of the exchange $\varepsilon$-matrix.\\

\noindent The choice of framing and the definition of cluster variables follow Section \ref{sec: cluster variables}. For the non-simply laced case, we will use a special class of weaves, generalizing the inductive weaves of Section \ref{sec: inductive weaves} that we introduce in Section \ref{sec: double inductive}.

\begin{remark}
In \eqref{eq: boundary intersection non simply laced} we matched the weights of {\bf cycles} $c'_j$ with {\bf coroots} $\rho_j^{\vee}$, while the  {\bf dual cycles} $c_i^{\vee}$ are matched with {\bf roots} $\rho_i$. This can be motivated as follows: in the definition of cluster variables in Theorem \ref{thm:cluster vars}  we evaluate the coroot $\chi_i(u)$ at $u=\prod A_i^{\gamma_i(e)}$, where the  cluster variables are weighted by Lusztig cycles. Thus cycles correspond to coroots, and dual cycles to roots.
\end{remark}


\subsection{Folding}
\label{sec: folding}

In order to understand the $(2d_{ij})$-valent vertices better, we can interpret non simply-laced rank 2 Dynkin diagrams by folding simply laced ones: $B_2$ is a folding of $A_3$ and $G_2$ is a folding of $D_4$. We will focus on the case of $B_2$ for reader's convenience, the case of $G_2$ is analogous.\\

\noindent The Dynkin diagram for $B_2$ has two nodes $1$ and $2$, we assume that $1$ corresponds to the long root. We can relate it to the Dynkin diagram for $A_3$ where the nodes $1$ and $3$ of the latter fold to the node $1$ in the former, and the nodes labeled by $2$ match. The $B_2$ braid relation $1212=2121$ corresponds to the braid {\em equivalence} $132132\sim 213213$ in $A_3$ which can be realized by the weave in Figure \ref{fig: B2A3}.

\begin{figure}[ht!]
\includegraphics[scale=2]{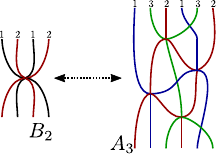}
\caption{$A_3$ weave unfolding the 8-valent vertex for $B_2$}
\label{fig: B2A3}
\end{figure}

In fact, there are two possible weaves here related by Zamolodchikov relation \cite[Section 4.2.6]{CGGS1}, and we can choose either one.  
Let us analyze the behaviour of Lusztig cycles under unfolding. First, the variables $t_i$ from \eqref{eq: lusztig} transform in the weave as follows:
$$
\left(t^a,t^a,t^b,t^c,t^c,t^d\right)\to \left(t^a,\frac{t^bt^c}{t^a+t^c},t^a+t^c,\frac{t^at^b}{t^a+t^c},t^c,t^d\right)\to
$$
$$
\left(t^a,\frac{t^bt^c}{t^a+t^c},t^a+t^c,\frac{t^ct^d(t^a+t^c)}{\pi_1},\frac{\pi_1}{t^a+t^c},\frac{t^at^bt^c}{\pi_1}\right)\to \left(t^a,\frac{t^bt^c}{t^a+t^c},\frac{t^ct^d(t^a+t^c)}{\pi_1},t^a+t^c,\frac{\pi_1}{t^a+t^c},\frac{t^at^bt^c}{\pi_1}\right)\to
$$
$$
\left(\frac{t^bt^{2c}t^d\pi_1}{\pi_2},\frac{\pi_2}{\pi_1},\frac{t^at^bt^c\pi_1}{(t^a+t^c)\pi_2},t^a+t^c,\frac{\pi_1}{t^a+t^c},\frac{t^at^bt^c}{\pi_1}\right)\to
$$
$$
\left(\frac{t^bt^{2c}t^d}{\pi_2},\frac{\pi_2}{\pi_1},\frac{\pi_2}{\pi_1},\frac{\pi_1^2}{\pi_2},\frac{t^at^bt^c}{\pi_1},\frac{t^at^bt^c}{\pi_1}\right).
$$
Here we have used the notation
$$
\pi_1:=t^at^b+(t^a+t^c)t^d,\quad \pi_2:=t^{2a}t^b+(t^a+t^c)^2t^d
$$
and employed the identities
$$
t^a\pi_1+t^ct^d(t^a+t^c)=\pi_2,\quad t^at^bt^c+\pi_2=(t^a+t^c)\pi_1.
$$
Note that the weights for the edges colored by $1$ and $3$ agree both in the input and the output.
By tropicalizing, we get precisely the equation \eqref{eq: phi B2 one}.
Similarly, any $B_2$ weave $\fW$ can be ``unfolded" to an $A_3$ weave $\fW'$ which has the following symmetry:

\begin{lemma}
\label{lem: Z2 action weave}
Let $\fW''$ be an $A_3$ weave obtained from $\fW'$ by swapping the colors for  $1$ and $3$. Then $\fW''$ is equivalent to $\fW'$ with added 4-valent vertices at the top and at the bottom.
\end{lemma}

\begin{proof}
This is a local check, so it is sufficient to check it for trivalent and 8-valent vertices in the $B_2$ weave $\fW$. A $2$-colored trivalent vertex in $\fW$ lifts to a 2-colored trivalent vertex in $\fW'$ or $\fW''$, so there is nothing to check. A $1$-colored trivalent vertex in $\fW$ to a pair of $1$- and $3$-colored trivalent vertices in $\fW'$, these are swapped in $\fW''$. Since we can move a $3$-colored strand through a $1$-colored trivalent vertex, and a $1$-colored strand through a $3$-colored trivalent vertex, we get the desired equivalence. Finally, an 8-valent vertex in $\fW$ lifts to a weave in Figure \ref{fig: B2A3} with reduced braid words on top and bottom, and any two such weaves are equivalent. 
\end{proof}

By abusing notations, one can say that there is a $\Z_2$ action on the weave $\fW'$ which sends it to $\fW''$ and adds 4-valent vertices at the top and at the bottom. We can summarize the properties of braid varieties and weaves under unfolding as follows:

\begin{proposition}
\label{prop: folding}
Let $\beta$ be a braid word for $B_2$, and $\beta'$ its unfolding to $A_3$ where we replace each $\sigma_1$ in $\beta$ (assume there are $n_1$ of these) by $\sigma_1\sigma_3$ in $\beta'$. Then the following holds:
\begin{itemize}
    \item[$(1)$] The group $H=(\Z_2)^{n_1}$ acts on $X(\beta')$ by swapping each $\sigma_1$ and $\sigma_3$, and swapping the corresponding $z$-variables.
    \item[$(2)$] The fixed point locus $X(\beta')^{H}$ is isomorphic to $X(\beta)$. Furthermore, the fixed point locus for the diagonal $\Z_2\subset H$ coincides with $X(\beta)$ as well.
    \item[$(3)$] Any $B_2$ weave $\fW$ for $\beta$ can be unfolded to an $A_3$ weave $\fW'$ for $\beta'$, and the action of $\Z_2$ extends to $\fW'$ as in Lemma \ref{lem: Z2 action weave}.
    \item[$(4)$] $B_2$ cycles lift to either one $\Z_2$-invariant cycle, or two $A_3$ cycles exchanged by the action of $\Z_2$. 
    \item[$(5)$] To calculate the entry $\varepsilon_{v,v'}$ of the exchange matrix for two trivalent vertices $v$ and $v'$ of $\fW$, one takes the intersection between the average of lifts of $\g_v$ and the sum of lifts of $\g_{v'}$ in $\fW'$. In this sense it is the restriction of the $A_3$ intersection form.
\end{itemize}
\end{proposition}

\begin{proof}
Parts (1)-(2)  are clear, (3) is a straightforward consequence of Lemma \ref{lem: Z2 action weave}, and (4) is clear. 

To prove (5), suppose that $\g_v$ and $\g_{v'}$ lift to $k$ and $k'$ cycles with total sums $\widetilde{\g_v}$ and $\widetilde{\g_{v'}}$ respectively. Let $\eta$ be an arbitrary slice of $\fW$ and $\eta'$ the corresponding slice of $\fW'$.  We claim that
$$
\sharp_{\eta}(\g^{\vee}_v\cdot \g_{v'})=\frac{1}{k}\sharp_{\eta'}\left(\widetilde{\g_{v}}\cdot \widetilde{\g_{v'}}\right).
$$

\noindent To lighten the notation, we will denote $C := \g_{v}$, $C^{\vee} := \g_{v}^{\vee}$ and $C' := \g_{v'}$. First we make some simple observations. Suppose that $\rho_i$ is a root associated to an edge in some slice. This edge has some color which is then associated with a simple root $\alpha$. By definition, $\rho_i$ is a Weyl group translate of $\alpha$. Edges with color $\alpha$ lift to $\langle \alpha^\vee, \alpha^\vee \rangle$ roots after unfolding, where we normalize the pairing $\langle -,- \rangle$ so that if $\rho^\vee$ is a short coroot, then $\langle \rho^\vee, \rho^\vee \rangle=1$. Thus an edge labelled by $\rho_i$ lifts to $\langle \rho_i^\vee,\rho_i^\vee \rangle$ roots after unfolding.

\noindent Now suppose that a root $\rho_i$ lifts to $a=\langle \rho_i^\vee,\rho_i^\vee \rangle$ roots $\widetilde{\rho}_{i,1},\ldots \widetilde{\rho}_{i,a}$ with the same weight $c_i$ in the unfolding, while a coroot $\rho_j^\vee$  lifts to $b=\langle \rho_j^\vee,\rho_j^\vee \rangle$ coroots $\widetilde{\rho}_{j,1}^\vee,\ldots \widetilde{\rho}_{j,b}^\vee$ with the same weight $c'_i$ in the unfolding.

\noindent We need a few facts:
\begin{itemize}
\item $c_i^\vee = \frac{\langle \rho_i^\vee,\rho_i^\vee \rangle}{k} c_i$, this is Lemma \ref{lem: langlands dual cycles}.
\item $\widetilde{\rho}_{i,1},\ldots \widetilde{\rho}_{i,a}$ are mutually orthogonal, so that $(\widetilde{\rho}_{i,x},\widetilde{\rho}^\vee_{i,y})=0$ unless $x=y$.
\item $(\rho_i,\rho_j^\vee) = (\widetilde{\rho}_{i,x}, \sum_{t=1}^{b}\widetilde{\rho}_{j,t}^\vee)$
for any lift $\widetilde{\rho}_{i,x}$ of $\rho_i$.
\end{itemize}

\noindent We are now ready to calculate
$$
c_i^\vee c'_j(\rho_i,\rho^{\vee}_j)=c_i^\vee c'_j\left(\widetilde{\rho}_{i,1},~ \sum_{t=1}^{b}\widetilde{\rho}_{j,t}^\vee\right) 
= \frac{a}{k}c_ic'_j\left(\widetilde{\rho}_{i,1},~\sum_{t=1}^{b}\widetilde{\rho}_{j,t}^\vee\right)
= \frac{1}{k}c_ic'_j\left(\sum_{t=1}^{a}\widetilde{\rho}_{i,t},~\sum_{t=1}^{b}\widetilde{\rho}_{j,t}^\vee\right).
$$
Hence the boundary intersections of $(C^{\vee},C')$ and $(\widetilde{C},\widetilde{C'})$ differ by a factor of $k$. 
Therefore the intersections at any $(2d_{ij})$-valent vertex differ by a factor $k$ as well, and the result follows.
\end{proof}

\noindent More generally, we can unfold any non simply-laced Dynkin diagram: $C_n$ unfolds to $A_{2n-1}$, $B_{n}$ unfolds to $D_{n+1}$, $G_2$ unfolds to $D_4$ and $F_4$ unfolds to $E_6$.  Proposition \ref{prop: folding} and its proof have a straightforward generalization to all these cases.
We define a diagonal matrix $D:=\operatorname{diag}(d_v)$ using Formula \eqref{eqn: multipliers}.

\begin{lemma}
\label{lem: skew symmetrizable}
The matrix $\varepsilon D^{-1}$ is skew-symmetric, so $\varepsilon$ is skew-symmetrizable.
\end{lemma}

\begin{proof}
Suppose that trivalent vertices $v,v'$ unfold to $d_v$ and $d_{v'}$ trivalent vertices, respectively. Let $\g_v,
\g_{v'}$ be the corresponding cycles, and let $\widetilde{\g_v},\widetilde{\g_{v'}}$ be the sum of all of their respective lifts. Then by Proposition \ref{prop: folding}(5) we get 
$$
\varepsilon_{v,v'}=\frac{1}{d_{v}}\left(\widetilde{\g_v},\widetilde{\g_{v'}}\right),\quad 
\varepsilon_{v',v}=\frac{1}{d_{v'}}\left(\widetilde{\g_{v'}},\widetilde{\g_{v}}\right),
$$
so $\varepsilon_{v,v'}d_{v'}^{-1}=-\varepsilon_{v',v}d_{v}^{-1}$ and the result follows.
\end{proof}


\subsection{Weave equivalence}

We would like  to relate different weaves by weave equivalences and mutations. The definition of weave mutation is unchanged, but the definition of weave equivalence is modified similarly to the 2-color relation in Soergel calculus \cite{EW}, see below. There is one such equivalence relation (generalizing 1212 from Figure \ref{fig: weave eq1}) for each rank 2 subdiagram, see Figure \ref{fig: weave equivalence B2 G2}. Informally, one can say that the weave equivalence allows one to push a trivalent vertex through a braid relation.


\begin{figure}[ht!]
 \begin{tikzpicture}
   \draw (1,2) node {$12121$};
   \draw (0,1) node {$21211$};
   \draw (2,1) node {$11212$};
   \draw (0,0) node {$2121$};
   \draw (2,0) node {$1212$};
   \draw (0.7,1.7)--(0,1.3);
   \draw (1.3,1.7)--(2,1.3);
   \draw (0,0.7)--(0,0.3);
   \draw (2,0.7)--(2,0.3);
   \draw (0.5,0)--(1.5,0);
 \end{tikzpicture}
 \qquad
 \qquad
 \begin{tikzpicture}
 \draw (1,2) node {$1212121$};
   \draw (0,1) node {$2121211$};
   \draw (2,1) node {$1121212$};
   \draw (0,0) node {$212121$};
   \draw (2,0) node {$121212$};
   \draw (0.7,1.7)--(0,1.3);
   \draw (1.3,1.7)--(2,1.3);
   \draw (0,0.7)--(0,0.3);
   \draw (2,0.7)--(2,0.3);
   \draw (0.5,0)--(1.5,0);
 \end{tikzpicture}
 \caption{Weave equivalences for $B_2$ (left) and $G_2$ (right) from braid word graphs
 }
 \label{fig: weave equivalence B2 G2}
\end{figure}

\begin{proposition}
\label{prop: weave equivalence non simply laced}
In any type, the weave equivalence does not change the $\varepsilon$-matrix, the intersection form and the cluster variables.
\end{proposition}

\begin{proof}
If a weave has no trivalent vertices inside, the intersection form can be computed using Lemma \ref{lem: boundary int} and, in particular, does not depend on the choice of braid relations for the fixed input and output.\\

\noindent Next, we need to check the equivalence relations in rank 2. In types $A_2$ and $A_1\times A_1$ this is done above. In type $B_2$, we unfold the 8-valent vertex to an $A_3$ weave as in Figure \ref{fig: B2A3}. For the weave equivalence, we have two cases: either we add a trivalent vertex labeled by 2, or we add a trivalent vertex labeled by 1 for $B_2$ which unfolds to a pair of trivalent vertices labeled by 1 and 3 for $A_3$. In the first case, after unfolding we get an $A_3$ weave with one trivalent vertex. By Theorem \ref{lem: demazure classification} any two such weaves are related by a sequence of (type $A$) weave equivalences and mutations. Since there is only one trivalent vertex, there are no mutations. In the second case, we have two trivalent vertices, but the corresponding type $A$ quiver has two frozen and no mutable vertices, so there are no mutations either.\\

\noindent Therefore by Lemma \ref{lem: quiver 1212} and Lemma \ref{lem: variables 1212} the cluster variables and the intersections between cycles in the unfolded weave do not change, hence they do not change for the $B_2$ weave as well. 
\end{proof}


\subsection{Double inductive weaves}
\label{sec: double inductive} 
We would like to encode ways of writing $\beta$ by adding letters on the left and on the right. This is reminiscent of the double-reduced words of Berenstein, Fomin and Zelevinsky \cite{BFZ}. We will notate such a way of writing $\beta$ by a \emph{double string} of entries of the form $iX$ where $i$ is the number of a node in the Dynkin diagram, and $X=L$ or $R$. The entry $iX$ means that we should add the braid letter $\sigma_i$ on the left if $X=L$ and on the right if $X=R$. For example, $(2L, 1R, 3R, 1L, 2L, 2R)$ encodes writing the positive braid word $\sigma_2 \sigma_1 \sigma_2 \sigma_1 \sigma_3 \sigma_2$ using the following string of subwords: $\underline{\sigma_2}$, $\sigma_2 \underline{\sigma_1}$, $\sigma_2 \sigma_1 \underline{\sigma_3}$, $\underline{\sigma_1} \sigma_2 \sigma_1 \sigma_3$, $\underline{\sigma_2} \sigma_1 \sigma_2 \sigma_1 \sigma_3$, $\sigma_2 \sigma_1 \sigma_2 \sigma_1 \sigma_3 \underline{\sigma_2}$.

Suppose that we can write $\beta$ using the double string $(i_1X_1, i_2X_2, \dots, i_lX_l)$. Let us call $\beta_k$ the $k$-th subword (of length $k$) coming from a double string. We may set $\beta_0=e$, the identity. Then $\beta_{k+1}=\sigma_{i_{k+1}} \beta_k$ or $\beta_k \sigma_{i_{k+1}}$ depending on whether $X_{k+1}=L$ or $R$, respectively.\\

\noindent Let us now construct the weave associated to a double string, that we call a {\bf double inductive weave}. At each stage we get a weave from $\beta_k$ to $u_k:=\delta(\beta_k)$. We start with the empty weave. If $\ell(u_{k+1})=\ell(u_k)+1$, then we just add a strand of color $i_{k+1}$ on the left or right, depending on whether $X_{k+1}=L$ or $R$. Otherwise, we have $\ell(u_{k+1})=\ell(u_k)$. In this case, 
we add a strand of color $i_{k+1}$ on the left or right, and this strand can form a trivalent vertex with an additional strand of color $i_{k+1}$. In both cases, we see that we get a weave from  $\beta_{k+1}$ to $u_{k+1}=\delta(\beta_{k+1})$. For example, the left inductive weave $\lind{\beta}$ is the weave associated to $(i_rL, i_{r-1}L, \dots, i_{1}L)$, while the right inductive weave $\rind{\beta}$ is associated to $(i_1R, i_2R, \dots, i_rR)$.\\

\noindent Note that in the first entry in the double string, the $L$ or $R$ is superfluous, and does not affect the resulting string of subwords or the corresponding weave. We will sometimes suppress $X_1$ or freely change it between $L$ and $R$.\\

\noindent We will often abbreviate the first $k$ entries in the double string by $\beta_k$ if we are not concerned with this part of the double string. For example, we might write a double string as $(\beta_k, i_{k+1}X_{k+1}, i_{k+2}X_{k+2}, \dots)$. It will be convenient to introduce a book-keeping device into our notation. Given a double string $(i_1X_1, i_2X_2, \dots, i_lX_l)$, let us write the $(k+1)$-st entry as $i_{k+1}X_{k+1}^+$ when $\ell(u_{k+1})=\ell(u_{k})+1$. In other words, we will add a superscript ``+'' to those entries that increase the length of the Demazure product. For example, if we are working in type $A_4$, the word $(2L, 1R, 3R, 1L, 2L, 2R)$ would be written $(2L^+, 1R^+, 3R^+, 1L^+, 2L, 2R^+)$.

\begin{remark}
Note that, given a double string for $\beta$, $(i_{1}X_1, i_{2}X_2, \dots)$ where $X_{i} \in \{R, L\}$, the cluster variables for the braid variety $X(\beta)$ are in correspondence with the steps that do not increase the length of the Demazure product, that is, with the complement of those entries that have a $+$ in the superscript.  Theorem \ref{thm: ind weave vars}(3) is valid for the double inductive weaves, with the same proof.
\end{remark}

\begin{theorem}\label{thm: double inductive}
Let $\fW_1$ and $\fW_2$ be double inductive weaves for the braid word $\beta$ in arbitrary type. Then, $\fW_1$ and $\fW_2$ are related by a sequence of weave equivalences and mutations.
\end{theorem}
\begin{proof}
 We will consider the following kinds of moves on double strings:

$$(i_1L, i_2X_2, \dots) \longleftrightarrow (i_1R, i_2X_2, \dots)$$
$$(\beta_k, iL, jR, \dots) \longleftrightarrow (\beta_k, jR, iL,  \dots)$$

\noindent First, observe that any two double strings for the same braid word $\beta$ are related by a series of the two moves above. The first move is trivial, as remarked before, and does not change the weave. The second move breaks down into several cases. We will break up the cases according to the lengths of $\ell(s_i*u_k)$, $\ell(u_k*s_j)$ and $\ell(u_{k+2})$, which we will now analyze. \\

\begin{enumerate}
\item {\bf Case 1:} $(\beta_k, iL^+, jR^+, \dots) \longleftrightarrow (\beta_k, jR^+, iL^+,  \dots)$

\noindent First, let us suppose that $\ell(s_i*u_k) = \ell(u_k)+1$, $\ell(u_k*s_j) = \ell(u_k)+1$ 
and $\ell(u_{k+2})=\ell(u_{k})+2$. Both weaves come from just adding an $i$ strand on the left and a $j$ strand on the right. Thus, the weave does not change, the cluster variables do not change, and the cluster variables are still attached to the same entries.\\

\item {\bf Case 2:} $(\beta_k, iL^+, jR, \dots) \longleftrightarrow (\beta_k, jR^+, iL,  \dots)$

\noindent This is the case where $\ell(s_i*u_k) = \ell(u_k)+1$ and $\ell(u_k*s_j) = \ell(u_k)+1$, but $\ell(u_{k+2})=\ell(u_{k})+1$.

\noindent We have that
\begin{align*} 
u_{k+2} &= s_i*u_k*s_j \\
&=s_i*u_k \\
&=s_iu_k \\
&=u_k*s_j \\
&=u_ks_j.
\end{align*}

\noindent From this, we get that $s_iu_k=u_ks_j$. Because $\ell(u_k) < \ell(u_ks_j)$, we know that $u_k$ cannot be written with an $s_j$ on the right. However, $s_iu_k$ \emph{can} be written with an $s_j$ on the right. This means that this $s_j$ must come from moving $s_i$ to the right through $u_k$ using a series of braid moves. Similarly, moving $s_j$ to the left through $u_k$ using a series of braid moves gives an $s_i$ on the left.\\

\noindent Let us now compare the weaves coming from the two different double strings: The weave for $(\beta_k, iL^+, jR, \dots)$ comes from adding an $i$ strand on the left, pulling it through $u_k$ using braid moves, and then merging with the $j$ strand on the right to get a trivalent vertex. The weave for $(\beta_k, jR^+, iL, \dots)$ comes from adding an $j$ strand on the right, pulling it through $u_k$ using braid moves, and then merging with the $i$ strand on the left to get a trivalent vertex.  These two weaves are related by a series of equivalences, 
see Figure \ref{fig: double weave equivalence} below.\\

\noindent Thus we have that the weaves are equivalent. The cluster variables stay the same, but the cluster variable attached to the entry $jR$ in $(\beta_k, iL^+, jR, \dots)$ becomes the cluster variable attached to the entry $iL$ in $(\beta_k, jR^+, iL, \dots)$.

An important specialization of this is when $u_k=w_0$. Under this specialization, we will have that $j=i^*$. This situation will arise repeatedly in Section \ref{sec: Richardson}, when we compare our work with previous work on cluster structures on Richardson varieties.

\begin{figure}[h!]
\centering
    \includegraphics[scale=1]{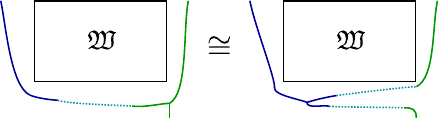}
    \caption{On the left, the weave for the double sequence $(\beta_k, iL^{+}, jR)$. On the right, the weave for $(\beta_k, jR^{+}, iR)$. These weaves are equivalent.}
    \label{fig: double weave equivalence}
\end{figure}

\item {\bf Case 3:} $(\beta_k, iL^+, jR, \dots) \longleftrightarrow (\beta_k, jR, iL^+,  \dots)$

\noindent This is the unique case where $\ell(s_i*u_k) = \ell(u_k)+1$ and $\ell(u_k*s_j) = \ell(u_k)$. In this case we must have that $\ell(u_{k+2}) = \ell(s_i*u_k*s_j)= \ell(u_k)+1$.\\


\noindent In this case, because adding $j$ to the right of $\beta_k$ results in a trivalent vertex, one can write a reduced word for $u_k$ with an $s_j$ on the right. This means that the trivalent vertex coming from adding $j$ on the right does not interact with adding a strand $i$ on the left. Thus, the weave does not change, the cluster variables do not change, and the cluster variables are still attached to the same entries.\\

\noindent There is a similar case with the roles of $L$ and $R$ reversed, which can be treated similarly.\\

\noindent Cases 4 and 5 will now deal with what happens when $\ell(s_i*u_k) = \ell(u_k)$ and $\ell(u_k*s_j) = \ell(u_k)$. In both these cases, we have that $\ell(s_iu_k)=\ell(u_k)-1$ and $\ell(u_ks_j)=\ell(u_k)-1$. Therefore we have that $\ell(s_iu_ks_j)=\ell(u_k)$ or $\ell(u_k)-2$. We deal with the latter case first.\\

\item {\bf Case 4:} $(\beta_k, iL, jR, \dots) \longleftrightarrow (\beta_k, jR, iL,  \dots)$ and $\ell(s_iu_ks_j)=\ell(u_k)-2$.

\noindent If $\ell(s_iu_ks_j)=\ell(u_k)-2$ means that $u_k$ has a reduced expression of the form $s_i \cdots s_j$. Thus adding an $i$ strand on the left and a $j$ strand on the right gives trivalent $i$ vertex on the left and a trivalent $j$ vertex on the right. These trivalent vertices do not interact with each other. Thus the resulting double inductive weave are identical, the cluster variables are the same, and they remain attached to the same entries.\\

\item {\bf Case 5:}$(\beta_k, iL, jR, \dots) \longleftrightarrow (\beta_k, jR, iL,  \dots)$ and $\ell(s_iu_ks_j)=\ell(u_k)$.

\noindent In this case, we have that $u_k=s_iv$ for some reduced word $v$ of length one less than $u_k$. Note that $v$ cannot be written with $s_j$ at the end. Thus $\ell(vs_j) = \ell(v)+1=\ell(u_k)$. Let $\gamma$ be the lift of $v$ to the braid group. From this, we have that $\ell(s_i*u_k*s_j)=\ell(u_k*s_j)$. This means that $u_k=\delta(s_i*u_k*s_j)=\delta(u_k*s_j)$. Therefore we have $u_k=vs_j$. This means that when we write $u_k$ with the strand $i$ on the left, and in order to use braid moves to write it with strand $j$ on the right, we have to pull the $i$ strand through $v$ to get the $j$ strand on the left.\\

\noindent Now we can compare the weaves on the two sides. The weave for $(\beta_k, iL, jR, \dots)$ comes from writing $u_k$ with an $i$ strand on the left. We add another $i$ strand and create a trivalent vertex. The $i$ strand on this trivalent vertex then gets pulled to the right using braid moves until it becomes a $j$ strand on the right, which merges with the $j$ strand added on the right to give another trivalent vertex, see Figure \ref{fig: double weave mutation} below.\\

\noindent There are two cluster variables. The first variable, which is attached to $iL$, has a cycle starting at the left $i$ trivalent vertex and ending on the right $j$ trivalent vertex. The second cluster variable, which is attached to $jR$, starts at the right $j$ trivalent vertex and goes downwards.\\

\noindent Mutation at the cycle corresponding to the first variable gives precisely the weave corresponding to $(\beta_k, jR, iL,  \dots)$. The cluster variable formerly attached to $iL$ mutates to become the one attached to $jR$. The variable formerly attached to $jR$ does not change, but it is now labelled by $iL$.\\

\noindent This case has some similarities to Case 2, with the role of $u_k$ in Case 2 now played by $v$. An important specialization of Case 5 is when $u_k=w_0$. Under this specialization we will again have that $j=i^*$. This situation will also arise repeatedly in Section \ref{sec: Richardson}.

\begin{figure}[h!]
\centering
    \includegraphics[scale=1]{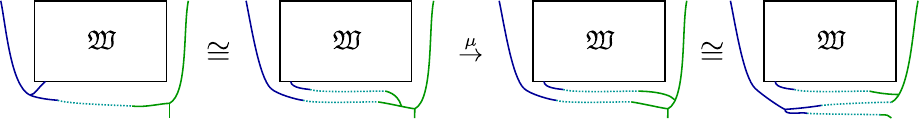}
    \caption{On the left, the weave for the double sequence $(\beta_k, iL, jR, \dots)$. On the right, the weave for $(\beta_k, jR, iL, \dots)$. These weaves are related by a mutation.}
    \label{fig: double weave mutation}
\end{figure}
\end{enumerate}

In summary, Cases 1, 3 and 4 are uninteresting. The moves
$$(\beta_k, iL^+, jR^+, \dots) \longleftrightarrow (\beta_k, jR^+, iL^+,  \dots)$$
$$(\beta_k, iL^+, jR, \dots) \longleftrightarrow (\beta_k, jR, iL^+,  \dots)$$
$$(\beta_k, iL, jR, \dots) \longleftrightarrow (\beta_k, jR, iL,  \dots) \textrm{ and } \ell(s_iu_ks_j)=\ell(u_k)-2$$
involve no changes in either cluster variables or which entry corresponds to which cluster variable.\\

\noindent Case 2 is mildly interesting in that the move
$$(\beta_k, iL^+, jR, \dots) \longleftrightarrow (\beta_k, jR^+, iL,  \dots)$$
changes the entry corresponding to the unique cluster variable, though the weave is unchanged.\\

\noindent Case 5 is the only move involving a mutation. In the move 
$$(\beta_k, iL, jR, \dots) \longleftrightarrow (\beta_k, jR, iL,  \dots) \textrm{ and } \ell(s_iu_ks_j)=\ell(u_k),$$
the cluster variable attached to $iL$ on the left mutates to the cluster variable attached to $jR$ on the right, while the cluster variable attached to $jR$ on the right does not change but becomes labelled by the cluster variable attached to $iL$ on the right.
\end{proof}

\begin{corollary}
\label{cor: double inductive}
In arbitrary type, the cluster seeds associated to any two double inductive weaves are mutation equivalent.
\end{corollary}

\begin{proof}
The proofs of Lemma \ref{lem: quiver mutation} and Lemma \ref{lem: variables mutation} still apply, so weave mutations correspond to the mutations of the exchange matrix and cluster variables. By Proposition \ref{prop: weave equivalence non simply laced}, weave equivalences do not change the exchange matrix or cluster variables. Now the result follows from Theorem \ref{thm: double inductive}.
\end{proof}


\subsection{Cluster structures in the non simply-laced case} With these results and notations, we are ready to prove Theorem \ref{thm:main} in the non-simply laced case for double inductive weaves.  

\begin{theorem}
\label{thm: non simply laced}
Let $\G$ be a simple algebraic group, $\beta\in \Br_{W}^{+}$ a positive braid word and $\mathfrak{w}: \beta \to \delta(\beta)$ a double inductive weave. Then we have
\[
\C[X(\beta)] \cong \up{\varepsilon_{\mathfrak{w}}} = \cluster{\varepsilon_{\mathfrak{w}}},
\]
where $\varepsilon_{\mathfrak{w}}$ is the skew-symmetrizable exchange matrix associated to $\mathfrak{w}$ in Section \ref{sec: non simply laced defs}.
\end{theorem}

\begin{proof}
The proof follows the argument for the simply laced case, as presented in Section \ref{sec: cluster variables}, quite closely. Thus we only list the key steps and necessary changes:
\begin{itemize}
    \item[(1)] Since we are considering double inductive weaves, where Theorem \ref{thm: double inductive} and Corollary \ref{cor: double inductive} apply, the cluster seeds associated to the left and right inductive weaves are mutation equivalent.\\
    
    \item[(2)] In the case of Bott-Samelson varieties, the results of Section \ref{sec:SW quiver} in the simply-laced case imply the corresponding results (for Bott-Samelson varieties) in the non-simply laced case, as follows. Assume the Dynkin diagram $\dynkin$ is obtained from $\dynkin'$ via folding. Note that the unfolding of the longest word in $W(\dynkin)$ is the longest word in $W(\dynkin')$. Thus, a braid of the form $\Delta\beta \in \Br(\dynkin)$ unfolds to $\Delta'\beta' \in \Br(\dynkin')$. As for the weaves, except possibly for $4$-valent vertices that do not influence the exchange matrix, the inductive weave $\rind{\Delta\beta}$ unfolds to $\rind{\Delta'\beta'}$; this follows from Remark \ref{rmk: ind weaves arms}. The result now follows since the exchange matrix $B$ for $\Conf(\beta)$ is obtained from that of $\Conf(\beta')$ via folding, see e.g. \cite[Section 3.6]{FockGoncharovII}.
    The results of Section \ref{sec: SW cluster variables} go through in the non-simply laced case with the same proofs. Thus, at this point we can conclude that Theorem \ref{thm:main} is true in the non-simply laced case for words of the form $\Delta\beta$.\\
    \item[(3)] The freezing argument from Lemma \ref{lem: freeze interesting} remains unchanged and applies in the non-simply laced case, from which we conclude the equality $\C[X(\beta)] = \up{\varepsilon_{\mathfrak{w}}}$ between the ring of functions and the upper cluster algebra.\\
    \item[(4)] Finally, in order to prove that cyclic rotation is a quasi-cluster transformation, we use the corresponding statement of Theorem \ref{thm: quasi cluster}, which follows from the simply laced case by unfolding. Note that if the weave $\fW$ in Figure \ref{fig:weaves rotation} is double inductive, then both $\fW_1$ and $\fW_2$ are double inductive as well. This proves that $\C[X(\beta)] \subseteq \cluster{\varepsilon_{\mathfrak{w}}}$ and thus $\C[X(\beta)]= \cluster{\varepsilon_{\mathfrak{w}}}$.
\end{itemize}

\end{proof}

\noindent Theorem \ref{thm: non simply laced} constructs cluster structures in arbitrary type. The only difference with Theorem \ref{thm:main} is that the latter states that any Demazure weave can be used to construct a cluster seed, whereas the former restricts to double inductive weaves. Let us now conclude Theorem \ref{thm:main} by providing the following generalization of Lemma \ref{lem: demazure classification} in arbitrary type.

\begin{proposition}
\label{prop: demazure classification B2}
Let $\fW_1, \fW_2: \beta \to \delta(\beta)$ be Demazure weaves in arbitrary type, where we have fixed a braid word for $\delta(\beta)$. Then $\fW_1$ and $\fW_2$ are related by a sequence of weave equivalences and mutations.
\end{proposition}

\begin{proof}
We follow the logic of \cite{Elias} and \cite[Section 4]{CGGS1}. It is sufficient to check all possible overlaps of the braid relations and Demazure moves $ii\to i$ and verify the statement for all Demazure weaves in these cases. It is proven in \cite[Lemma 5.1]{Elias}, in the language of minimal sets of ambiguities, that checking these overlaps is indeed sufficient in Type A and analogous arguments should apply for other types. Equivalently, we can draw the braid word graphs in all these cases and interpret the Demazure weaves as paths from top to bottom vertex. We need to 
check that, up to mutations, all cycles in these graphs are generated by pentagons as in Figure \ref{fig: weave equivalence B2 G2} and squares (for non-overlapping moves).\\

\noindent The overlap between two Demazure moves is a mutation. The overlap between a Demazure move and a braid relation (for example, $11212$ in type $B_2$) is covered by Figure \ref{fig: weave equivalence B2 G2}. This leaves us with the overlaps between two braid relations. To check these, we can restrict to a rank 2 subdiagram and consider the braid word graphs for $\beta=1212\ldots$ with $\ell(\beta)=d+k,k\le d-1$, where $d = d_{12}$ is the length of the braid relation.\\ 

\noindent We proceed by induction in $k$, the base case $k=1$ is our definition of equivalence, see Figure \ref{fig: weave equivalence B2 G2}. Assume that we verified the statement for all $\beta=1212\ldots$ with $\ell(\beta)\le d+k-1$, then we verified all overlaps of lengths at most $d+k-1$ and by the above argument any two weaves for an arbitrary braid word of length at most $d+k-1$ are equivalent.\\

\noindent Now consider $\beta=1212\ldots$ with $\ell(\beta)=d+k, k\le d-1$. We can apply $(k+1)$ different braid relations to $\beta$ and obtain braid words $\beta'_{a},1\le a\le k+1$. Note that $\ell(\beta'_a)=d+k$. The assumption $k\le d-1$ implies that we cannot apply any braid relations to $\beta'_{a}$ (except going back to $\beta$), so we must cancel double letters in all possible ways and obtain words $\beta''_{b}, 1\le b\le 2k$ of length $\ell(\beta''_b)=d+k-1$. Specifically, $\beta'_1$ is $\beta''_1=\underbrace{2121...}_{d+k-1}$ with one repeated letter, $\beta'_{k+1}$ is $\beta''_{2k}=\underbrace{1212...}_{d+k-1}$ with one repeated letter, and for  $2\le a\le k$ the word $\beta'_{a}$ is $\alpha=\underbrace{1212...}_{d+k-2}$ with two repeated letters, and can be simplified to two words $\beta''_{2a-2},\beta''_{2a-1}$ which can be further simplified to $\alpha$. We illustrate these words in Figure \ref{fig: braid graph top} for type $B_2$, $d=4$ and $k=3$.

\begin{figure}[ht!]
    \centering
     \begin{tikzpicture}
     \draw (0,3) node {$1212121$};
     
     \draw (-5,2) node {$2121121$};
     \draw (-2,2) node {$1121221$};
     \draw (2,2) node {$1221211$};
     \draw (5,2) node {$1211212$};
     
     \draw (-5,1) node {$212121$};
     \draw (-3,1) node {$121221$};
     \draw (-1,1) node {$112121$};
     \draw (1,1) node {$121211$};
     \draw (3,1) node {$122121$};
     \draw (5,1) node {$121212$};
     
     \draw (0,0) node {$12121$};


     \draw (-0.4,2.8)--(-5,2.2);
     \draw (-0.2,2.8)--(-2,2.2);
     \draw (0.2,2.8)--(2,2.2);
     \draw (0.4,2.8)--(5,2.2);
     \draw (-5,1.8)--(-5,1.2);
     \draw (5,1.8)--(5,1.2);
     \draw (-2.2,1.8)--(-3,1.2);
     \draw (-1.8,1.8)--(-1,1.2);
     \draw (1.8,1.8)--(1,1.2);
     \draw (2.2,1.8)--(3,1.2);
     \draw (-4.5,1)--(-3.5,1);
     \draw (-0.5,1)--(0.5,1);
     \draw (3.5,1)--(4.5,1);
     \draw (-0.4,0.2)--(-3,0.8);
     \draw (-0.2,0.2)--(-1,0.8);
     \draw (0.2,0.2)--(1,0.8);
     \draw (0.4,0.2)--(3,0.8);
     \end{tikzpicture}
    \caption{Braid words for type $B_2$, $d=4$ and $k=3$: $\beta$ on top, $\beta'_a$ and $\beta''_b$ on next two layers and $\alpha$ at the bottom.}
    \label{fig: braid graph top} 
\end{figure}

Consider an arbitrary path of braid words from $\beta$ to $\delta(\beta)=w_0$, it must pass through $\beta''_b$ for some $b$. By the assumption of induction, any two paths from $\beta''_b$ are equivalent, so we can choose a path from $\beta''_b$ to $w_0$ by first going to $\alpha$, and then following an arbitrary path to $w_0$. On the other hand, we can describe all cycles involving $\beta,\beta''_b$ and $\alpha$: there are $k$ pentagons (weave equivalences), $k-1$ squares (non-overlapping relations), and $k-2$ triangles of the form:
$$
\begin{tikzcd}
11\underbrace{21\ldots}_{d}=1\underbrace{12\ldots}_{d}1^* \arrow{dr} \arrow{rr}& &\underbrace{12\ldots}_{d}1^*1^* \arrow{dl}\\
 & 1\underbrace{21\ldots}_{d}=\underbrace{12\ldots}_{d}1^* & \\
\end{tikzcd}
$$
Here we denote by $1^*$ the index of the conjugate of the generator $s_1$ by $w_0,$ which is $1$ for even $d$ and $2$ for odd $d$. A straightforward verification shows that such a triangle can be obtained as a combination of three elementary equivalences (one of them corresponding to a commutative square in the braid word graph and two others corresponding to pentagons) and two mutations. Therefore, 
any two paths from $\beta$ to $\alpha$ are mutation equivalent, and any two paths from $\beta$ to $w_0$ are mutation equivalent.
\end{proof}

Theorem \ref{thm: non simply laced} and Proposition \ref{prop: demazure classification B2} now imply Theorem \ref{thm:main} in its entirety:

\begin{corollary}\label{cor: non simply laced}
Let $\G$ be a simple algebraic group, $\beta\in \Br_{W}^{+}$ a positive braid word and $\mathfrak{w}: \beta \to \delta(\beta)$ a Demazure weave. Then we have
\[
\C[X(\beta)] \cong \up{\varepsilon_{\mathfrak{w}}} = \cluster{\varepsilon_{\mathfrak{w}}}.
\]
where $\varepsilon_{\mathfrak{w}}$ is the skew-symmetrizable exchange matrix associated to $\mathfrak{w}$.
\end{corollary}


\subsection{Langlands dual seeds}

Consider a Demazure weave $\mathfrak{w}: \beta \to \delta(\beta)$ for a simple algebraic group $\G$. This gives us a cluster seed for the braid variety $X(\beta)$.
The Langlands dual group $\G^\vee$ has the same Weyl group and braid group. Therefore, $\mathfrak{w}$ can also be viewed as a weave $\beta \to \delta(\beta)$ for $\G^\vee$ and it also gives a seed for the corresponding braid variety for $\G^\vee$; let us refer to this variety $X^\vee(\beta)$. Let us study how the seeds for $X(\beta)$ and $X^\vee(\beta)$ obtained from $\mathfrak{w}$ are related to each other.

\begin{definition}[\cite{FockGoncharov_ensemble}]\label{def: langlands} Two cluster seeds $(I, I^{\uf}\!, \varepsilon, d)$ and $(\tilde{I}, \tilde{I}^{\uf}\!, \tilde{\varepsilon}, \tilde{d})$ are said to be Langlands dual if there is a bijection between $I$ and $\tilde{I}$ inducing a bijection between $I^{\uf}$ and $\tilde{I}^{\uf}$ such that

\begin{itemize}
    \item $\varepsilon_{ij} = -\tilde{\varepsilon}_{ji}$,
    \item $\tilde{d}_i = d_i^{-1}c$ for some constant $c$.
\end{itemize}
In other words, the exchange matrices are transposed and negated, while the multipliers are inverted up to rescaling.
\end{definition}

\begin{proposition} 
\label{prop: langlands}
Let $\fW$ be a weave for a braid word $\beta$. Then the corresponding seeds for the cluster varieties $X(\beta)$ and $X^\vee(\beta)$ are Langlands dual.
\end{proposition}

\begin{proof}
Let $v$ and $v'$ be trivalent vertices of $\fW$. Let $\varepsilon_{v,v'}$ be the corresponding entry in the exchange matrix for $X(\beta)$, and $\varepsilon^{\vee}_{v,v'}$ that for $X^{\vee}(\beta)$. We would like to check that $\varepsilon_{v,v'} = - \varepsilon^{\vee}_{v,v'}$.

This can be checked purely locally at trivalent vertices and at $(2d)$-valent vertices. In principle, this is a finite check that can just be done by hand, though it is somewhat tedious. We will give a conceptual proof for the most interesting case, the $(2d)$-valent vertices.

We use Equation \ref{eq: boundary intersection non simply laced} to conclude that for any slice $\eta$, we have
$$\sharp_{\eta}(\g^{\vee}_{v}\cdot \g_{v'}):=\frac{1}{2}\sum_{i,j = 1}^{r}\sign(j-i)c_{i}^\vee c'_{j}\cdot(\rho_{i}, \rho^{\vee}_{j}) = -\frac{1}{2}\sum_{i,j = 1}^{r}\sign(i-j)c'_{j} c_{i}^\vee \cdot(\rho^{\vee}_{j}, \rho_{i}) = - \sharp_{\eta}(\g_{v'}^{\vee} \cdot \g_v).$$

By taking $\eta$ to be a slice before and after any $(2d)$-valent vertex, we see that the local contribution to the intersection pairing at a vertex $\bar{v}$ satisfies
\[\sharp_{\bar{v}}(\g_{v}^{\vee}\cdot \g_{v'}) = -\sharp_{\bar{v}}(\g_{v'}\cdot \g_{v}^\vee).\]
as needed.

Suppose that at a trivalent vertex $\bar{v}$, we have that $\g_{v}$ has weight $c$ along the left vertex and $\g_{v'}$ has weight $c'$ along the right vertex. Then 
$$
\sharp_{\bar{v}}(\g_{v'}^{\vee} \cdot \g_v) = -cc' \frac{d_{\bar{v}}}{d_v}=-c^\vee c',
$$
so that again we have $\sharp_{\bar{v}}(\g_{v'}^{\vee} \cdot \g_v) = -\sharp_{\bar{v}}(\g_{v} \cdot \g_{v'}^{\vee})$

Finally, it is directly verified that the constant $c$ required by Definition \ref{def: langlands} can be taken to be the square ratio between the length of a long root and that of a short root, so $c = 2$ in types $BC$ and $F_4$, and $c = 3$ in type $G_2$.

\end{proof}

 Section \ref{sec:Poisson} below shows that braid varieties admit a cluster Poisson structure. Moreover, under the conditions of Lemma~\ref{lem: X structure} and the existence of a cluster DT-transformation, proven in Section \ref{sec:Poisson}, we can conclude that the braid varieties $X(\beta)$ and $X^\vee(\beta)$ are cluster dual.


\section{Properties and further results}\label{sec:properties}

This section collects a series of properties and results about the weaves and cluster structures presented thus far. These are additional facts that are not required for any of the previous results but might still be of independent interest. Each of the following subsections is also logically independent of each other.


\subsection{A characterization of frozen variables}

In this subsection, we give a combinatorial characterization of the trivalent vertices of a weave $\fW$ whose associated cluster variable is frozen. We start with the following lemma, which is a consequence of Corollary \ref{cor: cluster} and \cite[Theorem 2.2]{GLS}:

\begin{lemma}\label{lem: frozen nonvanishing}
Let $\fW: \beta \to \delta(\beta)$ be a weave and $v$ its trivalent vertex. Then, $v$ is frozen if and only if the cluster variable $A_{v}$ is nowhere vanishing on $X(\beta)$. 
\end{lemma}

\noindent Lemma \ref{lem: frozen nonvanishing} allows us to give a characterization of frozen trivalent vertices of a weave $\fW$ that has the combinatorial advantage of not referencing the cycle $\g_{v}$. It is also closer in spirit to the definition of frozen variables in \cite{GLSS2, GLSS1}. The construction is as follows. Let us suppose that a trivalent vertex $v$ of $\fW$ corresponds to a move
\[
\beta' = \beta_1\sigma_i\sigma_i\beta_2 \to \beta_1\sigma_i\beta_2.
\]
By definition, this trivalent vertex $v$ is said to be {\it Demazure frozen} if $\delta(\beta_1\beta_2) < \delta(\beta') = \delta(\beta)$. If $\tilde{z}$ denotes the variable on the right arm of the trivalent vertex $v$, then we have a decomposition of the form
$$
X(\beta')=(X(\beta_1\sigma_i\beta_2)\times \C^*)\sqcup (Y\times \C)
$$ 
for some algebraic variety $Y$, where the strata correspond to $\widetilde{z}\neq 0$ and $\widetilde{z}=0$ respectively. See \cite[Section 5.1]{CGGS1} for more details. In particular,
$v$ is Demazure frozen if and only if $Y$ is empty or, equivalently, the locus $\{\tilde{z} = 0\}$ is empty. 

\begin{lemma}\label{lem: frozen = demazure}
Let $\fW$ be a weave and $v\in\fW$ a trivalent vertex. Then, $v$ is frozen if and only if $v$ is Demazure frozen.
\end{lemma}
\begin{proof}
Let us assume first that $v$ is not frozen, that is, the locus $\{A_v \neq 0\}$ is nonempty. Now consider the collection of all vertices $v'$ that appear above $v$ on the weave, so that
\[
A_{v} = \widetilde{z}\cdot\prod_{v'}A_{v'}^{m_{v'}}
\]
for some nonnegative integers $m_{v'}$, cf. \eqref{eq: def cluster}. 
To check that $v$ is not Demazure frozen, it is enough  to check that the locus $\{\widetilde{z} = 0\}\cap \{\prod_{v'} A_{v'} \neq 0\}$ is nonempty or, equivalently, that $\{A_{v} = 0\} \not\subseteq \{\prod_{v'}A_{v'} = 0\}$. By assumption, $\{A_{v} \neq 0\} \neq \emptyset$ and by \cite[Theorem 1.3]{GLS} cluster variables are irreducible, so $A_{v}$ and $\prod_{v'}A_{v'}$ are coprime. Thus, $v$ is not Demazure frozen.
Conversely, assume that $v$ is not Demazure frozen. We want to check that $\{A_{v} = 0\} \neq \emptyset$. But by definition $v$ not being Demazure frozen means that the locus $\{\prod_{v'} A_{v'} \neq0\}\cap\{\widetilde{z} = 0\}$ is nonempty, and the result follows. 
\end{proof}

\noindent Lemma \ref{lem: frozen = demazure} can be used to give an upper bound on the number of frozen vertices of the cluster structure on $\C[X(\beta)]$.

\begin{proposition}\label{prop: bounds frozen}
Let $\beta \in \Br_{W}^{+}$ be a positive braid and $\fW: \beta \to \delta(\beta)$ a Demazure weave. Then the cluster structure $\cluster{\varepsilon_{\fW}} = \C[X(\beta)]$ has at most $\ell(\delta(\beta))$
frozen variables.
\end{proposition}

\noindent The upper bound in Proposition \ref{prop: bounds frozen} is sharp. Indeed, there are braid words such that $Q_{\fW}$ has exactly $\ell(\delta(\beta))$ frozen variables. For example, take any reduced word $\delta$ and let $\boldsymbol{\delta} \in \Br_{W}^{+}$ be obtained by repeating every letter of $\delta$ at least twice; then the left inductive weave $\lind{\boldsymbol{\delta}}$ has a quiver $Q_{\lind{\boldsymbol{\delta}}}$ which is a disjoint union of $\ell(\delta)$ linearly-oriented type $A$ quivers, each with one frozen variable.\\

\noindent Let us show Proposition \ref{prop: bounds frozen}. The non-simply laced case is proven similarly to the simply laced case by unfolding, so we will focus on the latter. In order to prove Proposition \ref{prop: bounds frozen} in the simply laced case, it is enough to show that the quiver $Q_{\lind{\beta}}$ for the left inductive weave has at most $\ell(\delta(\beta))$ frozen vertices. For each trivalent vertex $v$ of $\lind{\beta}$, we define a path $\iota(v)$ in the weave $\lind{\beta}$ as follows:

\begin{enumerate}
\item Start at $v$ and move downwards from this trivalent vertex.
\item If we reach another trivalent vertex, say $v_{2}$, the path $\iota(v)$ stops at $v_{2}$.
\item If the path $\iota(v)$ enters a hexavalent vertex from the upper left (resp. upper right, resp. upper center) edge, then it exists the hexavalent vertex from the lower right (resp. lower left, resp. lower middle) edge.
\item If the path $\iota(v)$ enters a tetravalent vertex from the upper left (resp. upper right) edge, then it exists the tetravalent vertex from the lower right (resp. lower left) edge.
\end{enumerate}

\noindent Note that $\iota(v)$ is, in general, different from the cycle $\g_{v}$. By definition, the trivalent vertex $v$ is said to \emph{fall down} if $\iota(v)$ does not stop, i.e., if $\iota(v)$ never reaches a trivalent vertex. Since we can always trace back $\iota(v)$ to $v$, we have an injection   from the set of trivalent vertices that fall down to the letters of (a reduced decomposition of) $\delta(\beta)$. Thus, Proposition \ref{prop: bounds frozen} follows from the following result. 

\begin{lemma}\label{lemma:fallsdown}
Let $v$ be a Demazure frozen trivalent vertex of the weave $\lind{\beta}$. Then $v$ falls down. 
\end{lemma}
\begin{proof}
 Note that if $v$ is a trivalent vertex in $\lind{\beta}$, then the left arm of $v$ goes straight up to $\beta$, without encountering any vertices. From here, it follows easily that the right arm of $v$ cannot lead directly to the middle strand of a hexavalent vertex. In fact, more is true. Assume that we have taken a trivalent vertex $v$ in the weave $\lind{\beta}$ and we have slided it up through tetra- and hexavalent vertices using moves from \cite[4.2.4]{CGGS1}. We obtain a weave $\widetilde{\mathfrak{w}}: \beta \to \delta(\beta)$ with a special trivalent vertex $\widetilde{v}$ on it.  Since all the weave moves are local, note that the part of the weave which is placed northeast of $\widetilde{v}$ is a weave of the form $\lind{\widetilde{\beta}}$ where $\widetilde{\beta}$ is a suffix of $\beta$. From here, it follows again that the right arm of $\widetilde{v}$ cannot directly lead to the middle strand of a hexavalent vertex. 

If we have two consecutive trivalent vertices $\beta_1s_is_is_i\beta_2 \to \beta_1s_is_i\beta_2 \to \beta_1s_i\beta_2$ then the top trivalent vertex is never Demazure frozen. Assume now that $v$ is a trivalent vertex that does not fall down, i.e., such that $\iota(v)$ stops at another trivalent vertex, say $v_1$. If $\iota(v)$ does not pass any hexavalent or tetravalent vertex, then by the observation at the beginning of this paragraph $v$ cannot be Demazure frozen. If it does, we slide $v_{1}$ through these hexavalent and tetravalent vertices to bring it next to $v$. These are all legal moves since, by the discussion above, we will never have to slide $v_{1}$ through the middle strand of a hexavalent vertex. Note that sliding $v_{1}$ does not affect the condition that defines $v$ being Demazure frozen. Thus, $v$ cannot be Demazure frozen. 
\end{proof}

\noindent The converse of Lemma \ref{lemma:fallsdown} does not hold: $v$ falling down in $\lind{\beta}$ does \emph{not} imply that $v$ is (Demazure) frozen. For example, in Figure \ref{fig:large example} below, which becomes a \emph{left} inductive weave after reflecting along a vertical line, the top trivalent vertex falls down but it is not frozen.

\begin{remark}
Note that in Lemma \ref{lemma:fallsdown} it is essential that we work with the inductive weave $\lind{\beta}$. For example, in Figure \ref{fig: cycles weights}, the topmost trivalent vertex is Demazure frozen but it does not fall down. 
\end{remark}


\subsection{Polynomiality of cluster variables} Theorem \ref{thm:main} proves that the algebra $\C[X(\beta)]$ is a cluster algebra. In particular, we have defined cluster variables and shown that they satisfy the corresponding exchange relations. In this subsection, we show that there is a way to lift the cluster variables in $\C[X(\beta)]$ to polynomials in $\C[z_1, \dots, z_r]$, where $\ell(\beta) = r$, in such a way that the exchange relations are still satisfied. Note that Corollary \ref{cor: braid via eqns} yields a projection $\pi: \C[z_1, \dots, z_r] \to \C[X(\beta)]$.  More precisely, we prove the following result:

\begin{theorem}\label{thm: polynomials}
Let $\beta = \sigma_{i_{1}}\cdots \sigma_{i_{r}} \in \Br^{+}_{W}$ and consider the projection $\pi: \C[z_1, \dots, z_r] \to \C[X(\beta)]$. Then, for each cluster variable $c \in \C[X(\beta)]$, there exists a polynomial $\tilde{c} \in \C[z_1, \dots, z_r]$ such that:
\begin{itemize}
    \item[$(1)$] $\pi(\tilde{c}) = c$.
    \item[$(2)$] The polynomials $\tilde{c}$ satisfy the cluster exchange relations: i.e.~ if $\mathbf{c} = \{c_1, \dots, c_{s}\}$ and $\mathbf{c}' = \{c'_1, \dots, c'_s\}$ are clusters in $\C[X(\beta)]$ related by a mutation in $k$ then, in $\C[z_1, \dots, z_r]$, we have:
    \[
    \tilde{c}_{k}\tilde{c}'_{k} = \prod_{i}\tilde{c}_{i}^{[\varepsilon_{ki}]_{+}} + \prod_{i}\tilde{c}_{i}^{-[\varepsilon_{ki}]_{-}}.
    \]
\end{itemize}
\end{theorem}

\noindent First, let us observe that the non-simply laced case of Theorem \ref{thm: polynomials} follows from the simply laced case since, by Proposition \ref{prop: folding} the cluster variables in the non-simply laced case can be obtained from those in the simply laced case by restricting to a closed subset. Thus, we focus in the simply laced case and we start proving Theorem \ref{thm: polynomials} in the case of $\Conf(\beta) = X(\Delta\beta)$. We denote by $w$'s the variables corresponding to $\Delta$, by $z_1, \dots, z_r$ the variables corresponding to $\beta$ and recall that
\[
\Conf(\beta) = \{(z_1, \dots, z_r) \mid B_{\beta}(z) \in \borel_{-}\borel\}.
\]

\noindent According to \cite{SW}, the frozen variables in $\C[\Conf(\beta)]$ are precisely $f_{i} := \Delta_{\omega_{i}}B_{\beta}(z)$, where $\Delta_{\omega_{i}}$ are generalized principal minors as in \cite{FZ, GLSkm}. So we can take this as the definition of $\tilde{f}_i$:
\[
\tilde{f}_i := \Delta_{\omega_i}B_{\beta}(z) \in \C[z_1, \dots, z_r] \subseteq \C[w_1, \dots, w_{\ell(w_0)}, z_1, \dots, z_r]
\]
Moreover, according to \cite{SW}, we have
\[
\C[\Conf(\beta)] = \C[z_1, \dots, z_r][\tilde{f}_{i}^{-1} \mid i \in \dynkin]
\]
so that
\[
\C[w_1, \dots, w_{\ell(w_0)}, z_1, \dots, z_r][\tilde{f}_{i}^{-1} \mid i \in \dynkin] = \C[w_1, \dots, w_{\ell(w_0)}]\otimes \C[\Conf(\beta)] = \C[w_1, \dots, w_{\ell(w_0)}] \otimes \C[X(\Delta\beta)]
\]
and Theorem \ref{thm: polynomials} for $X(\Delta\beta)$ follows if we show that cluster variables do not involve denominators in frozen variables. For this, the following lemma is useful.

\begin{lemma}\label{lem: polynomials}
Let $f(z) \in \C[\Conf(\beta)]$ be a cluster variable. Then $f(z)$ is a cluster variable of $\C[\Conf(\beta\sigma_{i})]$ for every $i \in \dynkin$.
\end{lemma}
\begin{proof}
First, let us assume that $f(z)$ belongs to a cluster associated to a weave $\mathfrak{w}$ of $\Delta\beta$. Extend this weave to a weave $\mathfrak{w}'$ of $\Delta\beta\sigma_{i}$ by adding an $i$-colored trivalent vertex on the bottom right of the weave. This adds a new cluster variable, but does not change the cluster variables that appeared before. 

In general, assume that $f(z) = \mu_{k_{1}}\mu_{k_{2}}\cdots \mu_{k_{\ell}}g(z)$, where $g(z)$ is a cluster variable in a cluster coming from a weave $\mathfrak{w}$ and $k_{1}, \dots, k_{\ell}$ are mutable vertices of the quiver $Q_{\mathfrak{w}}$. By Lemma \ref{lem: add trivalent quiver} and Remark \ref{rmk: add trivalent quiver}, this process will:
\begin{enumerate}
    \item Add a new frozen variable.
    \item Thaw some frozen variables of $Q_{\mathfrak{w}}$.
    \item Add new coefficients to the matrix $\varepsilon$, all of which involve only variables mentioned in the previous two items. 
\end{enumerate}
In particular, mutable variables of $Q_{\mathfrak{w}}$ do not have new incident variables in $Q_{\mathfrak{w}'}$. Since $k_{1}, k_{2}, \dots, k_{\ell}$ correspond to mutable variables of $Q_{\mathfrak{w}}$, this implies that the equality $f(z) = \mu_{k_{1}}\cdots \mu_{k_{\ell}}g(z)$ is also valid in $\C[\Conf(\beta)]$ and we are done. 
\end{proof}

\begin{proposition}\label{prop: polynomials}
Let $f(z) \in \C[\Conf(\beta)] = \C[z_1, \dots, z_r][f_{i}^{-1} \mid i \in \dynkin]$ be a cluster variable. Write
\[
f(z) = h(z)/g(z)
\]
where $h(z), g(z) \in \C[z_1, \dots, z_r]$ have no common factors, and $g(z)$ is a monomial in $f_{i}$'s. Then, $g(z) = 1$.
\end{proposition}
\begin{proof}
Let $i \in \dynkin$. By Lemma \ref{lem: polynomials}, $f(z)$ is also a cluster variable in $\C[\Conf(\beta\sigma_{i})]$. It is clear from the construction of the frozen variables, see also the proof of Proposition \ref{prop: conf quiver}, that $f_{i}(z)$ is a cluster variable in $\C[\Conf(\beta\sigma_i)]$ which is no longer frozen. Thus, $f_i(z)$ cannot divide $g(z)$ in $\C[z_1, \dots, z_{r+1}]$ or in $\C[z_1, \dots, z_r]$. The result follows. 
\end{proof}

\noindent Theorem \ref{thm: polynomials} for $\Delta\beta$ is now a consequence of Proposition \ref{prop: polynomials}. Let us now move on to general braid varieties.

\begin{proof}[Proof of Theorem \ref{thm: polynomials}]
Following the same argument as in the proof of Lemma \ref{lem: polynomials}, every cluster variable of $X(\beta)$ is also a cluster variable in $X(\Delta\beta)$ and, moreover, the exchange relations do not change. So the result follows from the corresponding statement on $X(\Delta\beta)$. 
\end{proof}

Theorem \ref{thm: polynomials} has the following geometric corollary.

\begin{corollary}
For every braid $\beta = \sigma_{i_{1}}\cdots \sigma_{i_{r}}$ there exists a principal open set $U \subseteq \C^{r}$ such that:
\begin{itemize}
    \item[$(1)$] The inclusion $\pi^{*}: X(\beta) \to \C^r$ factors through $U$.
    \item[$(2)$] There is a projection $\iota^{*}: U \to X(\beta)$ with section $\pi^{*}$. 
\end{itemize}
\end{corollary}
\begin{proof}
Let $f_1, \dots, f_k$ be the frozen variables in $X(\beta)$ and let $U := \{\prod_{i = 1}^{k} \tilde{f}_{i} \neq 0\} \subseteq \C^r$. By the starfish lemma, we have an embedding $\iota: \C[X(\beta)] \to \C[U] = \C[z_1, \dots, z_r][\tilde{f}_{i}^{-1}]$ sending the cluster variable $c$ to $\tilde{c}$. Now it is straightforward to verify (1) and (2).
\end{proof}

\noindent We refer the reader to Section \ref{sec: examples} below for several examples of cluster variables where it is straightforward to verify that the exchange relations are already valid in the polynomial algebra. 


\subsection{Local acyclicity and reddening sequences} The purpose of this subsection is to show that the cluster algebra $\C[X(\beta)]$ is always locally acyclic, in the sense of \cite{Muller}, and that it always admits a reddening sequence \cite{KellerDemonet}.\\

Let us first quickly discuss reddening sequences. Indeed, Lemma \ref{lem: add trivalent quiver} implies that the quivers we consider have reddening sequences as follows:

\begin{proposition}\label{prop: reddening}
Let $\fW: \beta \to \delta(\beta)$ be a Demazure weave. Its corresponding exchange matrix admits a reddening sequence.
\end{proposition}
\begin{proof}
By Theorem \ref{thm:mutation quivers}, it is enough to fix a weave $\fW: \beta \to \delta(\beta)$ and we fix the inductive weave $\lind{\beta}$. 
By Corollary \ref{cor: coincidence quivers} together with \cite[Section 4]{SW} (see also \cite[Corollary 4.9]{Cao}), $\rind{\Delta\beta}$ admits a maximal green sequence. Since $\lind{\Delta\beta}$ is mutation equivalent to $\rind{\Delta\beta}$, \cite[Corollary 3.2.2]{muller2015existence} implies that $\lind{\Delta\beta}$ has a reddening sequence. By Lemma \ref{lem: add trivalent quiver} and \cite[Theorem 3.1.3]{muller2015existence}, then so does $\lind{\beta}$.
\end{proof}

Assume that the exchange matrix of a cluster seed has full rank and its mutable part has a reddening sequence. Then, by works \cite{FockGoncharov_ensemble, GHKK}, the corresponding upper cluster algebra has a canonical basis of theta functions parameterized by the integral tropicalization of the dual cluster $\mathcal{X}$-variety. In the skew-symmetric case, the upper cluster algebra also has a generic basis parameterized by the same lattice \cite{Qin}. See \cite{KellerDemonet} for more details and references. Thus, Proposition \ref{prop: reddening} implies the following corollary, see also Theorem \ref{thm: FG duality} below.

\begin{corollary} \label{cor: 2 bases}
 The upper cluster algebra structure on $\C[X(\beta)]$ defined via Demazure weaves has a canonical basis of theta functions parameterized by the lattice of integral tropical points of the dual cluster $\mathcal{X}$-variety. If $\G$ is simply-laced, it also has a generic basis parameterized by the same lattice.
\end{corollary}
\begin{proof}
We only need to show that the exchange matrix has full rank. This follows from Corollary \ref{cor: full rank} below, which is independent of the intervening material.
\end{proof}

\begin{remark}
In fact, one expects that there is a precise link, close to being an equivalence, between the existence of a reddening sequence, local acyclicity, and the isomorphism between the upper cluster algebra and the cluster algebra, see \cite{Mills}.
\end{remark}

\noindent Let us now focus on local acyclicity; recall that locally acylic means that there exists a finite open cover
\[
X(\beta) = \bigcup_{i = 1}^{k}U_{i}
\]
where each $U_{i}$ is a cluster variety such that the mutable part of its associated quiver does not have directed cycles. Clearly,  to show that $X(\beta)$ is locally acyclic it is enough to provide such a decomposition such that each $U_{i}$ is itself a locally acyclic cluster variety.

\begin{theorem}\label{thm: loc acyclic}
For any positive braid word $\beta \in \Br_{W}^{+}$, the cluster structure on the braid variety $X(\beta)$ is locally acyclic.
\end{theorem}
\begin{proof}
We focus on the simply-laced case, the proof in the non-simply laced case is similar. As usual, let $\delta := \delta(\beta)$. We work by induction on $\ell(\beta) - \ell(\delta)$, which is the number of vertices on the quiver $Q_{\fW}$ for any weave $\fW: \beta \to \delta$. Since the quiver $Q_{\fW}$ always has at least one frozen vertex, the result is clear for $\ell(\beta) - \ell(\delta) \in \{0, 1, 2\}$.

In the general case, upon applying a cyclic rotation to $\beta$ we may assume that $\beta = \sigma_i\sigma_i\beta'$ for some positive braid word $\beta' \in \Br^{+}_{W}$. If $\delta = s_i\delta(\beta')$ then it is clear that we have $X(\beta) = \C^{\times} \times X(\beta')$, while $\ell(\beta') - \ell(\delta(\beta)') = \ell(\beta) - \ell(\delta) - 1$ and we may use induction to conclude that $X(\beta)$ is locally acyclic. So we will assume that $\delta = \delta(\beta')$. In this case, we may consider a weave $\fW: \beta \to \delta$ as in Figure \ref{fig: iibeta_weave}.

\begin{figure}[h!]
\centering
\includegraphics[scale=0.6]{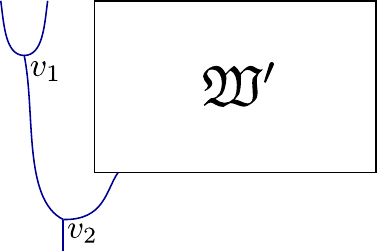}
\caption{A weave $\fW: \sigma_i\sigma_i\beta' \to \delta$, where $\fW'$ is a weave $\fW': \beta \to \delta$. Note that $v_1 \in Q_{\fW}$ is a mutable sink, while $v_2 \in Q_{\fW}$ is a frozen source}
\label{fig: iibeta_weave}.
\end{figure}
Locally around $v_1, v_2$ the quiver $Q_{\fW}$ looks as follows:
\begin{center}
\begin{tikzcd}
 v_1 & v'_1 \arrow{l}\\
 {\color{blue} v_2} \arrow{u} \arrow{ur} \arrow{r} & v'_k \arrow[dashed, dash]{u} \arrow{ul}
\end{tikzcd}
\end{center}
where $v'_1, \dots, v'_k$ are the trivalent vertices $v$ such that $\g_v$ has a nonzero weight at the right incoming leg of $v_2$. We will consider the elements:
\[
A_1 := A_{v_{1}},\qquad A_2 := \prod_{i = 1}^{k}A_{v'_{i}}
\]
Mutating at $v_1$, we obtain that the element $\displaystyle{\frac{1 + A_{v_{2}}A_{2}}{A_{1}}}$ is a regular function on $X(\beta)$ and therefore $A_{1}$ and $A_{2}$ cannot simultaneously vanish. In other words, $X(\beta) = U_1 \cup U_2$, where $U_i = \Spec(\C[X(\beta)][A_{i}^{-1}])$.
Let $Q_1$ be the quiver obtained from $Q_{\fW}$ by freezing the vertex $v_1$. We have (cf. \cite[Proposition 3.1]{Muller}):
\[
\cluster{Q_{1}} \subseteq \cluster{Q_{\fW}}[A_1^{-1}] = \up{Q_{\fW}}[A_1^{-1}] \subseteq \up{Q_{1}}.
\]
But $Q_{1}$ is easily seen to be a quiver for the braid word $\sigma_i\beta'$ with a disjoint frozen vertex. By Corollary \ref{cor: cluster}, $\cluster{Q_{1}} = \up{Q_{1}}$ and we conclude that $U_1 = \Spec(\cluster{Q_{1}}) = \C^{\times} \times X(\sigma_i\beta')$ is a cluster variety that, by induction, is locally acyclic. 

Similarly, let $Q_2$ be the quiver obtained from $Q_{\fW}$ by freezing the vertices $v'_1, \dots, v'_k$, so that
\[
\cluster{Q_{2}} \subseteq \cluster{Q_{\fW}}[A_2^{-1}] = \up{Q_{\fW}}[A_2^{-1}] \subseteq \up{Q_{2}},
\]
and $Q_2$ is easily seen to be the quiver $Q_{\fW'}$ with a disjoint quiver of the form $\square \rightarrow \bullet$, so $\cluster{Q_2} = \up{Q_2}$ and $U_2 = \Spec(\cluster{Q_2}) = X(\beta') \times X(\sigma_i^3)$ which, again by induction, is locally acyclic. The result follows. 
\end{proof}

A similar strategy to that of the proof of Theorem \ref{thm: loc acyclic} allows us to deduce more properties on the quiver $Q_{\fW}$ and the variety $X(\beta)$. First, let us recall that the class $\mathcal{P}'$ is the smallest class of quivers without frozen vertices that satisfies the following property:
\begin{itemize}
    \item The quiver with a single vertex belongs to $\mathcal{P}'$.
    \item If $Q \in \mathcal{P}'$, then any quiver mutation equivalent to $Q'$ also belongs to $\mathcal{P}'$.
    \item If $Q \in \mathcal{P}'$ and $Q'$ is obtained from $Q$ by adjoining a sink or a source, then $Q' \in \mathcal{P}'$.
\end{itemize}

See \cite{BucherMachacek, Cao, Ladkani}. We say that an ice quiver $Q$ belongs to $\mathcal{P}'$ if its mutable part $Q^{\uf}$ belongs to $\mathcal{P}'$.

\begin{proposition}
For any braid $\beta \in \Br^{+}_{W}$ and any weave $\fW: \beta \to \delta$, the quiver $Q_{\fW}$ belongs to the class $\mathcal{P}'$.
\end{proposition}
\begin{proof}
Since cyclic rotation does not change the (mutation class of the) mutable part of the weave $\fW$, see Lemma \ref{lem: rotation mutable}, we may assume that $\beta$ has the form $\beta = \sigma_i\sigma_i\beta'$ and take the weave $\fW$ as in Figure \ref{fig: iibeta_weave}, so that $Q_{\fW}$ is obtained from $Q_{\fW'}$ by adjoining a mutable sink and a frozen source and the result follows.
\end{proof}

Note that by \cite[Theorem 3.3]{BucherMachacek} this yields another (similar in spirit) proof of Proposition \ref{prop: reddening}.
By \cite[Theorem 4.6]{Ladkani}, resp. by \cite[Lemma 8.13]{muller2016skein}, we also get the following corollaries. 

\begin{corollary}
 For any braid $\beta \in \Br^{+}_{W}$ and any weave $\fW: \beta \to \delta$, the quiver $Q_{\fW}$ admits a unique non-degenerate potential (up to right equivalence). It is rigid and its Jacobian algebra is finite-dimensional.
\end{corollary}

\noindent This quiver with non-degenerate potential has been constructed geometrically in \cite[Section 2]{CasalsGao23} for the case of $\G=\SL_n$ and double Bott-Samelson cells. The terms in the potential are given by certain polygons bounded by curves (representing the Lusztig cycles) inside of the surface associated to the weave $\fW$.

\begin{corollary}
 For any braid $\beta \in \Br^{+}_{W}$ and any weave $\fW: \beta \to \delta$, any quantum cluster algebra whose exchange type is given by the quiver $Q_{\fW}$ equals its corresponding quantum upper cluster algebra.
\end{corollary}

\begin{figure}[h!]
\centering
    \includegraphics[scale=0.7]{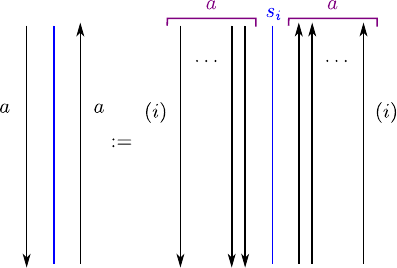}
    \caption{The topological 1-cycle near an $s_i$-edge of the weave and the shorthand notation of train tracks, where the number $a\in\mathbb{N}$ indicates $a$ parallel copies. The numbers $(i)$ in parentheses indicate that the segment in the plane parallel to an $s_i$-edge is lifted to the $i$th sheet of the branched cover. Note that the orientations are depicted.}
    \label{fig:Cycles_Notation}
\end{figure}


\subsection{Topological view on weave cycles}\label{sec:topological-cycles} Let us provide a topological interpretation of weaves and their cycles, building on \cite[Section 2]{CZ}; for this subsection we set $\G=\SL_{n+1}$. Given a weave $\fW\sse\R^2$ with $n$ colors, $s_1,\ldots,s_n\in S_{n+1}$, let $S(\fW)$ be the smooth surface obtained as a simple $(n+1)$-fold branched cover of $\R^2$ along the trivalent vertices of $\fW$, where the monodromy transposition around a trivalent vertex is declared to be $s_i$ if the (three) edges incident to the vertex are labeled with $s_i\in S_n$. The weave $\fW$ itself can then be interpreted as branch cuts for the projection $S(\fW)$ onto $\R^2$; there are more branch cuts than necessary but that is allowed and this choice appears naturally in this interpretation.\\

\begin{figure}[h!]
\centering
    \includegraphics[width=\textwidth]{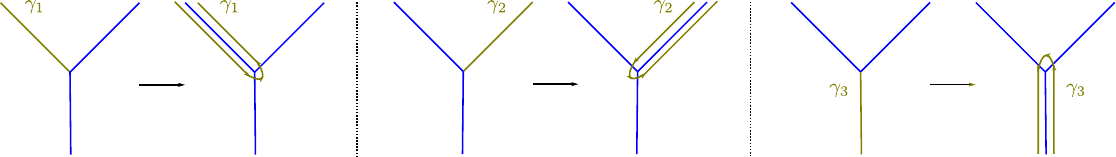}
    \caption{The projections to $\R^2$ of the relative 1-cycles $\gamma_1,\gamma_2,\gamma_3\sse S(\fW_{tri})$ near a trivalent vertex. Each cycle $\gamma_i$ has two projections, one contained in the weave $\fW_{tri}$ and the other is (the projection of) its generic perturbation.}
    \label{fig:Cycles_Trivalent_Perturb}
\end{figure}

First, at a generic horizontal slice of the weave $\fW$ a local 1-cycle on $S(\fW)$ of weight $a$ is defined according to Figure \ref{fig:Cycles_Notation}, with $a$ parallel copies at each side of an $s_i$-edge, lifting to sheet $i$. Figure \ref{fig:Cycles_Notation} also prescribes the orientations which are needed to compute signed intersections. Second, by construction, the cycles $\gamma_1,\gamma_2$ and $\gamma_3$ in Figure \ref{fig:Cycles_Trivalent_Perturb} lift to homonymous geometric relative 1-cycles on the surface $S(\fW_{tri})$ associated to the weave $\fW_{tri}$ given by a trivalent vertex, which is a 2-disk. Figure \ref{fig:Cycles_Trivalent_Perturb} actually depicts two projections to $\R^2$ of these cycles $\gamma_1,\gamma_2,\gamma_3\in S(\fW_{tri})$: a non-generic projection, literally above a weave edge, and a generic projection. The former provides neater descriptions of 1-cycles in terms of the edges of the weave itself, and the later is useful for computing intersection numbers, see \cite[Section 2]{CZ} and \cite[Section 3]{CW}. These two different projections of each $\gamma_i$ lift to (smoothly) isotopic, and thus homologous, 1-cycles. In the notation of Section \ref{sec: cycles}, $\gamma_1$ (resp.~$\gamma_2,\gamma_3$) geometrically realizes the weave cycle that has weight $1$ (resp.~$0,0$) on the top leftmost edge, has weight $0$ on the top rightmost edge (resp.~$0,0$) and weight $0$ (resp.~$0,1$) on the bottom edge. A weave cycle in $\fW_{tri}$ with arbitrary weights $(a,b,c)\in\Z^3$ can be realized geometrically be a linear combination of these $\gamma_1,\gamma_2,\gamma_3$: such relative 1-cycle $\gamma(a,b;c)$ can be drawn by taking $a$ disjoint copies of $\gamma_1$, $b$ disjoint copies of $\gamma_2$ and $c$ disjoint copies of $\gamma_3$, oriented appropriately according to signs.\\

\begin{figure}[h!]
\centering
    \includegraphics[scale=0.7]{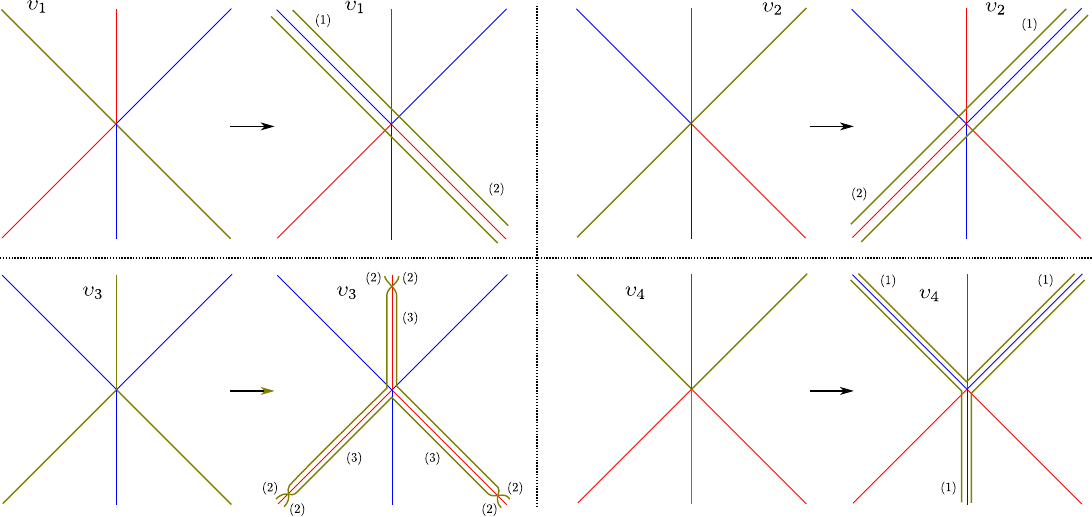}
    \caption{The projections to $\R^2$ of the relative 1-cycles $\upsilon_1,\upsilon_2,\upsilon_3,\upsilon_4\sse S(\fW_{hex})$ near a hexavalent vertex. The numbers in parentheses indicate the sheet, $1,2$ or $3$, to which that part of the segment is being lifted. Note that adjustments at the ends of $\upsilon_3$ need to be inserted so as to have boundary conditions match with other pieces of the cycle according to the rule of Figure \ref{fig:Cycles_Notation}.}
    \label{fig:Cycles_Hexavalent_Perturb}
\end{figure}

\begin{figure}[h!]
\centering
    \includegraphics[scale=0.7]{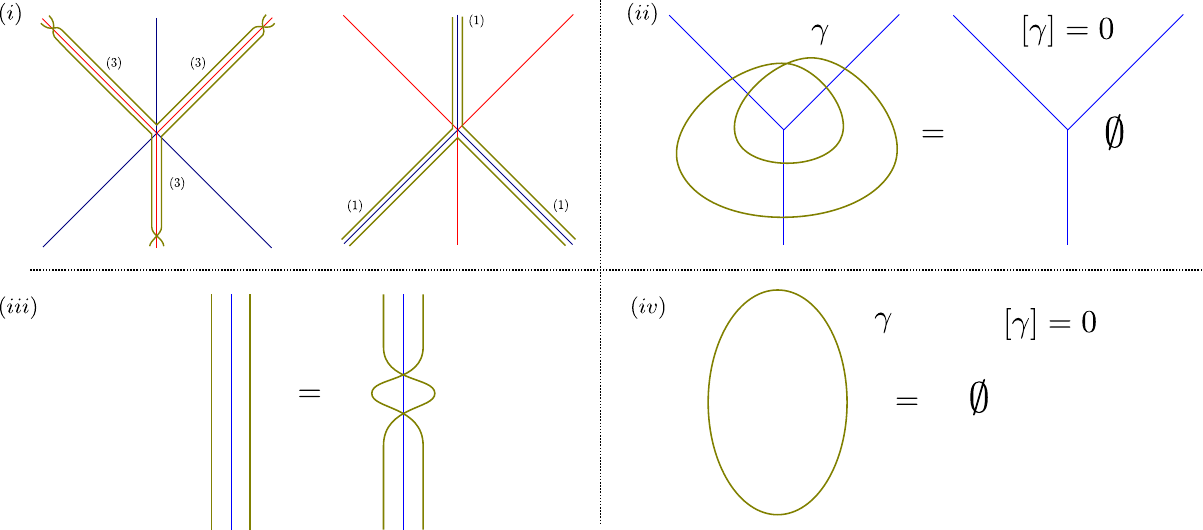}
    \caption{$(i)$ Two geometric relative 1-cycles associated to hexavalent vertices, in line with Figure \ref{fig:Cycles_Hexavalent_Perturb}. Parts $(ii),(iii)$ and $(iv)$ depict relations for the geometric 1-cycles that hold in the (relative) homology of $S(\fW)$. The curve $\gamma$ in $(ii)$ is lifted to sheets $i$ and $i+1$ if the blue edges are $s_i$-edges; same with the curves in $(iii)$. The curve $\gamma$ in $(iv)$ is meant to be anywhere in $R\setminus\fW$ and lifted to any sheet. In particular, both curves $\gamma$ in $(ii)$ and $(iv)$ are null-homologous in the first homology group $H_1(S(\fW),\Z)$.}
    \label{fig:Cycles_Rules}
\end{figure}

\noindent Third, Figure \ref{fig:Cycles_Hexavalent_Perturb} similarly depicts cycles $\upsilon_1,\upsilon_2,\upsilon_3,\upsilon_4$ that lift to (homonymous) geometric relative 1-cycles on the surface $S(\fW_{hex})$ associated to the weave $\fW_{hex}$ given by a hexavalent vertex, which consists of the (disjoint) union of three 2-disks. In the notation of Section \ref{sec: cycles}, $\upsilon_1$ (resp.~$\upsilon_2,\upsilon_3,\upsilon_4$) realizes the weave cycles with top weights $(1,0,0)$ (resp.~$(0,0,1),(0,1,0),(1,0,1)$) and bottom weights $(0,0,1)$ (resp.~$(1,0,0),(1,0,1),(0,1,0)$). Note that Figure \ref{fig:Cycles_Rules}.(i) also depicts the geometric cycles associated to those with top weights $(1,0,1)$, resp.~$(0,1,0)$, and bottom weights $(0,1,0)$, resp.~$(1,0,1)$, when the blue and red colors are exchanged. A weave cycle with arbitrary weights can be represented as a linear combination of these as well, which is geometrically represented by drawing copies of $\upsilon_i$ suitable superposed; denote this geometric 1-cycle by $\upsilon(a,b,c;a',b',c')$. In both cases of $S(\fW_{tri})$ and $S(\fW_{hex})$, we refer to these actual relative 1-cycles as being {\it geometric} cycles, in contrast to the (algebraically defined) weave cycles in Definition \ref{def:weave_cycle}. The following lemma states that the intersection numbers of these geometric cycles coincide with those intersection numbers defined in Section \ref{sec: cycles} for the respective weave cycles.

\begin{lemma}\label{lem:topological-intersection}
The algebraic intersections of the homology classes associated to the geometric 1-cycles in $S(\fW_{tri})$ and $S(\fW_{hex})$ described above coincide with the intersections of the corresponding weave cycles.
\end{lemma}

\begin{proof}
This readily follows by computing the geometric intersections of the $\gamma_i$ and $\upsilon_j$ cycles among themselves. From the generic projections from Figure \ref{fig:Cycles_Trivalent_Perturb}, it is immediate to see that the geometric intersection matrices are

$$(\langle\gamma_i,\gamma_j\rangle)=\left(\begin{array}{ccc}
0 & -1 & 1 \\
1 &     0 & -1\\
-1 & 1 & 0\\
\end{array}\right),\quad (\langle\upsilon_i,\upsilon_j\rangle)=\left(\begin{array}{ccccc}
0 & -1  & 0 & 0\\
1 &  0  & 0 & 0\\
0 & 0  & 0 & 0\\
0 & 0  & 0 & 0
\end{array}\right),$$
and these coincide with the intersections from Section \ref{sec: cycles}.
\end{proof}

\begin{figure}[h!]
\centering
    \includegraphics[width=\textwidth]{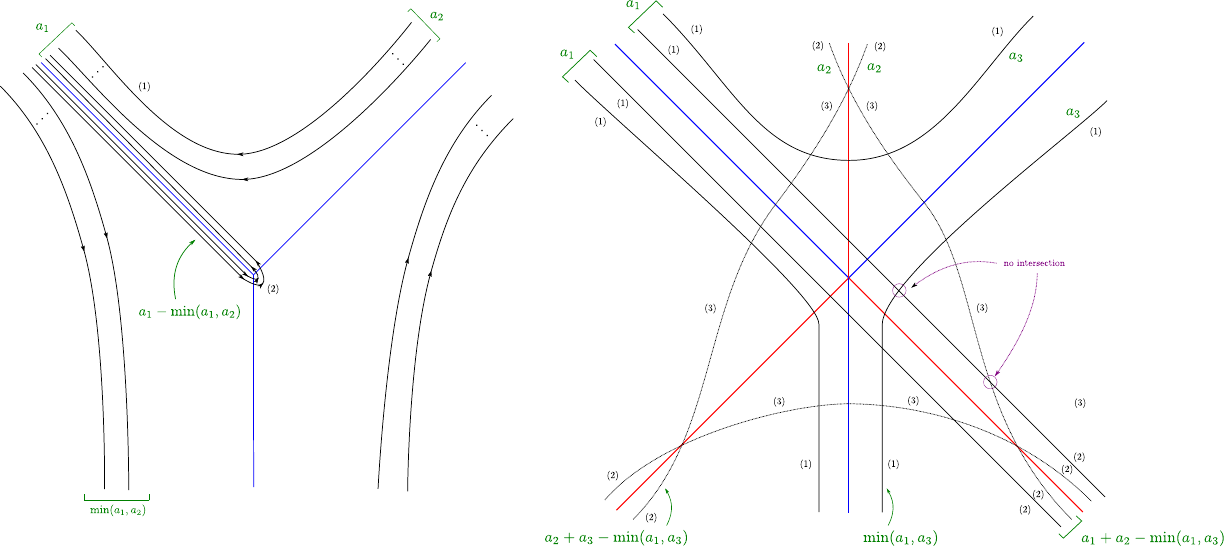}
    \caption{Embedded representative for the Lusztig cycles near a trivalent (left) and hexavalent vertices (right). The hexavalent picture uses the train track notation from Figure \ref{fig:Cycles_Notation}. In the hexavalent case, none of the intersections of the projection yield any geometric intersections in $S(\fW_{hex})$ as the branches near each intersection lift to different sheets. The trivalent picture is drawn in the case that $a_2<a_1$, the case $a_1<a_2$ is symmetric and the case $a_1=a_2$ would have no curves going near the trivalent vertex for $\fW_{tri}$. The hexavalent picture is drawn in the case that $a_3<a_1$, the case $a_1<a_3$ is also symmetric and the case $a_1=a_3$ would have no $\upsilon_1$-type curves going from the top left across to the bottom right.}
    \label{fig:Cycles_Lusztig}
\end{figure}

In general, the geometric realizations $\gamma(a,b,c)\sse S(\fW_{tri})$ and $\upsilon(a,b,c;a',b',c')\sse S(\fW_{hex})$ for arbitrary $a,b,c,a',b',c'\in\Z$ described above are {\it immersed} relative 1-cycles. For Lusztig weave cycles, as in Definition \ref{def:Lusztig_cycles}, we can find embedded relative 1-cycles geometrically representing them, as follows:

\begin{lemma}
Let $a_1,a_2,a_3\in\Z$ and . Then the relative 1-cycles $\gamma(a_1,a_2;\mbox{min}(a_1,a_2))\sse S(\fW_{tri})$ and  $$\upsilon(a_1,a_2,a_3;(a_2+a_3-\min(a_1,a_3),\min(a_1,a_3),a_1+a_2-\min(a_1,a_3))\sse S(\fW_{hex}).$$
are represented in homology by the embedded relative 1-cycles in Figure \ref{fig:Cycles_Lusztig}.
\end{lemma}

\begin{proof}
Let us describe the case of a trivalent vertex $\fW_{tri}$, the hexavalent case $\fW_{hex}$ is analogous. Consider the 1-cycle $\gamma(a_1,a_2;\mbox{min}(a_1,a_2))\sse S(\fW_{tri})$, with its projection onto $\R^2$ as $a_1$ disjoint unions of $\gamma_1$ (the perturbed version), $a_2$ disjoint unions of $\gamma_2$ (the perturbed version) and $\mbox{min}(a_1,a_2)$ disjoint unions of $\gamma_3$, also the perturbed version. These can be drawn so that the geometric intersections between $a_1\cdot \gamma_1$ and $a_2\cdot\gamma_2$ lie in the upper triangle of $R\setminus\fW_{tri}$, those between $a_1\cdot \gamma_1$ and $\mbox{min}(a_1,a_2)\cdot\gamma_3$ lie in the left triangle of $R\setminus\fW_{tri}$, and those between $a_2\cdot \gamma_2$ and $\mbox{min}(a_1,a_2)\cdot\gamma_3$ lie in the right triangle of $R\setminus\fW_{tri}$.\\

Consider the outmost copy of $\gamma_1$ and the outmost copy of $\gamma_2$ and perform a surgery at their unique intersection point so that one of the components is a curve that stays in the top triangle, as the ones appearing at the top of Figure \ref{fig:Cycles_Lusztig} (left). Iterate that procedure with the second outmost representatives, for a total of $\mbox{min}(a_1,a_2)$ times. Similarly, perform surgeries at the unique intersection of the outmost copy of $\gamma_1$ with the outmost copy of $\gamma_3$, and similarly for $\gamma_2$ and $\gamma_3$, and then iterate this procedure for a total of $\mbox{min}(a_1,a_2)$ times. The resulting 1-cycle geometrically represents $a_1\cdot\gamma_1+a_2\cdot \gamma_2+\mbox{min}(a_1,a_2)\cdot \gamma_3$. At this stage, the picture is that in Figure \ref{fig:Cycles_Lusztig} (left) plus a collection of closed immersed curves each of which winds around the trivalent vertex twice. It suffices to notice that these are null-homologous cycles, as indicated in Figure \ref{fig:Cycles_Rules}.(ii), and thus Figure \ref{fig:Cycles_Lusztig} (left) indeed represents this Lusztig cycle.
\end{proof}

\noindent We observe that computing intersections with these embedded representatives is rather immediate and yields the same results as in Section \ref{sec: cycles}, see Figure \ref{fig:Cycles_Trivalent_Lusztig_Intersections}. These local cycles from Figures \ref{fig:Cycles_Notation} and \ref{fig:Cycles_Lusztig} all glue globally to form geometric 1-cycles on $S(\fW)$: at a generic horizontal slice of the weave $\fW$ the cycle continue according to Figure \ref{fig:Cycles_Notation} and the boundary conditions match with those in Figure \ref{fig:Cycles_Lusztig}. For those Lusztig cycles that are contained in a compact region of $\fW$, the associated geometric 1-cycle is closed. For a Lusztig cycle that falls down, the associated geometric 1-cycle defines a relative 1-cycle. In general, these geometric 1-cycles can be simplified with the rules in Figure \ref{fig:Cycles_Rules}.$(ii)$, $(iii)$ and $(iv)$, plus other clear relations in homology, so as to obtain simpler representatives of their homology classes. For instance, a geometric 1-cycle might have several components, but if one of them is a curve $\gamma$ homologous to a curve as in Figure \ref{fig:Cycles_Rules}.$(ii)$ or $(iv)$, then that component $\gamma$ is null-homologous and can be erased.\\

\noindent Finally, there is substantial symplectic topology behind the theory of weaves, braid varieties and their cluster structures. The reader is referred to \cite{CasalsHonghao,CW,CZ} for that symplectic geometric interpretation and its relation to the microlocal theory of sheaves, and to \cite[Section 5]{CN} for its relation to Floer theory. In particular, see \cite[Section 4]{CW} for a discussion of how certain first homology lattices associated to $S(\fW)$ can arise as the natural $\mathcal{A}$- and $\mathcal{X}$-lattices.

\begin{figure}[h!]
\centering
    \includegraphics[width=\textwidth]{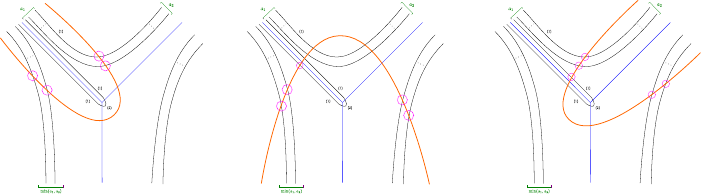}
    \caption{The intersections of $a_1\cdot\gamma_1+a_2\cdot \gamma_2+\mbox{min}(a_1,a_2)\cdot \gamma_3$ with $\gamma_1$, on the left, $\gamma_2$, on the right, and $\gamma_3$, center. Note that the intersections with $\gamma_1$ cancel.}
    \label{fig:Cycles_Trivalent_Lusztig_Intersections}
\end{figure}


\section{Cluster Poisson structures and Donaldson-Thomas transformations}\label{sec:Poisson}

This section proves Corollary \ref{cor:ensemble} and discusses DT-transformations.


\subsection{Braid varieties and $\CA$ and $\CX$-schemes}\label{ssec:braid_varieties_AX} Consider a seed datum $\s$ as in Section~\ref{sec: background}. Two schemes $\CA_\s$ and $\CX_\s$ are constructed in \cite{FockGoncharov_ModuliLocSys}, cf.~also \cite[Section 2]{GHK15}. The scheme $\CA_\s$ is constructed by gluing cluster tori $\mathbb{T}_{\mathbf{t}}$ using the mutation rule from Section~\ref{sec: background}, where $\mathbf{t}$ runs over all seeds mutation equivalent to $\s$. The scheme $\CX_\s$ is constructed by gluing dual cluster tori $\mathbb{T}^{\vee}_{\mathbf{t}}$, using a dual mutation rule 
described in \cite[Section 1.2]{FockGoncharov_ensemble}, cf.~ formula~\eqref{eq:mutation-X-Section-8} below. 
The pair $(\CA_\s,\CX_\s)$ can completed to a cluster ensemble, a notion first introduced in \cite[Section 1.2]{FockGoncharov_ensemble}, by choosing a birational map $p:\CA_\s\dashrightarrow\CX_\s$ induced by a map between the lattices corresponding to the tori $\mathbb{T}^{\vee}_{\mathbf{t}}$ and $\mathbb{T}_{\mathbf{t}}$, see~\cite[Section 2]{GHK15}. The choice is canonical in the absence of frozen variables. In general, these two schemes $\CA_\s$ and $\CX_\s$ are not equal nor $p$ is an isomorphism.\footnote{In fact, $\CA_\s$ is always separated but $\CX_\s$ might not be, cf.~\cite[Remark 4.2]{GHK15}.} The following two facts are relevant:\\
\begin{itemize}
    \item[(i)] By \cite[Theorem 3.14]{GHK15}, the ring of regular functions $\mathcal{O}_{\CA_\s}(\CA_\s)$ is the upper cluster algebra associated to $\s$ and the canonical map $\iota:\CA_\s\longrightarrow \Spec \mathcal{O}_{\CA_\s}(\CA_\s)$ is an open immersion.\\

    \item[(ii)] The scheme $\CX_\s$ carries a natural Poisson structure which is compatible with the Poisson structure on each torus $\mathbb{T}^{\vee}_{\mathbf{t}}$ given by the $\CX$-variables, cf.~\cite[Section 1.2]{FockGoncharov_ensemble}. On each cluster dual torus $\mathbb{T}^{\vee}_{\mathbf{t}}$, the $\CX$-variables $\{X_{\mathbf{t}}\}$ are dual to the $\{A_\mathbf{t}\}$-variables, cf.~ibid.
\end{itemize} 

Let $\beta$ be a positive braid word and $\s_\beta$ be the seed datum associated to any Demazure weave for $\beta$. The two facts above now specialize as follows:\\

\begin{itemize}
    \item[$(i_\beta)$] By Theorem \ref{thm: loc acyclic}, the cluster algebra $\C[X(\beta)]$ associated to $\s_\beta$ is locally acyclic and thus the upper cluster algebra coincides with the cluster algebra, cf.~\cite[Theorem 2]{Muller14_AequalsU}. By item $(i)$ above, this implies the existence of an isomorphism $\C[X(\beta)]\cong \mathcal{O}_{\CA_{\s_\beta}}(\CA_{\s_\beta})$. (Alternatively, see Corollary~\ref{cor: non simply laced}.) In general, the associated open immersion $\iota:\CA_{\s_\beta}\longrightarrow X(\beta)$ is not an isomorphism.\\

    \item[$(ii_\beta)$] The main result of the upcoming Subsection \ref{ssec:Poisson_structure} will be to
    construct,
    for
    our seed datum $\s=\s_\beta$,
    an explicit
    map $p:\CA_\s\longrightarrow\CX_\s$ making $(\CA_\s,\CX_\s)$ a cluster ensemble and, in addition, show that it is an isomorphism of schemes. In particular, the canonical (birational) cluster Poisson structure on $\CX_\s$ can be pull-backed to a Poisson structure on $\CA_\s$. By item $(i_\beta)$ above, this will imply that $X(\beta)$ admits a (birational) cluster Poisson structure, cluster with respect to the $\CX$-variables.\\
\end{itemize}


\subsection{Cluster Poisson structures and $X(\beta)$}\label{ssec:Poisson_structure} In this section, we explicitly construct a map $p:\CA_{s_\beta}\longrightarrow \CX_{s_\beta}$ for any choice of seed datum $\s_\beta$ associated to a Demazure weave for a positive braid word $\beta$.
We first construct it as a map of tori, in  matrix form, for an arbitrary given weave. We then proceed to show that this definition is compatible with mutations, thus giving the desired global map of schemes. 
By adapting  an argument from \cite{SG}, we  show that this map is in fact an isomorphism. As explained in item $(ii_\beta)$ of Subsection \ref{ssec:braid_varieties_AX}, this will imply that, in addition to $\C[X(\beta)]$ being a cluster algebra (as proven in previous sections), $X(\beta)$ also admits a cluster Poisson structure, also known as a cluster $\mathcal{X}$-structure.\\

First, we use the following result to construct the corresponding cluster $\CX$-variables as rational functions on $X(\beta)$ itself. We will later prove in Lemma \ref{lem: det=pm1} that the corresponding matrices for $p:\CA_{s_\beta}\longrightarrow \CX_{s_\beta}$, restricted to the cluster tori, are unimodular indeed.
\color{black}

\begin{lemma}\label{lem: X structure}
Let $(\varepsilon_{ij}) \in \mbox{Mat}(n, m)$, $n \leq m$, be the exchange matrix of a seed in a cluster algebra. Suppose that there exists an integer square matrix $(p_{ij}) \in \mbox{Mat}(m, m)$ such that the following two conditions are satisfied:

\begin{itemize}
\item[-] $p_{ij} = \varepsilon_{ij}$, unless both $i$ and $j$ are frozen;
\item[-] $\mathrm{det} (p_{ij}) = \pm 1$.
\end{itemize}
Then the collection of rational functions $(X_k), k \in [m],$ given by
\begin{equation}
\label{7.6.2022.1}
X_k := \prod (A_j) ^{p_{kj}}
\end{equation}
defines an initial seed of a cluster Poisson structure in the given cluster algebra.
\end{lemma}

\noindent Note that, by construction, $X_k$ are only \emph{rational} functions on $X(\beta)$, whereas $A_k$ are \emph{regular} functions.

\begin{proof} The proof is closely related to the calculations in \cite[Section 18]{SG}, as follows. Let $\Lambda$ be a free $\mathbb{Z}$-module with a basis $\{f_1, \ldots, f_m\}$. Let us set
\[
e_i :=\sum_{j \in [m]} p_{ij}f_j.
\]
Since $\det(p_{ij})=\pm 1$, $\{e_1,\ldots, e_m\}$ forms a new basis of $\Lambda$. Consider the algebraic torus $\mathcal{T}_\Lambda:= {\rm Hom}(\Lambda, \mathbb{G}_m)$. Each $v\in \Lambda$ corresponds to a character $T_v$ of $\mathcal{T}_\Lambda$. We set 
\[
X_i:= T_{e_i}, \qquad A_i:= T_{f_i}. 
\]
The character variables satisfy the defining identity \eqref{7.6.2022.1}. Following \cite[Lemma 18.2]{SG}, the mutation at $k\in [n]$ gives rise to a new unimodular matrix $(p_{ij}')$ such that 
\begin{equation}
\label{2022.7.14.1}
p_{ij}'=\left\{ \begin{array}{ll}
    -p_{ij} & \mbox{if $k=i$ or $k=j$}  \\
     p_{ij}+[p_{ik}]_+p_{kj}+p_{ik}[-p_{kj}]_+ & \mbox{otherwise}. 
\end{array}\right.
\end{equation}
Note that $p_{ij}'=\varepsilon_{ij}'$ unless both $i$ and $j$ are frozen. Recall that the cluster mutation $\mu_k$ gives rise to two new sets of variables
$\{A_i'\}$ and $\{X_i'\}$ such that
\begin{equation}
\label{eq:mutation-A-Section-8}
A_i'=\left\{\begin{array}{ll}
     A_i& \mbox{if } i\neq k  \\
     A_k^{-1}({\prod{A_j}^{[\varepsilon_{kj}]_+}+\prod A_j^{[-\varepsilon_{kj}]_+}}) & \mbox{if } i=k. 
\end{array}\right. 
\end{equation}
\begin{equation}
\label{eq:mutation-X-Section-8}
X_i'=\left\{\begin{array}{ll}
     X_i(1+X_k^{-{\rm sgn}(\varepsilon_{ik})})^{-\varepsilon_{ik}}& \mbox{if } i\neq k  \\
     X_k^{-1} & \mbox{if } i=k. 
\end{array}\right. 
\end{equation}
By Theorem 18.3 of \cite{SG}, quantum versions of the above mutations are defined via conjugations with the quantum dilogarithm series following monomial changes.
As a semi-classical limit, we obtain 
\begin{equation}
\label{7.6.2022.2}
X_i':= \prod_j (A_j')^{p_{ij}'}.
\end{equation}
In this way, we obtain a new algebraic torus with two sets of variables $\{A_i'\}$ and $\{X_i'\}$ related by \eqref{7.6.2022.2}. Now, repeating the same procedure to the newly obtained seeds and tori recursively, we obtain a cluster Poisson algebra (resp.~an upper cluster algebra) as the intersection of the Laurent polynomial rings of the $X$ (resp.~$A$) variables. These two algebras are isomorphic locally via the isomorphism given by the defining identities  \eqref{7.6.2022.1}.
\end{proof}

\noindent Note that the existence of the matrix $(p_{ij})$ as in Lemma \ref{lem: X structure} implies that the (non-square) matrix $(\varepsilon_{ij})_{i \in I^{\uf}, j \in I}$ has full rank. \\
 
Let us now specialize to the case of a braid variety $X(\beta)$. Consider any Demazure weave $\fW: \beta \to \delta(\beta)$. In previous sections, we defined a cluster algebra structure on $\mathbb{C}[X(\beta)]$ with an initial seed determined by the weave $\fW$. 
Let $E$ be the set of edges on the southern boundary of $\fW$. The ordered sequence of edges $e\in E$ corresponds to a reduced decomposition of $w=\delta(\beta)$, which further gives rise to an ordered list of positive roots $\rho_e$ as in \eqref{eq: def gamma}. Let $(\gamma_i)$ be the collection of cycles corresponding to the trivalent vertices of $\fW$. Recall the bilinear form $(\cdot, \cdot)$ on the root lattice defined via \eqref{22.7.20.1}. 
Following the notation of Lemma \ref{lem: X structure}, we have the exchange matrix 
\[
\varepsilon_{ij}= \sum_{v \mbox{ vertex of } \fW} \#_v(\gamma^{\vee}_i\cdot \gamma_j) + \frac{1}{2}\sum_{e, e' \in E} {\rm sign}(e'-e) \gamma_i^{\vee}(e)\gamma_j(e') \left(\rho_e, \rho^{\vee}_{e'}\right).
\]
where $i, j$ are trivalent vertices of the weave $\fW$. The second term corresponds to the boundary intersection number of $\gamma_i$ and $\gamma_j$ as in \eqref{eq: boundary intersection non simply laced}.\\

\noindent We now construct a suitable matrix $(p_{ij})$ from a weave $\fW$.
Set $\theta_i:=\theta_i(\fW),\theta_i^{\vee}:=\theta^{\vee}_i(\fW)$, where
\[
\theta_i(\fW):= \sum_{e\in E(\fW)} \gamma_i^{\vee}(e) \rho_e,\quad  \theta^{\vee}_i(\fW):= \sum_{e\in E(\fW)} \gamma_i(e) \rho^{\vee}_e.
\]
Note that $\theta_i, \theta_i^{\vee}\neq 0$ if and only if $i$ is frozen. 
We define $p_{ij}:=p_{ij}(\fW)$ where
\begin{equation} \label{eq: p_ij}
p_{ij}(\fW):=\varepsilon_{ij}-\frac{1}{2}\left(\theta_i, \theta^{\vee}_j\right)= \varepsilon_{ij} -\frac{1}{2}\sum_{e, e'\in E(\fW)} \gamma_i^{\vee}(e)\gamma_j(e')\left(\rho_{e}, \rho^{\vee}_{e'}\right).
\end{equation}
Note that $p_{ij}=\varepsilon_{ij}$ unless both $i$ and $j$ are frozen, as required by Lemma \ref{lem: X structure}.
\begin{lemma} \label{lem: p-integer}
For any Demazure weave $\fW$, the matrix $(p_{ij}(\fW))$ is an integer matrix.
\end{lemma}
\begin{proof}
Note that
\begin{align*}
p_{ij}&=  \sum_{v \mbox{ vertex of } \fW} \#_v(\gamma^{\vee}_i\cdot \gamma_j) + \sum_{e, e' \in E} \frac{{\rm sign}(e'-e)-1}{2} \gamma_i^{\vee}(e)\gamma_j(e')\left(\rho_{e}, \rho^{\vee}_{e'}\right)\\
& =\sum_{v \mbox{ vertex of } \fW} \#_v(\gamma^{\vee}_i\cdot \gamma_j) - \sum_{e' <e} \gamma_i^{\vee}(e)\gamma_j(e')\left(\rho_{e}, \rho^{\vee}_{e'}\right)  -\sum_{e\in E} \gamma_i^{\vee}(e)\gamma_j(e).
\end{align*}
since $(\rho_e,\rho^{\vee}_e)=2$. It is clear that $p_{ij}$ is an integer by the last expression.
\end{proof}

\begin{lemma}
The absolute value $|\det(p_{ij}(\fW))|$ is independent of the chosen Demazure weave $\fW$. 
\end{lemma}

\begin{proof}
It suffices to show that $|\det(p_{ij})|$ is invariant under the following three changes.
\begin{itemize}
    \item[$(i)$] {\it Weave equivalences}. The matrix $\varepsilon_{ij}$ and the vectors $\theta_i,\theta_i^{\vee}$ remain invariant under weave equivalences. Hence $p_{ij}$ remains the same.\\
    
    \item[$(ii)$] {\it Weave mutations.} Note that the vectors $\theta_i,\theta_i^{\vee}$ remain invariant under weave mutation. The matrix $\varepsilon_{ij}$ changes according to the mutation rule \eqref{eq: quiver mutation plus minus} for exchange matrices. Therefore the matrix $(p_{ij})$ changes as in \eqref{2022.7.14.1}. A direct check shows that $|\det(p_{ij})|$ is invariant.\\
    
    \item[$(iii)$] {\it Add a $(2d_{ij})$-valent vertex at the bottom of the weave.} It follows from Lemma \ref{lem: boundary int} that the matrix $(\varepsilon_{ij})$ is invariant. Meanwhile a direct local check (and folding in non simply-laced case) shows that the vectors $\theta_i,\theta_i^{\vee}$ are invariant as well. Therefore $(p_{ij})$ is invariant.\qedhere
\end{itemize}
\end{proof}

\begin{lemma}\label{lem: det=pm1}
For any Demazure weave $\fW$, $\mathrm{det}(p_{ij}(\fW)) = \pm 1$.
\end{lemma}
\begin{proof}
We work by induction on $\ell(\beta)$, the case $\ell(\beta) = 1$ is clear. So assume the result is true for $\beta$. If $\delta(\beta\sigma_k) = \delta(\beta)s_k$ then $X(\beta\sigma_k) = X(\beta)$ and the argument is done. Suppose otherwise. Then, we consider the weave for $\beta\sigma_k$ depicted in Figure \ref{fig:weaveXstructure} and let $e$ be the edge drawn in yellow.
\begin{figure}[ht!]
\includegraphics[scale=1]{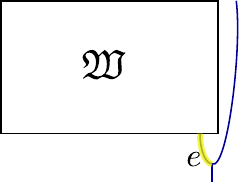}\caption{Weave for $\beta \sigma_k$ in the proof of Lemma \ref{lem: det=pm1}. The edge $e$ has been highlighted in yellow.}\label{fig:weaveXstructure}
\end{figure}

 The extra trivalent vertex corresponds to a cycle $\gamma_{m+1}$, with \[\theta_{m+1}=-\delta(\beta)(\alpha_k),\quad \theta^{\vee}_{m+1}=-\delta(\beta)\left(\alpha^{\vee}_k\right).\] Therefore
\[
p_{i,m+1}= \gamma_i(e), \qquad  p_{m+1, i}= -\gamma_i(e)- \left(\theta_i, \theta_{m+1}^{\vee}\right),\qquad  p_{m+1, m+1}=-1.
\]
The matrix $p_{ij}$ for $\beta\sigma_k$ has the form:
\[
\left(\begin{matrix} p_{11} & \cdots & p_{1m} & p_{1,m+1} \\ \vdots & \ddots & \vdots & \vdots \\ p_{m1} & \cdots & p_{mm} & p_{m,m+1} \\ p_{m+1,1} & \cdots & p_{m+1,m} & -1  \end{matrix}\right) \longrightarrow \left(\begin{matrix} p'_{11} & \cdots & p'_{1m} & p_{1,m+1} \\ \vdots & \ddots & \vdots & \vdots \\ p'_{m1} & \cdots & p'_{mm} & p_{m,m+1} \\ 0 & \cdots & 0 & -1  \end{matrix}\right)
\]
where the arrow means that we apply elementary matrix transformations, and \begin{align*}p'_{ij} &= p_{ij} + p_{i,m+1}p_{m+1,j}\\
&= \varepsilon_{ij}-\frac{1}{2}(\theta_i, \theta_j^{\vee})-\gamma^{\vee}_i(e)\gamma_j(e)-\gamma^{\vee}_i(e)(\theta_j, \theta^{\vee}_{m+1})\\
&=\left(\varepsilon_{ij}+\frac{1}{2}(\theta_i,\theta_{m+1}^{\vee})\gamma_j(e)-\frac{1}{2}(\theta_j, \theta^{\vee}_{m+1}) \gamma^{\vee}_{i}(e)\right)-\frac{1}{2}\left(\theta_i+\gamma_i^{\vee}(e)\theta_{m+1}, \theta_j^\vee+\gamma_j(e)\theta_{m+1}^\vee\right)\\
&=\varepsilon_{ij}'-\frac{1}{2}\left(\theta_i',(\theta_j')^\vee\right)
\end{align*}
coincides with the matrix for the weave $\fW: \beta \to \delta(\beta)$. The result now follows by induction.
\end{proof}

Recall that an exchange matrix is said to have really full rank if every element in $\Z^{I^{\uf}}$ is a linear combination of the columns of the rectangular matrix $(\varepsilon_{ij}: i\ \mathrm{mutable},j\ \mathrm{arbitrary})$, cf. \cite{LamSpeyer}. Note that a matrix that has really full rank has full rank.

\begin{corollary}\label{cor: full rank}
The exchange matrix $\varepsilon_{\fW}$ has really full rank.
\end{corollary}

\begin{proof}
If $i$ is mutable then $\varepsilon_{ij}=p_{ij}$, so the rectangular matrix $(\varepsilon_{ij}: i\ \mathrm{mutable},j\ \mathrm{arbitrary})$ consists of several rows of the matrix $p=(p_{ij})$. By Lemma \ref{lem: det=pm1} $p$ is unimodular and the result follows.
\end{proof}

\begin{theorem}\label{thm: cluster poisson}
The braid variety $X(\beta)$ admits a cluster Poisson structure.
\end{theorem}

\begin{proof} By Lemmas~\ref{lem: p-integer} and \ref{lem: det=pm1}, the matrix $(p_{ij})$ whose entries are given by the formula \eqref{eq: p_ij} satisfies the conditions of Lemma \ref{lem: X structure}. Thus, the collection of $\CX$-variables defined by the formula \eqref{7.6.2022.1} gives an initial seed of a cluster Poisson algebra. More precisely, there are two tori: the cluster torus in $X(\beta)$ that corresponds to the seed of the cluster algebra associated with the weave $\fW$, and the open torus in the cluster Poisson variety associated with the seed given by $\CX$-variables. The map $p$ defines an isomorphism from the former to the latter. In addition, this isomorphism is compatible with the respective mutations and thus induces a unimodular isomorphism from the corresponding cluster $\mathcal{A}$-variety to the corresponding cluster Poisson variety. It thus induces a unimodular isomorphism between their affinizations, whose domain is $X(\beta)$. Hence, the matrix $(p_{ij})$ gives a unimodular isomorphism between $X(\beta)$ and the affinization of a cluster Poisson variety. This endows $X(\beta)$ with a cluster Poisson structure by pulling back the cluster Poisson structure on the target along this isomorphism.
\end{proof}


\subsection{DT transformation}\label{sect: DT}

Thanks to Proposition \ref{prop: reddening}, together with the fact that the exchange matrix has maximal rank, the cluster Poisson variety $X(\beta)$ admits a Donaldson-Thomas (DT) transformation $\DT: X(\beta) \to X(\beta)$. In \cite[Section 4]{SW} an explicit geometric realization for the DT-transformation is presented for (double) Bott-Samelson varieties; this is used in \cite[Section 5]{CW} for a geometric description of the DT-transformation for grid plabic graphs of shuffle type. The goal of this section is to exhibit the DT-transformation explicitly for all braid varieties.\\

\noindent Let $\beta = \sigma_{i_1}\cdots \sigma_{i_{\ell}}$. Recall that we have the cyclic rotation
\[
\rho: X(\beta) \to X(\sigma_{i_{\ell}^{*}}\sigma_{i_{1}}\cdots \sigma_{i_{\ell-1}})
\]
that is a quasi-cluster transformation by Theorem \ref{thm: quasi cluster}. Applying this transformation $\ell(\beta)$ times we obtain $\rho^{\ell}: X(\beta) \to X(\beta^{*})$, where $\beta^{*} = \sigma_{i_{1}^{*}}\cdots \sigma_{i_{\ell}^{*}}$. On the other hand, since the map $i \mapsto i^{*}$ is an automorphism of the Dynkin diagram $\dynkin$, there is a group automorphism $*: \G \to \G$, $x \mapsto x^{*}$, satisfying $\borel^{*} = \borel$, and $x\borel \buildrel s_{i} \over \longrightarrow y\borel$ if and only if $x^{*}\borel \buildrel s_{i^{*}} \over \longrightarrow y^{*}\borel$. It follows that we have an isomorphism of varieties
\[
*: X(\beta) \to X(\beta^{*}).
\]
It is easy to see that this is an isomorphism of cluster varieties, as follows. Let $\fW: \beta \to \delta(\beta)$ be a weave. From the description of the cluster torus $T_{\fW} \subseteq X(\beta)$ in terms of distances of flags, it is easy to see that $T^{*}_{\fW} \subseteq X(\beta^{*})$ is the cluster torus $T_{\fW^{*}}$, where $\fW^{*}: \beta^{*} \to \delta(\beta^{*})$ is obtained by changing the color of every strand while keeping the shape of the weave intact. Obviously, the quivers $Q_{\fW}$ and $Q_{\fW^{*}}$ agree. The fact that the cluster variables also agree follows since these are defined in terms of distances of framed flags. 

As a slight modification and generalization of \cite[Section 1.5]{FZ}, we define the twist automorphism:
\[ D_{\beta} := *\circ\rho^{\ell}: X(\beta) \to X(\beta).
\]

\begin{theorem}\label{thm: DT}
The twist automorphism $D_{\beta}: X(\beta) \to X(\beta)$ is the $\DT$ transformation. 
\end{theorem}
\begin{proof}
As we have seen, the map $D_{\beta}$ is a quasi-cluster automorphism. It remains to show that, if $\fW: \beta \to \delta(\beta)$ is a weave and $D_{\beta}\fW: \beta \to \delta(\beta)$ a weave such that $D^{*}_{\beta}(T_{D_{\beta}\fW}) = T_{\fW}$, then the mutable parts of $Q_{\fW}$ and of $Q_{T_{D_{\beta}}\fW}$ are related by a reddening sequence of mutations. By \cite[Theorem 3.2.1]{Muller}, or \cite[Theorem 3.6]{goncharov2018donaldson}, it is enough to do this for a single weave. 

\begin{figure}[h!]
\centering
    \includegraphics[scale=0.5]{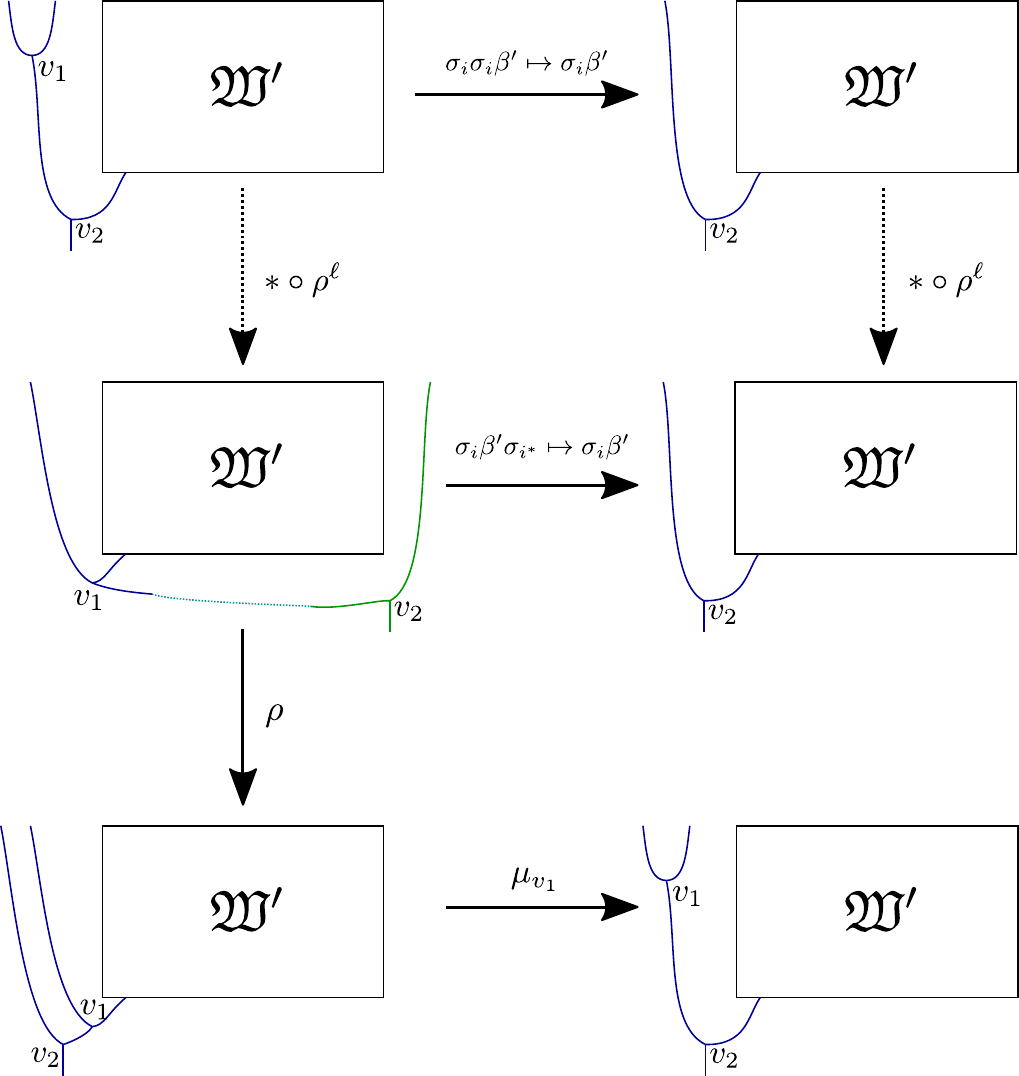}
    \caption{The weave $\fW$ for $\sigma_i\sigma_i\beta'$ (upper left) and that for $\sigma_i\beta'$ (upper right). The dotted arrows mean that we apply a cluster automorphism followed by a sequence of mutations. In the right dotted arrow, this sequence of mutations is a reddening sequence by inductive assumption.}
    \label{fig: DT}
\end{figure}

We work by induction on $\ell(\beta) - \ell(\delta)$, the case $\ell(\beta) - \ell(\delta) \in \{0, 1\}$ is clear. Let us assume for the time being that $\beta = \sigma_i\sigma_i\beta'$ for some positive braid word $\beta'$, where $\ell(\beta) = \ell + 1$. If $\delta(\beta) = s_i\delta(\beta')$ then we can reduce to the word $\beta'$ as in the proof of Theorem \ref{thm: loc acyclic}, so we assume that $\delta(\beta) = \delta(\beta')$. In this case, we may consider a weave $\fW$ as in the upper left corner of Figure \ref{fig: iibeta_weave}. Following the notation of that figure, the cycle corresponding to $v_1$ is a mutable sink and the cycle corresponding to $v_2$ a frozen source. By the inductive assumption, the DT transformation for $\sigma_i\beta'$ is $* \circ \rho^{\ell}$. We can apply the same transformation to $\beta$ to obtain the word $\sigma_i\beta'\sigma_{i^{*}}$. Note that the quiver for $X(\sigma_i\beta)$ is a subquiver of that for $X(\sigma_i\beta\sigma_{i^{*}})$, so we can apply a reddening sequence of mutations for $X(\sigma_i\beta)$ to $X(\sigma_i\beta\sigma_{i^{*}})$, see Figure \ref{fig: DT}.

Applying another cyclic shift to $\sigma_i\beta'\sigma_{i^{*}}$ we get a weave for $\beta = \sigma_i\sigma_i\beta'$ that is related to the starting weave by mutation at the sink $v_1$. This is a reddening sequence for the quiver that consists of the single vertex $v_1$. Since $v_1$ is a sink, it follows from Lemma 2.3 in \cite{BucherMachacek} that we have a reddening sequence for $Q_{\fW}$, so $\rho \circ * \circ \rho^{\ell}$ is the $\DT$ transformation for $\beta$. Now the result follows by observing that $\rho \circ * = * \circ \rho$.

In the general case, if $\ell(\beta) - \ell(\delta) > 0$ we can apply a sequence $\tau$ of braid moves (that can be interpreted as cluster automorphisms) and cyclic shifts (that are quasi-cluster automorphisms) to bring $\beta$ to the form $\sigma_i\sigma_i\beta'$. The diagram
\begin{center}
    \begin{tikzcd}
     X(\beta) \arrow["\tau"]{rr} \arrow["D_{\beta}"]{d}& & X(\sigma_i\sigma_i\beta') \arrow["D_{\sigma_i\sigma_i\beta'}"]{d} \\ X(\beta) \arrow["\tau"]{rr}  & & X(\sigma_i\sigma_i\beta')
    \end{tikzcd}
\end{center}
commutes and it follows that $D_{\beta}$ is indeed the DT transformation of $X(\beta)$.
\end{proof}

Thanks to Theorems \ref{thm:main} and \ref{thm: cluster poisson}, the braid variety $X(\beta)$ admits both a cluster $\mathcal{A}$- and a cluster $\mathcal{X}$-structure. Moreover, Proposition \ref{prop: langlands} together with \cite[Lemma 1.11]{FockGoncharov_ensemble} shows that, for any positive braid $\beta$, the pair 
$(X(\beta), X(\beta))$
is a cluster ensemble, i.e. $X(\beta)$ has both a cluster $\mathcal{A}$ and a cluster Poisson structure, related by a unimodular isomorphism.
Finally, since the exchange matrices have full rank (Corollary \ref{cor: full rank}) and the braid varieties admit a DT transformation, results of \cite{GHKK} and \cite{Qin} allow us to conclude the following result, which is an enhancement of Corollary \ref{cor: 2 bases}.

\begin{theorem} \label{thm: FG duality}
Let $\beta$ be a positive braid. Then the pair 
$(X(\beta), X(\beta))$
is a cluster ensemble such that the Fock-Goncharov cluster duality conjecture holds. In particular, $\C[X(\beta)]$ admits a canonical basis of theta functions naturally parameterized by the integral tropicalization of the dual braid variety $X^{\vee}(\beta)$. If $\G$ is simply-laced,  $\C[X(\beta)]$ also admits a generic basis parameterized by the same lattice.
\end{theorem}

\noindent Note that there are a number of schemes being discussed. On the one hand, there is the braid variety $X(\beta)$, which, by Theorems \ref{thm:main} and \ref{thm: cluster poisson} is an affinization of both a cluster $\mathcal{A}$-variety and a cluster Poisson variety, associated with a certain seed $\s$. These are related by a unimodular isomorphism. On the other hand, there is the variety $X^{\vee}(\beta)$, which admits the same pair of structures, but associated with the seed that is Langlands dual to $\s$. The cluster $\mathcal{A}$-variety (resp.~its affinization) on one side is dual, in the sense of Fock-Goncharov cluster duality, to the cluster Poisson variety on the other side (resp.~its affinization), and vice versa.


\section{Gekhtman-Shapiro-Vainshtein form}
\label{sec: form}

Since the cluster algebra $\C[X(\beta)]$ is locally acyclic, the canonical cluster $2$-form defined on the union of cluster tori extends to $X(\beta)$, see \cite[Theorem 4.4]{Muller}. The form on $X(\beta)$ is known as the Gekhtman-Shapiro-Vainshtein (GSV) form. In this section, we show that this GSV form may be constructed using the Maurer-Cartan form on the group $\G$ and the matrices $B_{\beta}$, similarly to \cite[Section 3]{CGGS1} and \cite{Mellit}.


\subsection{Construction of the form $\omega_\beta$ on $X(\beta)$} The construction of the 2-form $\omega_\beta$ on the braid variety $X(\beta)$ following Mellit \cite{Mellit}, see also \cite[Section 3]{CGGS1}), is as follows. Throughout this section, we assume without loss of generality that $\delta(\beta) = w_0$, cf. Lemma \ref{lem: loc closed}. Let $\theta$, resp.~$\theta^{R}$, denote the left (resp.~right) invariant $\mathfrak{g}$-valued form on $\G$, also known as the Maurer-Cartan form, and $\kappa: \mathfrak{g} \otimes \mathfrak{g} \to \C$ the Killing form on the Lie algebra $\mathfrak{g}$ of $\G$. These define a $2$-form on $\G \times \G$ by:
\[
(f|g) := \kappa(\theta(f)\wedge\theta^{R}(g))
\]

\noindent The 2-form $(f|g)$ satisfies the following ``cocycle condition'':
\begin{equation}
\label{eq: cocycle}
(f|g)+(fg|h)=(f|gh)+(g|h).
\end{equation}
Given a collection of $G$--valued functions $f_1,\ldots,f_{\ell}$, we define
\begin{equation}
    \label{eq: 2 form}
(f_1|\cdots |f_{\ell}):=(f_1|f_2)+(f_1f_2|f_3)+\ldots+(f_1\cdots f_{\ell-1}|f_{\ell}).
\end{equation}
By \eqref{eq: cocycle} this definition is associative in $f_i$. Using \eqref{eq: 2 form}, we define the 2-form $\omega_\beta$ on $X(\beta)$ for $\beta=\sigma_{i_1}\cdots\sigma_{i_\ell}$ to be the restriction of the form
$$
\omega:=(B_{i_1}(z_1)|\cdots |B_{i_{\ell}}(z_{\ell}))\in \Omega^2(\C^{\ell})
$$
to the braid variety $X(\beta)$. By definition, upon applying the map $B_{\beta}: \C^{\ell} \to \G$, the braid variety has its image contained in $w_0\borel$. Thus, similarly to \cite[Lemma 3.1]{CGGS1}, the restriction $\omega_\beta:=\omega|_{X(\beta)}$ yields a closed 2-form on $X(\beta)$.

\begin{remark}\label{rmk: SL}
In case $\G = \SL_n$, we have $\theta(f) = f^{-1}df$ and $\theta^{R}(g) = dgg^{-1}$. Moreover, if $\Upsilon: \G_1 \to \G_2$ is a homomorphism of Lie groups then $\Upsilon^{*}(\theta_{\G_2}) = \theta_{\G_1}$; similarly for the right-handed versions. We use these facts below, together with pinnings (Section \ref{sect: pinnings}) to reduce several calculations to the $\SL$-case.
\end{remark}


\subsection{Coincidence of the forms} Let us show that the closed $2$-form $\omega_\beta$ coincides with the GSV form on $X(\beta)$. We proceed via several lemmas studying the restrictions of the form to braid words of length $2$ and $3$, where we may assume we work in the $\SL$-case, see Remark \ref{rmk: SL} above. 

\begin{lemma}
\label{lem: left right invariant}
Suppose that $f=B_i(z)\chi_i(u)$. Then 
\begin{itemize}
\item[($1$)] The pullback of the left-invariant one-form along $f$ equals
$$f^{-1}df=\varphi_i\left(\begin{matrix}
u^{-1}du & 0\\
-u^{2}dz & -u^{-1}du\\
\end{matrix}\right)$$

\item[$(2)$] The pullback of the right-invariant one-form along $f$ equals $$df\cdot f^{-1}=\varphi_i\left(\begin{matrix}
-u^{-1}du & dz+2u^{-1}zdu\\
0& u^{-1}du\\
\end{matrix}\right)$$
\end{itemize}
\end{lemma}

\begin{proof}
We have
$$
f=\varphi_i\left(\begin{matrix}
uz & -u^{-1}\\
u & 0\\
\end{matrix}\right),\ 
f^{-1}=\varphi_i\left(\begin{matrix}
0 & u^{-1}\\
-u & uz\\
\end{matrix}\right),\ 
df=\varphi_i\left(\begin{matrix}
udz+zdu & u^{-2}du\\
du & 0\\
\end{matrix}\right)
$$
and the result follows.
\end{proof}

\begin{lemma}
\label{lem: form change 6 valent}
Suppose that $i$ and $j$ are adjacent. Then
$$
(B_i(z_1)|\chi_i(u_1)|B_{j}(z_2)|\chi_j(z_2)|B_i(z_3)|\chi_i(u_3))=\frac{du_1du_2}{u_1u_2}-\frac{du_1du_3}{u_1u_3}+\frac{du_2du_3}{u_2u_3}.
$$
\end{lemma}
\begin{proof}
It is easy to see that $(B_i(z)|\chi_i(u))=0$, so
$$
(B_i(z_1)|\chi_i(u_1)|B_{j}(z_2)|\chi_j(z_2)|B_i(z_3)|\chi_i(u_3))=(f_1|f_2|f_3)=(f_1|f_2)+(f_1f_2|f_3),
$$
where $f_1=B_i(z_1)\chi_i(u_1), f_2=B_{j}(z_2)\chi_j(z_2), f_3=B_{i}(z_3)\chi_i(z_3).
$ Now we can restrict to $SL_3$ and assume $i=1,j=2$. By Lemma \ref{lem: left right invariant} we get
$$
(f_1|f_2)=\Tr\left(\begin{matrix}u_1^{-1}du_1 & 0 & 0\\
-u_1^{2}dz_1 & -u_1^{-1}du_1 & 0\\
0 & 0 & 0\\
\end{matrix}\right)
\left(\begin{matrix}0 & 0 & 0\\
0 & -u_2^{-1}du_2 & dz_2+2u_2^{-1}z_2du_2 & \\
0 & 0 & u_2^{-1}du_2\\
\end{matrix}\right)=\frac{du_1du_2}{u_1u_2}.
$$
Similarly, one can compute
$$
f_1f_2=\left(
\begin{matrix}
u_1z_1 & -u_1^{-1}u_2z_2 & u_1^{-1}u_2^{-1}\\
u_1 & 0 & 0\\
0 & u_2 & 0\\
\end{matrix}
\right),
(f_1f_2)^{-1}=\left(
\begin{matrix}
0 & u_1^{-1} & 0\\
0 & 0 & u_2^{-1}\\
u_1u_2 & -u_1u_2z_1 & u_2z_2\\
\end{matrix}
\right)
$$
and
$$
d(f_1f_2)=\left(
\begin{matrix}
u_1dz_1+z_1du_1 & * & *\\
du_1 & 0 & 0\\
0 & du_2 & 0\\
\end{matrix}
\right),\ 
(f_1f_2)^{-1}d(f_1f_2)=
\left(
\begin{matrix}
u_1^{-1}du_1 & 0 & 0\\
0 & u_2^{-1}du_2 & 0\\
* & * & *\\
\end{matrix}
\right),
$$
so
$$
(f_1f_2|f_3)=\Tr\left(
\begin{matrix}
u_1^{-1}du_1 & 0 & 0\\
0 & u_2^{-1}du_2 & 0\\
* & * & *\\
\end{matrix}
\right)
\left(\begin{matrix}-u_3^{-1}du_3 & * & 0\\
0 & u_3^{-1}du_3 & 0\\
0 & 0 & 0\\
\end{matrix}\right)=-\frac{du_1du_3}{u_1u_3}+\frac{du_2du_3}{u_2u_3}.
$$
\end{proof}

\begin{lemma}
\label{lem: form change 3 valent}
Suppose that 
$$
B_i(z_1)\chi_i(u_1)B_i(z_2)\chi_i(u_2)=B_i(z_3)\chi_i(u_3)x_i(w),
$$
where $z_3=z_1-u_1^{-2}z_2^{-1},u_3=z_2u_1u_2,w=-z_2^{-1}u_2^{-2}$ as in \eqref{eq: trivalent framed}. Then
$$
(B_i(z_1)|\chi_i(u_1)|B_i(z_2)|\chi_i(u_2))-(B_i(z_3)|\chi_i(u_3)|x_i(w))=2\left(\frac{du_1du_3}{u_1u_3}+\frac{du_3du_2}{u_2u_3}-\frac{du_1du_2}{u_1u_2}\right).
$$
\end{lemma}

\begin{proof}
Let $f_k=B_i(z_k)\chi_i(u_k), k=1,2,3$ as above. Then by Lemma \ref{lem: left right invariant}
$$
(B_i(z_1)|\chi_i(u_1)|B_i(z_2)|\chi_i(u_2))=(f_1|f_2)=
$$
$$
\Tr\left(\begin{matrix}
u_1^{-1}du_1 & 0\\
-u_1^{2}dz_1 & -u_1^{-1}du_1\\
\end{matrix}\right)\left(\begin{matrix}
-u_2^{-1}du_2 & dz_2+2u_2^{-1}z_2du_2\\
0& u_2^{-1}du_2\\
\end{matrix}\right)=-2\frac{du_1du_2}{u_1u_2}-dz_1(u_1^{2}dz_2+2u_1^{2}u_2^{-1}z_2du_2).
$$
On the other hand,
$$
(B_i(z_3)|\chi_i(u_3)|x_i(w))=(f_3|x_i(w))=\Tr\left(\begin{matrix}
u_3^{-1}du_3 & 0\\
-u_3^{2}dz_3 & -u_3^{-1}du_3\\
\end{matrix}\right)\left(\begin{matrix}
0 & dw\\
0 & 0\\
\end{matrix}\right)=-u_3^{2}dz_3dw=
$$
$$
- z_2^2u_1^2u_2^2(dz_1+2u_1^{-3}z_2^{-1}du_1+u_1^{-2}z_2^{-2}dz_2)(z_2^{-2}u_2^{-2}dz_2+2z_2^{-1}u_2^{-3}du_2)=$$
$$
-dz_1(u_1^2dz_2+2u_1^2u_2^{-1}z_2du_2)-(2u_1^{-1}z_2^{-1}du_1dz_2+4u_1^{-1}u_2^{-1}du_1du_2+2z_2^{-1}u_2^{-1}dz_2du_2),
$$
therefore
$$
(B_i(z_1)|\chi_i(u_1)|B_i(z_2)|\chi_i(u_2))-(B_i(z_3)|\chi_i(u_3)|x_i(w))=2\left(\frac{du_1dz_2}{u_1z_2}+\frac{du_1du_2}{u_1u_2}+\frac{dz_2du_2}{z_2u_2}\right).
$$
Finally, $d\log(u_3)=d\log(u_1)+d\log(u_2)+d\log(z_2)$, so
$$
\frac{du_1du_3}{u_1u_3}+\frac{du_3du_2}{u_2u_3}-\frac{du_1du_2}{u_1u_2}=\frac{du_1dz_2}{u_1z_2}+\frac{dz_2du_2}{z_2u_2}+\frac{du_1du_2}{u_1u_2}.
$$
\end{proof}

\begin{theorem}
\label{thm: form}
Let $\beta$ be a positive braid word, $\fW$ a Demazure weave for $\beta$, $A_i$ the cluster variables for its associated cluster seed on $X(\beta)$, $\varepsilon_{ij}$ the coefficients of its exchange matrix, and $d_i$ the symmetrizers. Then the restriction of the 2-form $\omega_\beta\in\Omega^2(X(\beta))$ to the cluster chart corresponding to $\fW$ agrees, up to a constant factor of 2, with the Gekhtman-Shapiro-Vainshtein form \cite{GSV} defined by  
$$
\omega_{GSV}:=\sum_{i,j}d_i\varepsilon_{ij}\frac{dA_idA_j}{A_iA_j}.
$$
\end{theorem}

\begin{proof}Assume first that $\G$ is simply laced. 
We compute the 2-form $\omega$ at every cross-section of the weave using \eqref{eq: 2 form} and keep track of all the changes. At every edge $e$ we have $u=\prod A_i^{w_i(e)},$ so $d\log(u)=\sum w_i(e)d\log A_i$.

As we cross a 6-valent vertex with incoming $u$-variables $u_1,u_2,u_3$ and outgoing $u'_1,u'_2,u'_3$, by Lemma \ref{lem: form change 6 valent} the form changes by
$$
\left(\frac{du_1du_2}{u_1u_2}-\frac{du_1du_3}{u_1u_3}+\frac{du_2du_3}{u_2u_3}\right)-\left(\frac{du'_1du'_2}{u'_1u'_2}-\frac{du'_1du'_3}{u'_1u'_3}+\frac{du'_2du'_3}{u'_2u'_3}\right).
$$
As we cross a 3-valent vertex with incoming $u$-variables $u_1,u_2$ and outgoing $u_3$, by Lemma \ref{lem: form change 3 valent} the form changes by
$$
2\left(\frac{du_1du_3}{u_1u_3}+\frac{du_3du_2}{u_2u_3}-\frac{du_1du_2}{u_1u_2}\right)
$$
In both cases, this agrees with the definition of local intersection index up to a factor of 2.

\noindent It is easy to see that pushing a unipotent matrix to the right as in Lemma \ref{lem: slide weave} does not change the form. 
At the bottom of the weave, we are left with scalar permutation matrices and diagonal matrices $\chi_i(u)$.
By moving $\chi_i(u)$ to the left, we transform them to $\rho_i^{\vee}(u)$, and the form
$$
\left(\rho_1^{\vee}(u_1)|\cdots |\rho_{\ell}^{\vee}(u_\ell)\right),\quad \ell=\ell(\delta(\beta))
$$
agrees with the (skew-symmetrized) boundary intersection form as in Definition \ref{def: boundary intersection}.

In the non simply laced case, one needs to compute the form for $(2d_{ij})$-valent vertices. This follows from the simply laced case by folding, see Section \ref{sec: folding}. 
\end{proof}

\begin{remark}
\label{rem: roots same length}
In the above proof, we pull back the form from $\G\times \G$ to $\SL_2\times \SL_2$ and $\SL_3\times \SL_3$ using the pinning. The pullbacks of the left- and right-invariant $\mathfrak{g}$-valued forms agree with those for $\SL_2$ and $\SL_3$, but the Killing forms might differ by a factor. If $\G$ is simply laced then all simple roots have the same length and all the factors agree. Otherwise, one needs to scale the local intersection forms at trivalent vertices by the length of the corresponding simple root.
\end{remark}


\section{Comparison of cluster structures on Richardson varieties}\label{sec: Richardson}
The open Richardson variety is defined as the intersection of opposite Schubert cells $\mathcal{R}(v, w) := \mathcal{S}^{-}_{v} \cap \mathcal{S}_{w}$,
for $v \leq w$ in the Bruhat order, cf.~Subsection \ref{ssec:openRichardson}. For $\G$ simply-laced, B.~
Leclerc \cite{Leclerc} proposed a cluster structure for $\mathcal{R}(v, w)$  using additive categorification. This cluster structure is difficult to write down explicitly and, following an idea of J.~Schr\"oer, E.~M\'enard modified Leclerc's proposal in \cite{Menard} to give a more explicit construction of a seed for $\mathcal{R}(v, w)$. In this section, we show that M\'enard's cluster structure coincides with ours. As a consequence, the upper cluster algebra and cluster algebra constructed by M\'enard coincide with the ring of regular functions on the Richardson variety. Note that Leclerc and M\'enard consider strata in $\borel_-\backslash \G,$ while we  work with strata in $\G/\borel_+$. A detailed comparison between these versions of Richardson varieties can be found in \cite{GLtwist}, we will implicitly use the isomorphisms discussed there. In particular, we use that $\mathcal{R}(v, w) \cong \mathcal{R}(v^{-1}, w^{-1})$.\\

\noindent For open Richardson varieties, the cluster structure we obtain in Theorem \ref{thm:main} can be constructed by choosing reduced words for $w$ and $v^c :=v^{-1}w_0 = w_0(v^{-1})^*$, considering the right-to-left inductive weave for the braid variety $X(\beta(w)\beta(v^c))$ and applying the construction of cluster variables from Sections \ref{sec: cluster variables} and \ref{sec: non simply laced}. Since Subsection \ref{ssec:openRichardson} shows that $X(\beta(w)\beta(v^c))\cong\mathcal{R}(v, w)$ are isomorphic, it makes sense to compare these two (upper) cluster structures, that from Theorem \ref{thm:main} and that from \cite{Menard}. The following is the main result in this section:

\begin{theorem}\label{thm:Menard} Suppose $\G$ is simply-laced. The cluster structure on $\mathcal{R}(v, w)$ constructed by E.~M\'enard \cite{Menard} coincides with the cluster structure associated with the left inductive weave for $X(\beta(w)\beta(v^c))$, after an identification of strata in $\borel_-\backslash \G$ with strata in $\G/\borel_+$. In particular, it equals its upper cluster algebra.
\end{theorem}

Note that an advantage of the construction of the cluster structures in Theorem \ref{thm:main} is that we can write down the cluster variables explicitly as regular functions on the coordinate ring $\C[\mathcal{R}(v,w)]$. Now, E.~M\'enard's construction begins with a cluster structure on the unipotent cell $\mathcal{U}^{w} \cong \mathcal{R}(e,w) \cong \mathcal{R}(e, w^{-1})$, performs a sequence of mutations, and then removes some vertices. The proof of Theorem \ref{thm:Menard} is achieved by first interpreting his construction in terms of weaves. In fact, M\'enard's construction can be rephrased in terms of double-inductive weaves, as introduced in Subsection \ref{sec: double inductive}, as follows:
\begin{enumerate}
\item Start with a reduced word $\overline{w}$ for $w$ and choose the rightmost representative of $v$ as a subword of $\overline{w}$. This rightmost representative gives a reduced expression $\overline{v}$ for $v$ and we consider a reduced expression $\overline{v^c}$ for $v^c$. Then we have that $\overline{v^c}\overline{v}^*$ is a reduced expression for $w_0$. 
\item Consider the left inductive weave $\mathfrak{w}_1:=\lind{\beta(\overline{w})\beta(\overline{v^c}\overline{v}^*)}$. It defines a cluster seed for the braid variety $X(\beta(\overline{w})\beta(\overline{v^c}\overline{v}^*)) = X(\beta(\overline{w})\Delta).$
\item Via the twist automorphism, this seed is sent to the cluster seed for the braid variety $X(\Delta \overline{w}^*) \cong \Conf(w^*)$ given by the right inductive weave for the word $\overline{v} \overline{v^c}^* \overline {w}^*$. The latter is the initial seed for the cluster structure defined in \cite{SW}, see Section \ref{sec:double bs}. 
\item The variety $\Conf(w^*)$ is isomorphic to the unipotent cell $\mathcal{U}^{w}$, and by the work of Weng \cite{weng2016}, this seed agrees with the image under the twist of the initial cluster seed of the cluster structure defined in \cite{BFZ}, up to $n$ frozen variables. Since the twist map is an automorphism, the seed defined by the weave $\mathfrak{w}_1$ agrees with the initial seed of the cluster structure\footnote{The cluster structures in \cite{BFZ} are defined on double Bruhat cells. Explicit isomorphisms between certain reduced double Bruhat cells, including the unipotent cells, and suitable Richardson varieties can be found in \cite{BGY, GLtwist,  Leclerc}, see also \cite{WebsterYakimov}.} in \cite{BFZ}, up to frozens.
\item As proved in \cite{buan2009cluster, GLSkm}, this seed agrees with the one defined as the image under the cluster character map of the cluster-tilting object $V_{\overline{w}}.$ This is precisely the initial seed of the cluster structure on the unipotent cell $\mathcal{U}^{w}$ that M\'enard begins with. 
\item We then perform a sequence of mutations to go from the left inductive weave $\mathfrak{w}_1$ to another weave $\mathfrak{w}_2$. The weave $\mathfrak{w}_2$ comes from $\lind{\beta(\overline{w})\beta(\overline{v^c})}$ by adding letters of $\beta(\overline{v}^*)$ on the right, which yields a double-inductive weave. \item Then the deletion of vertices in M\'enard's quiver corresponds to removing the $\beta(\overline{v}^*)$ on the right. The deleted vertices correspond exactly to the cluster variables coming from the trivalent vertices that come from adding $\beta(\overline{v}^*)$ on the right. Note that because $\delta(\beta(\overline{w})\beta(\overline{v^c}))=w_0$, there is a cluster variable removed for every letter in the reduced word for $\overline{v}^*$.
\end{enumerate}


\subsection{Comparison of mutation sequences}

Let us start comparing our cluster structure with the construction of M\'enard, where we use the double inductive weaves of Section \ref{sec: double inductive}. Start with the left inductive weave $\mathfrak{w}_1:=\lind{\beta(\overline{w})\beta(\overline{v^c}\overline{v}^*)}$ and write 
$$\overline{w}=s_{i_l} s_{i_{l-1}} \cdots s_{i_2} s_{i_1},$$
$$\overline{v^c}=s_{j_m} s_{j_{m-1}} \cdots s_{j_2} s_{j_1}.$$
Let $\overline{v}=s_{k_n} s_{k_{n-1}} \cdots s_{k_2} s_{k_1}$ be the rightmost representative of $v$ as a subword of $w$, so that we have
$$\overline{v^*}=s_{k_n^*} s_{k_{n-1}^*} \cdots s_{k_2^*} s_{k_1^*}.$$

\noindent Let $1 \leq x_1 < x_2 < \cdots < x_n \leq l$ be the indices of the rightmost representative of $v$ as a subword of $w$. In other words, the $x_i$ are minimal such that $s_{i_{x_n}} s_{i_{x_{n-1}}} \cdots s_{i_{x_2}} s_{i_{x_1}}=v$. Thus we have that $i_{x_m}=k_m$. The weave $\mathfrak{w}_1$ is associated to the double string
$$(k_1^*, k_2^*L, k_3^*L, \dots, k_n^*L, j_1L, \dots, j_mL, i_1L, \dots, i_lL).$$
We wish to relate this to the weave associated to the double string 
$$(j_1L, \dots, j_mL, i_1L, \dots, i_lL, k_n^*R, \dots, k_1^*R).$$
By moving the $k$'s across one at a time we obtain a sequence of double strings
$$(k_1^*, k_2^*L, k_3^*L, \dots, k_n^*L, j_1L, \dots, j_mL, i_1L, \dots, i_lL),$$
$$(k_2^*, k_3^*L, \dots, k_n^*L, j_1L, \dots, j_mL, i_1L, \dots, i_lL, k_1^*R),$$
$$(k_3^*, \dots, k_n^*L, j_1L, \dots, j_mL, i_1L, \dots, i_lL, k_2^*R, k_1^*R),$$
$$\dots$$
$$(j_1L, \dots, j_mL, i_1L, \dots, i_lL, k_n^*R, \dots, k_1^*R).$$

\noindent This involves a sequence of cluster mutations which are now the object of our study. The {\it cases} that we will be referring to are those in the proof of Theorem \ref{thm: double inductive}. We first collect two simple lemmas:

\begin{lemma} For $1 \leq a \leq l$, let $u_a$ be the Demazure product $s_{i_a}*s_{i_{a-1}}* \cdots *s_{i_2}*s_{i_1}*v^c$. Then $\ell(u_a) > \ell(u_{a-1})$ if and only if $s_{i_a}$ is part of the rightmost representative of $v$.
\end{lemma}

\noindent This straightforward statement that can be directly checked, a proof can be found in \cite{Menard}. Replacing $v$ by $s_{k_b} \cdots s_{k_1}$, we get:

\begin{lemma}\label{lengthchange} For $1 \leq m \leq l$, let $u_{a,b}$ be the Demazure product $s_{i_a}*s_{i_{a-1}}* \cdots *s_{i_2}*s_{i_1}*v^c*s_{k_n}* \cdots * s_{k_{b+1}}$. Then, $\ell(u_{a,b}) > \ell(u_{a-1,b})$ if and only if $a$ is one of $x_1, x_2, \dots x_b$. \\
\noindent In particular, we have that $u_{a,b}=w_0$ for $a \geq x_b$.

\end{lemma}

\subsubsection{Moving $k_1^*R$ in the double string} Let us first analyze what happens as we move the entry $k_1^*R$ to the right in the double string. To begin with, the superscripts are placed as follows:
$$(k_1^*R^+, k_2^*L^+, k_3^*L^+, \dots, k_n^*L^+, j_1L^+, \dots, j_mL^+, i_1L, \dots, i_lL).$$

\noindent Thus the $k$'s and $j$'s have ``+'' superscripts, while the $i$' have none. This means that using Case 1 (from the proof of Theorem \ref{thm: double inductive}) we can move $k_1^*R$ across all the $k$'s and $j$'s without any mutations to get

$$(k_2^*L^+, k_3^*L^+, \dots, k_n^*L^+, j_1L^+, \dots, j_mL^+, k_1^*R^+, i_1L, \dots, i_lL).$$

\noindent In moving $k_1^*R^+$ further to the right in the double string, we can move $k_1^*R^+$ across $i_aL$ using Case 3 as long as the length of the Demazure product 
$$\ell(s_{i_a}*s_{i_{a-1}}* \cdots *s_{i_2}*s_{i_1}*v^c*s_{k_n}* \cdots * s_{k_{2}})$$
does not increase. Thus by Lemma \ref{lengthchange} there are no mutations until we hit $i_{x_1}L$:

$$(k_2^*L^+, k_3^*L^+, \dots, k_n^*L^+, j_1L^+, \dots, j_mL^+, \dots k_1^*R^+, i_{x_1}L, \dots).$$

\noindent Lemma \ref{lengthchange} also yields $u_{x_1,1}=s_{x_1}u_{x_1-1,1}=w_0=u_{x_1-1,1}s_{k_1^*}$, while $\ell(u_{x_1-1,1}) < \ell(w_0)$. Therefore, moving $k_1^*R$ across $i_{x_1}L$ involves Case 2:

$$(k_2^*L^+, k_3^*L^+, \dots, k_n^*L^+, j_1L^+, \dots, j_mL^+, \dots i_{x_1}L^+, k_1^*R, \dots).$$
At this point, $k_1^*R$ loses the ``+'' superscript. From this point forward, moving $k_1^*R$ across to the right only involves Cases 4 and 5. Because $u_{x_1,1}=w_0$, the Demazure product after this point will always be $w_0$. Therefore, we will have mutations precisely when $k_1^*R$ crosses a strand $i_aL$ with $i_a=k_1$ using the specialization of Case 5.

\subsubsection{Moving $k_2^*R$ in the double string} Let us analyze one more case before going to the general case. We want to understand what happens as we move the entry $k_2^*R$ to the right in the double string. To begin with, the superscripts are placed as follows:
$$(k_2^*R^+, k_3^*L^+, \dots, k_n^*L^+, j_1L^+, \dots, j_mL^+, \dots, i_{x_1}L^+, \dots, i_lL, k_1^*R).$$

\noindent Again, we can use Case 1 to move $k_2^*R$ across all the $k$'s and $j$'s without any mutations to get

$$(k_3^*L^+, \dots, k_n^*L^+, j_1L^+, \dots, j_mL^+, k_2^*R^+, \dots, i_{x_1}L^+, \dots, i_lL,  k_1^*R).$$

\noindent In moving $k_2^*R^+$ further to the right in the double string, we can move $k_2^*R^+$ across $i_aL$ using Case 3 as long as the length of the Demazure product 
$$\ell(s_{i_a}*s_{i_{a-1}}* \cdots *s_{i_2}*s_{i_1}*v^c*s_{k_n}* \cdots * s_{k_{3}})$$
does not increase and using Case 1 to move across $i_{x_1}L^+$. Thus by Lemma \ref{lengthchange} there are no mutations until we hit $i_{x_2}L$:

$$(k_3^*L^+, \dots, k_n^*L^+, j_1L^+, \dots, j_mL^+, \dots, i_{x_1}L^+, \dots, k_2^*R^+, i_{x_2}L, \dots, i_lL, k_1^*R).$$

\noindent Then again using Lemma \ref{lengthchange}, we see that $u_{x_2,2}=s_{x_2}u_{x_2-1,2}=w_0=u_{x_2-1,2}s_{k_2^*}$, while $\ell(u_{x_2-1,2}) < \ell(w_0)$. Therefore, moving $k_2^*R$ across $i_{x_2}L$ involves Case 2:

$$(k_3^*L^+, \dots, k_n^*L^+, j_1L^+, \dots, j_mL^+, \dots, i_{x_1}L^+, \dots, i_{x_2}L^+, k_2^*R, \dots, i_lL, k_1^*L).$$
At this point, $k_2^*R$ loses the ``+'' superscript. As in the previous discussion, from this point forward, moving $k_2^*R$ across to the right only involves Cases 4 and 5 and, because $u_{x_2,2}=w_0$, the Demazure product after this point will always be $w_0$. Thus we have mutations precisely when $k_2^*R$ crosses a strand $i_aL$ with $i_a=k_2$, using the specialization of Case 5.

\subsubsection{Moving a general term $k_b^*R$ in the double string} The argument continues similarly as the two discussions above. We begin with
$$(k_b^*R^+, k_{b+1}^*L^+, \dots, k_n^*L^+, j_1L^+, \dots, j_mL^+, \dots, i_{x_1}L^+, \dots, i_{x_{b-1}}L^+, \dots, i_lL, k_{b-1}^*R, \dots, k_1^*R).$$

\noindent Again, we can use Case 1 (in the proof of Theorem \ref{thm: double inductive}) to move $k_b^*R$ across all the $k$'s and $j$'s without any mutations. This yields the double string

$$(k_{b+1}^*L^+, \dots, k_n^*L^+, j_1L^+, \dots, j_mL^+, k_b^*R^+, \dots, i_{x_1}L^+, \dots, i_{x_{b-1}}L^+, \dots, i_lL, k_{b-1}^*R, \dots, k_1^*R).$$

\noindent In moving $k_b^*R^+$ further to the right in the double string, we can move $k_b^*R^+$ across $i_aL$ using Case 3 if the length of the Demazure product 
$$\ell(s_{i_a}*s_{i_{a-1}}* \cdots *s_{i_2}*s_{i_1}*v^c*s_{k_n}* \cdots * s_{k_{b+1}})$$
is not increasing, and using Case 1 to move across $i_{x_c}L^+$ for $c<b$. Lemma \ref{lengthchange} shows that there are no mutations until we hit $i_{x_b}L$:

$$(k_{b+1}^*L^+, \dots, k_n^*L^+, j_1L^+, \dots, j_mL^+, \dots, i_{x_1}L^+, \dots, i_{x_{b-1}}L^+, \dots, k_b^*R^+, i_{x_{b}}L, \dots, i_lL, k_{b-1}^*R, \dots, k_1^*R).$$

\noindent Lemma \ref{lengthchange} again shows that $u_{x_b,b}=s_{x_b}u_{x_b-1,b}=w_0=u_{x_b-1,b}s_{k_b^*}$, while $\ell(u_{x_b-1,b}) < \ell(w_0)$. Therefore, moving $k_2^*R$ across $i_{x_2}L$ involves Case 2, and we obtain:

$$(k_{b+1}^*L^+, \dots, k_n^*L^+, j_1L^+, \dots, j_mL^+, \dots, i_{x_1}L^+, \dots, i_{x_{b-1}}L^+, i_{x_{b}}L^+, \dots, k_b^*R, \dots, i_lL, k_{b-1}^*R, \dots, k_1^*R).$$

\noindent As above, $k_b^*R$ then loses the ``+'' superscript and continuing to move $k_b^*R$ across to the right only involves Cases 4 and 5. Since we have $u_{x_b,b}=w_0$ as before, the Demazure product after this point is $w_0$ and we have mutations precisely when $k_b^*R$ crosses a strand $i_aL$ with $i_a=k_b$.


\subsection{The mutation sequence and proof of Theorem \ref{thm:Menard}}
To summarize, when we move $k_b^*R$ across, we get no mutations until we reach $i_{x_b}$. At this point we have

$$(\dots, k_b^*R^+, i_{x_b}L, \dots) \longrightarrow (\dots, i_{x_b}L^+, k_b^*R, \dots),$$
which involves no mutation. Then moving $k_b^*R$ across the remaining $i_aL$ involves mutation only when we cross $i_a$ with the color $k_b$. Let us describe what this means in terms of quivers.\\

\noindent The quiver for the initial seed, which is attached to the weave for the double string 
$$(k_1^*R^+, k_2^*L^+, k_3^*L^+, \dots, k_n^*L^+, j_1L^+, \dots, j_mL^+, i_1L, \dots, i_lL).$$
has one cluster variable for each $i_a$ in the reduced word for $w$. One can associate each node of the Dynkin diagram with a color, and therefore we can color each of the vertices in the quiver: the vertex associated with $i_a$ will have the color $i_a$.

\noindent Let us fix a color and consider all the $i_a$ of that color. Let the indices be $a_1, a_2, \dots, a_N$. Now some subset of these $a_{b_1}, \dots, a_{b_M}$ belong to the rightmost representative of $v$ in $w$. Let us suppose that $k_{c_1}^*, \dots, k_{c_M}^*$ are the corresponding letters in $v^*$. Initially the vertices of our fixed color are labelled
$$i_{a_1}L, i_{a_2}L, \dots, i_{a_N}L.$$

\noindent We move $k_{c_1}^*R$ across $i_{a_{b_1}}L$ and our vertices are labelled
$$i_{a_1}L, i_{a_2}L, \dots, \widehat{i_{a_{b_1}}L}, k_{c_1}^*R, \dots, i_{a_N}L,$$
where the hat symbol means we skip that entry. We then mutate vertices $b_1$ up to $N-1$ to move $k_{c_1}^*R$ to the end:
$$i_{a_1}L, i_{a_2}L, \dots, \widehat{i_{a_{b_1}}L}, i_{a_{b_1}+1}L,\dots, i_{a_N}L, k_{c_1}^*R. $$

\noindent In the next step we move $k_{c_2}^*R$ across $i_{a_{b_2}}L$ and our vertices are labelled
$$i_{a_1}L, i_{a_2}L, \dots, \widehat{i_{a_{b_1}}L}, \dots, \widehat{i_{a_{b_2}}L}, k_{c_2}^*R, \dots, i_{a_N}L, k_{c_1}^*R.$$
Then $k_{c_2}^*R$ corresponds to the $b_2-1$-st entry, and we mutate vertices $b_2-1$ through $N-2$ to move it past $i_{a_N}L$ to end up with
$$i_{a_1}L, i_{a_2}L, \dots, \widehat{i_{a_{b_1}}L}, \dots, \widehat{i_{a_{b_2}}L}, i_{a_{b_2+1}}L, \dots, i_{a_N}L, k_{c_2}^*R, k_{c_1}^*R.$$

\noindent In general, the mutations come from moving $k_{c_d}^*R$ from 
$$i_{a_1}L, \dots, \widehat{i_{a_{b_1}}L}, \dots, \widehat{i_{a_{b_d}}L}, k_{c_d}^*R \dots, i_{a_N}L, k_{c_{d-1}}^*R, \dots, k_{c_1}^*R$$
to 
$$i_{a_1}L, \dots, \widehat{i_{a_{b_1}}L}, \dots, \widehat{i_{a_{b_d}}L}, \dots, i_{a_N}L, k_{c_d}^*R , k_{c_{d-1}}^*R, \dots, k_{c_1}^*R.$$
This involves mutating from vertices $b_d-(d-1)$ through $N-d$.\\

To summarize, when we move $k_{c_d}^*$ across the double string, we mutate only vertices of the color $k_{c_d}$. Moreover, we mutate a sequence of vertices of that color, starting at $b_d-(d-1)$ and ending at $N-d$, where the reflection $k_{c_d}^*$ is the $d$-th occurence of that color in the  representative of $v$ in $w$, and this letter in $v$ is the $b_d$-th occurence of that color in $w$. This is precisely the rule for mutation given by M\'enard.

\begin{remark} M\'enard's work \cite[Definitions 5.23 and 6.1]{Menard}  gives an explicit mutation sequence. The role played by $\gamma_m$ there is what we call $d$ above; the role played by $\beta_m$ is $b_d-d$ in our notation. Mutating from the ``$\beta_m$ from the first vertex'' of a color to the ``$\gamma_m$ from the last vertex'' means mutating from vertices $b_d-d+1$ to $N-d$.
\end{remark}

\noindent After this sequence of mutations, we end with the double string
$$(j_1L^+, \dots, j_mL^+, i_1L^{(+)}, \dots, i_lL^{(+)}, k_n^*R, \dots, k_1^*R).$$
In the above, we only have a ``+'' superscript on $i_aL$ when $i_a$ is part of the rightmost representative of $v$ in $w$. Note that all the $k$'s will correspond to cluster variables. Then, in M\'enard's algorithm, if the color $r$ occurs $X_r$ times in $v$, we delete the last $X_r$ vertices of that color. This corresponds exactly to deleting the vertices associated with $k_n^*R, \dots, k_1^*R$. Thus the mutation/deletion algorithm in \cite{Menard} leaves us with exactly the cluster structure associated to the double string
$$(j_1L, \dots, j_mL, i_1L, \dots, i_lL).$$
This is left inductive weave for $X(\beta(w)\beta(v^c))$, therefore giving our cluster structure on the Richardson variety $\mathcal{R}(v,w)$. This concludes the proof of Theorem \ref{thm:Menard}.


\section{Examples}
\label{sec: examples}

This section provides explicit examples of braid varieties, weaves and initial seeds for the cluster structures constructed in Theorem \ref{thm:main}. It contains three examples in Type A and one in Type B.

\begin{figure}[h!]
    \centering
    \includegraphics[scale=0.7]{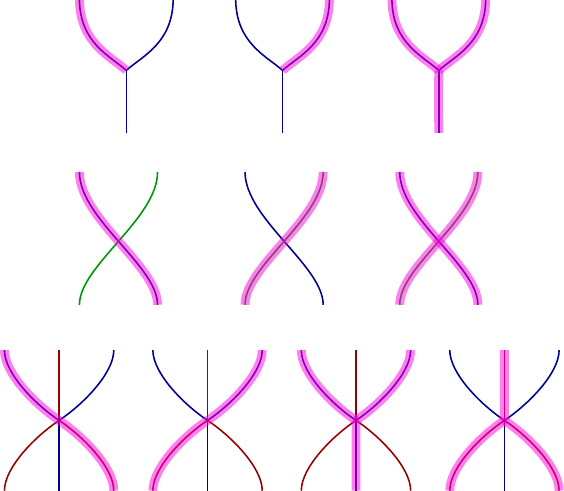}
    \caption{The propagation rules for Lusztig cycles whose weights are either $0$ or $1$. The edges with weight $1$ are colored in purple.}
    \label{fig:Lusztig-weight1-propagation-rules}
\end{figure}


\subsection{Lusztig cycles with weights 0 and 1}\label{ssec:first_example} Let us first focus on Lusztig cycles in weaves for the case that all weights of a Lusztig cycle are $0$ or $1$. This occurs already in many interesting examples, cf.~\cite[Section 2]{CW} and \cite[Section 7]{CZ}. In this case, instead of writing the numerical weights, we color the edges with weight $1$ in the Lusztig cycle and do {\it not} color the edges with weight $0$. Using this diagrammatic convention, the propagation rules for such Lusztig cycles are illustrated in Figure \ref{fig:Lusztig-weight1-propagation-rules}. The first row of Figure \ref{fig:Lusztig-weight1-propagation-rules} exhibits the cases near a trivalent vertex, the second row does so for a tetravalent vertex, and the third row presents the possibilities near a hexavalent vertex. We use these rules repeatedly in the examples in Sections \ref{sec:example-weight-2}, \ref{sec:braid-as-mutation}, \ref{sec:affine-example} and \ref{sec:nonsimplylaced-example} below.\\

\begin{figure}[h!]
    \centering
    \includegraphics{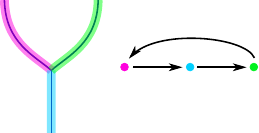}
    \caption{(Left) A trivalent vertex with three Lusztig cycles, each with weights 0 and 1, depicted in purple, green and blue respectively. (Right) The corresponding intersection quiver $Q_{\mathfrak{W}}$. The different colors represent distinct Lusztig cycles and each corresponds to the vertex in the quiver of the same color. The arrows in the quiver capture the intersections between these cycles.}
    \label{fig:arrows-trivalent-vertices}
\end{figure}

\begin{figure}[h!]
    \centering
    \includegraphics{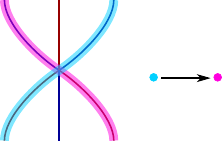}
    \caption{(Left) A hexavalent vertex with two Lusztig cycles, each with weights 0 and 1, depicted in purple and blue. (Right) The corresponding intersection quiver $Q_{\mathfrak{W}}$, the arrows in the quiver capture the intersections between these two cycles. As stated in the text, any other local intersections near a hexavalent vertex can be reduced to this case.}
    \label{fig:arrows-hexavalent-vertices}
\end{figure}

\noindent The intersections of Lusztig cycles with weights 0 and 1 are relatively simple. The only cases with non-zero local intersections occur near trivalent and hexavalent vertices. The determinant formulas in Definitions \ref{def:locint_3valent} and \ref{def:locint_6valent} lead to the intersection quivers in Figures \ref{fig:arrows-trivalent-vertices} and Figure \ref{fig:arrows-hexavalent-vertices}. The rules for intersections near a trivalent vertex are described in Figure \ref{fig:arrows-trivalent-vertices}, while an instance of a local intersection near a hexavalent vertex is depicted in Figure \ref{fig:arrows-hexavalent-vertices}. Note that the computation of the intersections at a hexavalent vertex can be reduced to Figure \ref{fig:arrows-hexavalent-vertices} by Lemma \ref{lem: add 101} (cf.~ proof of Lemma \ref{lem: boundary int}).


\subsection{A complete example}\label{sec:example-weight-2} Let us consider $\G=\SL_3$, the braid $\beta=11221122$ and the following three Demazure weaves for it. The first weave is depicted in Figure \ref{fig:weave_ex1}, where we marked the nonzero weights of the Lusztig cycle $\gamma_v$ associated to the topmost trivalent vertex $v$. Note that one of the edges, marked in yellow, has weight 2. The other two weaves are the left inductive and the right inductive weaves for $\beta$. 

\begin{figure}[h!]
\centering
\includegraphics[scale=1.1]{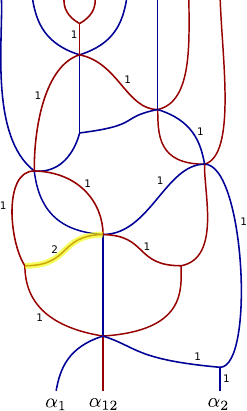}
\caption{A Demazure weave for $\beta=11221122$ and the Lusztig cycle $\gamma_v$ for the topmost trivalent vertex. At the bottom, we also include the sequence of roots $\rho_k$, where we denote $\alpha_{ij} := \alpha_{i} + \cdots + \alpha_{j}$ for $i < j$.}\label{fig:weave_ex1}
\end{figure}
\noindent For the Demazure weave in Figure \ref{fig:weave_ex1}, the mutable part of the quiver has type $A_2$ and there are three frozen variables, as follows:
\begin{center}
\begin{tikzcd}
A_1 \arrow{r}& A_2 \arrow{dl}\arrow{dr}& \\
{\color{blue} z_4} \arrow{u}& {\color{blue} F} \arrow{u}  \arrow[dashleftarrow]{r}& {\color{blue} z_6} \arrow[dashrightarrow, bend left]{ll}
\end{tikzcd}
\end{center}
\noindent where the dashed arrows have weight $1/2$ and the solid arrows have weight $1$. A direct computation yields the following three frozen variables
$$
z_4,~z_6,~\ \mathrm{ and }\ F:=-z_2z_5z_6z_8 + z_2z_4z_7z_8 - z_2z_4 + z_2z_8 - z_6z_8.
$$
The two mutable cluster variables are
$$
A_1:=-z_5z_6 + z_4z_7 + 1,\ A_2:=-z_5z_6z_8 + z_4z_7z_8 - z_4 + z_8.
$$
For the right inductive weave, the frozen variables are the same and the mutable cluster variables are 
$$
z_2,\quad A_3:=-z_2z_5z_6 + z_2z_4z_7 + z_2 - z_6.
$$
\noindent For the left inductive weave, the frozen variables are also the same and the cluster variables are $z_8$ and $A_2$. Note that
$$
A_1=\frac{A_3+z_6}{z_2},\ 
A_2=\frac{FA_1+z_4z_6}{A_3},\ 
z_8=\frac{A_2+z_4}{A_1},\
z_2=\frac{F+z_6z_8}{A_2},\
z_8=\frac{F+z_2z_4}{A_3}.
$$
In particular, we have a cycle of mutations
$$
(A_1,A_2)-(A_1,A_3)-(z_2,A_3)-(z_2,z_8)-(z_8,A_2)-(A_1,A_2).
$$


\subsection{Braid relation as a mutation}\label{sec:braid-as-mutation}

Let $\G=\SL_3$ and consider the braid word $\beta=112211121$ (compare with Figure \ref{fig:mutation 1}). The cluster variables for the right inductive weave $\rind{\beta}$ are  
$$
A_1=z_2,\ A_2=z_4,\ A_3=z_6,\ A_4=z_6z_7 - 1,
$$
$$
A_5=-z_2z_5z_6z_7 + z_2z_4z_8 + z_2z_5 + z_2z_7 - z_6z_7 + 1,\ A_6= z_4z_6z_7z_9 - z_4z_6z_8 - z_4z_9 - 1,
$$
and the quiver $Q_{\rind{\beta}}$ is
\begin{center}
\begin{tikzcd}
 & A_3 \arrow{d}& A_2 \arrow{d}\\
{\color{blue} A_5}\arrow{r} \arrow[dashrightarrow, bend left]{rr} & A_4\arrow{d}\arrow{r} & {\color{blue} A_6}\\
 & A_1 \arrow{ul} & 
\end{tikzcd}
\end{center}
Next, consider the six-valent vertex
$112211\underline{121}\to 112211212$
followed by the right inductive weave above (compare with Figure \ref{fig:mutation 2}). The cluster variables are the same as above except for 
$$\widetilde{A_4}:=-z_2z_5z_6 + z_2z_4z_9 + z_2 - z_6,$$
and the new quiver reads
\begin{center}
\begin{tikzcd}
 & A_3 \arrow[bend right]{dd} \arrow{dr} & A_2 \arrow{d}\\
{\color{blue} A_5} \arrow[dashleftarrow, bend right]{rr}  & \widetilde{A_4}\arrow{u}\arrow{l} & {\color{blue} A_6}   \arrow{l}\\
 & A_1 \arrow{u} & 
\end{tikzcd}
\end{center}
The quivers are related by a mutation at $A_4$ and indeed we also have the mutation identity 
$$
\widetilde{A_4}=\frac{A_3A_5+A_1A_6}{A_4}.
$$


\subsection{An example with an affine type cluster algebra}\label{sec:affine-example} Let $\G=\SL_4$ and consider the braid word $\beta=213223122132$. The corresponding right inductive weave $\fW$ and Lusztig cycles are shown in Figure \ref{fig:large example}. The corresponding (Legendrian) link is $\Lambda(\beta_{11})$ in \cite[Section 1.2]{CN}, where it is also referred to as $\Lambda(\tilde{A}_{2,1})$.
\begin{figure}[ht!]
\centering
    \includegraphics[scale=1.3]{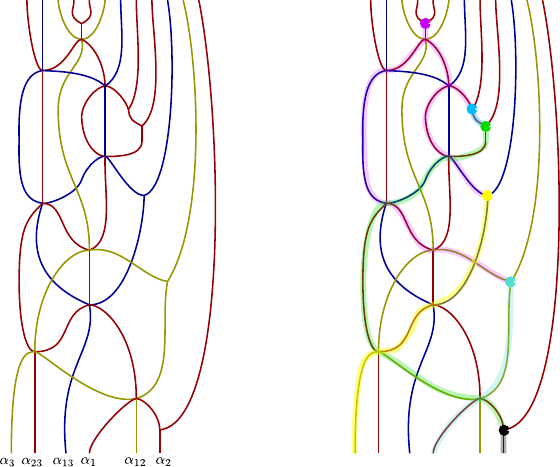}
    \caption{The inductive weave for $\beta = \sigma_2\sigma_1\sigma_3\sigma_2\sigma_2\sigma_3\sigma_1\sigma_2\sigma_2\sigma_1\sigma_3\sigma_2$ on the left, and the weave with its distinguished cycles on the right. Each cycle only takes weights $0$ or $1$, and we color the edges were the cycle takes weight $1$.}
    \label{fig:large example}
\end{figure}
A direct computation yields the following cluster variables, ordered from top to bottom:
$$
A_1=z_5, A_2=-z_6z_7 + z_5z_8, A_3=-z_6z_7z_9+z_5z_8z_9-z_5, 
A_4=-z_6z_9 + z_5z_{10},
$$
$$
A_5=-z_7z_9 + z_5z_{11}, A_6= z_6z_7z_{10}z_{11} - z_5z_8z_{10}z_{11} - z_6z_7z_9z_{12} + z_5z_8z_9z_{12} - z_8z_9 + z_7z_{10} + z_6z_{11} - z_5z_{12} + 1.
$$
The variables $A_1,A_2,A_3$ are mutable and $A_4,A_5,A_6$ are frozen.  The quiver $Q_{\mathfrak{W}}$ read from $\fW$ is
\begin{center}
    \begin{tikzcd}
    A_1 \arrow{rr} \arrow{dr} \arrow{drrr} & & A_2 \arrow {rr} & & A_3 \arrow[out=160, in=20, "2"]{llll} \arrow{dr}  &  \\ & {\color{blue} A_4} \arrow{urrr} \arrow[out=-20, in=-160, dashrightarrow]{rrrr} & & {\color{blue} A_5} \arrow{ur} & & {\color{blue} A_6} \arrow[dashleftarrow]{ll}.
    \end{tikzcd}
\end{center}

\noindent The mutable part of $Q_{\mathfrak{W}}$ is a quiver of affine type $A$. One can verify directly that mutating at all mutable vertices creates regular functions:
\[
\frac{A_1 + A_3}{A_2} = z_9, \qquad \frac{A_2A_4A_5 + A_1^2A_6}{A_3} = z_6z_7z_9 - z_5z_7z_{10} - z_5z_6z_{11} + z_5^2z_{12} - z_5,
\]
\[
\begin{array}{rl}
\displaystyle{\frac{A_2A_4A_5 + A_3^2}{A_1}} = & -z_6z_7z_8z_9^2 + z_5z_8^2z_9^2 + z_6z_7^2z_9z_{10} - z_5z_7z_8z_9z_{10} +   z_6^2z_7z_9z_{11} + \\ &  - z_5z_6z_8z_9z_{11} - z_5z_6z_7z_{10}z_{11} + z_5^2z_8z_{10}z_{11} + 2z_6z_7z_9 - 2z_5z_8z_9 + z_5
\end{array}
\]


\begin{figure}[ht!]
\centering
    \includegraphics[scale=1.1]{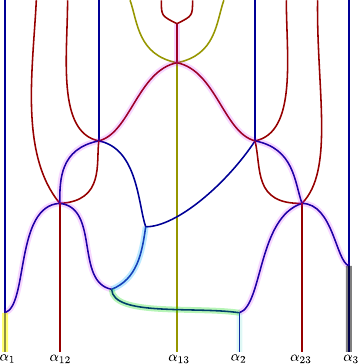}
    \caption{The weave $\fW'$ for $\sigma_1\sigma_2\sigma_2\sigma_1\sigma_3\sigma_2\sigma_2\sigma_3\sigma_1\sigma_2\sigma_2\sigma_1$.}
    \label{fig:roger_weave_cycles}
\end{figure}

\noindent Finally, let us now apply the cyclic rotation $\sigma_2\sigma_1\sigma_3\sigma_2\sigma_2\sigma_3\sigma_1\sigma_2\sigma_2\sigma_1\sigma_3\sigma_2 \mapsto \sigma_1\sigma_2\sigma_2\sigma_1\sigma_3\sigma_2\sigma_2\sigma_3\sigma_1\sigma_2\sigma_2\sigma_1$. A weave $\fW'$ for the latter word is given in Figure \ref{fig:roger_weave_cycles}.

\noindent  The cluster variables for $\fW'$ are the following regular functions, note that they are polynomials in $(z_i)$:
$$
B_1=z_7,\  B_2=-z_8z_9 + z_7z_{10},\ B_3=-z_4z_7 + z_3z_8,\ B_4=-z_3z_8z_9 + z_3z_7z_{10}-z_7,
$$
$$
B_5=-z_3z_8z_9z_{11} + z_3z_7z_{10}z_{11}-z_3z_7 - z_7z_{11},\ B_6=-z_8z_{11} + z_7z_{12}.
$$

\noindent Here $B_3,B_5,B_6$ are frozen and $B_1,B_2,B_4$ are mutable. The quiver has the form:
\begin{center}
    \begin{tikzcd}
     B_1 \arrow[out=20, in=160, "2"]{rrrr}{2} \arrow{drrrrr} & & B_2 \arrow{ll} & & B_4 \arrow{ll} \arrow{dl}& \\
      & {\color{blue} B_3} \arrow[dashrightarrow]{rr} \arrow{ul} & & {\color{blue} B_5}  \arrow[dashrightarrow]{rr} \arrow{ulll}& & {\color{blue} B_6} \\
    \end{tikzcd}
\end{center}
The cyclic rotation $z_i\to z_{i-2}$ sends $B_1$ to $A_1$, $B_2$ to $A_2$ and $B_6$ to $A_4$ . 


\subsection{An example in non-simply laced type}\label{sec:nonsimplylaced-example} Let $\G=B_2$ and consider the word 
$$\beta=\si_1\si_1\si_2\si_2\si_1\si_2\si_2\si_1\si_2 \in W(B_2),$$
where $W(B_2)$ is the Weyl group of type $B_2$. In Figure \ref{fig: B2 weave} we draw its right inductive weave, as well as the unfolding of this weave to $A_3$. Note that the quiver for the $A_3$ weave is given by:
\begin{center}
    \begin{tikzcd}
   & 2 \arrow{dr} & & 3 \arrow{ll} \arrow{dl} \arrow{dll} \arrow{dlll} &   \\
  1 \arrow[bend right]{rrrr}& 1' \arrow[bend right]{rr}& {\color{blue} 5} & {\color{blue} 4} \arrow{u} \arrow[dashrightarrow]{l} & {\color{blue} 4'} \arrow{ul} \arrow[dashrightarrow, out=-160, in=-20]{ll}
    \end{tikzcd}
\end{center}

\noindent This quiver has a $\mathbb{Z}_{2}$-symmetry $1 \leftrightarrow 1'$, $4 \leftrightarrow 4'$, and the exchange matrix for the $B_2$-weave is obtained from the quiver via folding. 

\begin{figure}[ht!]
\centering
    \includegraphics[scale=1.3]{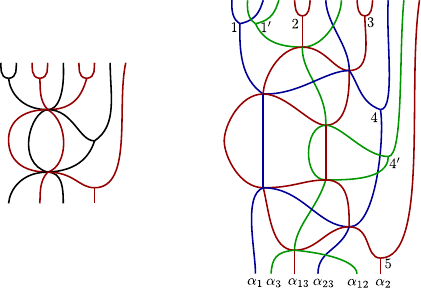}
    \caption{(Left) The right inductive weave for $\si_1\si_1\si_2\si_2\si_1\si_2\si_2\si_1\si_2 \in W(B_2)$. (Right) Its unfolding to $A_3$.}
    \label{fig: B2 weave}
\end{figure}

\noindent The symmetry acts on $z$-variables by swapping $z_1\leftrightarrow z_2$, $z_3 \leftrightarrow z_4$, $z_7 \leftrightarrow z_8$, $z_{11}\leftrightarrow z_{12}$ and fixing $z_5,z_6,z_9,z_{10},z_{13}$, so we have an inclusion of braid varieties
$$
X_{B_2}(\beta)\subset X_{A_3}(\si_1\si_3\si_1\si_3\si_2\si_2\si_1\si_3\si_2\si_2\si_1\si_3\si_2)
$$
where the left hand side is cut out by the equations
$$
z_1=z_2,\ z_3=z_4,\ z_7=z_8,\ z_{11}=z_{12}.
$$
The $A_3$ cluster variables are
$$
A'_1=z_3,\ A'_{1'}=z_4,\  A'_{2}=z_6, A'_3=z_{10}, 
$$
$$
A'_{4}=-z_4z_8z_{10} + z_4z_6z_{11} - z_{10}, A'_{4'}=-z_3z_7z_{10} + z_3z_6z_{12} - z_{10}, A'_{5}=-z_6z_{11}z_{12} + z_{6}z_{10}z_{13} - z_{10}.
$$
Restricting these to the $B_2$ braid variety yields
$$
A'_{1}|_{X_{B_2}(\beta)}=A'_{1'}|_{X_{B_2}(\beta)},\ A'_{4}|_{X_{B_2}(\beta)}=A'_{4'}|_{X_{B_2}(\beta)},
$$
as expected. The cluster variables for the $B_{2}$-braid variety are $A_1 = A'_{1}|_{X_{B_{2}}(\beta)} = A'_{1'}|_{X_{B_{2}}(\beta)}$, $A_2 = A'_{2}|_{X_{B_{2}}(\beta)}$, $A_3 = A'_{3}|_{X_{B_{2}}(\beta)}$, $A_4 = A'_{4}|_{X_{B_{2}}(\beta)} = A'_{4'}|_{X_{B_{2}}(\beta)}$ and $A_{5} = A'_{5}|_{X_{B_{2}}(\beta)}$. The exchange matrix and antisymmetrizer for $X_{B_{2}}(\beta)$ are given by
\[
\varepsilon = \left(\begin{matrix} 0 & 0 & -1 & 1 & 0 \\ 0 & 0 & -1 & 0 & 1 \\ 2 & 1 & 0 & -2 & 1 \\ -1 & 0 & 1 & 0 & 1/2 \\ 0 & -1 & -1 & -1 & 0\end{matrix}\right), \qquad d = (2, 1, 1, 2, 1)
\]
so that $\varepsilon_{ij}d^{-1}_{j}$ is skew-symmetric. Mutating at $A_{3}$, for example, we obtain the function
\[
\frac{A_{1}^{2}A_{2}A_{5} + A_{4}^{2}}{A_{3}} = \frac{(A'_{1}A'_{1'}A'_{2}A'_{5} + A'_{4}A'_{4'})|_{X_{B_{2}}(\beta)}}{A'_{3}|_{X_{B_{2}}(\beta)}}.
\]
It is indeed a regular function because it is the restriction of a regular function on the larger braid variety.

\bibliographystyle{plain}
\bibliography{main_revision.bib}

\end{document}